\documentclass[a4paper, 11pt]{article}
\usepackage{tocloft}
\usepackage{comment}
\usepackage[margin = 1in]{geometry}
\usepackage{aliascnt}

\usepackage{nomencl}

\usepackage{fancyhdr} 
\tocloftpagestyle{fancy}
\usepackage{ucs}
\usepackage[T1]{fontenc}
\usepackage{amsmath}
\usepackage{listings}
\usepackage{amssymb} 
\usepackage{amsthm}
\usepackage{amsfonts}
\usepackage{graphicx}
\usepackage{microtype}
\usepackage{mathdots}
\usepackage{bbold}
\usepackage{mathtools}
\usepackage[shortlabels]{enumitem}
\usepackage{yhmath}
\usepackage{marginnote}
\usepackage{mathrsfs}

\usepackage{mathpazo}

\usepackage{tikz}
\usepackage{tikz-cd}

\definecolor{blue}{RGB}{255,255,255}
\definecolor{linkgreen}{RGB}{33,100,60}
\definecolor{chilligreen}{RGB}{33,100,60}
\definecolor{chillired}{RGB}{140,00,10} 
\definecolor{chillidarkred}{RGB}{120,5,15}
\newcommand{\change}[1]{\textcolor{black}{#1}}
\newcommand{\fix}[1]{\textcolor{black}{#1}}
\newcommand{\cut}[1]{\textcolor{black}{#1}}

\usepackage[colorlinks=true, citecolor = chilligreen, urlcolor = chillired, backref=page]{hyperref}

\usepackage[nameinlink,capitalise]{cleveref}
\renewcommand*{\backref}[1]{}
\renewcommand*{\backrefalt}[4]{%
	\ifcase #1%
	\or (Cited on page~#2.)%
	\else (Cited on pages~#2.)%
	\fi%
}

\theoremstyle{definition}

\newtheorem{definition}{Definition}[subsection]

\newaliascnt{question}{definition}
\newtheorem{question}[question]{Question}
\aliascntresetthe{question}

\crefname{question}{Question}{Questions}

\newaliascnt{construction}{definition}
\newtheorem{construction}[construction]{Construction}
\aliascntresetthe{construction}
\crefname{construction}{Construction}{Constructions}

\newaliascnt{observation}{definition}

\aliascntresetthe{observation}
\crefname{observation}{Observation}{Observations}

\newaliascnt{conjecture}{definition}

\aliascntresetthe{conjecture}
\crefname{conjecture}{Conjecture}{Conjectures}

\newaliascnt{lemma}{definition}
\newtheorem{lemma}[lemma]{Lemma}
\aliascntresetthe{lemma}
\crefname{lemma}{Lemma}{Lemmas}

\newaliascnt{fact}{definition}
\newtheorem{fact}[fact]{Fact}
\aliascntresetthe{fact}
\crefname{fact}{Fact}{Facts}

\newaliascnt{setting}{definition}

\aliascntresetthe{setting}
\crefname{setting}{Setting}{Settings}

\newaliascnt{notation}{definition}
\newtheorem{notation}[notation]{Notation}
\aliascntresetthe{notation}
\crefname{notation}{Notation}{Notations}

\newaliascnt{remark}{definition}
\newtheorem{remark}[remark]{Remark}
\aliascntresetthe{remark}
\crefname{remark}{Remark}{Remarks}

\newaliascnt{corollary}{definition}
\newtheorem{corollary}[corollary]{Corollary}
\aliascntresetthe{corollary}
\crefname{corollary}{Corollary}{Corollaries}

\newaliascnt{theorem}{definition}
\newtheorem{theorem}[theorem]{Theorem}
\aliascntresetthe{theorem}
\crefname{theorem}{Theorem}{Theorems}

\newaliascnt{proposition}{definition}
\newtheorem{proposition}[proposition]{Proposition}
\aliascntresetthe{proposition}
\crefname{proposition}{Proposition}{Propositions}

\newaliascnt{example}{definition}
\newtheorem{example}[example]{Example}
\aliascntresetthe{example}
\crefname{example}{Example}{Examples}

\newaliascnt{recollection}{definition}

\aliascntresetthe{recollection}
\crefname{recollection}{Recollection}{Recollections}

\newaliascnt{folklore}{definition}
 
\aliascntresetthe{folklore}
\crefname{folklore}{assumption}{assumptions} 
\Crefname{folklore}{Assumption}{Assumptions}

\newaliascnt{assumption}{definition}
\newtheorem{assumption}[assumption]{Assumption}
\aliascntresetthe{assumption}
\crefname{assumption}{Assumption}{Assumptions}

\newcommand*\cocolon{%
        \nobreak
        \mskip6mu plus1mu
        \mathpunct{}%
        \nonscript
        \mkern-\thinmuskip
        {:}%
        \mskip2mu
        \relax
}

\newcommand{\Addresses}{{
  \bigskip
  \footnotesize

  R.~Quinn, \textsc{Utrecht Geometry Center, Universiteit Utrecht, The Netherlands}\par\nopagebreak
  \textit{E-mail address}: \texttt{r.quinn@uu.nl}

  \medskip

  Q.~Zhu, \textsc{Max Planck Institute for Mathematics, Bonn, Germany}\par\nopagebreak
  \textit{E-mail address}: \texttt{qzhu@mpim-bonn.mpg.de}

}}

\newtheorem{mainthm}{Theorem}

\newtheorem{mainex}[mainthm]{Example}
\newtheorem*{theorem*}{Theorem}
\newtheorem*{example*}{Example}
\newtheorem*{question*}{Question}
\newtheorem*{problem*}{Problem}
\newtheorem*{thmG*}{Theorem G}

\newcommand{\tb}{\textcolor{chillired}}

\newcommand{\coCart}[1]{\operatorname{coCart}\!\left(#1\right)}
\newcommand{\GL}{\operatorname{GL}}

\newcommand{\A}{\mathscr A}
\newcommand{\B}{\mathscr B}
\newcommand{\C}{\mathscr C}
\newcommand{\D}{\mathscr{D}}
\newcommand{\Ec}{\mathscr E}

\newcommand{\Sc}{\mathcal S}
\renewcommand{\P}{\mathscr{P}}
\renewcommand{\O}{\mathscr{O}}

\newcommand{\E}{\mathbb E}  % somehow there was \E as \mathscr E before but that was never used
\newcommand{\F}{\mathbb F}
\newcommand{\R}{\mathbb R}
\renewcommand{\S}{\mathbb S}
\newcommand{\N}{\mathbb N}
\newcommand{\Z}{\mathbb{Z}}
\newcommand{\one}{\mathbb 1}
\newcommand{\Q}{\mathscr Q}

\newcommand{\LM}{\mathcal{LM}}
\newcommand{\inert}{\mathrm{int}}

\newcommand{\tmf}{\operatorname{tmf}}
\newcommand{\ku}{\operatorname{ku}}
\newcommand{\kR}{\operatorname{ku}_{\mathbb{R}}}
\newcommand{\KU}{\operatorname{KU}}
\newcommand{\KR}{\operatorname{KU}_{\mathbb{R}}}
\newcommand{\MU}{\operatorname{MU}}
\newcommand{\MW}{\operatorname{MW}}
\newcommand{\MUR}{{\MU_{\mathbb{R}}}}

\newcommand{\MWR}{{\MW_{\mathbb{R}}}}

\newcommand{\uZ}{\underline{\Z}}
\newcommand{\CP}{\mathbb{CP}}
\newcommand{\CPR}{\mathbb{CP}_{\mathbb{R}}}

\newcommand{\Spaces}{\mathcal{S}}
\newcommand{\Sp}{\mathbf{Sp}}
\newcommand{\Alg}{\mathbf{Alg}}

\newcommand{\Seg}{\mathbf{Seg}}
\newcommand{\Cat}{\mathbf{Cat}}
\newcommand{\Fbrs}{\mathbf{Fbrs}}
\newcommand{\Orb}{\mathbf{Orb}}
\newcommand{\Glo}{\mathbf{Glo}}

\newcommand{\LMod}{\mathbf{LMod}}
\newcommand{\Mon}{\mathbf{Mon}}

\newcommand{\Op}{\mathbf{Op}}

\newcommand{\Ab}{\mathbf{Ab}}
\newcommand{\PSh}{\mathbf{PSh}}
\newcommand{\AlgPatt}{\mathbf{AlgPatt}}

\newcommand{\NS}{\mathrm{NS}}

\newcommand{\BP}{\operatorname{BP}}
\newcommand{\BPR}{{\BP_{\mathbb{R}}}}
\newcommand{\BU}{\operatorname{BU}}
\newcommand{\BSU}{\operatorname{BSU}}
\newcommand{\BUR}{{\BU_{\mathbb{R}}}}
\newcommand{\BSUR}{{\BSU_{\mathbb{R}}}}
\newcommand{\BGL}{\operatorname{BGL}}
\newcommand{\KO}{\operatorname{KO}}

\newcommand{\HZ}{\mathrm{H}{\mathbb{Z}}}
\newcommand{\uHZ}{\mathrm{H}\underline{\mathbb{Z}}}
\newcommand{\gr}{\operatorname{gr}}

\newcommand{\Map}{\operatorname{Map}}
\newcommand{\Fun}{\operatorname{Fun}}

\newcommand{\Span}{\operatorname{Span}}
\newcommand{\St}{\operatorname{St}}
\newcommand{\pr}{\operatorname{pr}}
\newcommand{\const}{\operatorname{const}}
\newcommand{\Pic}{\operatorname{Pic}}
\newcommand{\Th}{\operatorname{Th}}
\newcommand{\Un}{\operatorname{Un}}
\newcommand{\Res}{\operatorname{Res}}

\newcommand{\Ind}{\operatorname{Ind}}
\newcommand{\Nat}{\operatorname{Nat}}

\newcommand{\Ar}{\operatorname{Ar}}

\newcommand{\el}{\mathrm{el}}

\newcommand{\core}{\mathrm{core}}
\renewcommand{\phi}{\varphi}
\newcommand{\Or}{\operatorname{Or}}

\newcommand{\map}{\operatorname{map}}
\newcommand{\Infl}{\operatorname{Infl}}

\newcommand{\gl}{\mathrm{gl}}
\newcommand{\op}{\mathrm{op}}

\newcommand{\lax}{\mathrm{lax}}
\newcommand{\id}{\mathrm{id}}

\newcommand{\Coind}{\operatorname{Coind}}

\newcommand{\ind}{\operatorname{ind}}
\newcommand{\res}{\operatorname{res}}
\newcommand{\gp}{\mathrm{gp}}
\newcommand{\BOP}{\operatorname{BOP}}
\newcommand{\BUP}{\operatorname{BUP}}
\newcommand{\Gr}{\operatorname{Gr}}
\newcommand{\MOP}{\operatorname{MOP}}

\DeclareMathOperator*{\colim}{colim}

\makeatletter
\DeclareRobustCommand{\myuline}[2][0pt]{%
  \ifmmode
    \uline{\hphantom{#2}\kern-#1}%
    \kern#1%
    \mathllap{\mathpalette\my@cont@{#2}}%
  \else
    \uline{\phantom{#2}\kern-#1}%
    \kern#1%
    \llap{\contour{white}{#2}}%
  \fi
}

\newcommand{\ul}{\myuline}
\newcommand{\ol}{\overline}

\newcommand{\free}{\operatorname{free}}

\usepackage{stmaryrd}

\newcommand{\T}{\mathcal{T}}

\newcommand{\grp}{\mathrm{grp}}

\newcommand{\CST}{\operatorname{CST}}

\newcommand{\yo}{\text{\usefont{U}{min}{m}{n}\symbol{'210}}}

\DeclareFontFamily{U}{min}{}
\DeclareFontShape{U}{min}{m}{n}{<-> udmj30}{}

\usepackage[normalem]{ulem}
\usepackage{contour}
\usepackage{mathtools} % for \mathllap

\contourlength{0.8pt}

\newcommand{\my@cont@}[2]{\contour{white}{\mbox{$\m@th#1#2$}}}
\makeatother
\newcommand{\THR}{\mathrm{THR}}

\newcommand{\coker}{\operatorname{coker}}

\usepackage{fontawesome5}
\newcommand*{\redchili}{\textcolor{white}{\contour{chillired}{\faPepperHot}}}
\newcommand*{\greenchili}{\textcolor{white}{\contour{chilligreen}{\faPepperHot}}}

\begin{document}
	\title{\vspace{-0.5cm}\textbf{Multiplicative Equivariant Thom Spectra \& Structured Real Orientations}}
	\author{\textsc{Ryan Quinn \& Qi Zhu}}
    \date{} % make date not show up so we can put date in footnote
	\maketitle

            \hypersetup{linkcolor=chilligreen}

\vspace{-2em}
		\begin{abstract}
			 \noindent 
             For strongly even $\E_{\infty}^{C_2}$-rings $E$ we show that any homotopy ring map $\MU \to E^e$ lifts to an $\E_{\rho}$-map $\MU_{\R} \to E$.
             This refines the Hahn--Shi Real orientations of Lubin--Tate theories \(E_n\),
             the Hirzebruch level-\(n\) orientations of \(\tmf_1(n)\),
             and Quillen's idempotent to \(\E_\rho\)-maps.
             It allows us to provide the first structured version of $\BPR$ -- we show that it admits an $\E_{\rho}$-algebra structure. Furthermore, we extend these results to larger groups. In particular, for a finite group $C_2 \leq G$ the Hahn--Shi orientation $N_{C_2}^G \MUR \to E_n$ refines to a $\Coind_{C_2}^G \E_{\rho}$-map, and \(N^G_{C_2}\BPR\) admits a $\Coind_{C_2}^G \E_{\rho}$-algebra structure.
             \medskip \\Essential to this program is a robust theory of multiplicative equivariant Thom spectra, which we develop using parametrized higher algebra and fibrous patterns -- particularly we provide an equivariant version of Antolín-Camarena--Barthel's universal property for multiplicative Thom spectra and use this to deduce a multiplicative equivariant Thom isomorphism. We provide a number of categorical results of independent interest, most notably a distributive monoidal structure on parametrized left module categories.
		\end{abstract}
\hypersetup{linkcolor=chilligreen}
%------------------------------------------------------------------------------------------------------------
% chilli.png
\vspace{-1em}
\begin{figure}[ht!]
\centering
\includegraphics[width=90mm]{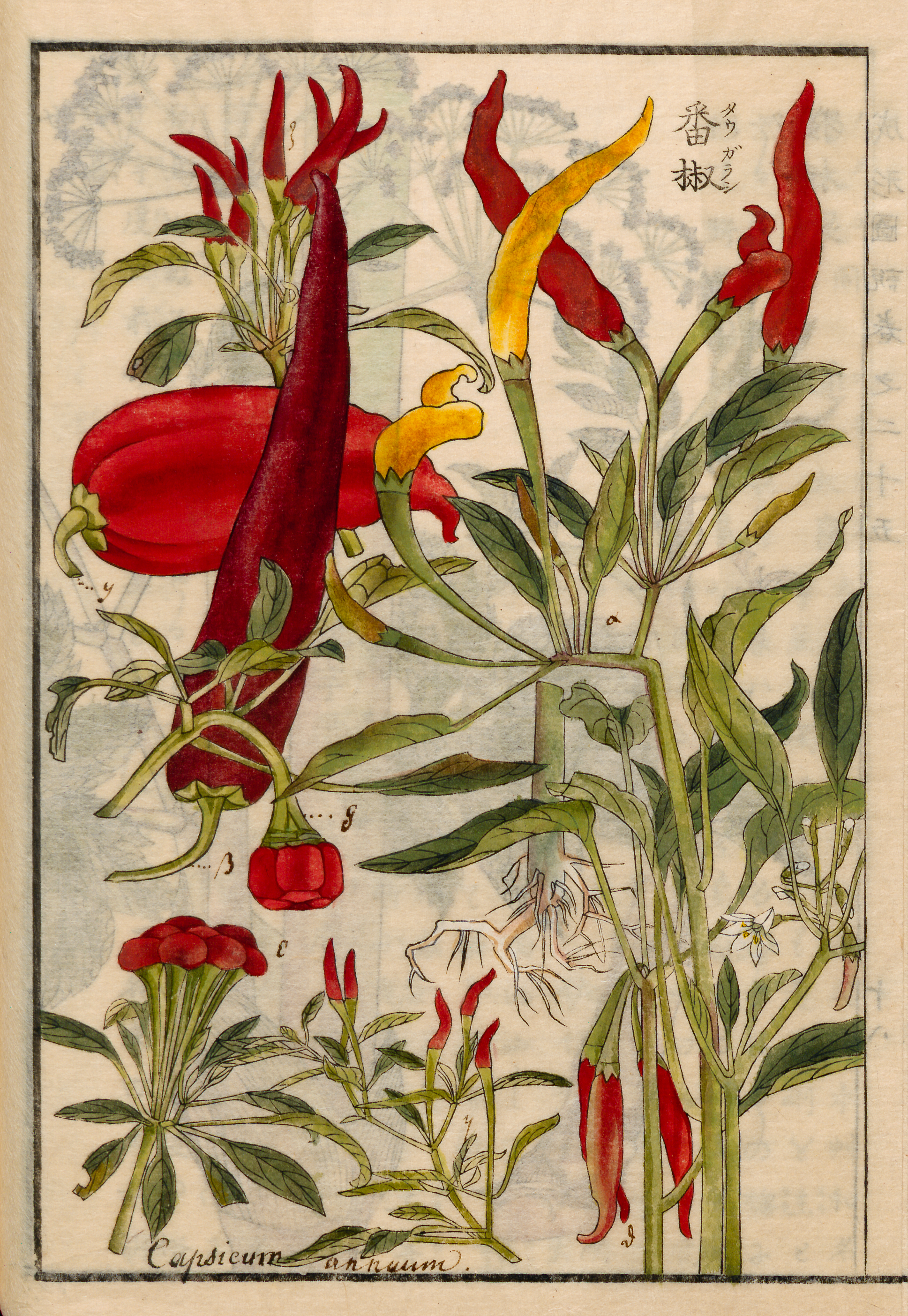}
\end{figure}
\vspace{-2em}
%------------------------------------------------------------------------------------------------------------

%------------------------------------------------------------------------------------------------------------
% date in footnote
\begingroup
\renewcommand\thefootnote{}% remove the number
\footnote{Illustration from the agricultural encyclopedia \emph{Seikei Zusetsu}, scan from \href{http://hdl.handle.net/1887.1/item:938310}{Leiden University Libraries}.}%
\footnote{\emph{Date}: \today}
\addtocounter{footnote}{-2}% don't step the counter
\endgroup
%------------------------------------------------------------------------------------------------------------

%------------------------------------------------------------------------------------------------------------
\thispagestyle{empty} % removed number from title page
\setcounter{tocdepth}{2}    % TOC shows up to \subsection
\setcounter{secnumdepth}{4} % Number up to \subsubsection
%------------------------------------------------------------------------------------------------------------
\section*{Summary of Results}
We develop obstruction-theoretic methods to produce structured orientations. As a consequence, we are able to give a first answer to the open question about multiplicative structures on the Real Brown--Peterson spectrum and its norms. Let $\rho$ denote the regular representation of $C_2$.
\begin{theorem*}[\cref{thm:multiplication_on_BPR}]
    Let $C_2 \leq G$. There is a $\Coind_{C_2}^G \E_{\rho}$-algebra structure on $N_{C_2}^G \BPR$.
\end{theorem*}
\noindent Prior to this work, not even the existence of an $\E_1$- or $\E_{\sigma}$-algebra structure was known on \(\BPR\).
One reason for the sparsity of results in this direction is the lack of computationally tractable equivariant obstruction theories. We obtain this structure on \(\BPR\) and its norms by developing a general obstruction theory for producing structured orientations.
\begin{theorem*}[\cref{thm: main lifting restated}]
    Let $E$ be a strongly even $\E_{\infty}^{C_2}$-ring spectrum. Then, any \(\E_2\)-ring map \(\MU\to E^e\) lifts uniquely to an \(\E_\rho\)-ring map \(\MUR\to E\). In particular, any homotopy ring map \(\MU\to E^e\) lifts uniquely to an \(\E_\rho\)-ring map.

\end{theorem*}
\noindent This immediately refines several orientations of interest, answering a question in \cite[Remark 5.7]{gabe2025realsyntomiccohomology}. 
\begin{example*}[\cref{corollary: hahn-shi hirzebruch n}, \cref{corollary: Hahn--Shi orientation for G}, \cref{cor:adamsopsMUR}]
    Let $C_2 \leq G$.
    \begin{enumerate}[(i)]
        \item The Hahn--Shi Real orientations \(N_{C_2}^G \MUR \to E_n\) admit lifts to $\Coind_{C_2}^G\E_{\rho}$-maps.
        \item The Hirzebruch level-\(n\) genera \(\MUR\to \tmf_1(n)\) admit lifts to $\E_{\rho}$-maps. 
        \item The Burklund--Hahn--Levy--Schlank Adams operations \(\Psi^{\ell}\) on \(\MU_{(2)}\) admit lifts to \(\E_\rho\)-maps \({\MUR}_{(2)}\to {\MUR}_{(2)} \).
    \end{enumerate}
\end{example*}
\noindent Our main tool is a multiplicative version of the equivariant Thom isomorphism.
\begin{theorem*}[\cref{theorem: monoidal thom iso}]
    Let $G$ be a finite group and $V$ be a $G$-representation. Suppose that $R \to A$ is a map of $\E_{\infty}^G$-ring spectra and that $f \colon X \to \ul{\Pic}_G(R)$ is a map of $\E_V$-monoidal spaces. An $\E_V$-$A$-orientation of $f$ gives rise to a Thom isomorphism 
    \[ A \otimes_R \Th_G(f) \simeq A \otimes \Sigma_+^{\infty}X \]
    of $\E_V$-$A$-algebras.
\end{theorem*}
\noindent We prove several results in parametrized higher algebra to obtain the above theorem. Most notably, we need good monoidal structures on parametrized left module categories.
\begin{theorem*}[\cref{theorem: LMod ---> O is O-monoidal}, \cref{theorem: LMod distributive}]
    Let $\ul{\C}^{\otimes}$ be a $G$-symmetric monoidal $G$-$\infty$-category and $\ul{\O}^{\otimes}$ be a $G$-$\infty$-operad, both satisfying certain conditions. Let $A \in \Alg_{\O \otimes \Infl_G \E_1}(\ul{\C})$.
    \begin{enumerate}[(i)]
        \item The pullback
        \begin{center}
    \begin{tikzcd}
        \ul{\LMod}_A^G(\ul{\C})^{\otimes} \arrow[r] \arrow[dr, phantom, very near start, "\lrcorner"] \arrow[d] & \ul{\Alg}_{\Infl_G \LM}(\ul{\C})^{\otimes} \arrow[d]
        \\ \ul{\O}^{\otimes} \arrow[r, "A", swap] & \ul{\Alg}_{\Infl_G \E_1}(\ul{\C})^{\otimes}
    \end{tikzcd}
\end{center}
    is an $\ul{\O}$-monoidal $G$-$\infty$-category.
        \item If $\ul{\C}^{\otimes}$ is $\ul{\O}$-distributive, then so is $\ul{\LMod}_A^G(\ul{\C})^{\otimes}$.
    \end{enumerate}
\end{theorem*}

   \newpage
		\tableofcontents
		\newpage
\section{Introduction}

The \(K(n)\)-local sphere \(L_{K(n)}\S\) is one of the fundamental objects of study in chromatic homotopy theory.
By Devinatz--Hopkins \cite{devinatzHopkins2004homotopyfixedpoints}, there is an equivalence \(L_{K(n)}\S\simeq E_n^{h\mathbb{G}}\), where \(E_n\) is a height \(n\) Lubin--Tate theory, and \( \mathbb{G}\) is the extended Morava stabilizer group.
Hopkins--Miller \cite{rezk1998hopkinsmiller} observed that, by taking finite subgroups \(G \leq \mathbb{G}\), one can produce higher height versions of \(\KO\), i.e.~in the same way that \(\KO\) detects substantial \(v_1\)-periodic information, the spectra 
\[ 
\mathrm{EO}_n(G) = E_n^{hG} \]
detect substantial \(v_n\)-periodic information. For example, for \(n=1\), this is closely related to \(\KO\), and for \(n=2\) this is closely related to \(\mathrm{TMF}\).
Thus, the \emph{higher real \(K\)-theories} \(\mathrm{EO}_n\) contain rich information about the stable homotopy groups of spheres.

\medskip
\noindent
Hill--Hopkins--Ravenel's landmark work \cite{hillhopkinsravenel2009Kervaire} began as an attempt to compute $E_4^{hC_8}$.
They showed that $E_4^{hC_8}$ would detect the Kervaire classes \cite{hill2010arf}.
However, the action on \(E_n\) is obtained through the Hopkins--Miller Theorem, rendering the associated homotopy fixed point spectral sequence intractable.
Nonetheless, there was good evidence this action could be made tractable.
At height \(1\), the Real Conner--Floyd orientation $\MU_{\R} \to \KU_{\R}$ relates the formal-group-theoretic action to the geometric complex conjugation action.
In celebrated work, Hahn--Shi \cite{hahnRealOrientationsLubin2020} produced higher height analogues of the Real Conner-Floyd orientation, namely the \emph{Hahn--Shi Real orientations} $\MU_{\R} \to E_n$ (and normed versions thereof). They thereby provided the first computation of $E_n^{hC_2}$ for $n > 2$, and opened the door to applying the Hill--Hopkins--Ravenel program to understanding higher real \(K\)-theories \cite{Beaudry_2020, beaudryHillShiZeng2021modelsLubinTate,Heard2021,hill2023slice,Duan_2025,carrick2025higherrealktheoriesfinite, duan2025periodicityfinitecomplexityhigher}.

\medskip
\noindent
However, from a multiplicative perspective, Real orientations are still poorly understood.
As pointed out in \cite[Remark 5.7]{gabe2025realsyntomiccohomology}, it was previously not known whether $E_n$ or $\tmf_1(n)$ admit $\E_{\rho}$-$\MU_{\R}$-orientations.
Such orientations are, among others, essential for applications in Real trace methods: an \(\E_\sigma\)-algebra structure is required to define \(\mathrm{THR}({-})\), and an \(\E_\rho\)-\(\MUR\)-orientation allows one to define \(\mathrm{THR}({-}/\MUR)\).
One of the main goals of this article is to construct structured orientations of the aforementioned examples.

\medskip
\noindent
The sparsity of multiplicative results in the Real setting stems from a lack of computationally tractable equivariant obstruction theories.
Even from a theoretical point of view, most work thus far has focused on the fully commutative case via $\N_{\infty}$-operads.
This excludes many of the most important ring spectra from consideration: 
Moore spectra $\S/p^k$, Morava \(K\)-theories \(K(n)\), Ravenel's \(X(n)\) spectra, and the Brown--Peterson spectrum \(\BP\) all do not admit \(\mathbb{E}_\infty\)-ring structures \cite{rognes2008galois, lawson2018secondary, antolinbarthel2019thom, devalapurkar2024higherchromaticthom, senger2024BP}.
Consequently, equivariant multiplicative refinements of these are inaccessible through \(\mathbb{N}_\infty\)-operads.

\medskip 
\noindent
We are led to study equivariant versions of $\E_n$, and to develop tools to answer the following guiding problem:
\begin{problem*}
    For which $C_2$-representations $V$ does $\BPR$ admit an $\E_V$-algebra structure?
\end{problem*}
\begin{proof}[Partial Answer]
    It admits an $\E_{\rho}$-algebra structure (\cref{thm:multiplication_on_BPR}).
\end{proof}
\noindent This question has received much attention in the non-equivariant setting \cite{basterraMandell2013BP, chadwickmandell, lawson2018secondary, hahnWilson2022redshift, senger2024BP, devalapurkar2025examplesdiskalgebras}, leading to many advances about structured ring spectra. Nonetheless, the Real question remained completely open.

\medskip
\noindent
\emph{Equivariant little disk operads} $\E_V$ -- like the non-equivariant ones -- come in a hierarchy controlling the amount of equivariant commutativity of their associated algebras. In particular, they specify norms and associated coherence maps according to the given representation. The specific structure of such algebras was studied by Hill \cite{Hill22disks}. As an example, consider the $C_2$-sign representation $\sigma$. Then, an $\E_{\sigma}$-algebra $A^{\otimes}$ can be viewed as a $C_2$-spectrum $A \in \Sp_{C_2}$ with underlying $\E_1$-algebra $\Res_e^{C_2}A$ and a left module structure $N_e^{C_2} \Res_e^{C_2}A \otimes A \to A$ \cite[Section 7.1]{horev2019genuineequivariantfactorizationhomology}.
We will use this module structure to discuss nilpotence results for $\E_{\sigma}$-algebras (\cref{sec:nilpotence}).

\medskip
\noindent
\change{Attacking the question about structures on $\BPR$ leads us to structured versions of \emph{equivariant Thom spectra}.}
In $\infty$-categorical language, Ando--Blumberg--Gepner--Hopkins--Rezk \cite{abghr2014infty} give an elegant formulation of Thom spectra: Let $R$ be an $\E_{\infty}$-ring spectrum, $X$ be a space and $f\colon X \to \Pic(R)$ be a map of spaces. Then, its Thom spectrum is defined as an $\infty$-categorical colimit
\[ \Th(f) = \mathrm{M}f = \colim \left( X \xrightarrow{f} \Pic(R) \to  \LMod_R \right). \]
This can be made equivariant via the recent surge of \emph{parametrized higher category} theory \cite{barwick2016parametrizedhighercategorytheory, cnossen2023parametrizedstabilityuniversalproperty, hilman2024parametrisedpresentability, martiniWolf2024colimits, stewart2025equivariantoperadssymmetricsequences}. All the defining notions admit parametrized analogues, so we can define equivariant Thom spectra as $G$-colimits in the exact same fashion. 
\cut{This motivates using the toolkit of parametrized higher category theory to transport results about non-equivariant Thom spectra to equivariant ones.} 
\medskip \\This $\infty$-categorical formalism allows a more accessible treatment of multiplicative structures on Thom spectra,
although they were already classically studied by Lewis \cite{lewisMaySteinberger1986equivariant}. Concurrently, Ando--Blumberg--Gepner \cite{andoblumberggepner2018parametrized} and Antolín-Camarena--Barthel \cite{antolinbarthel2019thom} analyzed the compatibility of $\Th$ with multiplicative objects. Notably, Antolín-Camarena--Barthel proved a universal property of multiplicative Thom spectra through operadic left Kan extensions and used this to obtain a multiplicative Thom isomorphism.

\medskip
\noindent
Let $f\colon X \to \ul{\Pic}_G(R)$ be a map of $G$-spaces and $A$ be an $R$-algebra. Its Thom spectrum is informally a twisted version of $R \otimes X$, and a detwist after a base change is called orientation. More precisely, an \emph{$A$-orientation} of a map $f\colon X \to \ul{\Pic}_G(R)$ takes one of the following equivalent forms (\cref{corollary: characterization of orientations}): 
\begin{enumerate}[(i)]
        \item A lift
        \begin{center}
        \begin{tikzcd}
        &\GL_1(\ul{\Pic}_G(R)_{\downarrow A}) \arrow[d]
                \\ X \arrow[r, "f", swap] \arrow[ur, dashed] & \ul{\Pic}_G(R)
            \end{tikzcd}
        \end{center}
        of $f$.
        \item A nullhomotopy of the composite 
        \begin{center}
            \begin{tikzcd}
                X \arrow[r, "f"] & \ul{\Pic}_G(R) \arrow[r, "\Ind_R^A"] & \ul{\Pic}_G(A).
            \end{tikzcd}
        \end{center}
        in $\Sc_G$.
        \item An $R$-module map $\Th_G(f) \to A$ such that for every $x\colon * \to X$ the adjoint $A$-module map corresponding to the $R$-module map 
        \begin{center}
            \begin{tikzcd}
                \Th_G(f \circ x) \arrow[r] & \Th_G(f) \arrow[r] & A
            \end{tikzcd}
        \end{center} is an equivalence.
\end{enumerate}
The importance of orientations lies in the resulting Thom isomorphisms $A \otimes_R \Th_G(f) \simeq A \otimes X$.
Each of the above notions admit structured refinements, called \emph{structured orientations}, inducing structured versions of the Thom isomorphism. An interest in structured versions of the Thom isomorphism already started long ago, see Mahowald--Ray \cite{mahowaldray1981thomiso}.

\medskip
\noindent
A part of this article consists of mimicking the work of Antolín-Camarena--Barthel \cite{antolinbarthel2019thom} via parametrized higher category theory to lift their results to the equivariant setting. This will provide basic tools to develop an obstruction theory for structured equivariant orientations.

\subsection{Main Results \& Outline}
\subsubsection*{Foundations}
The main goal of our foundational work is to develop a robust theory of multiplicative equivariant Thom spectra that allows us to discuss orientation theory (\cref{sec: structured orientations}), and an accompanying multiplicative Thom isomorphism (\cref{theorem: monoidal thom iso}).
Since orientations are about base changes between ring spectra, a good theory of modules over rings is essential.
This motivates our work on an equivariant left module category $\LMod$, and its monoidal properties (\cref{sec:monoidal_LMod}). In the subsequent discussion we will be a little vague in giving the precise assumptions since those are often quite technical to state. The precise statements can be found in the body of the article.
\medskip \\There is a notion of inflated $G$-$\infty$-operads (\cref{construction: inflated operads}), that turns an $\infty$-operad $\O^{\otimes}$ into a $G$-$\infty$-operad $\Infl_G \O^{\otimes}$. Algebras for $\Infl_G \O^{\otimes}$ are simply levelwise $\O$-algebras. So given a $G$-symmetric monoidal $G$-$\infty$-category $\ul{\C}^{\otimes}$ and an $\Infl_G \E_1$-algebra $A$ therein, one can mimic the classical setting to define
\begin{center}
    \begin{tikzcd}
        \ul{\LMod}_A^G(\ul{\C}) \arrow[r] \arrow[dr, phantom, very near start, "\lrcorner"] \arrow[d] & \ul{\Alg}_{\Infl_G \LM}(\ul{\C}) \arrow[d]
        \\ * \arrow[r, "A", swap] & \ul{\Alg}_{\Infl_G \E_1}(\ul{\C})
    \end{tikzcd}
\end{center}
where $\LM$ is the classical left-module operad \cite[Section 4.2.1]{lurie2017ha}. Stewart extended the Boardman--Vogt tensor product to $G$-$\infty$-operads (\cref{theorem: omnibus Alg theorem}) and equipped the above algebra categories with $G$-symmetric monoidal structures.
So given a $G$-$\infty$-operad $\O^{\otimes} \in \Op_{G, \infty}$, a map $\O^{\otimes} \to  \ul{\Alg}_{\Infl_G \E_1}(\ul{\C})^{\otimes}$ of $G$-$\infty$-operads corresponds to an $\O^{\otimes} \otimes \Infl_G \E_1$-algebra in $\ul{\C}^{\otimes}$.

\medskip
\noindent 
We use this observation to study monoidal structures on $\LMod$.
The core difficulty lies in proving that $\LMod$ is distributive in the sense of Nardin--Shah (\cref{def: distributivity}), which says that the $G$-monoidal structure is compatible with $G$-colimits in a suitable sense.

\begin{mainthm}[\cref{theorem: LMod ---> O is O-monoidal}, \cref{theorem: LMod distributive}]
    Let $\ul{\C}^{\otimes}$ be a $G$-symmetric monoidal $G$-$\infty$-category and $\ul{\O}^{\otimes}$ be a $G$-$\infty$-operad, both satisfying certain conditions. Let $A \in \Alg_{\O \otimes \Infl_G \E_1}(\ul{\C})$.
    \begin{enumerate}[(i)]
        \item The pullback
        \begin{center}
    \begin{tikzcd}
        \ul{\LMod}_A^G(\ul{\C})^{\otimes} \arrow[r] \arrow[dr, phantom, very near start, "\lrcorner"] \arrow[d] & \ul{\Alg}_{\Infl_G \LM}(\ul{\C})^{\otimes} \arrow[d]
        \\ \ul{\O}^{\otimes} \arrow[r, "A", swap] & \ul{\Alg}_{\Infl_G \E_1}(\ul{\C})^{\otimes}
    \end{tikzcd}
\end{center}
    is an $\O$-monoidal $G$-$\infty$-category.
        \item If $\ul{\C}^{\otimes}$ is $\ul{\O}$-distributive, then so is $\ul{\LMod}_A^G(\ul{\C})^{\otimes}$.
    \end{enumerate}
\end{mainthm}

\begin{proof}[Proof Sketch]
\hfill 
\begin{enumerate}[(i)]
    \item First, this pullback can be taken in $\Op_{G, \infty}$, so $\ul{\LMod}_A^G(\ul{\C})^{\otimes} \to \ul{\O}^{\otimes}$ is a map of $G$-$\infty$-operads. It remains to show that the left vertical arrow is a coCartesian fibration. Since coCartesian fibrations pull back, it suffices to show that the right vertical arrow is a coCartesian fibration. This is classical \cite[Lemma 4.5.3.6]{lurie2017ha} for each equivariant level
    \[ \ul{\Alg}_{\LM}(\ul{\C})^{\otimes}_H \to \ul{\Alg}_{\E_1}(\ul{\C})_H^{\otimes} \] 
    and to pass to the $G$-symmetric monoidal algebra categories we use a criterion by Haugseng--Melani--Safronov (\cref{corollary: criterion for map between symmetric monoidal categories to be cocartesian fibration}) for which we essentially need to check that norms, fiberwise tensor products and restrictions must preserve coCartesian edges.
    \item We equivariantize the following strategy, which henceforth becomes notationally much heavier. Consider the two colimit diagrams $I^{\triangleright}, J^{\triangleright} \to \LMod_A(\C)$ and the following diagram.
        \begin{center}
            \begin{tikzcd}
                (I \times J)^{\triangleright} \arrow[d, equal] \arrow[r] & I^{\triangleright} \times J^{\triangleright} \arrow[d, equal] \arrow[r] & \LMod_A(\C)^{\times 2} \arrow[r, "\otimes"] \arrow[d] & \LMod_{A \otimes A}(\C) \arrow[r, "A \otimes_{A \otimes A} -"] \arrow[d] & \LMod_A(\C)
                \\ (I \times J)^{\triangleright} \arrow[r] & I^{\triangleright} \times J^{\triangleright} \arrow[r] & \C^{\times 2} \arrow[r, "\otimes", swap]  & \C
            \end{tikzcd}
        \end{center}
        The last two functors of the top composite compose to the tensor product in $\LMod_A(\C)$ and we verify such a factorization for our equivariant setting (\cref{prop: indexed tensor products of LMod}). The main part of distributivity is to show that the top line is a colimit diagram. Distributivity of $\C^{\otimes}$ shows that the bottom line is a colimit diagram. Since the vertical forgetful functor reflects colimits, also the composite of the top three functors is a colimit diagram. But then, $A \otimes_{A \otimes A} -$ preserves this since it is a left adjoint.\qedhere
\end{enumerate}
\end{proof}

\noindent For us, this was the most difficult ingredient to get many tools from parametrized higher algebra running. 
Distributivity is required to use operadic left Kan extensions (\cref{theorem: operadic left Kan extension}). After setting up a suitable notion of $G$-Picard spaces (\cref{construction: Pic}), along with their $G$-monoidal structures, we can construct multiplicative equivariant Thom spectra. Start with a map $f\colon X^{\otimes} \to \ul{\Pic}_G^{\otimes}(A)$ between $\ul{\P}$-monoidal $G$-spaces for some (suitable) $G$-$\infty$-operad $\ul{\P}^{\otimes}$.\footnote{We write $(-)^{\otimes}$ to indicate that the monoidal structure is part of the data.} The operadic version of left Kan extensions then allows us to define
\begin{center}
        \begin{tikzcd}
            X^{\otimes} \arrow[r, "f"] \arrow[d] & \ul{\Pic}_{G}^{\otimes}(A) \arrow[r] & \ul{\LMod}_A^{G}(\ul{\Sp}_G)^{\otimes}
            \\ \ul{\P}^{\otimes} \arrow[urr, dashed, "\Th_{G}^{\otimes}(f)", swap, bend right]
        \end{tikzcd}
    \end{center}
which is a multiplicative enhancement of $G$-Thom spectra. In particular, this shows that a $\ul{\P}$-monoidal map $f\colon X^{\otimes} \to \ul{\Pic}_G^{\otimes}(A)$ gives rise to a $\ul{\P}$-algebra refinement of $\Th_G(f)$. This extension is essentially defined by a universal property, an equivariant version of Antolín-Camarena--Barthel's universal property.
In particular, this makes producing maps of out \(\Th_G^{\otimes}(f)\) tractable.
\begin{mainthm}[\cref{theorem: universal property of Th}] \label{intro: equivariant ACB}
    The functor 
    \[ \Th_G^{\otimes} \colon \Alg_{\ul{\P}}(\ul{\Sc}_G)_{/\ul{\Pic}_G^{\otimes}(A)} \to \Alg_{\ul{\P}}\left(\ul{\LMod}_A^G(\ul{\Sp}_G) \right) \] 
    admits an explicitly described right adjoint in terms of $G$-Picard spaces.
\end{mainthm}

\noindent If $\Th_G^{\otimes}(f)$ is an  $R$-module Thom spectrum with $\ul{\P}$-algebra structure, then a $\ul{\P}$-$A$-orientation is a $\ul{\P}$-map $\Th_G^{\otimes}(f) \to A$ such that on all fibers the adjoint $A$-module map is an equivalence (\cref{corollary: characterization of orientations}). This generalizes the classical Thom class notion and such maps can be understood through the aforementioned universal property (\cref{intro: equivariant ACB}). So it does not require much more to massage this into the multiplicative Thom isomorphism.

\begin{mainthm}[\cref{theorem: monoidal thom iso}] \label{intro: Thom isomorphism}
    A $\ul{\P}$-$A$-orientation of the map $f$ gives rise to an equivalence $A \otimes_R \Th_G^{\otimes}(f) \simeq A \otimes \Sigma_+^{\infty} X^{\otimes}$ of $\ul{\P}$-$A$-algebras.
\end{mainthm}

\noindent During the course of this categorical endeavour, we needed to prove a number of additional results, which we expect to be of independent interest. As such, we provide a monoidal structure on parametrized slice categories (\cref{sec:monoidal_slice}) through certain cotensors, put together results from the literature to give a weak version of parametrized monoidal straightening-unstraightening (\cref{theorem: microcosmic straightening-unstraightening}) and discuss monoidal equivariant versions of the $\GL_1$-functor, which universally takes the units of a multiplicative $G$-space (\cref{corollary: GL1 monoidal}).

\begin{remark} \label{remark: intro readability remark}
    For the sake of readability, we don't work in the highest possible generality. 
    \begin{enumerate}[(i)]
        \item Throughout the entire article, $G$ will always be a finite group. Our main reason for this is that the theory of parametrized higher algebra is not yet suitably worked out for non-finite groups. Once this changes, our results should go through mutatis mutandis.
        \item We always parametrize over $\Orb_G$ instead of a general parametrizing base $\infty$-category $\T$. If one is careful enough with the correct adjectives, like atomic orbital, existence of terminal objects, and so on, then our results also work in this higher generality.\footnote{In fact, this is how we wrote a preliminary draft before we decided to prioritize readability over generality.}
        \item We expect that the \cref{part:foundations} results are also true in a global equivariant context for normed global ring spectra by replacing the span pattern $\Span(\F_G)$ by a global analog. 
        \medskip \\The main difference lies in the distributivity of $\LMod$ since the Nardin--Shah setup of distributivity does not apply in the global setup. In this setting, there is a theory set up by Lenz--Linskens--Pützstück \cite{lenz2025normsequivarianthomotopytheory} who use drastically different language. Nonetheless, the same distributivity proof idea should apply here.
        \medskip \\The main work is already done by Lenz--Linskens--Pützstück who prove %over 5 pages 
        that $\ul{\Sp}_{\gl}^{\otimes}$ is distributive \cite[Theorem 5.10]{lenz2025normsequivarianthomotopytheory}.
        So this already provides the theory of global Thom spectra over the global sphere spectrum $\S_{\gl}$.
    \end{enumerate}
\end{remark}

\subsubsection*{Applications}
%Having concluded with abstract nonsense,
For applications we mostly focus on Real equivariant homotopy theory, but we formulate a number of results for more general finite groups $G$ out of independent interest.
\medskip \\Let $X$ be a strongly even $C_2$-spectrum (\cref{def: strongly even C2}). Let $\rho$ be the regular representation of \(C_2\).
Strong evenness implies the restriction map $\res_e^{C_2}\colon \pi_{*\rho}^{C_2}(X) \to \pi_{2*}^e(X)$ is an isomorphism,
whence $C_2$-equivariant information can be deduced from underlying non-equivariant information.
We extend this notion to positively indexed towers of $C_2$-spectra demanding that the first term and all associated graded pieces are strongly even (\cref{def: strongly even C2}).
By an induction and limiting argument, we obtain a similar conclusion that the restriction map is an isomorphism of the limit of towers (\cref{prop:strongly_even_c2_tower}). 

\medskip \noindent Combined with a cohomological slice tower argument (\cref{construction: cohomological slice tower}), we obtain:
\newpage
\begin{mainthm}[\cref{prop:checkingSAHSS}] \label{mainthm: general lifting}
    Let \(E \in \Sp_{C_2}\) be strongly even and $X \in \Sp_{C_2}$. Suppose the following:
    \begin{enumerate}[(i)]
        \item The $C_2$-spectrum \(X \in \Sp_{C_2} \) is slice bounded below.
        \item There is an indexing set $I$ such that $\uHZ\otimes X\simeq \uHZ\otimes \bigoplus_{i\in I}S^{n_i\rho}$ for certain $n_i \in \N$ where every degree appears finitely often.
    \end{enumerate}
    Then, $\ul{\map}_{\ul{\Sp}_{C_2}}(X, E)$ is strongly even. In particular, the restriction map
    \[
    \res_e^{C_2}\colon  \pi^{C_2}_{*\rho}(\myuline{\map}_{\ul{\Sp}_{C_2}}(X,E))\to \pi_{2*}(\map_{\Sp}(X,E))
    \]
    is an isomorphism.
\end{mainthm}
\noindent Together with a Thom isomorphism computation, we deduce the following lifting result for structured orientations.
\begin{mainthm}[\cref{thm:abstract_lifting_clean_c2}] \label{intro: abstract lifting}
    Let $R$ be an $\E_{\infty}^{C_2}$-ring spectrum and $E$ be a strongly even $\E_{\infty}^{C_2}$-algebra in $\ul{\LMod}_R^{C_2}$. Suppose that $X$ is an $n\rho$-loop space with a map $f \colon X \to \ul{\Pic}_{C_2}(R)$ of $n\rho$-loop spaces. Suppose the following:
    \begin{enumerate}[(i)]
        \item There exists an \(\E_{n\rho}\)-\(R\)-algebra map \(\mathrm{M}f\to E\).
        \item There is an indexing set $I$ such that $\uHZ\otimes \mathrm{B}^{n\rho}X\simeq \uHZ\otimes \bigoplus_{i\in I}S^{n_i\rho}$ for certain $n_i \in \N$ where every degree appears finitely often.
    \end{enumerate}
    %Suppose that there exists an $\E_{n\rho}$-$R$-algebra map $\mathrm{M}f \to E$ and suppose that $X$ satisfies certain finiteness and evenness conditions. 
    Then, the restriction map
    \[
    \res_e^{C_2}\colon \pi^{C_2}_{*\rho}\left(\myuline{\Map}_{\ul{\Alg}_{\E_{n \rho}}(\myuline{\LMod}^{C_2}_R)}(\mathrm{M}f,E) \right)\to \pi^{e}_{2*}\left(\myuline{\Map}_{\ul{\Alg}_{\E_{n \rho}}(\myuline{\LMod}^{C_2}_R)}(\mathrm{M}f,E) \right)
    \]
    is an isomorphism.
\end{mainthm}

\begin{proof}[Proof Sketch]
    After detwisting the left side of $\myuline{\Map}_{\ul{\Alg}_{\E_{n \rho}}(\myuline{\LMod}_R^{C_2})}(\mathrm{M}f,E)$ by the Thom isomorphism (\cref{intro: Thom isomorphism}), we can massage this term with a sequence of formal reformulations (\cref{thm: abstract akward lifting general}) giving rise to
    \[  \myuline{\Map}_{\ul{\Alg}_{\E_{n \rho}}(\myuline{\LMod}_R^{C_2})}(\mathrm{M}f, E) \simeq \Omega^\infty\myuline{\map}_{\ul{\Sp}_{C_2}}(\Sigma^\infty \mathrm{B}^{n\rho}X, \Sigma^{n\rho}\gl_1(E)) \]
    where $\ul{\map}_{\ul{\Sp}_{C_2}}$ denotes the $C_2$-mapping spectrum. Now, we can run a cohomological slice tower (\cref{mainthm: general lifting}).
\end{proof}

\noindent This is our main abstract lifting result, which we will mainly apply to the Real bordism spectrum $\MU_{\R}$ to study structured Real orientations. In particular, we upgrade Chadwick--Mandell's result about $\E_2$-orientations \cite[Theorem 1.2]{chadwickmandell} to the Real setting.

\begin{mainthm}[\cref{thm: main lifting restated}] \label{mainthm: lifting orientations}
    Let $E$ be a strongly even $\E_{\infty}^{C_2}$-ring spectrum. Then, any \(\E_2\)-ring map \(\MU\to E^e\) lifts uniquely to an \(\E_\rho\)-ring map \(\MUR\to E\). In particular, any homotopy ring map \(\MU\to E^e\) lifts uniquely to an \(\E_\rho\)-ring map.
%    \begin{enumerate}[(i)]
 %       \item Then, $E$ admits an $\E_{\rho}$-$\MU_{\R}$-orientation.
%        \item Any \(\E_2\)-ring map \(\MU\to E^e\) lifts uniquely to an \(\E_\rho\)-ring map \(\MUR\to E\). In particular, any homotopy ring map \(\MU\to E^e\) lifts uniquely to an \(\E_\rho\)-ring map. since $\MU \to E^e$ uniquely lifts to an $\E_2$-map by Chadwick--Mandell \cite[Theorem 1.2]{chadwickmandell}.
%    \end{enumerate}
\end{mainthm}

\begin{proof}[Proof Sketch]
    Let $\ol{\gamma}\colon  \MU_{\R} \to \ku_{\R}$ be the Conner--Floyd Real orientation of $\ku_{\R}$. Then, we consider the $C_2$-Thom spectrum
        \[ \MWR = \Th_{C_2} \left(\Omega^{\infty} \Sigma^{\rho} \MU_{\R} \to \Omega^{\infty} \Sigma^{\rho} \ku_{\R} \simeq \BUR \to \ul{\Pic}_{C_2}(\ul{\Sp}_{C_2}) \right), \]
        which is a $C_2$-Thom spectrum over a Real Wilson space. Hahn--Raksit--Wilson and Angelini-Knoll--Kong--Quigley realize that $\MW_{\R}$ has a $\E_{\infty}^{C_2}$-$\rho$-cellular decomposition \cite[Theorem 5.10]{gabe2025realsyntomiccohomology}, so that by obstruction theory there exists an $\E_{\infty}^{C_2}$-map $\MW_{\R} \to E$ (see \cref{theorem: C2 Einfty map from MW}).
        \medskip \\Inspired by a decomposition of Hahn--Raksit--Wilson \cite[Theorem 3.2.10]{hahn2024motivicfiltrationtopologicalcyclic} of $\MW$ into copies of $\MU$, we split $\MU_{\R}$ off of $\MW_{\R}$ as an $\E_{\rho}$-retract by lifting the responsible idempotent (\cref{prop: MWR splitting}).
        \medskip \\Given an \(\E_2\)-ring map \(\MU\to E^e\), we may precompose it with the map \(\MW\to \MU\) from the aforementioned splitting. Results of Hill--Hopkins on the homology of Real Wilson spaces ensure the finiteness and evenness conditions of \cref{intro: abstract lifting}, so we may use it obtain an \(\E_\rho\)-lift \(\MWR\to E\) of $\MW \to E^e$. Precomposing this with \(\MUR\to \MWR\) from the splitting mentioned above then gives an \(\E_\rho\)-map \(\MUR\to E\) lifting the given \(\E_2\)-ring map \(\MU\to E^e\).
        \medskip \\The statement about lifting homotopy ring maps \(\MU\to E^e\) now follows since any such map uniquely lifts to an $\E_2$-map by Chadwick--Mandell \cite[Theorem 1.2]{chadwickmandell}.
\end{proof}

\noindent
Moreover, this also provides a Real version of the Chadwick--Mandell result \cite[Theorem 1.2]{chadwickmandell} for \(\E^{C_2}_\infty\)-rings (\cref{remark: Real version of Chadwick Mandell}), as any Real orientation $\MU_{\R} \to E$ forgets down to a homotopy ring map $\MU \to E^e$ which then lifts to an $\E_{\rho}$-map $\MU_{\R} \to E$.
\medskip \\Our theory can then be immediately applied to equip familiar examples with more structure.
\begin{mainex}[\cref{corollary: hahn-shi hirzebruch n}, \cref{cor:adamsopsMUR}]
    \hfill 
    \begin{enumerate}[(i)]
        \item The Hahn--Shi Real orientations \cite{hahnRealOrientationsLubin2020} of Lubin--Tate theories \(\MUR\to E_n\) admit lifts to $\E_{\rho}$-maps.
        \item The Hirzebruch level-\(n\) genera \cite{MeierHirzebruch} \(\MUR\to \tmf_1(n)\) admit lifts to $\E_{\rho}$-maps. 
        \item The Burklund--Hahn--Levy--Schlank \cite{BHLS} Adams operations on \(\MU_{(2)}\) admit lifts to \(\E_\rho\)-operations on \({\MUR}_{(2)}\).
    \end{enumerate}
\end{mainex}

\noindent Let $H \leq G$ and $\O^{\otimes}$ be an $H$-$\infty$-operad. If $A$ is an $\O$-ring spectrum, then $N_H^G A$ always obtains the structure of an $\Coind_H^G \O$-algebra by parametrized abstract nonsense (\cref{construction: Coind}). This allows us to enhance all our $C_2$-level results to larger groups (\cref{corollary: MUG orientations}).

\begin{mainex}[\cref{corollary: Hahn--Shi orientation for G}]
    Let $C_2 \leq G$.
    \begin{enumerate}[(i)]
        \item The Hahn--Shi Real orientations $\MU^{( \! ( G ) \! )} = N_{C_2}^G \MUR \to E_n$ refine to $\Coind_{C_2}^G \E_{\rho}$-maps.
        \item Let $n > 1$ and $G = (\Z/n)^{\times}$. Then, the Hirzebruch level-$n$ genera induce $\Coind_{C_2}^G \E_{\rho}$-maps $\MU^{( \! ( G ) \! )} \to \tmf_1(n)$.
    \end{enumerate}
\end{mainex}

\noindent Finally, we use our results to provide a first structured version of the Real Brown--Peterson spectrum $\BP_{\R}$. Classically, $\BP$ is constructed through Quillen's idempotent $\MU_{(2)} \to \MU_{(2)}$ \cite{quillen1969complexcobordism}. By our lifting results (\cref{mainthm: lifting orientations}), this provides an $\E_{\rho}$-refinement of the Real Quillen idempotent $\MU_{\R(2)} \to \MU_{\R(2)}$ \cite{araki1979orientations}. It splits off an $\E_{\rho}$-version of $\BP_{\R}$.

\begin{mainthm}[\cref{thm:multiplication_on_BPR}]\label{mainthm:BPR}
    Let $C_2 \leq G$. Then, there is an $\Coind_{C_2}^G\E_{\rho}$-algebra structure on $\BP^{( \! ( G ) \! )} = N_{C_2}^G \BP_{\R}$.
\end{mainthm}

\subsection{Relation to Other Work}
This article touches many areas that have been discussed by other authors, all of whom we have benefited from.
\begin{enumerate}[(i)]
    \item \emph{Multiplicative equivariant Thom spectra.} These were considered by Horev--Klang--Zou \cite{hahn2024equivariantnonabelianpoincareduality}. We fill in some details and in particular avoid a strong version of parametrized monoidal straightening-unstraightening (\cref{folklore}), which is stated as folklore in their article \cite[Theorem A.6.1]{hahn2024equivariantnonabelianpoincareduality}. We take the further step of proving a universal property of multiplicative equivariant Thom spectra generalizing the Antolín-Camarena--Barthel universal property \cite{antolinbarthel2019thom}. Furthermore, we consider left module spectra over more base rings than just the sphere spectrum $\S$, which takes a considerable amount of additional work. In more classical language, multiplicative equivariant Thom spectra were already studied by Lewis \cite[Section X.6]{lewisMaySteinberger1986equivariant}.
    \item \emph{Equivariant monoidal module categories.} Monoidal parametrized versions of module categories have been considered by other authors like Linskens--Nardin--Pol, Yang, and P\"utzst\"uck \cite{linskens2022global,  yang2025filteredhochschildkostantrosenbergtheoremreal, yang2025normedCp, pützstück2025globalpicardspectraborel}. Moreover, Blumberg--Hill studied equivariant module categories through models \cite{Blumberg_2020}. We make the effort of working over all finite groups, incorporating (almost) all $G$-operads and prove the hard fact of distributivity (\cref{theorem: LMod distributive}).
    \item \emph{Structured orientations.} Our \cref{mainthm: lifting orientations} is similar to a conjectured result by Roytman \cite[Anticipated Theorem]{Roytman2023} who outlined an approach and towards which he made partial progress. We take a different approach. Roytman's approach is to directly generalize the tools used in Chadwick--Mandell's paper \cite{chadwickmandell} to the equivariant setting, and then directly generalize Chadwick--Mandell's argument. That is, Roytman's strategy is to show that Real orientations can be lifted to \(\E_\rho\)-maps. This approach requires equivariant versions of $\mathrm{TAQ}$ which have not appeared in the literature.
        
    \medskip
    \noindent
    We instead prove that \(\E_2\)-maps \(\MU\to E^e\) lift to \(\E_\rho\)-maps \(\MUR\to E\), which ultimately also shows that Real orientations can be lifted to $\E_{\rho}$-maps by first applying Chadwick--Mandell's original theorem \cite[Theorem 1.2]{chadwickmandell} before applying our result. We take this approach for three main reasons:
    \begin{itemize}
        \item it generalizes to more than \(\E_\rho\)-maps, we get obstruction theory for lifting \(\E_{2n}\)-maps to \(\E_{n\rho}\)-maps,
        \item it generalizes to arbitrary equivariant \(R\)-module Thom spectra, in particular it does not rely on the close connection between Real orientations and homotopy commutative ring maps,
        \item all of the technical machinery required follows once we generalize the work of Antolín-Camarena--Barthel \cite{antolinbarthel2019thom} to the equivariant setting.
    \end{itemize}
    On the other hand, Roytman's conjecture differs from our \cref{mainthm: lifting orientations} by relaxing the requirement of an $\E_{\infty}^{C_2}$-structure to an $\E_{\rho}$-structure. 
    Given the equivariant TAQ needed for Roytman's strategy, the results of our paper can also be sharpened.
\end{enumerate}

\subsection{Overview}
We split our article into two parts: \cref{part:foundations} concerns the categorical framework of multiplicative equivariant Thom spectra via parametrized higher algebra; \cref{part:applications} applies this to obtain concrete results, mostly about structured Real orientations.
The reader who is predominately interested in the applications -- and willing to believe that equivariant multiplicative Thom spectra work `as expected' -- is free to skip directly to \cref{part:applications}.
For the convenience of the reader, we sketch a proof for the $\E_{\rho}$-structure on $\BP_{\R}$ at the beginning of \cref{part:applications} that accommodates those who skip \cref{part:foundations}.
It demonstrates the main ideas of \cref{part:applications} while avoiding certain technicalities that arise when studying general orientations.
\medskip \\In \cref{sec:recollections} we recall the main notions from parametrized higher algebra and set up notation. In particular, we define singly-coloured, inflated, and coinduced $G$-$\infty$-operads, and give an extensive overview of Nardin--Shah distributivity.
\medskip \\In \cref{sec:categorical_groundwork} we begin by providing a monoidal structure on parametrized slice $\infty$-categories and proceed by assembling results from the literature to give a (weak) monoidal version of parametrized straightening-unstraightening. Then, we do the hard work of setting up a robust monoidal version of equivariant left module categories and check Nardin--Shah distributivity. We end with a discussion of equivariant grouplike algebras in which we recall equivariant versions of the recognition theorem and work out suitable versions of $\GL_1$ and $\Pic$.
\medskip \\In \cref{sec:param_Thom} we discuss parametrized multiplicative Thom spectra and prove an equivariant version of the Antolín-Camarena--Barthel universal property of multiplicative Thom spectra. We then set up (multiplicative) orientation theory through a universal example and notably prove a multiplicative equivariant Thom isomorphism.
\medskip \\In \cref{sec:StronglyEvenTowers_CohomologicalSlice} we review and study strongly even spectra, which we extend to towers. Furthermore we set up and study a cohomological version of the equivariant slice tower. Using these, we set up our main obstruction theory to study structured orientations of equivariant Thom spectra by reducing the problem to a Bredon cohomology computation via the multiplicative Thom isomorphism.
%\medskip \\In \cref{sec: structured orientations} we set up our main obstruction theory to study structured orientations of equivariant Thom spectra by reducing the problem to a Bredon cohomology computation via the multiplicative Thom isomorphism. We then specialize to self-orientations and study idempotents and their splittings in this regard.
\medskip \\In \cref{sec: Real orienations} we specialize to the Real setting for our main applications. We begin by producing $\E_{\rho}$-orientations for strongly even $\E_{\infty}^{C_2}$-rings through Real Wilson spaces.
Such orientations allow us to use the Thom isomorphism to produce preferred lifts of orientations.
In particular, we show that underlying complex orientations $\MU \to E^e$ of strongly even $\E_{\infty}^{C_2}$-ring spectra lift to $\E_{\rho}$-Real orientations $\MU_{\R} \to E$.
We apply this to equip examples from the literature with more structure: the Hahn--Shi orientations, the Hirzebruch level-$n$ genera, and the Adams operations on $\MU$.
Furthermore, we use this to refine Quillen's idempotent whence we construct an $\E_{\rho}$-algebra structure on $\BP_{\R}$. For $C_2 \leq G$ we enhance each of those to $\Coind_{C_2}^G \E_{\rho}$.
\medskip \\In \cref{sec:factorization_and_EV} we give additional applications. We give a formula for relative $\THR$ of equivariant Thom spectra, rephrase equivariant versions of the Hopkins--Mahowald theorem in terms of $\E_V$-quotients, and end with a curiosity related to the Hopkins--Mahowald theorem in the context of nilpotence, namely that $\MU_{\R}$ detects nilpotence for $\E_{\sigma}$-rings and that $\uHZ$ detects nilpotence for $\E_{\sigma} \otimes \E_{\infty}$-rings.
\medskip \\In \cref{sec: more equivariant higher algebra} we spell out that restrictions and norms of $G$-symmetric monoidal $G$-$\infty$-categories are symmetric monoidal functors. We conclude by explaining the existence of an $\E_{\infty}^{C_2}$-structured Real $J$-homomorphism $J_{\R}\colon  \BUR \to \ul{\Pic}_{C_2}(\ul{\Sp}_{C_2})$ via forthcoming work of Brink--Lenz.
\medskip \\In \cref{sec: more homological algebra} we record a proof of an equivariant cohomological UCT while relaxing a finiteness assumption. Moreover, we discuss a result to obtain even cohomology from even homology with suitable finiteness conditions.

\subsection{Notations \& Conventions}

We collect notational conventions here that we will use throughout the whole article.
\begin{enumerate}[(1)]
    \item Throughout, $G$ will always be a finite group.\footnote{The finiteness can be negotiable (\cref{remark: intro readability remark}).}
    \item We denote by $\rho$ the regular representation of $C_2$. More generally, $\rho_G$ is the regular representation of $G$. Moreover, we write $\sigma$ for the $1$-dimensional $C_2$-sign representation.
    \item Underlined categories are $G$-$\infty$-categories, e.g. $\ul{\Sp}_G$ is the $G$-$\infty$-category which on level $H \leq G$ is $\Sp_H$.
    \item Let $X,Y \in \Sp_G$. We  denote by $\ul{\map}_G(X, Y)$ the mapping $G$-spectrum refining the mapping $G$-space $\ul{\Map}_G(X, Y)$. We also write $\ul{\map}_{\ul{\Sp}_G}(X, Y)$ and $\ul{\Map}_{\ul{\Sp}_G}(X, Y)$.
    \item The $G$-$\infty$-operad $\E_{\infty}^G$ is the terminal $G$-$\infty$-operad. So $\E_{\infty}^G$-algebras admit all norms -- they are also known as normed algebras, ultracommutative algebras, or \(G\)-\(\E_\infty\)-algebras.
    \item Let $H \leq G$ be a subgroup. In any parametrized/equivariant setup, we denote by $\Res_H^G$ the restriction functor and by $\Ind_H^G$ resp.~$\Coind_H^G$ the left resp.~right adjoint -- this equivalently also appears as indexed coproducts resp.~products via parametrized colimits resp.~limits.
    %\marginpar{\tiny \raggedright Is there anything else we should add?}
\end{enumerate}
See also the \nameref{ION}.
\section*{Acknowledgments}
We are grateful to Max Blans, Emma Brink, Christian Carrick, Markus Hausmann, Kaif Hilman, Tobias Lenz, Catherine Li, Jiacheng Liang, Lennart Meier, Sven van Nigtevecht, Maxime Ramzi, Jaco Ruit, and Jan Steinebrunner for helpful conversations on this project. We would like to thank Christian, Kaif, Lennart, and  Markus for helpful feedback on an early draft.
\medskip \\Furthermore, we are grateful to the participants and organizers of the European Talbot 2025 for a \emph{magical} week, and particularly for the chili gong show in which we first presented our results to the public. \redchili \greenchili % We have not only dedicated our front picture to it but also our color-coding. 
\medskip \\RQ is funded by the NUI Travelling Studentship. QZ is supported by the Max
Planck Institute for Mathematics (MPIM) in Bonn and is thankful for its hospitality and financial support.

\newpage
\part{Foundations}
\label{part:foundations}
\begin{description}
    \item[\cref{sec:recollections}] This section recalls the main notions from parametrized higher algebra and sets up the notation.

    In \cref{sec:param_operads} we recall the notion of $G$-$\infty$-operads in the language of fibrous patterns due to Barkan--Haugseng--Steinebrunner but also mention underlying Nardin--Shah parametrized $\infty$-operads. We define singly colored $G$-$\infty$-operads.

    In \cref{sec:inflated_operads} we recall Stewart's construction of inflated $G$-$\infty$-operads and the algebras that they corepresent. Moreover, we discuss coinduced operads and the way they yield the natural structure on norms of algebras.

    In \cref{sec:distributivity} we give an extensive overview of Nardin--Shah distributivity in parametrized higher algebra. We spell out a model-independent viewpoint of parametrized fibers via parametrized straightening-unstraightening, spell out a projection formula as a consequence of distributivity, and recall operadic left Kan extensions.

    \item[\cref{sec:categorical_groundwork}] This is the main technical section of the article in which we assemble and prove all the categorical tools we will need.

    In \cref{sec:monoidal_slice} we describe a monoidal structure on parametrized slice $\infty$-categories in the large generality of fibrous patterns. We give a universal property of this construction.

    In \cref{sec:microcosmic} we assemble results from the literature to give a (weak) monoidal version of parametrized straightening-unstraightening.

    In \cref{sec:monoidal_LMod} we define parametrized left module categories and endow them with a suitably monoidal structure. The proof of this relies on an equivariant version of a result by Haugseng--Melani--Safronov to check that a functor of $\O$-monoidal $\infty$-categories is a coCartesian fibration. We then do the hard work of checking that $\LMod_A(-)$ preserves Nardin--Shah distributivity.

    In \cref{sec:Pic_and_grouplike} we first give an overview of equivariant grouplike algebras, notably equivariant versions of the recognition theorems by Cnossen--Haugseng--Lenz--Linskens, Guillou--May, and Juran; in particular, we spell out the interplay of the delooping functor with taking infinite loop spaces $\Omega^{\infty}$. We then work out an equivariantly multiplicative version of the $\GL_1$ functor and finally define equivariant Picard spaces together with a monoidal structure on them. We end by discussing a comma category $\Pic_G^{\otimes}(R)_{\downarrow A}$ involving Picard spaces and describe $\GL_1 (\Pic_G^{\otimes}(R)_{\downarrow A} )$.

    \item[\cref{sec:param_Thom}] This section combines the categorical work to discuss parametrized Thom spectra and the orientation theory that they control.

    In \cref{sec:universal_property_of_Thom} we recall the definition of parametrized Thom spectra and go on by giving several possibilities of multiplicatively enhancing this construction -- namely via Day convolution and via operadic left Kan extension. We use the operadic left Kan extension viewpoint to give an equivariant version of the Antolín-Camarena--Barthel universal property of multiplicative Thom spectra.

    In \cref{sec:abstract_orientation_theory} we set up (multiplicative) orientation theory through a universal example. We then prove some results about multiplicative orientations, most notably a multiplicative Thom isomorphism.
\end{description}

\section{Recollections on Parametrized Higher Algebra}\label{sec:recollections}

\subsection{Language of Parametrized Operads}\label{sec:param_operads}
Our article will build on the language of parametrized higher algebra \cite{nardinshah2022equivarianttopos} and fibrous patterns \cite{barkanhaugsengsteinebrunner2024envelopesalgebraicpatterns}. Parametrized higher category theory works for (potentially nice enough) base $\infty$-categories $\T$, but we will restrict to the $G$-orbit category $\Orb_G$ for a finite group $G$ for the sake of readability. We will begin by recalling some vocabulary while reminding the reader that the Barkan--Haugseng--Steinebrunner fibrous pattern language is supposed to be a (vast) generalization of $\infty$-operads. Please feel free to skip this and return to it to look up vocabulary.
\begin{enumerate}[(1)]
    \item Recall that the \emph{$G$-orbit category} $\Orb_G$ is the full subcategory of the $\infty$-category of $G$-spaces $\Sc_G$ spanned by the transitive $G$-sets $\{G/H \}_{H \leq G}$. A $G$-$\infty$-category is a functor $\Orb_G^{\op} \to \Cat_{\infty}$, or equivalently a coCartesian fibration over $\Orb_G^{\op}$ by straightening-unstraightening: $\Cat_{G, \infty} = \Fun(\Orb_G^{\op}, \Cat_{\infty}) \simeq \coCart{\Orb_G^{\op}}$.
    \item An \emph{algebraic pattern} is an $\infty$-category $\O$ equipped with a factorization system of so-called \emph{inert} and \emph{active} morphisms and a collection of \emph{elementary objects} \cite{chuHaugseng2021homotopycoherent}.  For $o \in \O$ we write $\O_{o/}^{\el} = \O^{\el} \times_{\O^{\inert}} \O_{o/}^{\inert}$ where the superscripts stand for elementary objects resp. inert edges. We write $\AlgPatt$ for the $\infty$-category of algebraic patterns.
    \item Let $\O \in \AlgPatt$ and $\C \in \Cat_{\infty}$. A functor $F \colon \O \to \C$ is a \emph{Segal $\O$-object} in $\C$ if for every $o \in \O$ the induced functor
    \begin{center}
        \begin{tikzcd}
            (\O_{o/}^{\el})^{\triangleleft} \arrow[r] & \O \arrow[r, "F"] & \C
        \end{tikzcd}
    \end{center}
    is a limit diagram. We write $\Seg_{\O}(\C)$ for the $\infty$-category of Segal $\O$-objects in $\C$ \cite{chuHaugseng2021homotopycoherent}.
    \item Let $\O \in \AlgPatt$. A \emph{fibrous $\O$-pattern} is a functor $\P \to \O$ admitting all coCartesian lifts of inert edges satisfying another condition akin to the classical $\infty$-operad conditions \cite[Definition 4.1.2]{barkanhaugsengsteinebrunner2024envelopesalgebraicpatterns}. We write $\Fbrs(\O)$ for the $\infty$-category of fibrous $\O$-patterns.
    \item Let $\O \in \AlgPatt$ and $\P, \Q \in \Fbrs(\O)$. Then, we write $\Alg_{\P/\O}(\Q) = \Fun_{\Fbrs(\O)}(\P, \Q)$. We write $\Alg_{\P}(\Q) = \Alg_{\P/\O}(\Q)$.
    \item There is an algebraic pattern structure on $\Span(\F_{G})$, which is an example of a span pattern. Inert resp. actives are backwards resp. forward maps and the elementary objects are those objects in $\Orb_G$ \cite[Definition 3.2.6]{barkanhaugsengsteinebrunner2024envelopesalgebraicpatterns}. This will be the most relevant algebraic pattern for us.
    \item We denote by $\Op_{G, \infty} = \Fbrs(\Span(\F_G))$ the $\infty$-category of $G$-$\infty$-operads.
    \item Let $\O^{\otimes} \in \Op_{G, \infty}$, then we write $\Mon_{\O}(\Cat_{\infty})$ for the subcategory of $\Fbrs(\O^{\otimes})$ spanned by the coCartesian fibrations with morphisms the ones preserving coCartesian edges. This is equivalent to $\Seg_{\O}(\Cat_{\infty})$.
    \item Given an $\O$-monoidal $\infty$-category $\C^{\otimes} \to \O^{\otimes}$ and a map $o \to o'$, then we will denote by $\bigotimes_{o \to o'}\colon  \C^{\otimes}_o \to \C^{\otimes}_{o'}$ the induced functor.
\end{enumerate}
We will study many aspects of our paper in the language of patterns, which was not the original way parametrized higher algebra was phrased. Rather, Nardin--Shah \cite{nardinshah2022equivarianttopos} define $G$-$\infty$-operads as coCartesian fibrations over a certain $G$-$\infty$-category $\myuline{\F}_{G, *}$ with certain operad conditions. On the other hand \cite[Proposition 5.2.14]{barkanhaugsengsteinebrunner2024envelopesalgebraicpatterns} shows that it is an equivalent notion, allowing us to pull back to the parametrized world along a map $\myuline{\F}_{G, *} \to \Span(\F_{G})$. Suitable comparison results can be found in the literature \cite{barkanhaugsengsteinebrunner2024envelopesalgebraicpatterns, pützstück2024parametrized, stewart2025equivariantoperadssymmetricsequences}. While we prefer to work with patterns, there are some parts where we cannot avoid working in Nardin--Shah's formalism without potentially expending much more effort in comparison results -- most notably when discussing distributivity. In such cases, will decorate objects in the Nardin--Shah formalism with a superscript $(-)^{\NS}$, e.g.~we will write $\Op_{G, \infty}^{\NS}$ and $\Mon_{\myuline{\O}}^{\NS}(\ul{\Cat}_{G, \infty})$.
\medskip \\Before writing out terminologies of underlying Nardin--Shah objects, we first define singly $G$-colored $G$-$\infty$-operads because we will need this to state the definition that follows.

\begin{definition}
    A $G$-$\infty$-operad $\O^{\otimes} \to \Span(\F_G)$ \tb{has a single $G$-color/is singly $G$-colored} if the projection map 
    \[ \O^{\otimes} \times_{\Span(\F_{G})} \Orb_G^{\op} \to \Orb_G^{\op} \]
    is an equivalence.
\end{definition}

\noindent In particular, we obtain a map $\Orb_G^{\op} \to \O^{\otimes}$ via the projection map. Moreover, there is a map $\myuline{\F}_{G, *} \to \Span(\F_{G})$.

\begin{construction}
    Let $\O^{\otimes} \in \Op_{G, \infty}, \C^{\otimes} \in \Mon_{\O}(\Cat_{\infty})$ and $A^{\otimes} \in \Alg_{\O}(\C)$.
    \begin{enumerate}[(i)]
        \item The underlying Nardin--Shah $G$-$\infty$-operad of $\O^{\otimes}$ is $\tb{\myuline{\O}^{\otimes}} = \O^{\otimes} \times_{\Span(\F_{G})} \myuline{\F}_{G, *} \to \myuline{\F}_{G, *}$.
        \item The underlying $\myuline{\O}$-monoidal $G$-$\infty$-category of $\C^{\otimes}$ is $\tb{\myuline{\C}^{\otimes}} = \C^{\otimes} \times_{\Span(\F_{G})} \myuline{\O}^{\otimes} \to \myuline{\O}^{\otimes}$ with the structure map $\myuline{\O}^{\otimes} \to \myuline{\F}_{G, *} \to \Span(\F_{G})$ to define the fiber product on the left.
        \item The underlying $G$-$\infty$-category of $\C^{\otimes}$ is $\tb{\myuline{\C}} = \C^{\otimes} \times_{\Span(\F_{G})} \Orb_G^{\op} \to \Orb_G^{\op}$. 
        \item Suppose now that $\O^{\otimes}$ has a single $G$-color and let $A^{\otimes}\colon  \O^{\otimes} \to \C^{\otimes}$ be an $\O$-algebra. Its underlying $G$-object is given by 
        \[ \tb{\myuline{A}} \colon \Orb_G^{\op} \simeq \O^{\otimes} \times_{\Span(\F_G)} \Orb_G^{\op} \to \C^{\otimes} \times_{\Span(\F_G)} \Orb_G^{\op} \simeq \myuline{\C}. \]
    \end{enumerate}
    Note for example that $\myuline{\C} \to \Orb_G^{\op}$ is really a $G$-$\infty$-category because $\C^{\otimes} \to \O^{\otimes}$ has all inert lifts and we are pulling back to a part of the inerts. By inspection, you can take underlying objects in various orders, e.g.~ $\myuline{\C}$ is also the underlying $G$-$\infty$-category of $\myuline{\C}^{\otimes}$.
    \medskip \\Moreover, we will also underline objects that already come from the parametrized higher algebra world.
\end{construction}
\noindent Usually, the Nardin--Shah and Barkan--Haugseng--Steinebrunner formalisms will be interchangable and don't make a crucial difference in our arguments. 
\medskip \\We end this subsection with a result on singly-colored $G$-$\infty$-operads.

\begin{lemma} \label{lemma: underlying T-category for single T-color}
    Let $\O^{\otimes} \in \Op_{G, \infty}$ have a single $G$-color and $\C^{\otimes} \in \Mon_{\O}(\Cat_{\infty})$. Then, its underlying $G$-$\infty$-category can also be described as $\myuline{\C} \simeq \C^{\otimes} \times_{\O^{\otimes}} \Orb_G^{\op}$.
\end{lemma}

\begin{proof}
    This is by pullback pasting applied to
    \begin{center}
        \begin{tikzcd}
            \myuline{\C} \arrow[r] \arrow[d] \arrow[dr, phantom, very near start, "\lrcorner"] & \C^{\otimes} \arrow[d]
            \\ \Orb_G^{\op} \arrow[dr, phantom, very near start, "\lrcorner"] \arrow[r] \arrow[d, equal] & \O^{\otimes} \arrow[d]
            \\ \Orb_G^{\op} \arrow[r] & \Span(\F_{G})
        \end{tikzcd}
    \end{center}
    where the bottom square is a pullback square since $\O^{\otimes}$ has a single $G$-color and the composite square is a pullback by definition. Thus, we deduce that the top square is a pullback square by pullback pasting.
\end{proof}

\subsection{Inflated \& Coinduced Operads}\label{sec:inflated_operads}
We already have many interesting operads $\O^{\otimes}$ from classical non-equivariant mathematics and wish to bump up these to inflated $G$-operads $\Infl_G \O^{\otimes}$. Algebras over $\Infl_G \O^{\otimes}$ should at each level be classical algebras over $\O^{\otimes}$. So with $\O^{\otimes} = \E_1^{\otimes}$ this will for example allow us to speak about a structure with a coherently associative multiplication on each level.

\begin{construction} \label{construction: inflated operads}
    By functoriality of $\Span$, the functor $i\colon\F \to \F_{G}, \ \myuline{1} \mapsto G/G$ induces another functor $i_*\colon  \Span(\F) \to \Span(\F_{G})$. The associated pullback functor $i^*\colon  \Op_{G, \infty} \to \Op_{\infty}$ then admits a left adjoint
    \[ \tb{\Infl_{G}} = i_!\colon  \Op_{ \infty} \simeq \Fbrs(\Span(\F)) \to \Fbrs(\Span(\F_{G})) \simeq \Op_{G, \infty}, \]
    called \tb{inflation}. For brevity, we will often omit the inflation functor from the notation. For example, there are the classical $\infty$-operads $\E_1^{\otimes}$ and $\LM^{\otimes}$, which we will view as $G$-$\infty$-operads $\E_1^{\otimes} = \Infl_{G} \E_1^{\otimes}$ and $\LM^{\otimes} = \Infl_{G} \LM^{\otimes}$ by inflation.
\end{construction}

\begin{proof}[Checks]
    Here are the technical fibrous pattern checks that we need to do to obtain the above functor.
    \begin{enumerate}[(i)]
        \item Stewart slightly generalizes \cite[Corollary 4.2.3]{barkanhaugsengsteinebrunner2024envelopesalgebraicpatterns} and gives a criterion for the existence of a left adjoint \cite[Proposition 2.23]{stewart2025equivariantoperadssymmetricsequences}. We check those, i.e.~we observe that $\Span(\F_{G})$ is soundly extendable \cite[Lemma A.8]{stewart2025equivariantoperadssymmetricsequences} and that $i_*$ is a Segal morphism, which can be checked straight from the definition \cite[Proposition A.19]{stewart2025equivariantoperadssymmetricsequences}. 
        \item Stewart also has a construction of the inflation functor \cite[below Proposition 3.22]{stewart2025equivariantoperadssymmetricsequences}, which also has right adjoint $i^*$, so these inflation functors agree.\qedhere
    \end{enumerate}
\end{proof}

\begin{remark}
    Consider the functor $p\colon\F_{G} \to \F, \ G/H \mapsto \myuline{1}$ for every $H \leq G$. This induces a functor $p^*\colon  \Op_{\infty} \to \Op_{G, \infty}$ which is given by pulling back along $\Span(\F_{G}) \to \Span(\F)$. This gives another way to construct $G$-$\infty$-operads from $\infty$-operads but one does not obtain inflated operads. For example, the $G$-$\infty$-operad $p^*\E_k^{\otimes}$ is what we believe should be called the \emph{$\E_k^G$-operad}.
\end{remark}

\begin{question}
    Is there a good theory of $\N_k$-operads similar to that of $\N_{\infty}$-operads \cite{blumberghill2015operadic}? We believe that in such a theory $\E_k^{G, \otimes} = p^* \E_k^{\otimes}$ should be the one leading to the most norms on algebras, and hence should be called the $\E_k^G$-operad.
\end{question}

\begin{theorem}[{\cite[Theorem D]{stewart2025equivariantoperadssymmetricsequences}}]
    Let $\O^{\otimes} \in \Op_{G, \infty}$ and $\C \in \Mon_{\Span(\F_G)}(\Cat_{\infty})$. There exists \label{theorem: omnibus Alg theorem} a functor 
    \[ \otimes\colon  \Op_{G, \infty} \times \Op_{G, \infty} \to \Op_{G, \infty} \] 
    refining the Boardman--Vogt tensor product such that:
    \begin{enumerate}[(i)]
        \item For $\O^{\otimes} \in \Op_{G, \infty}$ the functor $- \otimes \O^{\otimes}\colon  \Op_{G, \infty} \to \Op_{G, \infty}$ admits a right adjoint $\tb{\Alg_{\O}(-)^{\otimes}}$ refining the $G$-$\infty$-category of $\O$-algebras.
        \item We have $\Alg_{\O}(\C)^{\otimes} \in \Mon_{\Span(\F_G)}(\Cat_{\infty})$.
        \item If $\O^{\otimes}$ has a single $G$-color, then the forgetful functor $\Alg_{\O}(\C)^{\otimes} \to \C^{\otimes}$ is $G$-symmetric monoidal.
        \item Let $\O^{\otimes} \in \Op_{\infty}$ and $H \leq G$. Then, $\Alg_{\Infl_{G}\O}(\C)^{\otimes}_H \simeq \Alg_{\O}(\C_H)$ naturally.
    \end{enumerate}
\end{theorem}

\noindent Let $\O^{\otimes}, \P^{\otimes} \in \Op_{G, \infty}$. The adjunction in part (i) in particular shows 
\[ \Alg_{\O \otimes \P}(-)^{\otimes} \simeq \Alg_{\O}(\Alg_{\P}(-))^{\otimes} \]
by a Yoneda argument.
\medskip \\Let us end this subsection with a discussion on coinduced $G$-$\infty$-operads. Its purpose is that for $H \leq G$ with $\O^{\otimes} \in \Op_{H, \infty}$ and an $\O$-algebra $A$, its norm $N_H^G A$ naturally obtains an $\Coind_H^G \O$-algebra structure. The idea already appeared in \cite[Section 6.2]{blumberghill2015operadic} which is phrased in parameterized language in \cite{stewart2025tensorproductsequivariantcommutative}.
\begin{construction} \label{construction: Coind}
    Let $H \leq G$. 
    \begin{enumerate}[(i)]
        \item Postcomposing by the functor $\Res_H^G: \Span(\F_G) \to \Span(\F_H)$ induces the restriction map $\Res_H^G \colon \Op_{G,\infty} \to \Op_{H, \infty}$. This admits a right adjoint $\tb{\Coind_H^G} \colon \Op_{H, \infty} \to \Op_{G, \infty}$  by some parametrized higher category theory \cite[Section 1.3.1]{stewart2025tensorproductsequivariantcommutative}. It is the \tb{coinduction functor} of equivariant operads.
        \item Consider a singly $H$-colored $\O^{\otimes} \in \Op_{H, \infty}$ and the counit $\varepsilon \colon \Res_H^G \Coind_H^G \O^{\otimes} \to \O^{\otimes}$. Then, there is a commutative diagram of symmetric monoidal functors \cite[Construction 1.40]{stewart2025tensorproductsequivariantcommutative}
        \begin{center}
            \begin{tikzcd}
                \Alg_{\O}(\Res_H^G \C) \arrow[r, "\varepsilon^*"] \arrow[d] & \Alg_{\Res_H^G \Coind_H^G \O}(\Res_H^G \C) \arrow[r, "N_H^G"] \arrow[d] & \Alg_{\Coind_H^G\O}(\C) \arrow[d]
                \\ \C_H \arrow[r, equal] & \C_H \arrow[r, "N_H^G", swap] & \C_G
            \end{tikzcd}
        \end{center}
        The symmetric monoidality is by \cref{lemma: Res N sym mon} and the right square is via \cref{theorem: omnibus Alg theorem} \textcolor{chilligreen}{(iii)}. To apply the latter, we need that $\Coind_H^G \O^{\otimes}$ is singly $G$-colored, which follows from the next result (\cref{lemma: coinduction of singly colored}).
    \end{enumerate}
\end{construction}

\begin{lemma} \label{lemma: coinduction of singly colored}
    Let $H \leq G$ and consider a singly $H$-colored $\O^{\otimes} \in \Op_{H, \infty}$. Then, the $G$-$\infty$-operad $\Coind_H^G \O^{\otimes} \in \Op_{G, \infty}$ is singly $G$-colored.
\end{lemma}

\begin{proof}
    By the arguments in \cite[Section 1.3.1]{stewart2025tensorproductsequivariantcommutative} there is a functor $\ul{\Op}_{G, \infty} \to \ul{\Cat}_{G, \infty}$ which on level $K \leq G$ takes the underlying $K$-$\infty$-category and strongly preserves parametrized limits. Thus, the diagram
    \begin{center}
        \begin{tikzcd}
            \Op_{H, \infty} \arrow[d] \arrow[r, "\Coind_H^G"] & \Op_{G, \infty} \arrow[d]
            \\ \Cat_{H, \infty} \arrow[r, "\Coind_H^G", swap] & \Cat_{G, \infty}
        \end{tikzcd}
    \end{center}
    commutes. The left arrow sends $\O^{\otimes}$ to the terminal object by the singly colored condition, so the composite sends to the terminal object because $\Coind_H^G$ is a right adjoint. In particular, $\Coind_H^G \O^{\otimes}$ is sent to $\Orb_G^{\op}$, i.e.~is singly $G$-colored.
\end{proof}

\begin{example}
    Let $H \leq G$. Then, $\Coind_H^G \E_{\infty}^H \simeq \E_{\infty}^G$ because right adjoints preserve terminal objects. %Indeed, let $\O^{\otimes} \in \Op_{G, \infty}$, then
    %\[ \Map_{\Op_{G, \infty}}(\O^{\otimes}, \Coind_H^G \E_{\infty}^H) \simeq \Map_{\Op_{H, \infty}}(\Res_H^G \O^{\otimes}, \E_{\infty}^H) \simeq *\]
    %using that $\E_{\infty}^H$ is terminal in $\Op_{H, \infty}$. So $\Coind_H^G \E_{\infty}^H$ is terminal in $\Op_{G, \infty}$.
\end{example}

\subsection{Distributivity of Parametrized Monoidal Structures}\label{sec:distributivity}
One essential categorical tool in this article to realize multiplicative Thom spectra and their universal properties is through operadic left Kan extensions, which is roughly a lax monoidal version of extending by colimits. Such a notion should naturally require a compatibility of the monoidal structure with colimits. More specifically, the monoidal structure should distribute over colimits. The notion of distributivity was initiated by Nardin in his PhD thesis \cite{nardin2017stabilitydistributivity} and further worked out by Nardin--Shah \cite{nardinshah2022equivarianttopos}. 
\medskip \\For the convenience of the reader, we will now give an exposition of this theory. On the way, we spell out the relation of parametrized fibers with parametrized straightening-unstraightening (\cref{prop: parametrized fiber as image of straightening}) and prove a projection formula assuming distributivity (\cref{prop: projection formula from distributivity}). This is one of the sections where we are forced to work with the \cite{nardinshah2022equivarianttopos} formalism instead of the \cite{barkanhaugsengsteinebrunner2024envelopesalgebraicpatterns} formalism because distributivity in this generality is only worked out there.
\begin{remark}
    Recently, Lenz--Linskens--Pützstück \cite{lenz2025normsequivarianthomotopytheory} gave another treatment of distributivity in the context of normed categories, which e.g.~also captures examples in global equivariant homotopy theory. This is not possible with \cite{nardinshah2022equivarianttopos} since the global orbit category (for finite groups) $\Glo$ is not atomic orbital. On the other hand, the \cite{lenz2025normsequivarianthomotopytheory} version is also more restrictive in the sense that only normed setups are allowed, so in the equivariant world only the $\N_\infty$-operads are permitted and not all $G$-$\infty$-operads, which is the reason we are using \cite{nardinshah2022equivarianttopos}. It is expected that the \cite{lenz2025normsequivarianthomotopytheory} distributivity agrees with \cite{nardinshah2022equivarianttopos} in the settings where both make sense but there is currently no known proof of this \cite[Remark 3.15]{lenz2025normsequivarianthomotopytheory}.
\end{remark}
\noindent Non-equivariantly, a monoidal $\infty$-category $(\C, \otimes)$ is called \emph{distributive/monoidally cocomplete} \cite[Definition 2.4]{antolinbarthel2019thom} if the tensor product $- \otimes -$ commutes with colimits in each variable. A shorter way of phrasing this is to demand for two functors $F\colon I \to \C$, $G\colon J \to \C$ and their associated colimit diagrams $I^{\triangleright}, J^{\triangleright} \to \C$ that the composition
\begin{center}
    \begin{tikzcd}
        (I \times J)^{\triangleright} \arrow[r] & I^{\triangleright} \times J^{\triangleright} \arrow[r] & \C \times \C \arrow[r, "\otimes"] & \C 
    \end{tikzcd}
\end{center}
is a colimit diagram. Indeed, the cone point of this composite is $\colim_I{F} \otimes \colim_J{G}$ while the composite restricted to $I \times J$ is $F \otimes G$, so this composite being a colimit diagram means 
\[ \colim_{I} \colim_J F \otimes G \simeq \colim_I F \otimes \colim_J G. \] 
Nardin realized that suitably parametrizing this leads to a slick parametrized version of distributivity. To state his definition we first need a suitable source category for indexed tensor products. 
\medskip \\To define this, we need a parametrized version of fibers, which we will define as the images of parametrized straightening functors. 

\begin{proposition}[{\cite[Proposition 8.3]{barwick2016parametrizedhighercategorytheory}}] \label{prop: parametrized straightening}
    Let $\myuline{\C} \in \Cat_{G, \infty}$, then there is an equivalence
    \[ \tb{\myuline{\St}}\colon  \coCart{\myuline{\C}} \simeq \Map_{\Cat_{\infty}}(\myuline{\C}, \Cat_{\infty}) \simeq \Map_{\Cat_{G, \infty}}\left(\myuline{\C}, \myuline{\Cat}_{G, \infty} \right)    \]
    induced by the classical straightening-unstraightening equivalence and the forgetful-cofree adjunction. The inverse is denoted by $\tb{\myuline{\Un}}$.
\end{proposition}

\begin{definition}\label{def: parameterized fiber}
    Let $F\colon \myuline{\C} \to \myuline{\D}$ be a coCartesian fibration of $G$-$\infty$-categories and $d \in \myuline{\D}$ lying over $G/H \in \Orb_G^{\op}$. Then, we call $\tb{\myuline{\C}_{\myuline{d}}} = \myuline{\St}(F)(d) \in \Cat_{H, \infty}$ the \tb{parametrized fiber} over $d$.
\end{definition}

\begin{proposition} \label{prop: parametrized fiber as image of straightening}
    Let $F\colon \myuline{\C} \to \myuline{\D}$ be a coCartesian fibration of $G$-$\infty$-categories and $d \in \myuline{\D}$ lying over $G/H \in \Orb_G^{\op}$. Then, there is a pullback square
    \begin{center}
        \begin{tikzcd}      \myuline{\C}_{\myuline{d}} \arrow[r] \arrow[d] \arrow[dr, phantom, very near start, "\lrcorner"] & \myuline{\C} \arrow[d, "F"]
            \\ \Orb_H^{\op} \arrow[r] & \myuline{\D}
        \end{tikzcd}
    \end{center}
    where the bottom arrow is classified by $d \in \myuline{\D}_H$ by the Yoneda Lemma \cite[Lemma 2.2.7]{cnossen2023parametrizedstabilityuniversalproperty}.
\end{proposition}

\begin{proof}
    The diagram
    \begin{center}
        \begin{tikzcd}
            \Map_{\Cat_{\infty}}\left(\Orb_H^{\op},\Cat_{\infty}\right) \arrow[r, "\simeq"] \arrow[d] & \Map_{\Cat_{G, \infty}}\left(\Orb_H^{\op}, \myuline{\Cat}_{G, \infty} \right) \arrow[d, "\text{Yoneda}"]  
            \\ \coCart{\Orb_H^{\op}}  \arrow[r, equal] & \Cat_{H, \infty}
        \end{tikzcd}
    \end{center}
    commutes by \cite[Remark 2.2.15]{cnossen2023parametrizedstabilityuniversalproperty}. Let $T\colon  \myuline{\C} \times_{\myuline{\D}} \Orb_H^{\op} \to \Orb_H^{\op}$ be the pullback of $F$. Then, the above diagram evaluated at $\St(T)$ becomes
    \begin{center}
        \begin{tikzcd}
            \St(T) \arrow[r, mapsto] \arrow[d, mapsto] & \myuline{\St}(T) \arrow[d, mapsto]
            \\ G \arrow[r, no head, "\simeq", swap] & \myuline{\St}(T)(G/H) = \myuline{\C}_{\myuline{d}}
        \end{tikzcd}
    \end{center}
    where $\myuline{\St}(G/H) \simeq \myuline{\C}_{\myuline{d}}$ uses that $\Orb_H^{\op} \to \myuline{\D}$ is classified by $d \in \myuline{\D}_H$.
\end{proof}

\begin{remark}
    In particular, we recover Nardin-Shah's notion of parametrized fibers \cite[Notation 2.3.1]{nardinshah2022equivarianttopos} at least in the setting when $F\colon \myuline{\C} \to \myuline{\D}$ is a coCartesian fibration. They define
    \[ \myuline{\C}_{\myuline{d}} = * \times_{\myuline{\D}} \Ar^{\mathrm{cocart}}(\myuline{\D}) \times_{\myuline{\D}} \myuline{\C} \]
    which we identify with the aforementioned pullback (\cref{prop: parametrized fiber as image of straightening}) using the equivalence $\Ar^{\mathrm{cocart}}(\myuline{\D}) \simeq \myuline{\D} \times_{\Orb_G^{\op}} \Ar(\Orb_G^{\op})$, see \cite[Lemma 2.23]{shah2022parametrizedhighercategorytheory}.
\end{remark}

\noindent We will in particular need this in the setting of $\O$-monoidal $G$-$\infty$-categories where the Segal conditions allow us to give a more convenient description of those parametrized fibers.

\begin{theorem}[{\cite[Theorem 2.3.3]{nardinshah2022equivarianttopos}}] \label{theorem: equivariant Segal conditions}
    Let $\myuline{\O}^{\otimes} \in \Op_{G, \infty}^{\NS}$ and $\myuline{\C}^{\otimes} \in \Mon_{\myuline{\O}}^{\NS}(\Cat_{G, \infty})$. Let $o \in \myuline{\O}^{\otimes}$ be an object over $[f\colon U \to G/H] \in \myuline{\F}_{G, *}$ with $H \leq G$. Let $U \simeq \coprod_i G/H_i$ with $H_i \leq G$ be an orbit decomposition and suppose that $o$ corresponds to $(o_i)_i$ under the equivalence $\myuline{\C}^{\otimes}_U \simeq \prod_i \myuline{\C}^{\otimes}_{H_i}$. Then, there is an equivalence
    \[ \myuline{\O}_{\myuline{o}}^{\otimes} \simeq f_* \coprod_i \myuline{\O}^{\otimes}_{\myuline{o}_i} \]
    of $H$-$\infty$-categories, equivalently, we coinduce up each of the $\myuline{\C}_{\myuline{o}_i}^{\otimes}$ to a $H$-$\infty$-category and take their product.
\end{theorem}

\begin{definition} \label{def: indexed tensor product}
    Let $\myuline{\O}^{\otimes} \in \Op_{G, \infty}^{\NS}$ and $\myuline{\C}^{\otimes} \in \Mon_{\myuline{\O}}^{\NS}(\Cat_{G, \infty})$ and let $U \to G/H$ be some map in $\F_{G}$ corresponding to a map in $\myuline{\F}_{G, *}$ over $G/H$. Let $o \to o'$ be a coCartesian lift of that map. Then, we associate to it an \tb{indexed tensor product} functor 
    \[ \tb{\underline{\bigotimes}_{o \to o'}} = \myuline{\St}(\myuline{\C}^{\otimes})(o \to o')\colon  \myuline{\C}^{\otimes}_{\myuline{o}} \to \myuline{\C}^{\otimes}_{\myuline{o}'} \]
    of $H$-$\infty$-categories.
\end{definition}
\noindent Now, we have the language to recall the definition of distributivity.

\begin{definition}[{\cite[Definition 3.2.3, 3.2.4]{nardinshah2022equivarianttopos}}] \label{def: distributivity}
    Let $\myuline{\O}^{\otimes} \in \Op_{G, \infty}^{\NS}$.
    \begin{enumerate}[(i)]
        \item Let $f\colon U \to V$ be a map in $\F_{G}$ as well as $\myuline{\C} \in \Cat_{(\Orb_G)_{/U}, \infty}$ and $\myuline{\D} \in \Cat_{(\Orb_G)_{/V}, \infty}$. A $(\Orb_G)_{/V}$-functor $F\colon f_*\myuline{\C} \to \myuline{\D}$ is called \tb{distributive} if for every pullback square
        \begin{center}
            \begin{tikzcd}
                U' \arrow[r, "f'"] \arrow[d, "g'", swap] \arrow[dr, phantom, very near start, "\lrcorner"] & V' \arrow[d, "g"]
                \\ U \arrow[r, "f", swap] & V
            \end{tikzcd}
        \end{center}
        in $\F_G$ and every $(\Orb_G)_{/U'}$-colimit diagram $p\colon \myuline{I}^{\myuline{\triangleright}} \to g'^* \myuline{\C}$ the $(\Orb_G)_{/V'}$-functor
        \begin{center}
            \begin{tikzcd}
                (f'_* \myuline{I})^{\myuline{\triangleright}} \arrow[r] & f'_*(\myuline{I}^{\myuline{\triangleright}}) \arrow[r, "f'_* p"] & f_*' g'^* \myuline{\C} \arrow[r, no head, "{\simeq, \mathrm{BC}}"] & g^* f_* \myuline{\C} \arrow[r, "g^* F"] & g^* \myuline{\D}
            \end{tikzcd}
        \end{center}
        is an $(\Orb_G)_{/V'}$-colimit-diagram.
        \item Let $\myuline{\C}^{\otimes} \in \Mon_{\myuline{\O}}^{\NS}(\Cat_{G, \infty})$ and suppose that for all $H \leq G$ and $o' \in \myuline{\O}^{\otimes}_H$ the parametrized fiber $\myuline{\C}_{\myuline{o}'}^{\otimes} \in \Cat_{H, \infty}$ is $H$-cocomplete. Then, $\myuline{\C}^{\otimes}$ is said to be \tb{($\myuline{\O}$-)distributive} if for every coCartesian lift $o \to o'$ of a map in $\myuline{\F}_{G, *}$ over some $G/H \in \Orb_G$ corresponding to $f\colon U \to G/H$ the functor $\underline{\bigotimes}_{o \to o'}\colon  \myuline{\C}_{\myuline{o}}^{\otimes} \to \myuline{\C}_{\myuline{o}'}^{\otimes}$ is distributive.
    \end{enumerate}
\end{definition}

\noindent Part (ii) is well-defined: If $U \simeq \coprod_i G/H_i$ with $H_i \leq G$ is an orbit decomposition, then \cref{theorem: equivariant Segal conditions} provides an equivalence $\myuline{\C}_{\myuline{o}}^{\otimes} \simeq f_* \coprod_i \myuline{\C}_{\myuline{o}_i}^{\otimes}$, so the source of $\underline{\bigotimes}_{o \to o'}\colon  \myuline{\C}_{\myuline{o}}^{\otimes} \to \myuline{\C}_{\myuline{o}'}^{\otimes}$ is really coinduced up as demanded in (i).
\medskip \\For an (ordinary) symmetric monoidal $\infty$-category the distributivity of 
\[ \bigotimes_{\langle 2 \rangle \to \langle 1 \rangle} = - \otimes -\colon  \C^{\otimes} \times \C^{\otimes} \to \C^{\otimes} \] unravels as exactly the classical distributivity notion mentioned in the introduction of this subsection.

\begin{theorem}[{\cite[Corollary 3.28]{nardin2017stabilitydistributivity}}] \label{theorem: nardin Sp is distributive}
    There is a (unique) distributive $G$-symmetric monoidal structure on $\myuline{\Sp}_G$.
\end{theorem}

\begin{example} \label{example: S Cat distributive}
    There exist $G$-cartesian $G$-symmetric monoidal structures $\myuline{\Sc}_{G}^{\times}$ and $\myuline{\Cat}_{G, \infty}^{\times}$ which are distributive \cite[Proposition 3.2.5]{nardinshah2022equivarianttopos}.
\end{example}

\begin{lemma} \label{lemma: distributivity along pullback}
    Let $\myuline{\C}^{\otimes} \to \myuline{\F}_{G, *}$ be a distributive $G$-symmetric monoidal $\infty$-category, as well as $\ul{\O}^{\otimes} \in \Op_{G, \infty}^{\NS}$. Then, the pullback 
    \begin{center}
        \begin{tikzcd}
            \myuline{\C}^{\otimes} \times_{\myuline{\F}_{G, *}} \myuline{\O}^{\otimes} \arrow[r] \arrow[d] \arrow[dr, phantom, very near start, "\lrcorner"] & \myuline{\C}^{\otimes} \arrow[d]
            \\ \myuline{\O}^{\otimes} \arrow[r] & \myuline{\F}_{G, *}
        \end{tikzcd}
    \end{center}
    is $\myuline{\O}$-distributive. 
\end{lemma}

\begin{proof}
    First, this pullback can be taken in $\Op_{G, \infty}$ since the maps involved are maps of $G$-$\infty$-operads. Moreover, we pull back a coCartesian fibration. Thus, $\myuline{\C}^{\otimes} \times_{\myuline{\F}_{G, *}} \myuline{\O}^{\otimes}$ is indeed an $\myuline{\O}$-monoidal $G$-$\infty$-category. So now we can discuss distributivity. 
    \medskip \\The point is that distributivity is a notion that concerns the parametrized fibers but the parametrized fibers of $\myuline{\C}^{\otimes} \times_{\myuline{\F}_{G, *}} \myuline{\O}^{\otimes} \to \myuline{\O}^{\otimes}$ can be computed as certain parametrized fibers of $\myuline{\C}^{\otimes} \to \myuline{\F}_{G, *}$ and so the distributivity of $\myuline{\C}^{\otimes}$ yields the $\myuline{\O}$-distributivity of $\myuline{\C}^{\otimes} \times_{\myuline{\F}_{G}} \myuline{\O}^{\otimes}$.
\end{proof}

\noindent The following projection formula result will be important to analyze the existence of indexed coproducts in parametrized left module categories (\cref{prop: categorical properties of LMod}). This will in particular recover the classical projection formula for equivariant spectra using that $\myuline{\Sp}_G^{\otimes}$ is distributive (\cref{theorem: nardin Sp is distributive}). % \marginpar{\tiny \raggedright Can we describe what it means for norms to distribute with colimits?}

\begin{proposition}\label{prop: projection formula from distributivity}
    Let $\myuline{\C}^{\otimes} \in \Mon_{\myuline{\F}_{G, *}}^{\NS}(\Cat_{G, \infty})$ be distributive and $H \leq K \leq G$. For $A \in \myuline{\C}_{K}^{\otimes}$ and $X \in \myuline{\C}_H^{\otimes}$ we have 
    $\Ind_H^{K}(\Res_H^{K}A \otimes X) \simeq A \otimes \Ind_H^{K}X$.
\end{proposition}

\begin{proof}
    We consider $\Orb_H^{\op} \to \Orb_K^{\op}$ as an $K$-$\infty$-category. By the Yoneda Lemma, the objects $A, X$ correspond to $K$-functors
    \[ \myuline{A}\colon  \Orb_K^{\op} \to \myuline{\C}_{\myuline{K}} \quad \text{and} \quad \myuline{X}\colon  \Orb_H^{\op} \to \myuline{\C}_{\myuline{K}}. \]
    Taking the corresponding colimit cones we obtain colimit diagrams $(\Orb_K^{\op} )^{\myuline{\triangleright}}, (\Orb_H^{\op} )^{\myuline{\triangleright}} \to \myuline{\C}_{\myuline{K}}$. For distributivity (\cref{def: distributivity}) we consider the pullback square
    \begin{center}
        \begin{tikzcd}
            G/K \amalg G/K \arrow[r] \arrow[d, equal] \arrow[dr, phantom, very near start, "\lrcorner"]  & G/K \arrow[d, equal]
            \\ G/K \amalg G/K \arrow[r] & G/K
        \end{tikzcd}
    \end{center}
    in $\F_{G}$ and the parametrized colimit diagram
    \[ p\colon  \left(\Orb_K^{\op}, \Orb_H^{\op} \right)^{\myuline{\triangleright}} \to \myuline{\C}_{\underline{G/K \amalg G/K}}^{\otimes} \simeq \left(\myuline{\C}_{\myuline{K}}, \myuline{\C}_{\myuline{K}} \right). \]
    By distributivity the composite
    \begin{center}
        \begin{tikzcd}
            \left(\Orb_K^{\op} \times \Orb_H^{\op} \right)^{\myuline{\triangleright}} \arrow[r] & \Orb_K^{\myuline{\triangleright},\op} \times \Orb_H^{\myuline{\triangleright}, \op} \arrow[r] & \myuline{\C}_{\myuline{K}}^{\otimes} \times \myuline{\C}_{\myuline{K}}^{\otimes} \arrow[r, "\otimes"] & \myuline{\C}_{\myuline{K}}^{\otimes}
        \end{tikzcd}
    \end{center}
    is a $K$-colimit diagram which unravelled means
    \[ \underset{\Orb_K^{\op} \times \Orb_H^{\op}}{\myuline{\colim}} ( \myuline{A} \otimes \myuline{X}) \simeq \underset{\Orb_K^{\op}}{\myuline{\colim}} \ \myuline{A} \otimes \underset{\Orb_H^{\op}}{\myuline{\colim}} \ \myuline{X}. \]
    where $\ul{\colim}$ denotes $K$-colimits. For the rest of the proof we want to verify that evaluating this on $K/K \in \Orb_K$ precisely yields the desired formula $\Ind_H^{K}(\Res_H^{K}A \otimes X) \simeq A \otimes \Ind_H^{K}X$.
    \medskip \\Let us discuss the right side. By definition, the functors
    \[ \underset{\Orb_H^{\op}}{\myuline{\colim}}\colon  \myuline{\Fun}_{K}\left(\Orb_H^{\op}, \myuline{\C}_{\myuline{K}}^{\otimes} \right) \to \myuline{\C}_{\myuline{K}}^{\otimes} \quad \text{and} \quad \underset{\Orb_K^{\op}}{\myuline{\colim}}\colon  \myuline{\Fun}_{K} \left(\Orb_K^{\op}, \myuline{\C}_{\myuline{K}}^{\otimes} \right) \to \myuline{\C}_{\myuline{K}}^{\otimes} \]
    are the $K$-left adjoints to the precomposition by the trivial map \cite[Definition 3.1.1]{hilman2024parametrisedpresentability}. In particular, the second functor is equivalent to the identity functor which explains the equivalence $\myuline{\colim}_{\Orb_K^{\op}} \myuline{A} \simeq \myuline{A}$. We now check that $\myuline{\colim}_{\Orb_H^{\op}}$ evaluated on $K/K$ is $\Ind_{H}^{K}$. To make sense of this, we first of all compute the source of the functor as
    \[ \myuline{\Fun}_{K} \left(\Orb_H^{\op}, \myuline{\C}_{\myuline{K}}^{\otimes} \right)_{K} \simeq \Fun_{K} \left(\Orb_H^{\op} \times \Orb_K^{\op}, \myuline{\C}_{\myuline{K}} \right) \simeq \Fun_{K} \left(\Orb_H^{\op}, \myuline{\C}_{\myuline{K}}^{\otimes} \right)\]
    by \cite[Corollary 2.2.9]{cnossen2023parametrizedstabilityuniversalproperty}, so we are looking at the functor
    \[ \underset{\Orb_H^{\op}}{\myuline{\colim}}\colon  \Fun_{K} \left(\Orb_H^{\op}, \myuline{\C}_{\myuline{K}}^{\otimes} \right) \to \myuline{\C}_{K}^{\otimes} \]
    whose right adjoint is informally described by sending $Y \in \myuline{\C}_{K}^{\otimes}$ to the natural transformation $(\const {\Res_L^{K} Y})_{L \leq K}$. Thus, we compute
    \begin{align*}
        \Map_{\myuline{\C}_{K}^{\otimes}} \left( \underset{\Orb_H^{\op}}{\myuline{\colim}} \ \myuline{X}, Y \right) &\simeq \Map_{\Fun_K\left(\Orb_H^{\op}, \myuline{\C}_{\myuline{K}}^{\otimes} \right)} \left(\myuline{X}, (\const {\Res_L^{K}Y})_{L \leq K} \right)
        \\ &= \Map_{\Nat \left(\Map_{\Orb_K}(-, K/H), \myuline{\C}_{\myuline{K}}^{\otimes} \right)} \left(\myuline{X}, (\const {\Res_L^{K}Y})_{L \leq K} \right)
        \\ &\simeq \Map_{\myuline{\C}_H^{\otimes}}(X, \Res_H^{K}Y)
    \end{align*}
    by the Yoneda Lemma. By adjunction we conclude $\myuline{\colim}_{\Orb_H^{\op}} \ \myuline{X} \simeq \Ind_H^{K}X$. This demystifies the terms on the right side of the projection formula. 
    \medskip \\For the left side, an adjunction argument allows us to first compute a colimit over $\Orb_K^{\op}$ and then over $\Orb_H^{\op}$. The $\Orb_K^{\op}$-part only takes care of $\ul{A}$ which is unchanged as in the previous paragraph. However, we must apply $\Res_H^G$ to view $\ul{A} \otimes \ul{X}$ as an $\Orb_H^{\op}$-diagram and then the effect of $\ul{\colim}_{\Orb_H^{\op}}$ is $\Ind_H^K$ as demonstrated above.
\end{proof}

\noindent Let us end this subsection by recalling a parametrized version of operadic left Kan extensions.

\begin{theorem}[{\cite[Proposition 4.3.3, Theorem 4.3.4]{nardinshah2022equivarianttopos}}] \label{theorem: operadic left Kan extension}
    Consider $\myuline{\O}^{\otimes} \in \Op_{G, \infty}^{\NS}$ as well as $\myuline{\C}^{\otimes}, \myuline{\D}^{\otimes} \in \Mon_{\myuline{\O}}^{\NS}(\ul{\Cat}_{G, \infty})$ with structure map $p\colon  \myuline{\C}^{\otimes} \to \myuline{\O}^{\otimes}$. Let $\myuline{\D}^{\otimes}$ be $\myuline{\O}$-distributive.
    \begin{enumerate}[(i)]
        \item The restriction functor sits in an adjunction 
    \begin{center}
        \begin{tikzcd}
        \Alg_{\C/\O}(\D)  \arrow[r, "p_!" {name = L}, shift left = 1.5] & \Alg_{\O/\O}(\D) \arrow[l, "p^*" {name = R}, shift left = 1.5]
        \arrow[phantom,from=L,to=R,"\scriptscriptstyle\tiny\bot"]
        \end{tikzcd}
    \end{center}
        The left adjoint $p_!$ is called \tb{operadic left Kan extension}.
        \item Let $\myuline{\O}^{\otimes}$ have a single $G$-color and $F \in \Alg_{\C/\O}(\D)$. Then, the underlying $G$-object of $p_! F$ is computed as the $G$-colimit of $F|_{\myuline{\C}}\colon  \myuline{\C} \to \myuline{\D}$.
    \end{enumerate}
\end{theorem}

\begin{proof}
    Slightly more detailed, part (ii)  combines the cited results with \cite[Proposition 10.8]{shah2022parametrizedhighercategorytheory}.
\end{proof}

\noindent In particular, this shows that the parametrized colimit of a lax $\myuline{\O}$-monoidal functor inherits an $\myuline{\O}$-algebra structure.

\section{Categorical Groundwork}\label{sec:categorical_groundwork}

\subsection{Monoidality of Parametrized Slice Categories}\label{sec:monoidal_slice}
Lurie uses simplicial methods to construct a monoidal structure on slice $\infty$-categories \cite[2.2.2.1]{lurie2017ha} which Antolín-Camarena--Barthel characterized by a universal property in \cite[Lemma 2.12]{antolinbarthel2019thom}. 
\medskip \\We generalize this to the pattern setting through a model-independent approach. In particular, this yields a notion in parametrized higher category theory. For $G$-symmetric monoidal $G$-$\infty$-categories and $G$-commutative algebras this has already been claimed in the literature without proof \cite[Section A.5]{hahn2024equivariantnonabelianpoincareduality}. We are grateful to Jan Steinbrunner for suggesting the following to us.
\medskip \\Classically for an $\infty$-category $\C$, the over-slice of $\C$ over some object $c \in \C$ is the pullback
\begin{center}
    \begin{tikzcd}
        \C_{/c} \arrow[r] \arrow[d] \arrow[dr, phantom, very near start, "\lrcorner"] & \C^{[1]} \arrow[d, "t"]
        \\ * \arrow[r, "c", swap] & \C 
    \end{tikzcd}
\end{center}
This can be generalized to a higher algebraic setting by replacing the cotensor construction $\C^{[1]} = \Fun([1], \C)$ in $\Cat_{\infty}$ by one in $\coCart{\F_*}$ and by replacing $c \in \C$ by an algebra. More generally, we can perform all of this for fibrous patterns, which is the purpose of this section.
\medskip \\This subsection uses the language of patterns more intensely than in all other sections and the reader can freely skip it by assuming the existence of a monoidal structure on $G$-$\infty$-categories (\cref{prop: slice is O-monoidal}) and the universal properties (\cref{theorem: universal property of monoidal slice}, \cref{prop: abstract antolin-camarena-barthel}) it comes with.
\medskip \\The main ingredient to carry out the generalization to fibrous patterns is the cotensor which is well-studied in a more general setting, as we will now recall from \cite[Section 5.3]{barkanhaugsengsteinebrunner2024envelopesalgebraicpatterns}. Let $\B \in \Cat_{\infty}$ and $\B_0 \subset \B$ be a wide subcategory. Then, $\tb{\Cat_{\infty/\B}^{\B_0\text{-cocart}}}$ denotes the subcategory of $\Cat_{\infty/\B}$ having coCartesian lifts over $\B_0$ and maps preserving those.

\begin{proposition}[{\cite[Construction 5.3.1, Proposition 5.3.2, Proposition 5.3.11(i)]{barkanhaugsengsteinebrunner2024envelopesalgebraicpatterns}}] \label{prop: omnibus tensor-cotensor}
    Consider $\B \in \Cat_{\infty}$ and a wide subcategory $\B_0 \subset \B$ as well as $\C \in \Cat_{\infty}$ and $(\Ec \to \B) \in \Cat_{\infty/\B}^{\B_0\text{-cocart}}$.
    \begin{enumerate}[(i)]
        \item Then, $\Cat_{\infty/\B}^{\B_0\text{-cocart}} \in \LMod_{\Cat_{\infty}}(\Cat_{\infty}^{\times})$ with tensoring given by
        \[ \Cat_{\infty} \times \Cat_{\infty/\B}^{\B_0\text{-cocart}} \to \Cat_{\infty/\B}^{\B_0\text{-cocart}}, \ (\C, \Ec \to \B) \mapsto (\Ec \times \C \to \B). \]
        \item The functor $- \otimes (\Ec \to \B)\colon  \Cat_{\infty} \to \Cat_{\infty/\B}^{\B_0\text{-cocart}}$ is left adjoint to $\Fun_{/\B}^{\B_0\text{-cocart}}(\Ec \to \B, -)$.
        \item The functor $\C \otimes -\colon  \Cat_{\infty/\B}^{\B_0\text{-cocart}} \to \Cat_{\infty/\B}^{\B_0\text{-cocart}}$ is left adjoint to a cotensoring, with the cotensor of $\C \in \Cat_{\infty}$ and $\Ec \to \B$ given by
        \[ \tb{\Ec_{/\B}^{\C}} = \Fun(\C, \Ec) \times_{\Fun(\C, \B)} \B \to \B \]
        where the pullback is formed along the constant diagram $\B \to \Fun(\C, \B)$.
        \item Let $\O \in \AlgPatt, \P \in \Fbrs(\O)$ and $\C \in \Cat_{\infty}$. Then, the cotensor $\P_{/\O}^{\C}$ in $\Cat_{\infty/\O}^{\text{int-cocart}}$ is again a fibrous pattern.
        \item Let $\O$ be a sound pattern, $\D^{\otimes} \in \Seg_{\O}(\Cat_{\infty})$, $\C \in \Cat_{\infty}$. Then, $(\D^{\otimes})^{\C}_{/\O} \in \Seg_{\O}(\Cat_{\infty})$.
    \end{enumerate}
\end{proposition}

\begin{proof}
    Part (v) follows from (iii) and (iv) since $(\D^{\otimes})^{\C}_{/\O} \to \O$ is a fibrous pattern and a coCartesian fibration, so it is a Segal $\O$-$\infty$-category given the soundness \cite[Observation 4.1.10]{barkanhaugsengsteinebrunner2024envelopesalgebraicpatterns}.
\end{proof}

\noindent With this cotensor we may first define a slice construction in great generality of which in particular the fibrous pattern version of it will be of interest to us in this paper.

\begin{definition} \label{def: slice pattern}
    Let $\O \in \Cat_{\infty}$ and $(\D \to \O) \in \Cat_{\infty/\O}$ and $A\colon \O \to \D$ in $\Cat_{\infty}$. Then, we define $\tb{\D_{/A}} = \D^{[1]}_{/\O} \times_{\D} \O \to \O$, the \tb{(over)-slice $\infty$-category} of $\D$ over $A$ as the pullback
    \begin{center}
        \begin{tikzcd}
            \D_{/A} \arrow[r] \arrow[dr, phantom, very near start, "\lrcorner"] \arrow[d] & {\D_{/\O}^{[1]}} \arrow[d, "\mathrm{target} \circ \pr_{\Ar(\D)}"]
            \\ \O \arrow[r, "A", swap] & \D
        \end{tikzcd}
    \end{center}
    in $\Cat_{\infty}$ where the left leg is the structure map.
\end{definition}

\begin{lemma} \label{lemma: fibers of slice}
    Let $\O \in \Cat_{\infty}$ and $(\D \to \O) \in \Cat_{\infty/\O}$ and $A\colon \O \to \D$ in $\Cat_{\infty}$. For $o \in \O$ we have $(\D_{/A})_o \simeq (\D_o)_{/A(o)}$.
\end{lemma}

\begin{proof}
    Recall that $(\D_{/A})_o$ denotes the fiber of $\D_{/A} \to \O$ at $o \in \O$. Consider the following diagram:
    \begin{center}
        \begin{tikzcd}
            (\D_{/A})_o \arrow[r] \arrow[dr, phantom, very near start, "\lrcorner"] \arrow[d] & \D_{/\O}^{[1]} \arrow[dr, phantom, very near start, "\lrcorner"] \arrow[r] \arrow[d] & \O \arrow[d]
            \\ \D_{/A(o)} \arrow[r] \arrow[dr, phantom, very near start, "\lrcorner"]\arrow[d] & \Ar(\D) \arrow[d] \arrow[r] & \Ar(\O)
            \\ * \arrow[r, "A(o)", swap] & \D
        \end{tikzcd}
    \end{center}
    Here, the top right square is a pullback square by construction of the cotensor (\cref{prop: omnibus tensor-cotensor} \textcolor{chilligreen}{(iii)}), the bottom square is a pullback by definition of slice categories and the left composite rectangle is a pullback square by definition of $(\D_{/A})_o$. In particular, the top left square is a pullback square by pullback pasting. 
    \medskip \\Thus, the top composite pullback rectangle yields $(\D_{/A})_o \simeq \D_{/A(o)} \times_{\Ar(\O)} \O$. We further factor $\D_{/A(o)} \to \Ar(\O)$ through $\O_{/o}$ and with it factor the pullback into two steps as follows:
    \begin{center}
        \begin{tikzcd}
            (\D_{/A})_o \simeq \D_{/A(o)} \times_{\Ar(\O)} \O \arrow[dr, phantom, very near start, "\lrcorner"]  \arrow[r] \arrow[d] & \O \times_{\Ar(\O)} \O_{/o} \arrow[r] \arrow[d] \arrow[dr, phantom, very near start, "\lrcorner"]  & \O \arrow[d] \arrow[dd, bend left = 60, equal]
            \\ \D_{/A(o)} \arrow[r] & \O_{/o} \arrow[r] \arrow[d] \arrow[dr, phantom, very near start, "\lrcorner"] & \Ar(\O) \arrow[d]
            \\ & * \arrow[r, "o", swap] & \O
        \end{tikzcd}
    \end{center}
    The bottom square is also a pullback by construction and the right composite arrow is $\id_{\O}$. Thus, $\O \times_{\Ar(\O)} \O_{/o} \simeq *$. Therefore,
    \begin{align*}
        (\D_{/A})_o &\simeq \D_{/A(o)} \times_{\O_{/o}} *
        \\ &\simeq (\Ar(\D) \times_{\D} *) \times_{\Ar(\O) \times_{\O} *} (\Ar(*) \times_* *)
        \\ &\simeq (\Ar(\D) \times_{\Ar(\O)} \Ar(*)) \times_{\D \times_{\O} *} (* \times_* *)
        \\ &\simeq \Ar(\D_o) \times_{\D_o} *
        \\ &\simeq (\D_o)_{/A(o)}
    \end{align*}
    as desired.
\end{proof}

\begin{example} \label{example: classical parametrized slice}
    Let $\myuline{\D} \to \Orb_G^{\op}$ be a $G$-$\infty$-category and $d\colon  \Orb_G^{\op} \to \myuline{\D}$ be an $G$-object therein. Then, $\myuline{\D}_{/d}$ recovers the usual parametrized slice \cite[Notation 4.29]{shah2022parametrizedhighercategorytheory} and for $H \leq G$ we obtain $(\myuline{\D}_{/d})_H \simeq (\myuline{\D}_H)_{/d_H}$ by the previous Lemma (\cref{lemma: fibers of slice}).
\end{example}

\noindent The most important example for us will be slicing Segal $\O$-$\infty$-categories over an algebra. We now show that these remain $\O$-Segal.

\begin{proposition} \label{prop: slice is O-monoidal}
    Let $\O^{\otimes} \in \AlgPatt$ with $\D^{\otimes} \in \Fbrs(\O^{\otimes})$ and $A \in \Alg_{\O}(\D)$.
    \begin{enumerate}[(i)]
        \item Then, $\D^{\otimes}_{/A} \to \O^{\otimes}$ is also a fibrous pattern.
        \item If $\D^{\otimes} \to \O^{\otimes}$ is furthermore a coCartesian fibration, then so is $\D^{\otimes}_{/A} \to \O^{\otimes}$.
    \end{enumerate}
    In particular, if $\O^{\otimes}$ is soundly extendable and $\D^{\otimes} \in \Seg_{\O}(\Cat_{\infty})$, then $\D_{/A}^{\otimes} \in \Seg_{\O}(\Cat_{\infty})$.
\end{proposition}

\begin{proof}
    The 'in particular' part is a consequence of \cite[Remark 4.2.7]{barkanhaugsengsteinebrunner2024envelopesalgebraicpatterns}. 
    \medskip \\In the course of this proof we will need to find coCartesian edges of certain functors. For the convenience of the reader, we will write informal proofs in terms of diagrams which are formalized through mapping spacifying the argument.
    \begin{enumerate}[(i)]
        \item Pullbacks in $\Fbrs(\O^{\otimes})$ which are computed in $\Cat_{\infty}$ \cite[Lemma 4.1.13]{barkanhaugsengsteinebrunner2024envelopesalgebraicpatterns}, so we consider the pullback square
        \begin{center}
            \begin{tikzcd}
                \D_{/A}^{\otimes} \arrow[r] \arrow[d] \arrow[dr, phantom, very near start, "\lrcorner"] & (\D^{\otimes})^{[1]}_{/\O^{\otimes}} \arrow[d]
                \\ \O^{\otimes} \arrow[r, "A", swap] & \D^{\otimes}
            \end{tikzcd}
        \end{center}
        Since $A$ is an $\O$-algebra, the bottom map is a map of fibrous patterns. Thus, it suffices to show that the right vertical map is one as well. 
        \medskip \\Let $o \to o'$ be an inert map in $\O^{\otimes}$ and $(d_1 \to d_2)$ be a map in $(\D^{\otimes})_{/\O^{\otimes}}^{[1]}$. Let $d_1 \to d_1'$ and $d_2 \to d_2'$ be coCartesian lifts of $o \to o'$, which exist because $\D^{\otimes} \in \Fbrs(\O^{\otimes})$. Then, there is a unique factorization
        \begin{center}
            \begin{tikzcd}
                & d_1' \arrow[dr, dashed, "\exists !"] & & & o' \arrow[dr, equal]
                \\ d_1 \arrow[ur] \arrow[r] & d_2 \arrow[r] & d_2' & o \arrow[r, equal] \arrow[ur] & o \arrow[r] & o'
            \end{tikzcd}
        \end{center}
        lying over the diagram on the right. In particular, the dashed arrow lies over $\id_{o'}$. Thus, we have constructed an edge $(d_1 \to d_2) \to (d_1' \to d_2')$, which we claim to be a coCartesian lift. Indeed, consider 
        \begin{center}
            \begin{tikzcd}
                & (d_1' \to d_2') \arrow[dr, dashed] & & & o' \arrow[dr]
                \\ (d_1 \to d_2) \arrow[rr] \arrow[ur] & & (d_1'' \to d_2'') & o \arrow[rr] \arrow[ur] & & o''
            \end{tikzcd}
        \end{center}
        with the left diagram lying over the right. We need to show that there exists a unique dashed arrow. The arrows $d_1' \to d_1''$ resp. $d_2' \to d_2''$ exist uniquely for the $(-)_1$ resp. the $(-)_2$ parts of the diagrams. So we need to show that
        \begin{center}
            \begin{tikzcd}
                d_1' \arrow[r] \arrow[d] & d_2' \arrow[d]
                \\ d_1'' \arrow[r] & d_2''
            \end{tikzcd}
        \end{center}
        commutes. For this consider the following left diagram lying over the right diagram.
        \begin{center}
            \begin{tikzcd}
                & d_1' \arrow[dr, dashed, "\exists !"] & & & o' \arrow[dr]
                \\ d_1 \arrow[r] \arrow[ur] & d_1'' \arrow[r] & d_2'' & o \arrow[ur] \arrow[r] & o'' \arrow[r, equal] & o''
            \end{tikzcd}
        \end{center}
        By the universal property of the coCartesian edge $d_1 \to d_1'$ there exists a unique dashed arrow but each of the two composites in our square above fits. Thus, the square must commute.
        \medskip \\This shows that the coCartesian lift of an inert edge of $(\D^{\otimes})_{/\O^{\otimes}}^{[1]} \to \O^{\otimes}$ agrees with pointwise coCartesian lifts of $\D^{\otimes} \to \O^{\otimes}$. In particular, $(\D^{\otimes})_{/\O^{\otimes}}^{[1]} \to \D^{\otimes}$ must preserve inert edges, i.e. is a map in $\Fbrs(\O^{\otimes})$.
        \item We compute the coCartesian lift over every edge.
        \medskip \\Let $o \to o'$ be a map in $\O^{\otimes}$ and $(d \to A(o))$ lie over $o$. Let $d \to d'$ be a coCartesian lift of $o \to o'$ starting in $d$ provided by the coCartesian fibration $\D^{\otimes} \to \O^{\otimes}$ and let $A(o) \to A'$ be the coCartesian lift of $o \to o'$ starting in $A(o)$. By the universal property of coCartesian fibrations there is a unique factorization
        \begin{center}
            \begin{tikzcd}
                & A' \arrow[dr, dashed, "\exists !"] & & & o' \arrow[dr, "o'"]
                \\ A(o) \arrow[ur] \arrow[rr, "A(o \to o')", swap] & & A(o') & o \arrow[rr] \arrow[ur] & & o'
            \end{tikzcd}
        \end{center}
        where the left diagram lies over the right. By the universal property of coCartesian edges, there is a unique factorization 
        \begin{center}
            \begin{tikzcd}
                & d' \arrow[dr, dashed, "\exists !"] & & & o' \arrow[dr, equal]
                \\ d \arrow[ur] \arrow[r] & A(o) \arrow[r] & A' & o \arrow[r, equal] \arrow[ur] & o \arrow[r] & o'
            \end{tikzcd}
        \end{center}
        where the left diagram lies over the right. Altogether, this provides a map
        \begin{center}
            \begin{tikzcd}
                d \arrow[d] \arrow[rr] & & d' \arrow[dl] \arrow[d, dashed]
                \\ A(o) \arrow[r] & A' \arrow[r] & A(o'),
            \end{tikzcd}
        \end{center}
        which we claim to be the coCartesian lift of $o \to o'$ starting in $(d \to A(o))$. Note that the diagram commutes by construction. Moreover, the bottom map is $A(o \to o')$ by definition, so this is really a map in the slice category.
        \medskip \\Indeed, consider
        \begin{center}
            \begin{tikzcd}
                & o' \arrow[dr]
                \\ o \arrow[rr] \arrow[ur] & & o''
            \end{tikzcd}
        \end{center}
        and some $(d \to A(o)) \to (d'' \to A(o''))$. We need to find a unique factorization
        \begin{center}
            \begin{tikzcd}
                & (d' \to A' \to A(o')) \arrow[dr, dashed]
                \\ (d \to A(o)) \arrow[ur] \arrow[rr] & & (d'' \to A(o''))
            \end{tikzcd}
        \end{center}
        lying over that previous triangle. By the universal property of the coCartesian lift $d \to d'$, there is a unique map $d' \to d''$ making the $d \to d' \to d''$ part commute. Moreover, $A(o') \to A(o'')$ must be $A(o' \to o'')$ for this to be a map in the slice category. By construction it also makes the $A(o) \to A(o') \to A(o'')$ part commute. It remains to check that
        \begin{center}
            \begin{tikzcd}
                d' \arrow[r] \arrow[d] & A(o') \arrow[d]
                \\ d'' \arrow[r] & A(o'')
            \end{tikzcd}
        \end{center}
        commutes. For this we consider the diagram
        \begin{center}
            \begin{tikzcd}
                & d' \arrow[dr, dashed, "\exists !"] & & & o' \arrow[dr]
                \\ d \arrow[ur] \arrow[r] & d'' \arrow[r] & A(o'') & o \arrow[ur] \arrow[r] & o'' \arrow[r, equal] & o''
            \end{tikzcd}
        \end{center}
        where the left diagram lies over the right.
        \medskip \\By all the commutative diagrams already provided in that triangle of arrows above, we see that both ways of going the square are allowed for the dashed arrow. By uniqueness it follows that they agree. \qedhere
    \end{enumerate}
\end{proof}

\begin{corollary}
    Let $\O^{\otimes} \in \Op_{G, \infty}$ and $\D^{\otimes} \in \Mon_{\O}(\Cat_{\infty})$ with $A \in \Alg_{\O}(\D)$. Let $o \to o'$ be a map in $\O^{\otimes}$ and $(X_o \to A(o)) \in (\D_o^{\otimes})_{/A(o)}$, then the $(o \to o')$-indexed tensor product in $\D_{/A}^{\otimes}$ is given by 
    \[ \bigotimes_{o \to o'} (X_o \to A(o)) \simeq \left( \bigotimes_{o \to o'} X_o \to \bigotimes_{o \to o'} A(o) \xrightarrow{\mu} A(o') \right) \]
    where $\mu$ is the algebra multiplication.
\end{corollary}

\begin{proof}
    Recall that $\bigotimes_{o \to o'} (X_o \to A(o))$ denotes the coCartesian edge starting in $(X_o \to A(o))$ lying over $o \to o'$. In the proof of \cref{prop: slice is O-monoidal} we described the coCartesian edges of $\D^{\otimes}_{/A}$, so the indexed tensor product can be immediately read off.
\end{proof}

\begin{example}
    More concretely, if $\C^{\otimes} \in \Mon_{\Span(\F_{G})}(\Cat_{\infty})$ and $A \in \Alg_{\Span(\F_G)}(\C)$, then the indexed product is described in the following more familiar fashion:
    \begin{enumerate}[(i)]
        \item Let $\nabla\colon G/H \amalg G/H \to G/H$ be the fold map and $X \to A_H, Y \to A_H$ be maps in $\C_H$. Then, 
        \[ \bigotimes_{\nabla} (X \to A_H, Y \to A_H) \simeq (X \otimes Y \to A_H \otimes A_H \to A_H). \]
        \item Let $H \leq K$ and $G/H \to G/K$ be the projection map and $X \to A_H$ be a map in $\C_H$. Then,
        \[ \bigotimes_{G/H \to G/K} (X \to A_H) \simeq \left( N_H^K X \to N_H^K A_H \to A_K \right) \]
        where $N_H^K$ is given by the indexed tensor product $\bigotimes_{G/H \to G/K}$.
    \end{enumerate}
    In particular, this is the usual tensor product on slice categories for $G = e$ which is exemplified in (i).
\end{example}

\begin{proposition}
    Let $\O^{\otimes} \in \Op_{G, \infty}$ with $\D^{\otimes} \in \Mon_{\O}(\Cat_{\infty})$ and $A \in \Alg_{\O}(\D)$. Then, the underlying $G$-$\infty$-category of $\D^{\otimes}_{/A}$ is the slice $G$-$\infty$-category $\ul{\D_{/A}}$.
\end{proposition}

\begin{proof}
    For the sake of the following computation, let us write $S = \Span(\F_G)$. We omit writing the map to $\Orb_G^{\op}$ and perform in $\coCart{\Orb_G^{\op}}$ the following computation:
    \begin{align*}
        \D_{/A}^{\otimes} \times_{S} \Orb_G^{\op} &\simeq \left((\D^{\otimes})_{/S}^{[1]} \times_{\D^{\otimes}} S) \right) \times_{S} \Orb_G^{\op}
        \\ &\simeq \left((\D^{\otimes})^{[1]}_{/S} \times_{S} \Orb_G^{\op} \right) \times_{\D^{\otimes} \times_{S} \Orb_G^{\op}} \left( S \times_{S} \Orb_G^{\op} \right)
        \\ &\simeq \left((\D^{\otimes})^{[1]}_{/S} \times_{S} \Orb_G^{\op} \right) \times_{\ul{\D}} \Orb_G^{\op}
        \\ &\simeq \left( \left(\Ar( \D^{\otimes}) \times_{\Ar(  S)} S  \right)\times_{S} \Orb_G^{\op} \right) \times_{\ul{\D}} \Orb_G^{\op}
        \\ &\simeq \left( \Ar(  \D^{\otimes}) \times_{\Ar(  S)}  \Orb_G^{\op} \right) \times_{\ul{\D}} \Orb_G^{\op}
        \\ &\simeq \left( \Ar(  \D^{\otimes}) \times_{\Ar(  S)} \Ar(  \Orb_G^{\op}) \times_{\Ar(  \Orb_G^{\op})} \Orb_G^{\op} \right) \times_{\ul{\D}} \Orb_G^{\op}
        \\ &\simeq \left( \Ar(  \D^{\otimes} \times_{S} \Orb_G^{\op})  \times_{\Ar(  \Orb_G^{\op})} \Orb_G^{\op} \right) \times_{\ul{\D}} \Orb_G^{\op}
        \\ &\simeq \left( \Ar(  \ul{\D})  \times_{\Ar(  \Orb_G^{\op})} \Orb_G^{\op} \right) \times_{\ul{\D}} \Orb_G^{\op}
        \\ &\simeq \ul{\D}^{[1]}_{/\Orb_G^{\op}} \times_{\ul{\D}} \Orb_G^{\op},
    \end{align*}
    which is the classical parametrized slice (\cref{example: classical parametrized slice}).
\end{proof}

\noindent So we are really putting an $\O$-monoidal structure on $\myuline{\D_{/A}}$.

\begin{theorem}[Universal Property of Monoidal Slice] \label{theorem: universal property of monoidal slice} Let  $\O^{\otimes} \in \AlgPatt, \C^{\otimes}, \D^{\otimes} \in \Fbrs(\O^{\otimes})$ as well as $A \in \Alg_{\O}(\D)$ and $F \in \Fun_{\Fbrs(\O^{\otimes})}(\C^{\otimes}, \D^{\otimes})$. Let $p\colon \C^{\otimes} \to \O^{\otimes}$ denote the structure map. Then, there are equivalences:
    \begin{enumerate}[(i)]
        \item $\Fun_{\Fbrs(\O^{\otimes})}(\C^{\otimes}, \D_{/A}^{\otimes}) \simeq \Fun_{\Fbrs(\O^{\otimes})}(\C^{\otimes}, \D^{\otimes})_{/A \circ p}$.
        \item $\Map_{\Fun_{\Fbrs(\O^{\otimes})}(\C^{\otimes}, \D^{\otimes})}(F, A \circ p) \simeq \Map_{\Fbrs(\O^{\otimes})_{/\D^{\otimes}}}(\C^{\otimes}, \D^{\otimes}_{/A})$. 
    \end{enumerate}
\end{theorem}

\begin{proof}
    \hfill 
    \begin{enumerate}[(i)]
        \item We compute
        \begin{align*}
        \hspace{-19.82pt}\Fun_{\Fbrs(\O^{\otimes})}(\C^{\otimes}, \D_{/A}^{\otimes}) &\simeq \Fun_{\Fbrs(\O^{\otimes})}\left(\C^{\otimes}, (\D^{\otimes})_{/\O^{\otimes}}^{[1]} \times_{\D^{\otimes}} \O^{\otimes} \right)
        \\ &\simeq \Fun_{\Fbrs(\O^{\otimes})}\left(\C^{\otimes}, (\D^{\otimes}_{/\O})^{[1]} \right) \times_{\Fun_{\Fbrs(\O^{\otimes})}(\C^{\otimes}, \D^{\otimes})} \Fun_{\Fbrs(\O^{\otimes})}(\C^{\otimes}, \O^{\otimes})
        \\ &\simeq \Fun_{\Fbrs(\O^{\otimes})}(\C^{\otimes} \times [1], \D^{\otimes}) \times_{\Fun_{\Fbrs(\O^{\otimes})}(\C^{\otimes}, \D^{\otimes})} *
        \\ &\simeq \Fun_{\Fbrs(\O^{\otimes})}(\C^{\otimes}, \D^{\otimes})^{[1]} \times_{\Fun_{\Fbrs(\O^{\otimes})}(\C^{\otimes}, \D^{\otimes})} *
        \\ &\simeq \Fun_{\Fbrs(\O^{\otimes})}(\C^{\otimes}, \D^{\otimes})_{/A \circ p}.
    \end{align*}
    where we have used the tensor-cotensor adjunction (\cref{prop: omnibus tensor-cotensor}) a number of times. In the second equivalence we used that $\Fun_{\Fbrs(\O^{\otimes})}(\C^{\otimes}, -)$ is a right adjoint (\cref{prop: omnibus tensor-cotensor} \textcolor{chilligreen}{(ii)}) and thus preserves limits. In the third equivalence, we used a Yoneda-type argument to argue that the tensor-cotensor adjunction enriched to functor categories and the fourth equivalence is also by a Yoneda argument.
        \item Using (i) we obtain
        \begin{align*}
            \Map_{\Fun_{\Fbrs(\O)}(\C^{\otimes}, \D^{\otimes})}(F, A \circ p) &\simeq \{F \} \times_{\Fun_{\Fbrs(\O)}(\C^{\otimes}, \D^{\otimes})} \Fun_{\Fbrs(\O)}(\C^{\otimes}, \D^{\otimes})_{/A \circ p}
            \\ &\simeq \{F \} \times_{\Fun_{\Fbrs(\O)}(\C^{\otimes}, \D^{\otimes})} \Fun_{\Fbrs(\O)}(\C^{\otimes}, \D^{\otimes}_{/A})
            \\ &\simeq \Map_{\Fbrs(\O^{\otimes})_{/\D^{\otimes}}}(\C^{\otimes}, \D^{\otimes}_{/A})
        \end{align*}
        as desired.
    \end{enumerate}
\end{proof}

\begin{remark}
    Let $\O^{\otimes} \in \Op_{\infty}, \D^{\otimes} \in \Mon_{\O}(\Cat_{\infty})$ and $A \in \Alg_{\O}(\D)$. Then, our monoidal slice construction $\D^{\otimes}_{/A}$ agrees with Lurie's slice construction $\D_{/A_{\O}}$ in \cite[Notation 2.2.2.3]{lurie2017ha} because our universal property (\cref{theorem: universal property of monoidal slice} \textcolor{chilligreen}{(ii)}) specializes to a universal property that Lurie's monoidal slice also satisfies for ordinary $\infty$-operads \cite[Lemma 2.12]{antolinbarthel2019thom}.
\end{remark}

\noindent We end this section by giving a relation between operadic left Kan extensions and the slice monoidal structure. This is a generalization of \cite[Theorem 2.13]{antolinbarthel2019thom} and is the abstract version of the universal property of multiplicative equivariant Thom spectra that is to come in later sections (\cref{theorem: universal property of Th}).

\begin{proposition} \label{prop: abstract antolin-camarena-barthel}
    Let $\ul{\O}^{\otimes} \in \Op_{G, \infty}^{\NS}$ with $\ul{\C}^{\otimes}, \ul{\D}^{\otimes} \in \Mon_{\ul{\O}}^{\NS}(\ul{\Cat}_{G, \infty})$  and $A \in \Alg_{\ul{\O}}(\ul{\D})$. Let $p\colon \ul{\C}^{\otimes} \to \ul{\O}^{\otimes}$ denote the structure map and $F \in \Alg_{\ul{\C}/\ul{\O}}(\ul{\D})$. Suppose that $\ul{\D}^{\otimes}$ is $\ul{\O}$-distributive. Then, there is an equivalence 
    \[ \Map_{\Alg_{\ul{\O}}(\ul{\D})}(p_! F, A) \simeq \Map_{\Fbrs(\ul{\O}^{\otimes})_{/\ul{\D}^{\otimes}}}(\ul{\C}^{\otimes}, \ul{\D}_{/A}^{\otimes}) \]
    where we endow $\ul{\D}^{\otimes}_{/A}$ with the slice monoidal structure.
\end{proposition}

\begin{proof}
    We compute
    \begin{align*}
        \Map_{\Alg_{\ul{\O}}(\ul{\D})}(p_! F, A) &\simeq \Map_{\Alg_{\ul{\C}/\ul{\O}}(\ul{\D})}(F, A \circ p)
        \\ &\simeq \Map_{\Fbrs(\ul{\O}^{\otimes})_{/\ul{\D}^{\otimes}}}(F, \ul{\D}_{/A}^{\otimes} \to \ul{\D}^{\otimes})
    \end{align*}
    by operadic left Kan extension (\cref{theorem: operadic left Kan extension}) and the universal property of the monoidal slice category (\cref{theorem: universal property of monoidal slice}).
\end{proof}

\noindent Colloquially, a map $p_! F \to A$ corresponds to a lift
\begin{center}
        \begin{tikzcd}
            & \ul{\D}_{/A}^{\otimes} \arrow[d]
            \\ \ul{\C}^{\otimes} \arrow[r, "F", swap] \arrow[ur, dashed] & \ul{\D}^{\otimes}
        \end{tikzcd}
    \end{center}
over $\ul{\O}^{\otimes}$. This characterizes the operadic left Kan extension $p_! F$.

\subsection{Microcosmic Parametrized Monoidal Straightening-Unstraightening}\label{sec:microcosmic}
Let $X \in \Sc_{G}$, then (parametrized) straightening-unstraightening (\cref{prop: parametrized straightening}) gives rise to an equivalence $(\myuline{\Sc}_{G})_{/X} \simeq \myuline{\PSh}_{G}(X)$. For $\O^{\otimes} \in \Op_{G, \infty}$ and $X^{\otimes} \in \Alg_{\myuline{\O}}(\myuline{\Sc}_{G})$, both sides naturally enhance to $\O$-monoidal $\infty$-categories, namely via the slice monoidal structure (\cref{prop: slice is O-monoidal}) and (parametrized) Day convolution (\cref{theorem: omnibus day convolution}).
\medskip \\One naturally expects the equivalence to enhance to an $\O$-monoidal equivalence. Indeed, this statement has already been stated in the literature without proof \cite[Theorem A.6.1]{hahn2024equivariantnonabelianpoincareduality}. We will not need a result as strong as this, and will put together a weaker statement on a \emph{microcosmic} level with terminology inspired by \cite{ramzi2022monoidalgrothendieckconstructioninftycategories}. 
\medskip \\Here, the word 'microcosmic' essentially points to the result on the level of algebras. We don't claim much originality in this subsection, the claim will follow by putting together a number of results from the literature. Indeed, the following result from Stewart should already be viewed as a microcosmic parametrized monoidal straightening-unstraightening result. We also learned towards the end of the preparation of this article that our presentation is very similar to Cnossen's \cite[Proposition A.3]{cnossen2023twistedambidexterityequivarianthomotopy}, although we phrase our discussion in the more general setting of $G$-$\infty$-operads.

\begin{proposition}[{\cite[Corollary 1.52]{stewart2025tensorproductsequivariantcommutative}}] \label{prop: Alg(...) = Mon(...)}
    Let $\O^{\otimes} \in \Op_{G, \infty}$. There exist cartesian $G$-symmetric monoidal structures on $G$-$\infty$-categories with $G$-products for which there are equivalences
    \[ \Alg_{\myuline{\O}}(\myuline{\Cat}_{G, \infty}) \simeq \Mon_{\O}(\Cat_{\infty}) \quad \text{and} \quad \Alg_{\myuline{\O}}(\myuline{\Sc}_{G}) \simeq \Mon_{\O}(\Sc) \]
    where $\myuline{\Cat}_{G, \infty}$ and $\myuline{\Sc}_{G}$ are endowed with the cartesian $G$-symmetric monoidal structure.
\end{proposition}

\noindent For $\O^{\otimes} = \E_{\infty}^G$ this has also already appeared in \cite[Proposition A.23]{hilman2024mcduff}.

\begin{theorem}[{\cite[Section 3]{nardinshah2022equivarianttopos}}] \label{theorem: omnibus day convolution}
    Consider $\O^{\otimes} \in \Op_{G, \infty}$ as well as $\C^{\otimes}, \D^{\otimes} \in \Mon_{\O}(\Cat_{\infty})$. Suppose that $\D^{\otimes}$ is $\ul{\O}$-distributive. Then, there exists an $\O$-monoidal $\infty$-category $\tb{\Fun^{\O}(\C, \D)^{\otimes}}$ whose underlying $G$-$\infty$-category is $\myuline{\Fun}_{G}(\C, \D)$ such that 
    \[ \Alg_{\myuline{\O}/\myuline{\O}}\left(\myuline{\Fun}^{\O}(\C, \D) \right) \simeq \Alg_{\ul{\C}/\ul{\O}}(\ul{\D}). \]
\end{theorem}

\noindent For example, $\myuline{\Sc}_{G}^{\times}$ is distributive (\cref{example: S Cat distributive}), so also its pullback along $\O^{\otimes} \to \Span(\F_{G})$ by \cref{lemma: distributivity along pullback}, so for $\O^{\otimes} \in \Op_{G, \infty}$ and $X \in \Mon_{\O}(\Sc)$ we find $\tb{\PSh_{G}^{\O}(X)^{\otimes}} \in \Mon_{\O}(\Sc)$ whose underlying $G$-$\infty$-category is $\myuline{\PSh}_{G}(X)$.

\begin{corollary} \label{theorem: microcosmic straightening-unstraightening}
    Let $\O^{\otimes} \in \Op_{G, \infty}$
    \begin{enumerate}[(i)]
        \item  Let $\C^{\otimes} \in \Mon_{\O}(\Cat_{\infty})$. Then, $\Alg_{\myuline{\O}/\myuline{\O}}\left(\myuline{\Fun}^{\O}(\myuline{\C}, \myuline{\Cat}_{G, \infty}) \right) \simeq \Mon_{\C}(\Cat_{\infty})$.
        \item Let $X^{\otimes} \in \Mon_{\O}(\Sc)$. Then, $\Alg_{\ul{\O}/\ul{\O}} (\ul{\PSh}_{G}^{\O}(X) ) \simeq \Mon_X(\Sc) \simeq \Mon_{\O}(\Sc)_{/X^{\otimes}}$.
    \end{enumerate}
    
\end{corollary}

\begin{proof}
    \hfill 
    \begin{enumerate}[(i)]
        \item By the universal property of Day convolution (\cref{theorem: omnibus day convolution}) we compute 
    \begin{align*}
        \Alg_{\myuline{\O}/\myuline{\O}}\left(\ul{\Fun}^{\O}(\myuline{\C}, \myuline{\Cat}_{G, \infty}) \right) &\simeq \Alg_{\myuline{\C}/\myuline{\O}}(\myuline{\O}^{\otimes} \times_{\myuline{\F}_{G, *}} \myuline{\Cat}_{G, \infty}^{\times})
        \\ &\simeq \Alg_{\myuline{\C}}\left(\myuline{\Cat}_{G, \infty} \right)
        \\ &\simeq \Mon_{\C}(\Cat_{\infty})
    \end{align*}
    where we use \cref{prop: Alg(...) = Mon(...)} in the last step. Note in the first equivalence that we take the $\O$-monoidal version of $\ul{\Cat}_{G, \infty}$. Then, the second equivalence is for example via  $\Fbrs(\ul{\O}^{\otimes}) \simeq \Fbrs(\ul{\F}_{G, *})_{/\O^{\otimes}}$ \cite[Corollary 4.1.17]{barkanhaugsengsteinebrunner2024envelopesalgebraicpatterns}.

    \item By the universal property of Day convolution (\cref{theorem: omnibus day convolution}) and abusing $X \simeq X^{\op}$ we can thus compute 
    \begin{align*}
        \Alg_{\myuline{\O}/\myuline{\O}}\left(\myuline{\PSh}_{G}^{\O}(X) \right) &\simeq \Alg_{\myuline{X}/\myuline{\O}}(\myuline{\O}^{\otimes} \times_{\myuline{\F}_{G, *}}\myuline{\Sc}_{G}^{\times}) 
        \\ &\simeq \Alg_{\myuline{X}}(\myuline{\Sc}_{G}) 
        \\ &\simeq \Mon_{X}(\Sc)
    \end{align*}
    where we use \cref{prop: Alg(...) = Mon(...)} in the last step. Moreover, we get $\Mon_X(\Sc) \simeq \Mon_{\O}(\Sc)_{/X^{\otimes}}$ because over spaces every edge is coCartesian.\qedhere
    \end{enumerate}
\end{proof}

\begin{corollary} \label{corollary: real microcosmic}
    Let $\O^{\otimes} \in \Op_{G, \infty}$ and $X^{\otimes} \in \Mon_{\Span(\F_G)}(\Sc)$. Then, 
    \[ \Alg_{\ul{\O}}(\ul{\PSh}^{\O}_G(X)) \simeq \Alg_{\ul{\O}}(\ul{\Sc}^G_{/X}) \] 
    where $\ul{\Sc}^G_{/X}$ is endowed with the slice monoidal structure (\cref{prop: slice is O-monoidal}).
\end{corollary}

\begin{proof}
    Let us first remark that $\ul{\PSh}_G(X)^{\otimes} \times_{\F_{G,*}} \ul{\O}^{\otimes} \simeq \ul{\PSh}^{\O}(X)^{\otimes}$ where the second term implicitly uses $X^{\otimes} \times_{\F_{G,*}} \ul{\O}^{\otimes}$, i.e. the underlying $\ul{\O}$-monoidal space of $X^{\otimes}$. This follows by the explicit construction of Day convolution through norms \cite[Definition 3.1.6]{nardinshah2022equivarianttopos} and one can compare them through the universal property of these norms \cite[Proposition 3.1.7]{nardinshah2022equivarianttopos}.
    \medskip \\With this in our arsenal, we can write out
    \begin{align*}
        \Alg_{\ul{\O}} \left(\ul{\PSh}_G(X) \right) &\simeq \Alg_{\ul{\O}/\ul{\O}}\left(\ul{\PSh}_G(X)^{\otimes} \times_{\F_{G,*}} \ul{\O}^{\otimes} \right) 
        \\ &\simeq \Alg_{\ul{\O}/\ul{\O}}(\ul{\PSh}_G^{\O}(X)) 
        \\ &\simeq \Alg_{\ul{\O}}(\ul{\Sc}^G)_{/X^{\otimes}} 
        \\ &\simeq \Alg_{\ul{\O}}(\ul{\Sc}^G_{/X})
    \end{align*}
    where we use the previous result (\cref{theorem: microcosmic straightening-unstraightening}) and the universal property of the slice monoidal structure (\cref{theorem: universal property of monoidal slice} \textcolor{chilligreen}{(i)}).
\end{proof}

\noindent This is really the result that should be called microcosmic straightening-unstraightening. It is the equivalence you obtain by applying $\Alg_{\O}$ to a stronger macrocosmic straightening-unstraightening $\ul{\PSh}_G^{\O}(X)^{\otimes} \simeq (\ul{\Sc}^G_{/X})^{\otimes}$.

\subsection{Monoidality of Parametrized Left Module Categories}\label{sec:monoidal_LMod}
We equivariantize Lurie's $\infty$-category of left modules $\LMod_A(\C)$ via parametrized higher algebra to $\LMod_A^G(\C)$. By adapting a criterion about coCartesian fibrations of Haugseng--Melani--Safronov \cite[Lemma A.1.8]{haugsengmelanisafronov2022shiftedisotropic} to the parametrized setting we endow $\LMod_A^G(\C)$ with a multiplicative structure. We then spend much effort in showing that this multiplicative structure is distributive in the sense of Nardin--Shah (\cref{def: distributivity}) culminating in the most technical proof of this article. Throughout, we are careful to make this work for more $G$-$\infty$-operads than only the terminal one.
\medskip \\Let us first recall parts of the classical setting. Let $\C^{\otimes}$ be a symmetric monoidal $\infty$-category. There is a preferred map between $\infty$-operads $\E_1 \to \LM$ such that for $A \in \Alg_{\E_1}(\C)$ the pullback
\begin{center}
    \begin{tikzcd}
        \LMod_A(\C) \arrow[r] \arrow[d] \arrow[dr, phantom, very near start, "\lrcorner"] & \Alg_{\LM}(\C) \arrow[d]
        \\ * \arrow[r, "A", swap] & \Alg_{\E_1}(\C)
    \end{tikzcd}
\end{center}
is the $\infty$-category of left $A$-modules in $\C$ \cite[Definition 4.2.1.13]{lurie2017ha}. Recall that $\Alg_{\LM}(\C)$ is the $\infty$-category with objects $(R,M)$ encoding an $\E_1$-algebra $R$ that acts on some object $M \in \C$. Observe that $\LMod_A(\C)$ is the fiber of a coCartesian fibration:
\begin{lemma}[Lurie] \label{lemma: AlgLM ---> AlgE1 is coCartesian fibration}
    Let $\C^{\otimes}$ be a symmetric monoidal $\infty$-category whose underlying $\infty$-category $\C$ admits geometric realizations and whose tensor product distributes over these.
    \begin{enumerate}[(i)]
        \item Then, $\Alg_{\LM}(\C) \to \Alg_{\E_1}(\C)$ is a coCartesian fibration.
        \item A coCartesian edge over $A \to B$ starting in $(A, M)$ is given by $(A,M) \to (B, B \otimes_A M)$.
    \end{enumerate} 
\end{lemma}

\begin{proof}
    This is \cite[Lemma 4.5.3.6, Proposition 4.6.2.17]{lurie2017ha}.
\end{proof}

\noindent Lurie then proves the existence of a monoidal structure in suitable setups \cite[Corollary 4.8.5.20]{lurie2017ha} and goes on with life.

\begin{remark}
    This is not exactly what Lurie does in \cite[Definition 4.2.1.13]{lurie2017ha}. His approach allows for more general inputs than a (symmetric) monoidal $\infty$-category $\C^{\otimes}$ but in case the input is a (symmetric) monoidal $\infty$-category $\C$, then his construction is equivalent to the one claimed above.
    \medskip \\In that case, Lurie takes the same pullback square but with $\Fun^{\lax}_{/\E_1}(\LM^{\otimes}, \C^{\otimes} \times_{\F_*} \E_1^{\otimes})$ instead of $\Alg_{\LM}(\C)$ \cite[Example 4.2.1.16]{lurie2017ha}. We show that these are actually equivalent. Consider the composite of pullback squares
    \begin{center}
        \begin{tikzcd}
            \Fun_{/\E_1}^{\lax}(\LM^{\otimes}, \C^{\otimes} \times_{\F_*} \E_1^{\otimes}) \arrow[r] \arrow[d] \arrow[dr, phantom, very near start, "\lrcorner"] & \Fun_{/\F_*}^{\lax}(\LM^{\otimes}, \C^{\otimes} \times_{\F_*} \E_1^{\otimes}) \arrow[r] \arrow[d] \arrow[dr, phantom, very near start, "\lrcorner"] & \Fun_{/\F_*}^{\lax}(\LM^{\otimes}, \C^{\otimes}) \arrow[d]
            \\ * \arrow[r] & \Fun_{/\F_*}^{\lax}(\LM^{\otimes}, \E_1^{\otimes}) \arrow[r] & \Fun_{/\F_*}^{\lax}(\LM^{\otimes}, \F_*) 
        \end{tikzcd}
    \end{center}
    where the left pullback is for example \cite[Definition 2.1.3.1]{lurie2017ha}. The right square is a pullback square since $\Fun_{\Op_{\infty}}(\LM^{\otimes}, -)$ preserves pullbacks. We conclude that the composite rectangle is a pullback diagram and since $\Fun_{/\F_*}^{\lax}(\LM^{\otimes}, \F_*) \simeq *$, we obtain
    \[ \Fun_{/\E_1}^{\lax}(\LM^{\otimes}, \C^{\otimes} \times_{\F_*} \E_1^{\otimes}) \simeq \Fun_{/\F_*}^{\lax}(\LM^{\otimes}, \C^{\otimes}) = \Alg_{\LM}(\C) \]
    demonstrating that our exposition agrees with Lurie's.
\end{remark}

\noindent Our idea is to refine the pullback square in two ways, namely in a parametrized direction and in a higher algebra direction.

\begin{construction} \label{construction: LMod}
    Let $\C^{\otimes} \in \Mon_{\Span(\F_{G})}(\Cat_{\infty})$ with $\O^{\otimes} \in \Op_{G, \infty}$ and $A \in \Alg_{\O \otimes \E_1}(\C)$ corresponding to a map $\O^{\otimes} \to \Alg_{\E_1}(\C)^{\otimes}$ in $\Op_{G, \infty}$ by adjunction (\cref{theorem: omnibus Alg theorem}). Furthermore, note that $\Alg_{\LM}(\C)^{\otimes} \to \Alg_{\E_1}(\C)^{\otimes}$ is a map in $\Op_{G, \infty}$ since they have the pointwise monoidal structure (\cref{theorem: omnibus Alg theorem}). We denote by $\tb{\LMod_A^{G}(\C)^{\otimes}}$ the pullback
    \begin{center}
        \begin{tikzcd}
            \LMod_A^{G}(\C)^{\otimes} \arrow[r] \arrow[d] \arrow[dr, phantom, very near start, "\lrcorner"] & \Alg_{\LM}(\C)^{\otimes} \arrow[d]
            \\ \O^{\otimes} \arrow[r] & \Alg_{\E_1}(\C)^{\otimes}
        \end{tikzcd}
    \end{center}
    taken in $\Op_{G, \infty}$.
\end{construction}

\noindent For $H \leq G$, this recovers a fiberwise left module construction 
\[ \LMod_A^{G}(\C)^{\otimes}_H \simeq \LMod_{A_H}(\C_H) \] 
by \cref{theorem: omnibus Alg theorem} \textcolor{chilligreen}{(iii)} whose restrictions are computed in the underlying $G$-$\infty$-category of $\C^{\otimes}$ by another application of \cref{theorem: omnibus Alg theorem} \textcolor{chilligreen}{(iii)}, so in particular, the underlying $G$-$\infty$-category $\myuline{\LMod}_A^{G}(\C)^{\otimes}$ can be viewed as a $G$-parametrized version of Lurie's left module category.
\begin{remark}
    There are already other parametrized treatments of module categories, for example by Linskens--Nardin--Pol \cite[Appendix A]{linskens2022global} or Pützstück \cite[Section 5]{pützstück2025globalpicardspectraborel}. Both of these references have global equivariant applications in mind and in particular focus on the fully commutative ('ultracommutative' or 'normed') setup, while our intention is to give a general $\O$-monoidal treatment in a non-global setup. However, the ideas in \cite{linskens2022global} are quite close to our treatment.
\end{remark}
\noindent We will now demonstrate that our construction endows $\LMod_A^G(\C)$ with an $\O$-monoidal structure. Its proof depends on the following result generalizing \cite[Lemma 1.10]{ramzi2022monoidalgrothendieckconstructioninftycategories} which is based on \cite[Lemma A.1.8]{haugsengmelanisafronov2022shiftedisotropic}. We also took inspiration from \cite[Lemma B.6]{reutter2025enrichedinftycategoriesmarkedmodule}.

\begin{lemma}
    \label{lemma: criterion for map between Segal objects to be a coCartesian fibration} Let $\O \in \Op_{G, \infty}$ with $\C^{\otimes}, \D^{\otimes} \in \Mon_{\O}(\Cat_{\infty})$ and let $F\colon \C^{\otimes} \to \D^{\otimes}$ be a map in $\Mon_{\O}(\Cat_{\infty})$. Suppose the following two conditions:
    \begin{enumerate}[(i)]
        \item For every $H \leq G$ and $o \in \O_H^{\otimes}$ the induced functor $F_o\colon  \C_o^{\otimes} \to \D_o^{\otimes}$ is a coCartesian fibration.
        \item Let $o \to o'$ be a map in $\O^{\otimes}$ such that $o'$ lies over an orbit. Then, $\bigotimes_{o \to o'}\colon  \C_o^{\otimes} \to \C_{o'}^{\otimes}$ sends $F_o$-coCartesian edges to $F_{o'}$-coCartesian edges.
    \end{enumerate}
    Then, $\C^{\otimes} \to \D^{\otimes}$ is a coCartesian fibration.
\end{lemma}

\begin{proof}
    The criterion given in \cite[Lemma A.1.8]{haugsengmelanisafronov2022shiftedisotropic} specialized to our setting states that the following conditions altogether imply that $\C^{\otimes} \to \D^{\otimes}$ is a coCartesian fibration.
    \begin{enumerate}[(1)]
        \item The functors $\C^{\otimes} \to \O^{\otimes}$ and $\D^{\otimes} \to \O^{\otimes}$ are coCartesian fibrations.
        \item The functor $F\colon  \C^{\otimes} \to \D^{\otimes}$ preserves coCartesian edges over $\O^{\otimes}$.
        \item For each $o \in \O^{\otimes}$ the induced map $F_o\colon  \C_o^{\otimes} \to \D_o^{\otimes}$ is a coCartesian fibration.
        \item Let $o \to o'$ be a map in $\O^{\otimes}$. Then, the induced map $\C_o^{\otimes} \to \C_{o'}^{\otimes}$ takes $F_o$-coCartesian edges to $F_{o'}$-coCartesian edges.
    \end{enumerate}
    Parts (1) and (2) are already contained in our assumptions. So let us confirm (3) and (4).
    \begin{enumerate}
        \item[(3)] Suppose $o \in \O_X^{\otimes}$ where we have an orbit decomposition $X \simeq \coprod_i G/H_i$ with $H_i \leq G$. In that regard, $\O_X^{\otimes} \simeq \prod_i \O_{H_i}^{\otimes}$ and $o$ corresponds to some tuple $(o_i)_i$ under this correspondence. Then, the map becomes %\marginpar{\tiny \raggedright Do I need to spell out the Segal condition more explicitly?}
        \[ F_o\colon  \C_o^{\otimes} \simeq \prod_i \C_{o_i}^{\otimes} \xrightarrow{\prod_i F_{o_i}} \prod_i \D_{o_i}^{\otimes} \simeq \D_o^{\otimes} \]
        which is a product of coCartesian fibrations by (i) and thus itself a coCartesian fibration.
        \item[(4)] Suppose $o' \in \O_Y^{\otimes}$ with orbit decompositions $Y \simeq \prod_j G/K_j$. Then, $o'$ correspond to $(o'_j)_j$ and the map in question becomes
        \[ \C_o^{\otimes} \to \C_{o'}^{\otimes} \to \prod_{j} \C_{o_j}^{\otimes}. \]
        So it suffices that the component maps $\C_o^{\otimes} \to \C_{o_j}^{\otimes}$ send $F_o$-coCartesian edges to $F_{o_j}$-coCartesian edges which is assumption (ii).
    \end{enumerate}
    So we win.
\end{proof}

\begin{corollary} \label{corollary: criterion for map between symmetric monoidal categories to be cocartesian fibration}
    Let $\C^{\otimes}, \D^{\otimes} \in \Mon_{\Span(\F_{G})}(\Cat_{\infty})$ and consider a map $F\colon \C^{\otimes} \to \D^{\otimes}$ in $\Mon_{\Span(\F_G)}(\Cat_{\infty})$. Suppose the following two conditions:
    \begin{enumerate}[(i)]
        \item For every $H \leq G$ the induced functor $F_H\colon  \C_H^{\otimes} \to \D_H^{\otimes}$ is a coCartesian fibration.
        \item Norms, fiberwise tensor products and restrictions preserve coCartesian edges, i.e. for $H \leq K \leq G$ the maps
        \[ \Res_H^K\colon  \C_K^{\otimes} \to \C_H^{\otimes}, \ \otimes\colon  \C_H^{\otimes} \times \C_H^{\otimes} \to \C_H^{\otimes}, \ N_H^K\colon  \C_H^{\otimes} \to \C_K^{\otimes} \]
        preserve coCartesian edges over the respective fibers of $\D^{\otimes}$.
    \end{enumerate}
    Then, $\C^{\otimes} \to \D^{\otimes}$ is a coCartesian fibration.
\end{corollary}

\begin{proof}
     We want to check the conditions in \cref{lemma: criterion for map between Segal objects to be a coCartesian fibration}. Note here that $\bigoplus_{G/K \leftarrow G/H} = \Res_H^K$. %\marginpar{\tiny \raggedright Maybe come up with better notation. Should maybe also just rewrite the entire proof. Need to also note that equivalences can be ignored which is why we can take subgroups.} 
     An induction argument reduces allows us to only demand fiberwise tensor products of two objects. We furthermore note that the following are automatic:
    \begin{itemize}
            \item The forwards maps $\emptyset \to G/H$ induce units $\one\colon  * \to \C_{H}^{\otimes}$ and this certainly preserves coCartesian edges since it sends $\id_*$ to $\id_{\one}$.
            \item The backwards maps $G/H \leftarrow \emptyset$ induce maps $\C_H^{\otimes} \to *$ which preserves coCartesian edges since every edge in $*$ is $\id_*$-coCartesian.
            \item The backwards fold %\marginpar{\tiny \raggedright Might need a reference for this or for all the bullet points. It should follow from some Segal conditions. e.g.~in Bastiaan's lecture notes, these are part of the definition of an operad.} 
            maps $G/H \leftarrow \coprod_i G/H$ induce diagonal maps $\C_H^{\otimes} \to \prod_i \C_H^{\otimes}$ which hence preserve coCartesian edges.
        \end{itemize}
        This is why there is no need to furthermore demand these in (ii).   
\end{proof}

\begin{theorem} \label{theorem: LMod ---> O is O-monoidal}
    Let $\C^{\otimes} \in \Mon_{\Span(\F_G)}(\Cat_{\infty})$ whose underlying $G$-$\infty$-category has fiberwise geometric realizations and whose norms, fiberwise tensorings with one object and restrictions commute over these fiberwise geometric realizations. Let $\O^{\otimes} \in \Op_{G, \infty}$ and $A \in \Alg_{\O \otimes \E_1}(\C)$. Then, $\LMod_A^{G}(\C)^{\otimes} \to \O^{\otimes}$ is an $\O$-monoidal $\infty$-category.
\end{theorem}

\begin{proof}
    It suffices to prove that $\LMod_A^{G}(\C)^{\otimes} \to \O^{\otimes}$ is a coCartesian fibration. Since coCartesian fibrations pull back, it suffices to show that $\Alg_{\LM}(\C)^{\otimes} \to \Alg_{\E_1}(\C)^{\otimes}$ is a coCartesian fibration.
    \medskip \\To do so we will check the criteria in \cref{corollary: criterion for map between symmetric monoidal categories to be cocartesian fibration}.
    \begin{enumerate}[(i)]
        \item By \cref{theorem: omnibus Alg theorem} \textcolor{chilligreen}{(i)}, for $H \leq G$ we have
        \[ \Alg_{\LM}(\C)^{\otimes}_H \simeq \Alg_{\LM}(\C_H) \to \Alg_{\E_1}(\C_H)\simeq \Alg_{\E_1}(\C)^{\otimes}_H. \]
        Using the assumption that fiberwise tensorings commute with fiberwise geometric realizations it follows that this is a coCartesian fibration from \cref{lemma: AlgLM ---> AlgE1 is coCartesian fibration}.
        \item We now have to show that norms, restrictions and fiberwise tensor products preserve coCartesian edges. We do this via the explicit description of coCartesian edges (\cref{lemma: AlgLM ---> AlgE1 is coCartesian fibration} \textcolor{chilligreen}{(ii)}).
        \begin{itemize}
            \item Fiberwise tensor products: Let $H \leq G$ and consider a coCartesian edge 
            \[ ((A_1, M_1), (A_2, M_2)) \to ((B_1, B_1 \otimes_{A_1} M_1), (B_2, B_2 \otimes_{A_2} M_2)) \]
            in $\Alg_{\LM}(\C)^{\otimes}_H \times \Alg_{\LM}(\C)^{\otimes}_H$ over $(A_1, A_2) \to (B_1, B_2)$ in $\Alg_{\E_1}(\C)^{\otimes}_H \times \Alg_{\E_1}(\C)^{\otimes}_H$. We wish to show that 
            \begin{align*}
                (A_1, M_1) \otimes (A_2, M_2) &\simeq (A_1 \otimes A_2, M_1 \otimes M_2)
                \\ &\to \left(B_1 \otimes B_2, (B_1 \otimes_{A_1} M_1) \otimes (B_2 \otimes_{A_2} M_2) \right)
                \\ &\simeq (B_1, B_1 \otimes_{A_1} M_1) \otimes (B_2, B_2 \otimes_{A_2} M_2)
            \end{align*}
            is a coCartesian edge over $A_1 \otimes A_2 \to B_1 \otimes B_2$. For this, we need that the natural map
            \[ (B_1 \otimes B_2) \otimes_{A_1 \otimes A_2} (M_1 \otimes M_2) \to (B_1 \otimes_{A_1} M_1) \otimes (B_2 \otimes_{A_2} M_2) \]
            is an equivalence. This follows by writing out the bar construction and using (repeatedly) that tensoring with one object preserves geometric realizations and is symmetric monoidal.
            \item Norms:
            Let $H \leq G$ and consider a coCartesian edge 
            \[ (A, M) \to (B, B \otimes_A M) \] 
            in $\Alg_{\LM}(\C)^{\otimes}_H$ over $A \to B$ in $\Alg_{\E_1}(\C)_H^{\otimes}$. We wish to show that
            \[ (N_H^{G} A, N_H^{G}M) \simeq N_H^{G}(A, M) \to N_H^{G}(B, B \otimes_A M) \simeq (N_H^{G}B, N_H^{G}(B \otimes_A M)) \]
            is a coCartesian edge over $N_H^{G}A \to N_H^{G}B$. For this, we need that the natural map $N_H^{G}B \otimes_{N_H^{G}A} N_H^{G}M \to N_H^{G}(B \otimes_A M)$ is an equivalence. But that's true because $N_H^{G}$ is symmetric monoidal (\cref{lemma: Res N sym mon}) and commutes with geometric realizations by assumption. 
            \item Restrictions: This  is the exact same argument as for norms using \cref{lemma: Res N sym mon} again.
        \end{itemize}
    \end{enumerate}
    This finishes the proof.
\end{proof}

\begin{remark} \label{remark: AlgLM ---> AlgE1 is cartesian fibration}
    The same strategy shows that $\Alg_{\LM}(\C)^{\otimes} \to \Alg_{\E_1}(\C)^{\otimes}$ is a cartesian fibration.
\end{remark}

\begin{corollary} \label{corollary: functoriality of LMod in A}
    Let $\C^{\otimes} \in \Mon_{\Span(\F_{G})}(\Cat_{\infty})$ and $\O^{\otimes} \in \Op_{G, \infty}$. Then, taking left modules enhances to a functor
    \[ \LMod_{(-)}^{G}(\C)^{\otimes}\colon  \Alg_{\O \otimes \E_1}(\C) \to \Alg_{\O}(\Cat_{\infty}) \]
    where the functoriality is through relative tensor products.
\end{corollary}

\begin{proof}
    We have seen that $\Alg_{\LM}(\C)^{\otimes} \to \Alg_{\E_1}(\C)^{\otimes}$ is a coCartesian fibration and symmetric monoidal (\cref{theorem: LMod ---> O is O-monoidal}).  By parametrized microcosmic monoidal straightening-unstraightening (\cref{prop: Alg(...) = Mon(...)}) this corresponds to a lax $G$-symmetric monoidal functor $\myuline{\Alg}_{\E_1}(\C)^{\otimes} \to \myuline{\Cat}_{G, \infty}^{\times}$. We win by applying $\Alg_{\O}(-)$.
\end{proof}

\noindent Those conditions required in this theorem will be carried around in most of the subsequent results so that we have an $\O$-monoidal structure on $\LMod_A^{G}(\C)$, so let us introduce a terminology for it.
\begin{definition} \label{def: O-module datum}
    Let $\O^{\otimes} \in \Op_{G, \infty}$. An \tb{$\O$-module datum} consists of a pair
    \[ \left(\C^{\otimes} \in \Mon_{\Span(\F_G)}(\Cat_{\infty}), \ A \in \Alg_{\O \otimes \E_1}(\C) \right) \] 
    such that the underlying $G$-$\infty$-category of $\C^{\otimes}$ has fiberwise geometric realizations and whose norms, fiberwise tensorings with one object and restrictions commute over these fiberwise geometric realizations.
\end{definition}
\noindent So the content of \cref{theorem: LMod ---> O is O-monoidal} is that an $\O$-module datum $(\C^{\otimes}, A)$ gives rise to an $\O$-monoidal $\infty$-category $\LMod_A^G(\C)^{\otimes}$. 
\medskip \\The device running the Thom spectrum engine later is operadic left Kan extensions (\cref{theorem: operadic left Kan extension}), which depends on a distributivity property (\cref{def: distributivity}), so we will spend the rest of the remaining section showing that the our left module category is $\O$-distributive. Since distributivity is treated in the setting of $G$-$\infty$-categories, we will work with the underlying $G$-symmetric monoidal $G$-$\infty$-categories in the rest of the section.
\medskip \\To talk about distributivity one already needs the existence of colimits, which is the main content of the following result. In particular, this needs a certain projection formula condition for which we recall that it is implied by distributivity (\cref{prop: projection formula from distributivity}).
\begin{proposition} \label{prop: categorical properties of LMod}
    Let $(\C^{\otimes}, A)$ be an $\E_0$-module datum. 
    \begin{enumerate}[(i)]
        \item The forgetful functor $\myuline{\LMod}_A^{G}(\C) \to \myuline{\C}$ is conservative.
        \item Let $A \to B$ be a map of $\E_1$-algebras. Then, the levelwise relative tensor products assemble into a $G$-left adjoint $B \otimes_A -\colon  \myuline{\LMod}_A^{G}(\C) \to \myuline{\LMod}_B^{G}(\C)$.
        \item Let $\myuline{\C}$ be $G$-presentable and levelwise distributive. Assume moreover that projection formulas hold, i.e. for $H \leq K \leq G$ and objects $c \in \C_H, c' \in \C_{K}$ the preferred map
        \[ \Ind_{H}^{K} \left(\Res_H^{K}c' \otimes c \right) \to c' \otimes \Ind_H^{K}c \]
        is an equivalence. Then, the $G$-$\infty$-category $\myuline{\LMod}_A^{G}(\C)$ is $G$-presentable and the $G$-functor $\myuline{\LMod}_A^{G}(\C) \to \myuline{\C}$ strongly preserves and strongly reflects $G$-colimits.
    \end{enumerate}
\end{proposition}

\begin{proof}
    \hfill 
    \begin{enumerate}[(i)]
        \item Conservativity is a levelwise statement where it is \cite[Corollary 4.2.3.2]{lurie2017ha}.
        \item The coCartesian fibration $\myuline{\Alg}_{\LM}(\C) \to \myuline{\Alg}_{\E_1}(\C)$ (see \cref{theorem: LMod ---> O is O-monoidal}) straightens to a $G$-functor $\myuline{\Alg}_{\E_1}(\C) \to \myuline{\Cat}_{G, \infty}$ and is the straightened version of the $\LMod$ construction. So the map $A \to B$ yields a $G$-functor, which we denote by
        \[ B \otimes_A -\colon  \myuline{\LMod}_A^{G}(\C) \to \myuline{\LMod}_B^{G}(\C). \] Levelwise, this is the classical relative tensor products, i.e. by $B_H \otimes_{A_H} -$ on the level $H \leq G$. So it is a relative adjunction over $\Orb_G^{\op}$ with levelwise right adjoint given by restriction, since this is levelwise so \cite[Proposition 7.3.2.6]{lurie2017ha}. To check that it is a $G$-adjunction we still need to check that the right adjoints assemble into a $G$-functor \cite[Corollary 2.2.7]{hilman2024parametrisedpresentability}, but they do because they come from the cartesian straightening of the cartesian fibration $\myuline{\Alg}_{\LM}(\C) \to \myuline{\Alg}_{\E_1}(\C)$ (see \cref{remark: AlgLM ---> AlgE1 is cartesian fibration}).
        \item We begin with $G$-cocompleteness, which is equivalent to showing that $\myuline{\LMod}_A^{G}(\C)$ is fiberwise cocomplete, that restrictions preserve fiberwise colimits and that restrictions admit left adjoints satisfying the Beck--Chevalley condition \cite[Theorem 3.1.9]{hilman2024parametrisedpresentability}. We first note that the assumptions are enough to conclude that $\myuline{\LMod}_A^{G}(\C)$ is fiberwise presentable \cite[Corollary 4.2.3.7(1)]{lurie2017ha} and in particular fiberwise cocomplete.
        \medskip \\Now let $H \leq K \leq G$, then $\LMod_{A_{K}}(\C_{K}) \to \LMod_{A_H}(\C_H)$ preserves colimits because the forgetful functor to $\C_{K}$ resp. $\C_H$ creates colimits \cite[Corollary 4.2.3.7(2)]{lurie2017ha} but $\C_{K} \to \C_H$ preserves colimits by $G$-cocompleteness of $\myuline{\C}$. Altogether, $\myuline{\LMod}_A^G(\C)$ is $G$-presentable \cite[Theorem 6.1.2]{hilman2024parametrisedpresentability}.
        \medskip \\The same argument shows that $\LMod_{A_{K}}(\C_{K}) \to \LMod_{A_H}(\C_H)$ also preserves limits. The adjoint functor theorem thus ensures a left adjoint $\ind_H^{K}\colon  \LMod_{A_H}(\C_H) \to \LMod_{A_{K}}(\C_{K})$. We will first show that it is computed underlying, i.e. that the square
        \begin{center}
            \begin{tikzcd}
                \LMod_{A_H}(\C_H) \arrow[r, "\ind_H^{K}"] \arrow[d, "U", swap] &\LMod_{A_{K}}(\C_{K}) \arrow[d, "U"]
                \\ \C_H \arrow[r, "\Ind_H^{K}", swap] & \C_{K}
            \end{tikzcd}
        \end{center}
        commutes. Since $\LMod_{A_H}(\C_H)$ is generated by free $A_H$-modules under sifted colimits \cite[Proposition 4.7.3.14]{lurie2017ha} and all functors in this square preserve colimits, it suffices to check this on free $A_H$-modules. So let $c \in \C_H$ and $A_H \otimes c \in \LMod_{A_H}(\C_H)$ be the free $A_H$-module on $c$. Let $M \in \LMod_{A_{K}}(\C_{K})$. For the sake of the following computation, let us introduce the notation $\res_H^K\colon  \LMod_{A_K}(\C_K) \to \LMod_{A_H}(\C_H)$. Then, we compute 
        \begin{align*}
            \Map_{\LMod_{A_{K}}(\C_{K})}(\ind_H^{K}(A_H \otimes_{\one} c), M) &\simeq \Map_{\LMod_{A_H}(\C_H)}(A_H \otimes_{\one} c, \res_H^{K} M)
            \\ &\simeq \Map_{\C_H}(c, U\res_H^{K}M)
            \\ &\simeq \Map_{\C_H}(c, \Res_H^{K} UM)
            \\ &\simeq \Map_{\C_{K}}(\Ind_H^{K}c, UM)
            \\ &\simeq \Map_{\LMod_{A_{K}}(\C_{K})}(A_{K} \otimes_{\one} \Ind_H^{K}c, M).
        \end{align*}
        Thus, we discover $\ind_H^{K}(A_H \otimes c) \simeq A_{K} \otimes \Ind_H^{K}c$. By 
        %\marginpar{\tiny \raggedright Is the word 'natural' enough to make this coherent?} 
        the projection formula we may thus write down a chain of natural equivalences
        \begin{align*}
            U\ind_H^{K}(A_H \otimes_{\one} c) &\simeq U(A_{K} \otimes_{\one} \Ind_H^{K}c) 
            \\ &\simeq U A_{K} \otimes U \Ind_{H}^{K}c
            \\ &\simeq UA_K \otimes \Ind_K^H Uc
            \\ &\simeq (UA_H)_K \otimes \Ind_K^H Uc
            \\ &\simeq \Ind_H^{K}(UA_H \otimes Uc) 
            \\ &\simeq \Ind_H^{K}U(A_H \otimes_{\one} c).
        \end{align*}
        Since restrictions and inductions of left module categories are computed underlying, we deduce that the Beck--Chevalley maps are computed underlying where it is an equivalence because $\C$ is $G$-cocomplete. On the other hand, the forgetful functor is conservative by (i), so the Beck--Chevalley maps for $\myuline{\LMod}_A^G$ are also equivalences.
        \medskip \\To see that $\myuline{\LMod}_A^{G}(\C) \to \myuline{\C}$ strongly preserves $G$-colimits, we need to see that this is fiberwise the case and that induction is computed underlying. The fiberwise part is \cite[Corollary 4.2.3.7(2)]{lurie2017ha} again and the induction part was the commutative diagram above.
        \medskip \\Since the forgetful functor strongly preserves $G$-colimits and is conservative, it also strongly reflects $G$-colimits \cite[Lemma 2.2.12]{hilman2024parametrisednoncommutativemotivesequivariant}. \qedhere
    \end{enumerate}
\end{proof}

\begin{remark}
    \hfill 
    \begin{enumerate}[(i)]
        \item We were not able to write down the indexed coproduct functor
    \[ \ind_H^{K}\colon  \LMod_{A_H}(\C_H) \to \LMod_{A_{K}}(\C_{K}) \]
    by hand and needed presentability to do so.
        \item One might try to factor it through $\LMod_{\Ind_H^{K}A}(\C_{K})$ by looking for natural maps induced by the adjunctions. However, this already fails because $\Ind_H^{K}$ is typically not lax symmetric monoidal (it would more naturally be oplax) and so $\Ind_H^{K}A$ does not even naturally obtain an algebra structure.
    \medskip \\Informally, given some $A_H$-module $M$, we want an action of $A_{K}$ on $\Ind_H^{K}M$. Inducing up our given action yields a map
    \[ \Ind_H^{K}(A_H \otimes M) \to \Ind_H^{K}M \]
    which is not yet of the form $A_{K} \otimes \Ind_H^{K}M \to \Ind_H^{K}M$ but that's where the projection formula comes to the rescue, providing us with a map
    \begin{center}
        \begin{tikzcd}
            A_H \otimes \Ind_H^K M \arrow[r, "\simeq", no head] & \Ind_H^K (A_H \otimes M) \arrow[r] & \Ind_H^K M.
        \end{tikzcd}
    \end{center}
    This is why we needed to assume the projection formula for this construction.
    \end{enumerate} 
\end{remark}

\noindent We now describe in full generality the indexed tensor product on left module categories. 

\begin{proposition} \label{prop: indexed tensor products of LMod}
    Let $\O^{\otimes} \in \Op_{G, \infty}$ have a single $G$-color and $(\C^{\otimes}, A)$ an $\O$-module datum. Consider a coCartesian lift $o' \to o$ in $\myuline{\O}^{\otimes}$ of a map in $\myuline{\F}_{G, *}$ over some $G/H \in \Orb_G$ corresponding to $f\colon U \to G/H$. 
    \medskip \\If $U \simeq \coprod_{i=1}^n G/H_i$ with $H_i \leq G$ is an orbit decomposition, then the map $\underline{\bigotimes}_{o' \to o}$ factors as
    \begin{center}
        \begin{tikzcd}
             \myuline{\LMod}_{A}^{G}(\C)_{\myuline{o}'}^{\otimes} \arrow[r] \arrow[d] & \myuline{\LMod}_{\underline{\bigotimes}_{U \to G/H} \myuline{A}^{\otimes}(o')}^{G}(\C)_{\myuline{H}}  \arrow[rr, "A \otimes_{\underline{\bigotimes}_{o' \to o} A} -"] \arrow[d] & & \myuline{\LMod}_A^{G}(\C)_{\myuline{H}}^{\otimes}
            \\ \myuline{\C}_{\myuline{o}'}^{\otimes} \arrow[r, "\underline{\bigotimes}_{o' \to o}", swap] & \myuline{\C}_{\myuline{o}}^{\otimes}
        \end{tikzcd}
    \end{center}
    where the left square commutes.
\end{proposition}

\begin{proof}[Proof Sketch]
Since the proof of this result is notationally quite heavy, we will first sketch an argument of this result in the classical setting of symmetric monoidal $\infty$-categories for the convenience of the reader. In the actual proof we will then parametrize all ingredients that appear.
\medskip \\Let $\C^{\otimes} \to \F_*$ be a symmetric monoidal $\infty$-category and $A^{\otimes} \in \Alg_{\E_{\infty}}(\C)$, then the special case for $\langle 2 \rangle \to \langle 1 \rangle$ is
\begin{center}
    \begin{tikzcd}
        \LMod_A(\C) \times \LMod_A(\C) \arrow[r] \arrow[d] & \LMod_{A \otimes A}(\C) \arrow[rr, "A \otimes_{A \otimes A}-"] \arrow[d] & & \LMod_A(\C)
        \\ \C \times \C \arrow[r, "\otimes", swap] & \C
    \end{tikzcd}
\end{center}
To obtain the top factorization we note that
\[ \St \left(\LMod_A(\C)^{\otimes} \to \F_* \right) \simeq \left( \F_* \xrightarrow{A^{\otimes}} \Alg_{\E_1}(\C)^{\otimes} \to \Cat_{\infty} \right) \]
and the top composite is the effect of this functor on $\langle 2 \rangle \to \langle 1 \rangle$. This is first sent to $(A,A) \to A$ in $\Alg_{\E_1}(\C)^{\otimes}$ which we can factor through the coCartesian lift $(A,A) \to A \otimes A \to A$. That already recover the top line of the diagram. The second map induces a relative tensor product functor (\cref{lemma: AlgLM ---> AlgE1 is coCartesian fibration}) and we are left to check that the first map
\[ \LMod_A(\C) \times \LMod_A(\C) \to \LMod_{A \otimes A}(\C) \]
is compatible with the underlying tensor product. It suffices to check that
\begin{center}
    \begin{tikzcd}
        \LMod_A(\C) \times \LMod_A(\C) \arrow[r] \arrow[d] & \LMod_{A \otimes A}(\C) \arrow[d]
        \\ \Alg_{\LM}(\C) \times \Alg_{\LM}(\C) \arrow[r] & \Alg_{\LM}(\C)
    \end{tikzcd}
\end{center} 
commutes because $\Alg_{\LM}(\C)^{\otimes} \to \C^{\otimes}$ is symmetric monoidal. Naively, our idea is that the square should be the naturality square of some natural transformation between suitably chosen functors $\Alg_{\E_1}(\C)^{\otimes} \to \Cat_{\infty}$ evaluated at $(A, A) \to A \otimes A$. In reality, our choices won't give naturality squares for all maps in $\Alg_{\E_1}(\C)^{\otimes}$ but there will be some partial naturality, in particular for those $(A,A) \to A$ that we need! We will check this via some coCartesian edge techniques.
\end{proof}
\begin{proof}[Proof of \cref{prop: indexed tensor products of LMod}]
    By applying $- \times_{\Span(\F_{G})} \myuline{\F}_{G, *}$ to the defining pullback square of $\LMod_A^{G}(\C)^{\otimes}$ we obtain the square
    \begin{center}
        \begin{tikzcd}
            \myuline{\LMod}_A^{G}(\C)^{\otimes} \arrow[r] \arrow[d] \arrow[dr, phantom, very near start, "\lrcorner"] & \myuline{\Alg}_{\LM}(\C)^{\otimes}\arrow[d, "q"]
            \\ \myuline{\O}^{\otimes} \arrow[r, "\myuline{A}^{\otimes}", swap] & \myuline{\Alg}_{\E_1}(\C)^{\otimes}
        \end{tikzcd}
    \end{center}
    In particular,
    \[ \myuline{\St} \left(\myuline{\LMod}_A^{G}(\C)^{\otimes} \to \myuline{\O}^{\otimes} \right) \simeq \left( \myuline{\O}^{\otimes} \xrightarrow{\myuline{A}^{\otimes}} \myuline{\Alg}_{\E_1}(\C)^{\otimes} \xrightarrow{\myuline{\St}(q)} \myuline{\Cat}_{\T, \infty} \right) \]
    and the map $\underline{\bigotimes}_{o' \to o}$ that we want to understand is the image of $o' \to o$ under this functor, so 
    \[ \underline{\bigotimes}_{o' \to o} \simeq \myuline{\St}(q)(\myuline{A}^{\otimes}(o') \to \myuline{A}^{\otimes}(o)). \]
    We 
    factor $\myuline{A}^{\otimes}(o') \to \myuline{A}^{\otimes}(o)$ over a the lift of $U \to G/H$ along $\myuline{\Alg}_{\myuline{\E}_1}(\C)^{\otimes} \to \myuline{\F}_{G, *}$ depicted as follows:
    \begin{center}
        \begin{tikzcd}
            & \underline{\bigotimes}_{U \to G/H} \myuline{A}^{\otimes}(o') \arrow[dr] & & & G/H \arrow[dr, equal]
            \\ \myuline{A}^{\otimes}(o') \arrow[ur] \arrow[rr] & & \myuline{A}^{\otimes}(o) & U \arrow[rr] \arrow[ur] & & G/H
        \end{tikzcd}
    \end{center}
    By applying $\myuline{\St}(q)$ on the left square the map $\underline{\bigotimes}_{o' \to o}$ factors as the composite
    \begin{center}
        \begin{tikzcd}
            \myuline{\LMod}_A^{G}(\C)^{\otimes}_{\myuline{o}'} \arrow[r] & \myuline{\LMod}_{\underline{\bigotimes}_{U \to G/H} \myuline{A}^{\otimes}(o')}^{G}(\C)_{\myuline{H}} \arrow[r] & \myuline{\LMod}_A^{G}(\C)^{\otimes}_{\myuline{o}}
        \end{tikzcd}
    \end{center}
    by functoriality where the last term is also equivalent to the $H$-$\infty$-category $\myuline{\LMod}_{\myuline{A}^{\otimes}(o)}^{G}(\C)_{\myuline{H}}$ since $\O^{\otimes}$ has a single $G$-color. In particular, the second map is the effect of functoriality of $\myuline{\Alg}_{\E_1}(\C) \to \myuline{\Cat}_{G, \infty}$ on a map of algebras over the same $G/H \in \Orb_G^{\op}$ which is given by levelwise relative tensor products.
    \medskip \\Now about the first map. Consider the following functors: 
    \begin{center}
        \begin{tikzcd}
            \myuline{\Alg}_{\LM}(\C)^{\otimes} \arrow[d, "q", swap]
            \\ \myuline{\Alg}_{\E_1}(\C)^{\otimes} \arrow[d, "p", swap] \arrow[r, "\Phi"] & \myuline{\Cat}_{G, \infty}
            \\ \myuline{\F}_{G, \infty} \arrow[r, "\Psi", swap] & \myuline{\Cat}_{G, \infty}
        \end{tikzcd}
    \end{center}
    where $\Phi$ and $\Psi$ are the respective parametrized straightenings, i.e. $\Phi = \myuline{\St}(q)$ and $\Psi = \myuline{\St}(pq)$. Ideally,\footnote{In the end, we will not get a natural transformation but will still get some partial naturality.} we would now like
    to demonstrate the existence of a $G$-natural transformation\footnote{I.e.~a map in the $\infty$-category $\Fun_{G}(\Phi, \Psi p)$. See \cite[Remark 2.2.3]{cnossen2023parametrizedstabilityuniversalproperty} for an unravelled version of this notion.} $\Phi \Rightarrow \Psi p$ because naturality with respect to $\myuline{A}^{\otimes}(o') \to \underline{\bigotimes}_{U \to G/H} \myuline{A}^{\otimes}(o')$ then yields a commutative diagram
    \begin{center}
        \begin{tikzcd}
            \myuline{\LMod}_A^{G}(\C)_{\myuline{o}'}^{\otimes} \arrow[r] \arrow[d] & \myuline{\LMod}_{\underline{\bigotimes}_{U \to t} \myuline{A}^{\otimes}(o')}^{G}(\C)_{\myuline{H}} \arrow[d]
            \\ \myuline{\Alg}_{\LM}(\C)^{\otimes}_{\myuline{U}} \arrow[r, "\underline{\bigotimes}_{U \to G/H}", swap] & \myuline{\Alg}_{\LM}(\C)_{\myuline{H}}^{\otimes}
        \end{tikzcd}
    \end{center}
    and moreover there is a commutative diagram
    \begin{center}
        \begin{tikzcd}
            \myuline{\Alg}_{\LM}(\C)^{\otimes}_{\myuline{U}} \arrow[r, "\underline{\bigotimes}_{U \to G/H}"] \arrow[d] & \myuline{\Alg}_{\LM}(\C)^{\otimes}_{\myuline{H}} \arrow[d]
            \\ \myuline{\C}_{\myuline{U}}^{\otimes} \arrow[r, "\underline{\bigotimes}_{U \to G/H}", swap] & \myuline{\C}_{\myuline{H}}^{\otimes}
        \end{tikzcd}
    \end{center}
    because $\Alg_{\LM}(\C)^{\otimes} \to \C^{\otimes}$ is $G$-symmetric monoidal (\cref{theorem: omnibus Alg theorem}). Pasting these two squares yields the desired factorization.
    \medskip \\To find a natural transformation $\Phi \Rightarrow \Psi p$ it would be enough to give a map $\myuline{\Un}(\Phi) \to \myuline{\Un}(\Psi p)$ in $\coCart{\myuline{\Alg}_{\E_1}(\C)^{\otimes}}$. By the behaviour of (parametrized) unstraightenings with compositions, we obtain a pullback square
    \begin{center}
        \begin{tikzcd}
            \myuline{\Un}(\Psi p) \arrow[r] \arrow[d] \arrow[dr, phantom, very near start, "\lrcorner"] & \myuline{\Un}(\Psi) \arrow[d]
            \\ \myuline{\Alg}_{\E_1}(\C)^{\otimes} \arrow[r, "p", swap] & \myuline{\F}_{\T, *}
        \end{tikzcd}
    \end{center}
    where we really use that we know this for the ordinary unstraightenings and reduce it to that case because the parametrized unstraightening of a functor to $\myuline{\Cat}_{G, \infty}$ corresponds to the unstraightening of the corresponding functor to $\Cat_{\infty}$ (see \cref{prop: parametrized straightening}). By definition, we have $\myuline{\Un}(\Psi) \simeq \myuline{\Alg}_{\LM}(\C)^{\otimes}$, so 
    \[ \myuline{\Un}(\Psi p) \simeq \myuline{\Alg}_{\LM}(\C)^{\otimes} \times_{\myuline{\F}_{\T, *}} \myuline{\Alg}_{\E_1}(\C)^{\otimes}. \]
    Thus, we are considering the map
    \begin{center}
        \begin{tikzcd}
            \myuline{\Alg}_{\LM}(\C)^{\otimes} \arrow[rr, "{(\id, q)}"] \arrow[dr, "q", swap] & & \myuline{\Alg}_{\LM}(\C)^{\otimes} \times_{\myuline{\F}_{G, *}} \myuline{\Alg}_{\E_1}(\C)^{\otimes} \arrow[dl, "\pr_2"]
            \\ & \myuline{\Alg}_{\E_1}(\C)^{\otimes}
        \end{tikzcd}
    \end{center}
    which -- as it turns out -- is only a functor over $\myuline{\Alg}_{\E_1}(\C)^{\otimes}$, but not a map of coCartesian fibrations, meaning that this does not induce a natural transformation $\Phi \Rightarrow \Psi p$. Nonetheless, we are still able to salvage some naturality squares out of this, namely exactly with respect to the maps $\myuline{A}^{\otimes}(o') \to \underline{\bigotimes}_{U \to G/H} \myuline{A}^{\otimes}(o')$; these are precisely the ones we need.
    \medskip \\We pull back this diagram along $\phi\colon [1] \to \myuline{\Alg}_{\E_1}(\C)^{\otimes}$ where we pick out any $p$-coCartesian edge in $\myuline{\Alg}_{\E_1}(\C)^{\otimes}$. We claim now that 
    {\small \begin{center}
        \begin{tikzcd}
            \myuline{\Alg}_{\LM}(\C)^{\otimes} \times_{\myuline{\Alg}_{\E_1}(\C)^{\otimes}} {[1]} \arrow[rr, "{\widetilde{(\id, q)}}"] \arrow[dr, "\widetilde{q}", swap] & & \myuline{\Alg}_{\LM}(\C)^{\otimes} \times_{\myuline{\F}_{G, *}} \myuline{\Alg}_{\E_1}(\C)^{\otimes} \times_{\myuline{\Alg}_{\E_1}(\C)^{\otimes}} {[1]}  \arrow[dl, "\widetilde{\pr}_2"]
            \\ & {[1]} 
        \end{tikzcd}
    \end{center} }
    \noindent is a map of coCartesian fibrations. Indeed, consider a $\widetilde{q}$-coCartesian edge and we wish to show that its image under $\widetilde{(\id, q)}$ is $\widetilde{\pr}_2$-coCartesian. Equivalently \cite[Proposition 2.4.1.3(1)]{lurie2009htt}, pick a $q$-coCartesian edge $x \to y$ over $\phi$ and we wish to show that $(x,qx) \to (y,qy)$ is $\pr_2$-coCartesian. Note that $\pr_2$ is the projection map of a pullback square
    \begin{center}
        \begin{tikzcd}
            \myuline{\Alg}_{\LM}(\C)^{\otimes} \times_{\myuline{\F}_{G, *}} \myuline{\Alg}_{\E_1}(\C)^{\otimes} \arrow[dr, phantom, very near start, "\lrcorner"] \arrow[r] \arrow[d, "\pr_2", swap] & \myuline{\Alg}_{\LM}(\C)^{\otimes} \arrow[d, "pq"]
            \\ \myuline{\Alg}_{\E_1}(\C)^{\otimes} \arrow[r] & \myuline{\F}_{G, *}
        \end{tikzcd}
    \end{center}
    and being a coCartesian edge can be checked before pulling back \cite[Proposition 2.4.1.3(1)]{lurie2009htt}. In other words, we are asking if $x\to y$ is $pq$-coCartesian. By assumption, $x \to y$ is $q$-coCartesian and moreover $q(x \to y) = \phi$ is $p$-coCartesian. Thus, $x \to y$ is $pq$-coCartesian \cite[Proposition 2.4.1.3(3)]{lurie2009htt}.
    \medskip \\This yields naturality for $p$-coCartesian edges and in particular gives the naturality for the commutative squares that we wanted.
\end{proof}

\begin{example}
    Let $R \in \Alg_{\E_{\infty}^G}(\ul{\Sp}_G)$ and $M \in \LMod_{R^H}(\Sp_H)$. If we write $N_R^{H \to G}$ for the norm in $\ul{\LMod}_R(\ul{\Sp}_G)^{\otimes}$, then 
    \[ M \otimes_{R^H} M \simeq R^H \otimes_{R^H \otimes R^H} (M \otimes M) \quad \text{and} \quad N_R^{H \to G}M \simeq R \otimes_{N_H^G R} N_H^G M. \]
    are special cases of \cref{prop: indexed tensor products of LMod}.
\end{example}

\noindent Now, we can finally prove distributivity of $\LMod$. We are grateful to Kaif Hilman for suggesting the proof strategy.

\begin{theorem} \label{theorem: LMod distributive}
    Let $\O^{\otimes} \in \Op_{G, \infty}$ have a single $G$-color and $(\C^{\otimes}, A)$ be an $\O$-module datum. Suppose that $\myuline{\C}^{\otimes}$ is distributive and $G$-presentable. Then, $\myuline{\LMod}_A^{G}(\C)^{\otimes} \to \myuline{\O}^{\otimes}$ is $\O$-distributive.
\end{theorem}

\begin{proof}
    Let $H \leq G$ and $o \in \O^{\otimes}_H$. We then have $\myuline{\LMod}_A^{G}(\C)^{\otimes}_{\myuline{o}} \simeq \myuline{\LMod}_A^{G}(\C)_{\myuline{H}}$ since $\O^{\otimes}$ has a single $G$-color (\cref{lemma: underlying T-category for single T-color}) and this is $H$-cocomplete by \cref{prop: categorical properties of LMod} \textcolor{chilligreen}{(ii)} and \cref{prop: projection formula from distributivity}. 
    \medskip \\Now again let $H \leq G$ and $\alpha\colon  o' \to o$ in $\myuline{\O}^{\otimes}$ be a coCartesian lift of a map in $\myuline{\F}_{G, *}$ over $G/H$ corresponding to $f\colon U \to G/H$. We need to show that the associated pushforward $H$-functor $\underline{\bigotimes}_{o' \to o}\colon  \myuline{\LMod}_A^{G}(\C)_{\myuline{o}'}^{\otimes} \to \myuline{\LMod}_A^{G}(\C)_{\myuline{o}}^{\otimes}$ is distributive. For this, consider any pullback square
    \begin{center}
        \begin{tikzcd}
            U' \arrow[r, "f'"] \arrow[d, "g'", swap] \arrow[dr, phantom, very near start, "\lrcorner"] & V' \arrow[d, "g"]
            \\ U \arrow[r, "f", swap] & G/H
        \end{tikzcd}
    \end{center}
    in $\F_{G}$ and $p\colon  \myuline{I}^{\myuline{\triangleright}} \to g'^* \myuline{\LMod}_A^{G}(\C)_{\myuline{o}'}^{\otimes}$ be a $(\Orb_G)_{/U'}$-colimit diagram. Let $U \simeq \coprod_i G/H_i$ with $H_i \leq G$ be an orbit decomposition, so $\myuline{\LMod}_A^{G}(\C)_{\myuline{o}'}^{\otimes} \simeq f_* \coprod_i \myuline{\LMod}_A^{G}(\C)_{\myuline{H}_i}$ by \cref{theorem: equivariant Segal conditions} and using that $\O^{\otimes}$ only has a single $G$-color. We need to prove that the composite
    {\small \begin{center}
        \begin{tikzcd}
            (f_*' \myuline{I})^{\myuline{\triangleright}} \arrow[r] & f'_*(\myuline{I}^{\myuline{\triangleright}}) \arrow[r] & f'_* g'^* \coprod_i \myuline{\LMod}_A^{G}(\C)_{\myuline{H}_i}^{\otimes} \arrow[r, no head, "\simeq"] & g^* f_* \coprod_i \myuline{\LMod}_A^{G}(\C)_{\myuline{H}_i}^{\otimes} \arrow[r] & g^* \myuline{\LMod}_A^{G}(\C)_{\myuline{o}}^{\otimes}
        \end{tikzcd}
    \end{center} }
    \noindent is a colimit diagram. The last map in the above composite as
    {\small \begin{center}
        \begin{tikzcd}
            g^* f_* \myuline{\LMod}_A^{G}(\C)_{\myuline{o'}}^{\otimes} \arrow[r] & g^*\myuline{\LMod}_{\underline{\bigotimes}_{o' \to o} A(o')}^{G}(\C)_{\myuline{H}} \arrow[r] & g^* \myuline{\LMod}_A^{G}(\C)_{\myuline{H}}^{\otimes}
        \end{tikzcd}
    \end{center} }
    \noindent by \cref{prop: indexed tensor products of LMod}.
    \medskip \\Then, there is a commutative diagram
    {\small \begin{center}
        \begin{tikzcd}
            (f_*' \myuline{I})^{\myuline{\triangleright}} \arrow[r] & f'_*(\myuline{I}^{\myuline{\triangleright}}) \arrow[r] & f'_* g'^* \coprod_i \myuline{\LMod}_A^{G}(\C)_{\myuline{H}_i}^{\otimes} \arrow[r, no head, "\simeq"] \arrow[d] & g^* \myuline{\LMod}_A^{G}(\C)_{\myuline{o'}}^{\otimes} \arrow[r] \arrow[d] & g^*\myuline{\LMod}_{\underline{\bigotimes}_{o' \to o}}^{G}(\C)_{\myuline{H}}^{\otimes} \arrow[d]
            \\ (f_*' \myuline{I})^{\myuline{\triangleright}} \arrow[u, equal] \arrow[r] & f_*'(\myuline{I}^{\myuline{\triangleright}}) \arrow[u, equal] \arrow[r] & f'_* g'^* \myuline{\C}_{\myuline{o'}}^{\otimes} \arrow[r, no head, "\simeq", swap] & g^* \myuline{\C}_{\myuline{o'}}^{\otimes} \arrow[r] & g^* \myuline{\C}_{\myuline{H}}^{\otimes}
        \end{tikzcd}
    \end{center} }
    \noindent induced by the forgetful functors by \cref{prop: indexed tensor products of LMod}. Since the forgetful functor strongly preserves colimits (\cref{prop: categorical properties of LMod} \textcolor{chilligreen}{(ii)}) and by assumption $\myuline{I}^{\myuline{\triangleright}} \to g'^* \myuline{\LMod}_A^{G}(\C)_{\myuline{o}'}^{\otimes}$ is a $(\Orb_G)_{/U}$-colimit diagram, we deduce that 
    \begin{center}
        \begin{tikzcd}
            \myuline{I}^{\myuline{\triangleright}} \arrow[r] & g'^* \myuline{\LMod}_A^{G}(\C)_{\myuline{o'}}^{\otimes} \arrow[r] & \myuline{\C}_{\myuline{o'}}^{\otimes} 
        \end{tikzcd}
        \end{center}is also a $(\Orb_G)_{/U'}$-colimit diagram, so since $\myuline{\C}^{\otimes}$ is $\myuline{\O}$-distributive, it follows that the bottom line of the above diagram is a colimit diagram. On the other hand, the forgetful functor $\myuline{\LMod}_A^{G}(\C)^{\otimes} \to \myuline{\C}^{\otimes}$ strongly reflects colimits (\cref{prop: categorical properties of LMod}), with which the upper line is also a colimit diagram.
    \medskip \\To get 
    back to our desired composite, we are only left to postcompose the top line by the map $g^*\ul{\LMod}_{\bigotimes_{i=1}^n N_{H_i}^H A_{H_i}}^{G}(\C)_{\myuline{H}}^{\otimes} \to g^* \ul{\LMod}_A^{G}(\C)_{\myuline{H}}^{\otimes}$, which is a $(\Orb_G)_{/V'}$-left adjoint (\cref{prop: categorical properties of LMod} \textcolor{chilligreen}{(ii)}) and thus strongly preserves $(\Orb_G)_{/V'}$-colimits.
\end{proof}

\noindent Let us enhance our terminology of $\O$-module data for the sake of brevity.

\begin{definition} \label{def: distributive module datum}
    A \tb{distributive module datum} consists of a triple $(\O^{\otimes}, \C^{\otimes}, A)$ such that $\O^{\otimes} \in \Op_{G, \infty}$ with a single $G$-color and $(\C^{\otimes}, A)$ is an $\O$-module datum (\cref{def: O-module datum}) with distributive and $G$-presentable $\ul{\C}^{\otimes}$.
\end{definition}

\noindent So the content of \cref{theorem: LMod distributive} is precisely that distributive module data yield distributive $\LMod$ constructions.

\subsection{Parametrized Grouplike Spaces \& Parametrized Picard Spaces}\label{sec:Pic_and_grouplike}
Grouplike monoidal $G$-spaces will be star players in this article.
It's the language to phrase Picard spaces in which will be an ingredient in the definition of Thom spectra and it is also involved in the recognition theorem, which will allow us to run computational arguments. We will recall and set the foundations in this subsection.
\medskip \\We generalize \cite[Definition 7.1]{andoblumberggepner2018parametrized} and define:

\begin{definition} \label{def: grouplike}
    \hfill
    \begin{enumerate}[(i)]
        \item An $\E_1$-$G$-space is \tb{grouplike} if it is grouplike at every level.
        \item Let $\O^{\otimes} \in \Op_{G, \infty}$ and $\eta\colon  \Infl_{G}\E_1^{\otimes} \to \O^{\otimes}$ be a map in $\Op_{G, \infty}$. Let $X \in \Alg_{\O}(\myuline{\Sc}_{G})$, then $X$ is \tb{grouplike} (with respect to $\eta$) if $\eta^* X$ is a grouplike $\E_1$-$G$-space.
    \end{enumerate}
    \noindent We write $\tb{\myuline{\Alg}_{\O}^{\gp}(\myuline{\Sc}_{G})} \subseteq \myuline{\Alg}_{\O}(\myuline{\Sc}_{G})$ for the full subcategory of (levelwise) grouplike $\O$-monoidal spaces.\footnote{See \cite[Definition 2.2.1]{nardinshah2022equivarianttopos} for a definition of parametrized algebra categories.}
\end{definition}

\noindent Note that we are not yet talking about $G$-subcategories. We will verify that it forms a $G$-subcategory later (\cref{lemma: Alggp G-subcategory}).

\begin{remark} \label{remark: branko grouplike}
    Let $V$ be a finite-dimensional real $G$-representation. Juran considers a notion of grouplike algebras that allows more $\E_V$-operads than we are considering here \cite[Definition 2.6]{juran2025genuineequivariantrecognitionprinciple}. If $A \in \Alg_{\E_V}(\myuline{\Sc}_G)$, then $A^H \in \Alg_{\E_{\dim{V^H}}}(\Sc)$ and Juran calls $A$ grouplike if $\pi_0(A^H)$ is a group for all $H$ such that $\dim{V^H} \geq 1$. So unlike us, he allows algebras which need not admit a levelwise monoid structure. On the other hand, we allow slightly more general $G$-$\infty$-operads and not just the $\E_V$-operads.
\end{remark}

\noindent Classically, grouplike multiplicative spaces were studied by May with regards to his recognition theorem \cite{may1972iteratedloopspaces}. Recently, such results were also proven in the equivariant setting, which we now briefly recall before moving to more categorical matters.

\begin{theorem}[{\cite[Theorem A]{cnossen2024normedequivariantringspectra}, \cite[Theorem A]{juran2025genuineequivariantrecognitionprinciple}}, \cite{Guillou2012EquivariantIL}]\label{thm: recognition}
    \hfill 
    \begin{enumerate}[(i)]
        \item There is a \(G\)-symmetric monoidal equivalence of \(G\)-$\infty$-categories $
    \myuline{\Sp}^{\otimes}_{G, \geq 0}\simeq \myuline{\Alg}^{\gp}_{\E^G_\infty}(\Sc_G)^{\otimes}$
    which is given levelwise by $\Omega^{\infty}$.
    \item Let $V$ be a finite-dimensional real $G$-representation and \(X, Y \in \Alg_{\E_V}^{\gp}(\myuline{\Sc}_{G}) \) in the sense of Juran (\cref{remark: branko grouplike}).
    Then, there exists a functor $\mathrm{B}^V \colon \Alg_{\E_V}^{\gp}(\myuline{\Sc}_{G, *}) \to \Sc^{G}_{*}$ which induces an equivalence of \(G\)-spaces 
    \[ \myuline{\Map}_{\myuline{\Alg}^{\gp}_{\E_V}(\Sc_G)}(X,Y)\simeq \myuline{\Map}_{\myuline{\Sc}_{G, *}}(\mathrm{B}^V X,\mathrm{B}^V Y). \]
    \end{enumerate}
\end{theorem}

\begin{lemma} \label{lem:delooping_suspension}
    Let $E \in \Sp^{G}_{\geq 0}$ and $V$ be a $G$-representation. Then, $\mathrm{B}^{V} \Omega^{\infty}E \simeq \Omega^{\infty} \Sigma^{V}E$.
\end{lemma}

\begin{proof}
    First, we note the adjunction
    \begin{center}
        \begin{tikzcd}
            \Sp^G_{\geq 0} \arrow[r, "" {name = L}, shift left = 1.5, hookrightarrow] & \Sp^G \arrow[l, "\tau_{\geq 0}" {name = R}, shift left = 1.5] \arrow[r, "\Sigma^V" {name = L2}, shift left = 1.5] & \Sp^G \arrow[l, shift left = 1.5, "\Omega^V" {name = R2}]
        \arrow[phantom,from=L,to=R,"\scriptscriptstyle\tiny\bot"]
        \arrow[phantom,from=L2,to=R2,"\scriptscriptstyle\tiny\bot"]
        \end{tikzcd}
    \end{center}
    which restricts to an adjunction $\Sigma^V \dashv \tau_{\geq 0} \Omega^V$ on $\Sp_{\geq 0}^G$. Equipped with this, we perform a Yoneda argument: Let $X \in \Sp^G_{\geq 0}$, then
    \begin{align*}
        \Map_{\Alg_{\E_{\infty}^G}^{\gp}(\myuline{\Sc}_G)}(\mathrm{B}^{V} \Omega^{\infty}E, \Omega^{\infty}X) &\simeq \Map_{\Alg_{\E_{\infty}^G}^{\gp}(\myuline{\Sc}_G)}(\Omega^{\infty}E, \Omega^{V} \Omega^{\infty}X)
        \\ &\simeq \Map_{\Alg_{\E_{\infty}^G}^{\gp}(\myuline{\Sc}_G)}(\Omega^{\infty}E,  \Omega^{\infty} \Omega^{V}X)
        \\ &\simeq \Map_{\Sp^{G}_{\geq 0}}(E, \tau_{\geq 0}\Omega^{V}X)
        \\ &\simeq \Map_{\Sp^{G}_{\geq 0}}(\Sigma^{V}E, X)
        \\ &\simeq \Map_{\Alg_{\E_{\infty}^G}^{\gp}(\myuline{\Sc}_G)}(\Omega^{\infty}\Sigma^{V}E, \Omega^{\infty}X)
    \end{align*}
    where we use the recognition theorem (\cref{thm: recognition}) to argue that $\Omega^{\infty}X$ for $X \in \Sp^G_{\geq 0}$ hits all of $\Alg_{\E_{\infty}^G}^{\gp}(\Sc_G)$.
\end{proof}

\begin{example} \label{example: equivariant Bott periodicity ku}
    By (Real) equivariant Bott periodicity \cite{Atiyah66Real} we have
    \[
    \Omega^\infty \kR \simeq \Z\times \BUR,\quad
    \Omega^\infty \Sigma^\rho \kR \simeq \BUR,\quad
    \Omega^\infty \Sigma^{2\rho} \kR \simeq \BSUR,
    \]
    and hence \( \mathrm{B}^\rho \BUR\simeq \BSUR\).
\end{example}

\noindent Let us now move to categorical matters.

\begin{lemma} \label{lemma: Alggp G-subcategory}
    The subcategory $\myuline{\Alg}_{\O}^{\gp}(\myuline{\Sc}_{G}) \subseteq \myuline{\Alg}_{\O}(\myuline{\Sc}_{G})$ is a $G$-subcategory.
\end{lemma}

\begin{proof}
    On the level $H \leq G$ the $G$-$\infty$-category $\myuline{\Alg}_{\O}(\myuline{\Sc}_{G})$ is given by $\Alg_{\myuline{\O}_{\myuline{H}}}(\myuline{\Sc}_{H})$ \cite[Definition 2.2.1]{nardinshah2022equivarianttopos} \cite[Corollary 2.2.11(3)]{cnossen2023parametrizedstabilityuniversalproperty}. We must check that the restriction functor preserves grouplike algebras. For this, we observe for $H \leq K \leq G$ that the square
    \begin{center}
        \begin{tikzcd}
            \Alg_{\myuline{\O}_{\myuline{K}}}(\myuline{\Sc}_{K}) \arrow[r] \arrow[d, "\eta^*", swap] & \myuline{\Alg}_{\myuline{\O}_{\myuline{H}}}(\myuline{\Sc}_{H}) \arrow[d, "\eta^*"]
            \\ \Alg_{\Infl_{K} \E_1}(\myuline{\Sc}_{K}) \arrow[r] & \Alg_{\Infl_{H} \E_1}(\myuline{\Sc}_{H})
        \end{tikzcd}
    \end{center}
    commutes, at least evaluated on objects. Indeed, if $A\colon  \myuline{\O}_{\myuline{K}}^{\otimes} \to \myuline{\Sc}_{K}^{\times}$ is an object in the top left, then going around the square is given by the two equivalent formulations
    \[ \left( \Infl_{K} \E_1|_{\Orb_H^{\op}} \to \myuline{\O}_{\myuline{K}}^{\otimes}|_{\Orb_H^{\op}} \to \myuline{\Sc}_{K}|_{\Orb_H^{\op}} \right) \simeq \left( \Infl_{K} \E_1 \to \myuline{\O}_{\myuline{K}}^{\otimes} \to \myuline{\Sc}_{K} \right)|_{\Orb_H^{\op}}. \]
    On the other hand, the restriction of a grouplike algebra is grouplike again since being grouplike is a levelwise condition.
\end{proof}

\begin{lemma} \label{lemma: Alggp presentable}
    The $G$-$\infty$-category $\myuline{\Alg}_{\Infl_{G}\E_1}^{\gp}(\myuline{\Sc}_{G})$ is $G$-presentable.
\end{lemma}

\begin{proof}
    We need to show that it is $G$-cocomplete and fiberwise presentable \cite[Theorem 6.1.3(7)]{hilman2024parametrisedpresentability}.
    \begin{itemize}
        \item $G$-cocompleteness: Since $\myuline{\Sc}_{G}^{\times}$ is distributive, we deduce that $\myuline{\Alg}_{\Infl_{G}\E_1}(\myuline{\Sc}_{G})$ is $G$-cocomplete \cite[Theorem 5.1.4]{nardinshah2022equivarianttopos}. We show that $G$-colimits of grouplike algebras in $\myuline{\Alg}_{\Infl_{G}\E_1}(\myuline{\Sc}_{G})$ are grouplike again.
        \medskip \\On the fiber $H \leq G$ we are looking at $\Alg_{\E_1}^{\gp}(\Sc_H) \subseteq \Alg_{\E_1}(\Sc_H)$ by \textbf{\ref{theorem: omnibus Alg theorem}} \textcolor{chilligreen}{(iii)}. The fiberwise colimit statement is \cite[Remark 5.2.6.9]{lurie2017ha}, so we still need to discuss indexed coproducts. For $H \leq K \leq G$ this is the left adjoint of $\Res_H^K\colon  \Alg_{\E_1}(\Sc_K) \to \Alg_{\E_1}(\Sc_H)$, also known as $\Ind_H^K$. Since the symmetric monoidal structure on $\Sc_H$ resp. $\Sc_K$ is given by the pointwise cartesian symmetric monoidal structure, this is equivalently the functor
        \[ \Res_H^K\colon  \Fun \left(\Orb_K^{\op}, \Alg_{\E_1}(\Sc) \right) \to \Fun \left(\Orb_H^{\op}, \Alg_{\E_1}(\Sc) \right) \]
        induced by precomposition by $j\colon  \Orb_H^{\op} \to \Orb_K^{\op}$. So its left adjoint $\Ind_H^K$ is given by left Kan extension. Let $A \in \Alg_{\E_1}^{\gp}(\Sc_H) \simeq \Fun \left(\Orb_H^{\op}, \Alg_{\E_1}^{\gp}(\Sc) \right)$ and $L \leq K$. Using the left Kan extension formula we compute
        \[ (\Ind_H^K A)_L \simeq \colim \left( j_{\downarrow L} \to \Orb_K^{\op} \xrightarrow{A} \Alg_{\E_1}(\Sc) \right), \]
        which is a colimit of grouplike $\E_1$-spaces and thus itself grouplike \cite[Remark 5.2.6.9]{lurie2017ha}. Lastly, $\Res_H^K\colon  \Alg_{\E_1}^{\gp}(\Sc_K) \to \Alg_{\E_1}^{\gp}(\Sc_H)$ preserves colimits because colimits of functor categories are computed pointwise.
        \medskip \\Altogether, we have shown $G$-cocompleteness \cite[Theorem 3.1.9]{hilman2024parametrisedpresentability}.
        \item Fiberwise presentability: Since $\myuline{\Sc}_G$ is distributive and fiberwise presentable, we may deduce that $\myuline{\Alg}_{\Infl_G \E_1}(\myuline{\Sc}_G)$ is also fiberwise presentable \cite[Theorem 5.1.4]{nardinshah2022equivarianttopos}. Moreover, the inclusion 
        \[ \myuline{\Alg}_{\Infl_G \E_1}^{\gp}(\myuline{\Sc}_G) \hookrightarrow \myuline{\Alg}_{\Infl_G \E_1}(\myuline{\Sc}_G) \] 
        fiberwise preserves colimits and limits because it classically levelwise admits left and right adjoints. Thus, by the $\infty$-categorical reflection principle \cite[Theorem 6.2]{ragimov2022inftycategoricalreflectiontheoremapplications}, $\myuline{\Alg}_{\Infl_{G}\E_1}^{\gp}(\myuline{\Sc}_{G})$ is also fiberwise presentable.
    \end{itemize}
    We're done.
\end{proof}

\begin{proposition} \label{prop: inclusion grouplike}
    The inclusion
    \[ \myuline{\Alg}_{\Infl_{G}\E_1}^{\gp}(\myuline{\Sc}_{G}) \hookrightarrow \myuline{\Alg}_{\Infl_{G}\E_1}(\myuline{\Sc}_{G}). \]
    is a $G$-symmetric monoidal $G$-left adjoint.
\end{proposition}

\begin{proof}
    In the proof of the previous result (\cref{lemma: Alggp presentable}) we have seen that this inclusion functor is a $G$-cocontinuous. By $G$-presentability (\cref{lemma: Alggp presentable}) we can invoke the adjoint functor theorem \cite[Theorem 6.2.1]{hilman2024parametrisedpresentability} to obtain a right adjoint. We still need to check that the inclusion is $G$-symmetric monoidal.
    \medskip \\Fiberwise this follows because $\pi_0$ commutes with finite products. We need to discuss indexed products, so let $H \leq K \leq G$ and $X \in \Alg_{\E_1}^{\gp}(\Sc_H) \simeq \Fun \left(\Orb_H^{\op}, \Alg_{\E_1}^{\gp}(\Sc) \right)$ and we wish to show that $\Coind_H^K X$ is grouplike as well. Let $L \leq K$. We compute
    \begin{align*}
        \pi_0^L \Coind_H^K X &\cong \pi_0 \Map_{\Sc} \left(*, (\Coind_H^K X)^L \right)
        \\ &\cong \pi_0 \Map_{\Sc_L}(*, \Res_L^K \Coind_H^K X)
        \\ &\cong \pi_0 \Map_{\Sc_K}(\Ind_L^K *, \Coind_H^K X)
        \\ &\cong \pi_0 \Map_{\Sc_H}(\Res_H^K \Ind_L^K *, X)
    \end{align*}
    By the double coset formula, $\Res_H^K \Ind_L^K *$ is a finite coproduct of orbits, so the above term is a finite product of equivariant $\pi_0$'s of $X$, which is grouplike.
\end{proof}

\begin{remark}
    While 
    \cite[Lemma 7.2]{andoblumberggepner2018parametrized} is stated for a general (coherent) $\infty$-operad $\O$, the proof a priori only works for $\E_k$-operads. They cite \cite[Remark 5.1.3.5]{lurie2012ha}\footnote{This is \cite[Remark 5.2.6.9]{lurie2017ha}.} only works for $\E_k$-operads. Moreover, they state that products of invertible objects remain invertible, but it is not clear why this should be true since the operations induced by $\O$ need not be compatible with the $\E_1$-operations that are picked out. Nonetheless, this argument still works for $\E_n$-operads by working slightly harder.
\end{remark}

\begin{definition} \label{def: GL1}
    We will denote the $G$-right adjoint by $\tb{\GL_1}\colon  \myuline{\Alg}_{\Infl_{G}\E_1}(\myuline{\Sc}_{G}) \to \myuline{\Alg}_{\Infl_{G}\E_1}^{\gp}(\myuline{\Sc}_{G})$.
\end{definition}

\begin{remark}
    Since $\GL_1$ is fiberwise, say for $H \leq G$, a right adjoint 
    \[ \Fun \left(\Orb_H^{\op}, \Alg_{\E_1}(\Sc) \right) \to \Fun \left(\Orb_H^{\op}, \Alg_{\E_1}^{\gp}(\Sc) \right), \] 
    it is induced by the classical $\GL_1\colon  \Alg_{\E_1}(\Sc) \to \Alg_{\E_1}^{\gp}(\Sc)$ \cite[Lemma 7.2]{andoblumberggepner2018parametrized}. As such, it fiberwise takes the subalgebra of tensor-invertible objects.
\end{remark}

\begin{corollary} \label{corollary: GL1 monoidal}
    Let $\O^{\otimes} \in \Op_{G, \infty}$. Then, the functor $\GL_1$ induces a $G$-right adjoint 
    \[ \GL_1\colon  \ul{\Alg}_{\O \otimes \Infl_{G}\E_1}(\myuline{\Sc}_{G}) \to \ul{\Alg}_{\O \otimes \Infl_{G}\E_1}^{\gp}(\myuline{\Sc}_{G}) \]
    where being grouplike is with respect to the map $\Infl_{G} \E_1 \to \O^{\otimes} \otimes \Infl_{G}\E_1$ induced by the unit of the Stewart's Boardman--Vogt monoidal structure. 
\end{corollary}

\begin{proof}
    By \cref{prop: inclusion grouplike} the left adjoint of $\GL_1\colon  \myuline{\Alg}_{\Infl_{G}\E_1}(\myuline{\Sc}_{G}) \to \myuline{\Alg}_{\Infl_{G}\E_1}^{\gp}(\myuline{\Sc}_{G})$ is $G$-symmetric monoidal, so $\GL_1$ inherits a $G$-lax symmetric monoidal structure \cite[Corollary D.5]{stewart2025tensorproductsequivariantcommutative}. In particular, we can apply $\Alg_{\O}(-)$ and we conclude by the universal property of Stewart's Boardman--Vogt tensor product (\cref{theorem: omnibus Alg theorem} \textcolor{chilligreen}{(i)}) to obtain a $G$-right adjoint
    \[ \GL_1\colon  \ul{\Alg}_{\O \otimes \Infl_{G} \E_1}(\myuline{\Sc}_{G}) \to \ul{\Alg}_{\O}\left(\myuline{\Alg}_{\Infl_{G}\E_1}^{\gp}(\myuline{\Sc}_{G}) \right). \]
    Let us still identify $\Alg_{\O}\left(\myuline{\Alg}_{\Infl_{G}\E_1}^{\gp}(\myuline{\Sc}_{G}) \right)$. Consider the commutative diagram 
    \begin{center}
        \begin{tikzcd}
            \Alg_{\O} \left( \myuline{\Alg}_{\Infl_{G}\E_1}(\myuline{\Sc}_{G}) \right) \arrow[r, "\simeq"] & \Alg_{\O \otimes \Infl_{G}\E_1}(\myuline{\Sc}_{G})
            \\ \Alg_{\Infl_{G} \E_0}\left(\myuline{\Alg}_{\Infl_{G}\E_1}(\myuline{\Sc}_{G}) \right) \arrow[r, "\simeq", swap] \arrow[u] & \Alg_{\Infl_{G} \E_1}(\myuline{\Sc}_{G}) \arrow[u]
        \end{tikzcd}
    \end{center}
    Then, $\Alg_{\O}\left(\myuline{\Alg}^{\gp}_{\Infl_{G}\E_1}(\myuline{\Sc}_{G}) \right)$ correspond to those objects in the top left corner of the square such that pulling back along the left vertical map yields a grouplike object. By commutativity of the diagram this corresponds to an object in the top right corner of the square such that pulling back along the right vertical map yields a grouplike object. In other words,
    \[ \Alg_{\O}\left(\myuline{\Alg}_{\Infl_{G}\E_1}^{\gp}(\myuline{\Sc}_{G}) \right) \simeq \Alg_{\O \otimes \Infl_{G}\E_1}^{\gp}(\myuline{\Sc}_{G}), \]
    so we win.
\end{proof}

\begin{remark} \label{remark: gl1}
    \hfill 
    \begin{enumerate}[(i)]
        \item This result is already new for the classical case $G = *$ and $\infty$-operads $\O$ such that $\O \otimes \E_1$ is not an $\E_n$-operad. 
        \item On the other hand, Juran also discusses $\GL_1$ for his notion of grouplike $G$-spaces \cite[Proposition 4.1]{juran2025genuineequivariantrecognitionprinciple} and $\E_V$-operads.
        \item If $R \in \Alg_{\E_{\infty}^G}(\myuline{\Sp}_G)$, then $\GL_1 R$, by definition, deloops to $\tb{\gl_1 R} \in \Sp^G_{\geq 0}$ by the recognition theorem (\cref{thm: recognition}).
    \end{enumerate}
\end{remark}

\noindent Here is a folklore computation of the homotopy groups of $\GL_1 R$.

\begin{lemma}\label{lem: homotopy of gl1}
    Let \(R \in \Alg_{\E_{1}}(\myuline{\Sc}_G) \) and let $V$ be a $G$-representation with $V^G \neq 0$. Then,
    \[ \pi^G_0(\GL_1(R)) \cong \pi^G_0(R)^\times \quad \text{and} \quad \pi^G_V(\GL_1(R)) \cong \pi^G_V(R). \]
\end{lemma}
\begin{proof}
    Consider the pullback square of \(G\)-spaces
    \[
    \begin{tikzcd}
        \GL_1(R)\ar[r]\ar[d] \arrow[dr, phantom, very near start, "\lrcorner"] & \Omega^\infty R \ar[d]\\
        \myuline{\pi}_0(R)^\times \ar[r] &\myuline{\pi}_0(R).
    \end{tikzcd}
    \]
    In particular, $\GL_1(R)^G$ consists of the the connected components of $(\Omega^{\infty}R)^G$ given by $\pi_0^G(R)^{\times}$. This yields the first part of the statement.
    \medskip \\For the other part we need to argue that
    \[ \pi_V^G(\GL_1 R, 1) \cong \pi_V^G(\Omega^{\infty}R, 1) \cong \pi_V^G(R). \]
    Indeed, the elements in the first group can be written as a compatible system of based continuous maps $S^{V^H} \to (\GL_1 R)^H$ for all $H \leq G$ up to based homotopy. Since $V^G \neq 0$, we infer $V^H \neq 0$ for all $H \leq G$, so $S^{V^H}$ is connected. On the other hand, $(\GL_1 R)^H$ is a union of connected components of $(\Omega^{\infty}R)^H$. Since we are considering based maps with the same basepoint in the target, we deduce that the groups must be the same.
\end{proof}

\begin{lemma} \label{lemma: core monoidal}
    Let $\O^{\otimes} \in \Op_{G, \infty}$ and $\C^{\otimes} \in \Mon_{\O}(\Cat_{\infty})$. Then, the $\O$-monoidal structure of $\C^{\otimes}$ restricts onto an $\myuline{\O}$-monoidal structure on $\myuline{\C}^{\core}$.
\end{lemma}

\begin{proof}
    The functor $(-)^{\core}\colon  \Cat_{\infty} \to \Sc$ induces a $G$-right adjoint $(-)^{\myuline{\core}}\colon \myuline{\Cat}_{\infty, G} \to \myuline{\Sc}_{G}$ on the corresponding cofree $G$-$\infty$-categories \cite[Example 2.3.3]{cnossen2023parametrizedstabilityuniversalproperty}. Thus, it preserves parametrized products and hence corresponds to a $G$-symmetric monoidal functor $\myuline{\Cat}_{\infty, G}^{\times} \to \myuline{\Sc}_{G}^{\times}$ by \cite[Corollary A.21]{stewart2025tensorproductsequivariantcommutative}. This induces a functor
    \[ (-)^{\myuline{\core}}\colon  \Alg_{\myuline{\O}}\left(\myuline{\Cat}_{\infty, G}^{\times} \right) \to \Alg_{\myuline{\O}}(\myuline{\Sc}_{G}^{\times}) \]
    which endows $\myuline{\C}^{\myuline{\core}}$ with an $\myuline{\O}$-monoidal structure.
\end{proof}

\begin{construction} \label{construction: Pic}
    Let $\O^{\otimes} \in \Op_{G, \infty}$ and $\C^{\otimes} \in \Mon_{\O \otimes \E_1}(\Cat_{\infty})$. 
    \begin{enumerate}[(i)]
        \item We denote by $\tb{\myuline{\Pic}_{G}(\C)} = \GL_1(\myuline{\C}^{\myuline{\core}})$ the \tb{Picard $G$-space} of $\C$.
        \medskip \\In particular, it is a $G$-space and enhances to an $(\O \otimes \E_1)$-submonoidal $G$-$\infty$-category $\myuline{\Pic}_G^{\otimes}(\C) \subseteq \myuline{\C}^{\otimes}$ by the hard work of this subsection (\cref{corollary: GL1 monoidal}, \cref{lemma: core monoidal}). 
        \item Let $\P^{\otimes} \to \O^{\otimes} \otimes \E_1^{\otimes}$ be a map in $\Op_{G, \infty}$ and $A \in \Alg_{\O \otimes \E_1}(\C)$. We define $\tb{\Pic_{G}^{\otimes}(\C)_{\downarrow A}}$ as the pullback 
    \begin{center}
        \begin{tikzcd}
            \ul{\Pic}_{G}^{\otimes}(\C)_{\downarrow A} \arrow[r] \arrow[d] \arrow[dr, phantom, very near start, "\lrcorner"] & \ul{\C}^{\otimes}_{/A} \arrow[d]
            \\ \ul{\Pic}_{G}^{\otimes}(\C) \arrow[r] & \ul{\C}^{\otimes}
        \end{tikzcd}
    \end{center}
    in $\Mon_{\O \otimes \E_1}^{\NS}(\Cat_{\infty})$.
    \item Let $(\C^{\otimes}, A)$ be an $(\O \otimes \E_1)$-module datum. We write $\tb{\ul{\Pic}_{G}^{\otimes}(A)} = \ul{\Pic}_{G}^{\otimes}(\LMod_A(\C))$ and likewise for similar variants. 
    \end{enumerate}
    If we wish to use the Barkan--Haugseng--Steinebrunner formalism of equivariant higher algebra, then we will omit the underlines.
\end{construction}

\noindent Part (ii) is merely an example of a (monoidal enhancement) of a comma $\infty$-category but we single out this specimen because it takes a crucial role in our subsequent discussions of Thom spectra. Since this discussion will often involve left module categories, we shorten the notation in (iii), but we note that $\C$ is implicit in the notation.

\begin{remark}
    Pützstück has also constructed monoidal structures on parametrized Picard spaces, most notably on versions appearing in global equivariant homotopy theory \cite{pützstück2025globalpicardspectraborel}. He focuses on the ultracommutative setting while we discuss general $G$-$\infty$-operads.
\end{remark}

\begin{lemma} \label{lemma: pullback Pic and downarrow}
    Let $\O^{\otimes} \in \Op_{G, \infty}$ and $(\C^{\otimes}, A)$ be an $(\O \otimes \E_1)$-module datum and consider a map $\P^{\otimes} \to \O^{\otimes} \otimes \E_1^{\otimes}$ in $\Op_{G, \infty}$. Let $R \to A$ be a map in $\Alg_{\P \otimes \E_1}(\C)$. Then, there is a pullback square
    \begin{center}
        \begin{tikzcd}
            \ul{\Pic}_G^{\otimes}(R)_{\downarrow A} \arrow[r, "\Ind_R^A"] \arrow[d] \arrow[dr, phantom, very near start, "\lrcorner"] & \ul{\Pic}^{\otimes}_G(A)_{\downarrow A} \arrow[d]
            \\ \ul{\Pic}_G^{\otimes}(R) \arrow[r, "\Ind_R^A", swap] & \ul{\Pic}_G^{\otimes}(A)
        \end{tikzcd}
    \end{center}
    in $\Mon_{\P}^{\NS}(\Cat_{\infty})$.
\end{lemma}

\begin{proof}
    Note that $\Ind_R^A$ is $\P$-monoidal (\cref{corollary: functoriality of LMod in A}). The commutativity of the square defines a comparison map 
    \[ \ul{\Pic}_G^{\otimes}(R)_{\downarrow A} \to \ul{\Pic}_G^{\otimes}(A)_{\downarrow A} \times_{\ul{\Pic}_G^{\otimes}(A)} \ul{\Pic}_G^{\otimes}(R) \]
    in $\Mon_{\P}^{\NS}(\Cat_{\infty})$. It thus suffices to check that this is an equivalence on underlying $\infty$-categories. Consider the commutative diagram
    \begin{center}
        \begin{tikzcd}
            \ul{\Pic}_G(R)_{\downarrow A} \arrow[r] \arrow[d] \arrow[dr, phantom, very near start, "\lrcorner"] & \ul{\LMod}_{R/A}^G \arrow[d] \arrow[r, "\Ind_R^A"] & \ul{\LMod}_{A/A}^G \arrow[d]
            \\ \ul{\Pic}_G(R) \arrow[r] & \ul{\LMod}_R^G \arrow[r, "\Ind_R^A", swap] & \ul{\LMod}_A^G
        \end{tikzcd}
    \end{center}
    The right square is a pullback square by a parametrized version of a general category theory result: If $L\colon  \C \leftrightarrows \D\colon  R$ is an adjunction, then the composite rectangle
    \begin{center}
        \begin{tikzcd}
            \C_{/Rd} \arrow[r, "L"] \arrow[d] & \D_{/LRd} \arrow[r, "\epsilon_d"] \arrow[d] & \D_{/d} \arrow[d]
            \\ \C \arrow[r] & \D \arrow[r, equal] & \D 
        \end{tikzcd}
    \end{center}
    is a pullback square.\footnote{Here are two ways of checking this: 
    \begin{enumerate}
        \item Check that the comparison map $\C_{/Rd} \to \D_{/d} \times_{\D} \C$ is essentially surjective and fully faithful. 
        \item Unstraighten both sides of $\Map_{\C}(-, Rd) \simeq \Map_{\D}(L-, d)$.
    \end{enumerate}} 
    \medskip \\Thus, we obtain
    \begin{align*}
        \ul{\Pic}_G(R)_{\downarrow A} &\simeq \ul{\LMod}^G_{A/A} \times_{\ul{\LMod}_A^G} \ul{\Pic}_G(R) 
        \\ &\simeq \ul{\LMod}_{A/A}^G \times_{\ul{\LMod}_A^G} \ul{\Pic}_G(A) \times_{\ul{\Pic}_G(A)} \ul{\Pic}_G(R)
        \\ &\simeq \ul{\Pic}_G(A)_{\downarrow A}  \times_{\ul{\Pic}_G(A)} \ul{\Pic}_G(R)
    \end{align*}
    as desired.
\end{proof}

\begin{corollary} \label{corollary: GL1 Pic(R)A as pullback}
    Let $\O^{\otimes} \in \Op_{G, \infty}$ and $(\C^{\otimes}, A)$ be an $(\O \otimes \E_1)$-module datum and consider a map $\P^{\otimes} \to \O^{\otimes} \otimes \E_1^{\otimes}$ in $\Op_{G, \infty}$. Let $R \to A$ be a map in $\Alg_{\P \otimes \E_1}(\C)$. Then, 
    \[ \GL_1(\ul{\Pic}_G^{\otimes}(R)_{\downarrow A}) \simeq \ul{\Pic}_G^{\otimes}(R) \times_{\ul{\Pic}_G^{\otimes}(A)} * \] 
    in $\Mon_{\P}(\Sc)$.
\end{corollary}

\begin{proof}
    The functor $\GL_1$ is a Bousfield colocalization, so applying it to the pullback in the previous Lemma (\cref{lemma: pullback Pic and downarrow}) yields 
    \[ \GL_1(\ul{\Pic}_G^{\otimes}(R)_{\downarrow A}) \simeq \ul{\Pic}_G^{\otimes}(R) \times_{\ul{\Pic}_G^{\otimes}(A)} \GL_1(\ul{\Pic}_G^{\otimes}(A)_{\downarrow A}). \]
    Since $\ul{\Pic}_G^{\otimes}(A)_{\downarrow A}$ is already a space, we have $
        \GL_1(\ul{\Pic}_G^{\otimes}(A)_{\downarrow A}) \simeq \GL_1 \left((\ul{\Pic}_G^{\otimes}(A)_{\downarrow A})^{\core} \right)$ and applying $\GL_1 \circ (-)^{\core}$ to the defining pullback square of $\ul{\Pic}_G^{\otimes}(A)_{\downarrow A}$ yields a pullback square
        \begin{center}
            \begin{tikzcd}
                \GL_1(\ul{\Pic}_G^{\otimes}(A)_{\downarrow A}) \arrow[r] \arrow[dr, phantom, very near start, "\lrcorner"] \arrow[d] & \GL_1 \core(\ul{\LMod}^G_A(\C)_{/A}^{\otimes}) \arrow[d]
                \\ \GL_1 (\ul{\Pic}_G^{\otimes}(A)) \arrow[r] & \GL_1 \core(\ul{\LMod}_A^G(\C)^{\otimes})
            \end{tikzcd}
        \end{center}
        which unravels to
        \begin{center}
            \begin{tikzcd}
                \GL_1(\ul{\Pic}_G^{\otimes}(A)_{\downarrow A}) \arrow[r] \arrow[dr, phantom, very near start, "\lrcorner"] \arrow[d] & \ul{\Pic}_G^{\otimes}(A)_{/A} \arrow[d]
                \\ \ul{\Pic}_G^{\otimes}(A) \arrow[r] & \ul{\Pic}_G^{\otimes}(A) 
            \end{tikzcd}
        \end{center}
        so $\GL_1(\ul{\Pic}_G^{\otimes}(A)_{\downarrow A}) \simeq \ul{\Pic}_G^{\otimes}(A)_{/A} \simeq *$.
\end{proof}

\section{Multiplicative Parametrized Thom Spectra}\label{sec:param_Thom}

\subsection{Setup \& Universal Property of Multiplicative Parametrized Thom Spectra}\label{sec:universal_property_of_Thom}
We have finally assembled all ingredients to equivariantize the universal property of multiplicative Thom spectra from Antolín-Camarena--Barthel \cite{antolinbarthel2019thom}. Once the language is set up, our proofs are essentially the same as in \cite[Section 3.1]{antolinbarthel2019thom}, but we will write them out for the convenience of the reader (\cref{theorem: universal property of Th}, \cref{corollary: Ind of Th}). Let us first recall a parametrized definition of (multiplicative) equivariant Thom spectra and then prove a parametrized version of the Antolín-Camarena--Barthel universal property.
\medskip \\A parametrized version of the Ando--Blumberg--Gepner--Hopkins--Rezk \cite{abghr2014infty} Thom spectrum definition gives:

\begin{definition} \label{def: parametrized thom spectrum}
    Let $A \in \Alg_{\E_1}(\C)$ and suppose that $\C$ is $G$-presentable, then we define the \tb{parametrized Thom object} functor $\tb{\Th_G}\colon  \myuline{\Sc}_{G/\ul{\Pic}_G(A)} \to \myuline{\LMod}_A^G(\C)$ as the parametrized Yoneda extension 
    \begin{center}
        \begin{tikzcd}
            \ul{\Pic}_{G}(A) \arrow[r] \arrow[d, hookrightarrow, "\yo_G", swap] & \myuline{\LMod}_A^{G}(\C)
            \\ \myuline{\PSh}_G(\ul{\Pic}_G(A)) \arrow[ur, dashed, "\Th_{G}", swap]
        \end{tikzcd}
    \end{center}
    using $\myuline{\Sc}_{G/\ul{\Pic}_{G}(A)} \simeq \myuline{\PSh}_{G}(\ul{\Pic}_{G}(A))$ by parametrized straightening-unstraightening (\cref{prop: parametrized straightening}).
\end{definition}

\begin{remark}
    \hfill 
    \begin{enumerate}[(i)]
        \item Thom spectra are also often denoted by $\Th_G(f) = \tb{\mathrm{M}f}$, motivated by classical examples such as $\mathrm{MO}, \MU, \mathrm{MSpin}, \cdots$ (and their equivariant versions). We will use this notation in  \cref{part:applications}.
        \item For $H \leq G$ we recover the non-parametrized $H$-Thom spectrum functor 
        \[ \Sc_{H/\ul{\Pic}_{H}(A)} \to \LMod_{A_H}(\C_H), \]
        which is given by the parametrized colimit
        \[ \Th_H(X \to \ul{\Pic}_{H}(A)) \simeq \myuline{\colim} \left(X \to \ul{\Pic}_{H}(A) \to \myuline{\LMod}_A^H(\C) \right), \]
        as demonstrated in \cite[Corollary 5.2.3]{hahn2024equivariantnonabelianpoincareduality}.
    \end{enumerate}
\end{remark} 

\begin{example} \label{example: thom along trivial}
    The following also enhance multiplicatively.
    \begin{enumerate}[(i)]
        \item Let $c\colon  X \to \ul{\Pic}_G(R)$ be equivalent to the trivial map picking out some $M \in \Pic_G(R)$. Its Thom object computed as the constant colimit, so 
        \[ \Th_G(c) \simeq  \ul{\colim}_X M = X \otimes M \simeq \Sigma_+^{\infty}X \otimes M. \]
        \item One computes 
        \[ \Th_G \left(X \times Y \to X \xrightarrow{f} \ul{\Pic}_G(A) \right) \simeq \Th_G(f) \otimes_A (A \otimes Y) \] 
        by symmetric monoidality of $\Th_G$ in the slice monoidal structure and (i). This does not need full parametrized (macrocosmic) monoidal straightening-unstraightening to pass from presheaves, instead we only need the classical statement \cite[Corollary 4.9]{ramzi2022monoidalgrothendieckconstructioninftycategories} at the $G$-level.
    \end{enumerate}
\end{example}

\noindent We will now mimick the Thom spectrum definition (\cref{def: parametrized thom spectrum}) multiplicatively.

\begin{construction} \label{construction: multiplicative thom spectra}
    Let $(\O^{\otimes} \otimes \E_1^{\otimes}, \C^{\otimes}, A)$ be a distributive module datum (\cref{def: distributive module datum}). Then we define the $(\O \otimes \E_1)$-monoidal \tb{multiplicative $G$-Thom object} functor 
    \[ \tb{\Th_G^{\otimes}}\colon  \ul{\PSh}_G(\ul{\Pic}_G(A))^{\otimes} \to \ul{\LMod}_A^G(\C)^{\otimes} \] 
    as the multiplicative parametrized Yoneda extension
    \begin{center}
        \begin{tikzcd}
            \ul{\Pic}_{G}^{\otimes}(A) \arrow[r] \arrow[d, "\yo_G^{\otimes}", swap] & \ul{\LMod}_A^G(\C)^{\otimes}
            \\  \ul{\PSh}_{G}(\ul{\Pic}_{G}(A))^{\otimes} \arrow[ur, dashed, "\Th_{G}^{\otimes}", swap]
        \end{tikzcd}
    \end{center}
    via an $(\O \otimes \E_1)$-monoidal version of \cite[Corollary 6.0.12]{nardinshah2022equivarianttopos}\footnote{The result is stated for $G$-symmetric monoidal $G$-$\infty$-categories because they passed the adjunctions to $G$-commutative algebras. Passing to $\O$-algebras allows more general operads.} using distributivity of $\LMod_A^G(\C)^{\otimes}$ (see \cref{theorem: LMod distributive}).
\end{construction}

\noindent Since $\Th_G^{\otimes}$ is $G$-colimit preserving \cite[Corollary 6.0.12]{nardinshah2022equivarianttopos}, this extends the parametrized Thom spectrum construction (\cref{def: parametrized thom spectrum}).
\medskip \\We required an additional $\E_1$ to get $\Pic_G$ running (\cref{construction: Pic}), which is why we will always need $(\O \otimes \E_1)$ to get started. But once $\Pic_G^{\otimes}$ is constructed, we can restrict its multiplicative structure to potentially demand less structure from the other participating objects. That's why we will often carry around an $G$-$\infty$-operad map $\P^{\otimes} \to \O^{\otimes} \otimes \E_1^{\otimes}$.

\begin{corollary} \label{corollary: multiplicative thom spectrum}
    Let $(\O^{\otimes} \otimes \E_1^{\otimes}, \C^{\otimes}, A)$ be a distributive module datum and consider a map $\P^{\otimes} \to \O^{\otimes} \otimes \E_1^{\otimes}$ in $\Op_{G, \infty}$. There is a functor 
    \[ \Th_{G}^{\otimes}\colon  \Mon_{\P}(\Sc)_{/\Pic_{G}^{\otimes}(A)} \to \Alg_{\ul{\P}/\ul{\P}}(\ul{\LMod}_A(\C)) \to \Alg_{\ul{\P}}(\ul{\LMod}_A(\C)) \]
    extending the parametrized Thom spectrum functor $\Th_{G}$.
\end{corollary}

\begin{proof}
    We apply $\Alg_{\ul{\P}/\ul{\P}}$ to $\Th_{G}^{\otimes}$ and use microcosmic straightening-unstraightening (\cref{theorem: microcosmic straightening-unstraightening}) to identify the left side. The second functor is the forgetful functor.
\end{proof}

\noindent This shows the monoidality in maps $X^{\otimes} \to \ul{\Pic}_G^{\otimes}(A)$, which is not immediate from the Day convolution monoidal structure on $\ul{\PSh}_G(\ul{\Pic}_G(A))^{\otimes}$.

\begin{remark}
    Our construction of the multiplicative parametrized Thom spectra is made to have good functoriality properties but there are other ways to obtain the multiplicative structure on parametrized Thom spectra. The following is one categorical level lower with the downside that it comes with less functoriality but with the upside of a different universal property, which is the one of particular interest to us in this article.
    \begin{enumerate}[(i)]
        \item Let $(\O^{\otimes} \otimes \E_1^{\otimes}, \C^{\otimes}, A)$ be a distributive module datum and let $\P^{\otimes} \to \O^{\otimes} \otimes \E_1^{\otimes}$ be a map in $\Op_{G, \infty}$. Given an $\P$-monoidal map $f\colon X^{\otimes} \to \ul{\Pic}_{G}^{\otimes}(A)$, we obtain the multiplicative Thom spectrum by operadic left Kan extension (\cref{theorem: operadic left Kan extension}) 
    \begin{center}
        \begin{tikzcd}
            X^{\otimes} \arrow[r, "f"] \arrow[d] & \ul{\Pic}_{G}^{\otimes}(A) \arrow[r] & \ul{\LMod}_A^{G}(\C)^{\otimes}
            \\ \ul{\P}^{\otimes} \arrow[urr, dashed, "\Th_{G}^{\otimes}(f)", swap, bend right]
        \end{tikzcd}
    \end{center}
    which is the technique used in \cite{antolinbarthel2019thom} to obtain an $(\O \otimes \E_1)$-algebra structure on $\Th_{G}(f)$.
    \item Mimicking \cite[Proposition 7.8]{carmelicnossenramziyanovski2025characters} with the language of parametrized higher category theory shows that the two monoidal structures on equivariant Thom spectra agree.
    \end{enumerate}
\end{remark}

\noindent Having set up the multiplicative equivariant Thom spectrum functor, we can now give a universal property, which is an equivariant version of \cite[Theorem 3.5]{antolinbarthel2019thom}.

\begin{remark}
    Let $\O^{\otimes} \in \Op_{G, \infty}$ and let $(\C^{\otimes}, R)$ be an $\O$-module datum. Then, consider an algebra $A \in \Alg_{\O}(\LMod_R)$. It comes with an $\O$-map from the initial object $R \to A$ by \cite[Theorem 5.2.11(i)]{nardinshah2022equivarianttopos}. 
\end{remark}

\begin{theorem} \label{theorem: universal property of Th}
    \fix{Let $(\O^{\otimes} \otimes \E_1^{\otimes}, \C^{\otimes}, R)$ be a distributive module datum, $\P^{\otimes} \to \O^{\otimes} \otimes \E_1^{\otimes}$ be a map in $\Op_{G, \infty}$ and $f \colon X^{\otimes} \to \Pic_G^{\otimes}(R)$ be a map in $\Mon_{\P}(\Sc)$. Let $A \in \Alg_{\P}(\LMod_R)$.} There is a natural equivalence 
    \[ \Map_{\Alg_{\P}(\LMod_R)}(\Th_G^{\otimes}(f), A) \simeq \Map_{\Mon_{\P}(\Sc)_{/\Pic_{G}^{\otimes}(R)}} \left(X, \Pic_G^{\otimes}(R)_{\downarrow A} \right), \]
    i.e. there is an adjunction $\Th_{G}^{\otimes}(-) \dashv \Pic_{G}^{\otimes}(\C)_{\downarrow (-)}$.
\end{theorem}

\begin{proof}
    We write $i\colon \Pic_{G}^{\otimes}(R) \to \LMod_R^{G}(\C)^{\otimes}$ and perform the following chain of natural equivalences: \begin{align*}
        \Map_{\Alg_{\P}(\LMod_R)}(\Th_{G}^{\otimes}(f), A) &\simeq \Map_{\Fbrs(\P^{\otimes})_{/\LMod_R^{G}(\C)^{\otimes}}}(X^{\otimes}, \LMod_R^{G}(\C)^{\otimes}_{/A})
        \\ &\simeq \Map_{\Fbrs(\P^{\otimes})}(X^{\otimes}, \LMod_R^{G}(\C)_{/A}^{\otimes}) \times_{\Map_{\Fbrs(\P^{\otimes})}(X, \LMod_R^{G}(\C)^{\otimes})} \{if \}
        \\ &\simeq \Map_{\Fbrs(\P^{\otimes})}(X^{\otimes}, \Pic_{G}^{\otimes}(R)_{\downarrow A}) \times_{\Map_{\Fbrs(\P^{\otimes})}(X, \Pic_{G}^{\otimes}(R))} \{f \}
        \\ &\simeq \Map_{\Mon_{\P}(\Sc)}(X^{\otimes}, \Pic_{G}^{\otimes}(R)_{\downarrow A}) \times_{\Map_{\Mon_{\P}(\Sc)}(X, \Pic_{G}^{\otimes}(R))} \{f \}
        \\ &\simeq \Map_{\Mon_{\P}(\Sc)_{/\Pic_{G}^{\otimes}(R)}}(X^{\otimes}, \Pic_{G}^{\otimes}(R)_{\downarrow A}).
    \end{align*}
    The first equivalence is \cref{prop: abstract antolin-camarena-barthel}. In the second and last equivalence we wrote out mapping spaces of slice $\infty$-categories. The third equivalence comes from pullback pasting
    \begin{center}
        \begin{tikzcd}
            \bullet \arrow[r] \arrow[d] \arrow[dr, phantom, very near start, "\lrcorner"] & \Map_{\Fbrs(\P^{\otimes})}(X^{\otimes}, \Pic_{G}^{\otimes}(R)_{\downarrow A}) \arrow[r] \arrow[d] \arrow[dr, phantom, very near start, "\lrcorner"] & \Map_{\Fbrs(\P^{\otimes})}(X^{\otimes}, \LMod_R^{G}(\C)_{/A}^{\otimes}) \arrow[d]
            \\ * \arrow[r, "f", swap] & \Map_{\Fbrs(\P^{\otimes})}(X^{\otimes}, \Pic_{G}^{\otimes}(R)) \arrow[r, "i_*", swap] & \Map_{\Fbrs(\P^{\otimes})}(X^{\otimes}, \LMod_R^{G}(\C)^{\otimes})
        \end{tikzcd}
    \end{center}
    where the right square is a pullback, hence the left one is a pullback if and only if the total rectangle is a pullback. The fourth equivalence is because the objects are space-valued fibrous patterns, so they are automatically coCartesian fibrations and maps automatically preserve coCartesian edges, which means that we may pass to $\Mon_{\P}(\Sc)$.
\end{proof}

\noindent Colloquially, a map $\Th_{G}^{\otimes}(f) \to A$ corresponds to a lift
\begin{center}
        \begin{tikzcd}
            & \Pic_{G}^{\otimes}(R)_{\downarrow A} \arrow[d]
            \\ X \arrow[r, "f", swap] \arrow[ur, dashed, "\exists !"] & \Pic_{G}^{\otimes}(R)
        \end{tikzcd}
    \end{center}
in $\Mon_{\P}(\Sc)$.

\begin{construction} \label{construction: Ind operads}
    Let $\O^{\otimes}, \Q^{\otimes} \in \Op_{G, \infty}$ and $\O^{\otimes} \to \Q^{\otimes}$ be a map in $\Op_{G, \infty}$. Let $\myuline{\C}^{\otimes}$ be a distributive $\Q$-monoidal $G$-$\infty$-category. Then, $\Res_{\ul{\O}}^{\ul{\Q}}\colon  \Alg_{\ul{\Q}}(\ul{\C}) \to \Alg_{\ul{\O}}(\ul{\C})$ preserves limits since limits are computed underlying \cite[Theorem 5.1.3]{nardinshah2022equivarianttopos}.  Moreover, these algebra categories are presentable by distributivity \cite[Theorem 5.1.4(4)]{nardinshah2022equivarianttopos}, so the adjoint functor theorem yields a left adjoint $\tb{\Ind_{\O}^{\Q}}\colon  \Alg_{\ul{\O}}(\ul{\C}) \to \Alg_{\ul{\Q}}(\ul{\C})$.
\end{construction}

\begin{corollary} \label{corollary: Ind commutes with Th}
    Let $(\O^{\otimes} \otimes \E_1^{\otimes}, \C^{\otimes}, R)$ be a distributive module datum and consider a map $\P^{\otimes} \to \O^{\otimes} \otimes \E_1^{\otimes}$ in $\Op_{G, \infty}$. Consider a map $f:X^{\otimes} \to \Pic_{G}^{\otimes}(R)$ in $\Mon_{\P}(\Sc)$ and another map $\P^{\otimes} \to \Q^{\otimes} \to \O^{\otimes} \otimes \E_1^{\otimes}$ in $\Op_{G, \infty}$. Then, there is a natural equivalence 
    \[ \Ind_{\P}^{\Q}(\Th_{G}^{\otimes}(f)) \simeq \Th_{G}^{\otimes}\left( \widetilde{f} \colon \Ind_{\P}^{\Q}X^{\otimes} \to \ul{\Pic}_G^{\otimes}(R) \right). \]
    where $\widetilde{f} \colon \Ind_{\P}^{\Q}X^{\otimes} \to \ul{\Pic}_G^{\otimes}(R)$ is the $\Q$-map associated to $X^{\otimes} \to \ul{\Pic}_G^{\otimes}(R)$.
\end{corollary}

\begin{proof}
    We wish to employ a Yoneda argument and proceed by using the universal property (\cref{theorem: universal property of Th}). Indeed, we compute
    \begin{align*}
        \Map_{\Alg_{\E_V}(\LMod_R)} \left(\Ind_{\E_0}^{\E_V} \Th^{\otimes}_G(f), A \right) &\simeq \Map_{\Alg_{\E_0}(\LMod_R)}(\Th^{\otimes}_G(f), A)
        \\ &\simeq \Map_{\Mon_{\E_0}(\Sc)_{/\Pic_G^{\otimes}(R)}} \left(X^{\otimes}, \Pic_G^{\otimes}(R)_{\downarrow A} \right)
        \\ &\simeq \Map_{\Mon_{\E_V}(\Sc)_{/\Pic_G^{\otimes}(R)}}\left(\Ind_{\P}^{\Q}X^{\otimes}, \Pic_G^{\otimes}(R)_{\downarrow A} \right)
        \\ &\simeq \Map_{\Alg_{\E_V}(\LMod_R)}(\Th_G^{\otimes}(\ol{f}), A).
    \end{align*}
    The third equivalence is obtained by writing out the mapping spaces of slice categories and then applying the universal property of $\Ind_{\P}^{\Q}$ by noting that on the right side we have grouplike objects.
\end{proof}

\noindent We will give a slight variant of this result in the grouplike setting in the next subsection (\cref{corollary: Ind of Th}).

\subsection{Abstract Orientation Theory}\label{sec:abstract_orientation_theory}
By definition (\cref{def: parametrized thom spectrum}), a Thom object is a $G$-colimit of some map $X \to \ul{\Pic}_G(R)$ of $G$-spaces, which can be viewed as a twisted version of $X \otimes R$ (see \cref{example: thom along trivial}). Let $R \to A$ be another algebra map, along which we can base change 
\begin{center}
    \begin{tikzcd}
        X \arrow[r] & \ul{\Pic}_G(R) \arrow[r, "\Ind_R^A"] & \ul{\Pic}_G(A).
    \end{tikzcd}
\end{center} 
The base change could have had the effect of making this composite nullhomotopic, which ultimately leads to a detwisting of the Thom spectrum after base changing (\cref{example: thom along trivial}). The detwisting is realized by such a nullhomotopy, which is then rightfully called \emph{orientation}.
\medskip \\In this subsection, we will study the structure of multiplicative orientations along with the phenomena that it comes with -- importantly, the multiplicative Thom isomorphism (\cref{theorem: monoidal thom iso}). Throughout, we will always start with the same data, for which we introduce some additional terminology for the sake of brefity.
\begin{notation} \label{notation: orientation datum}
    An \tb{orientation datum} is a tuple $(\O^{\otimes}, \P^{\otimes}, \C, R, A, f)$ consisting of 
    \begin{itemize}
        \item a distributive module datum $(\O^{\otimes} \otimes \E_1^{\otimes}, \C^{\otimes}, R)$; see \cref{def: distributive module datum},
        \item a map $\P^{\otimes} \to \O^{\otimes} \otimes \E_1^{\otimes}$ in $\Op_{G, \infty}$,
        \item a map $f\colon X^{\otimes} \to \Pic_G^{\otimes}(R)$ in $\Mon_{\P}(\Sc)$,
        \item \fix{an algebra $A \in \Alg_{\P \otimes \E_1}(\C)$ with a $\P^{\otimes} \otimes \E_1^{\otimes}$-map $R \to A$.}\footnote{Here, $R$ was by definition $\O^{\otimes} \otimes \E_2^{\otimes}$ which we restrict to $\P^{\otimes} \otimes \E_1^{\otimes}$.}
    \end{itemize}
    \fix{In particular, $A \simeq R \otimes_R A \in \Alg_{\P}(\LMod_R)$.}
\end{notation}

\begin{remark}
    \fix{We warn the reader that the necessary data for the results of this section can be quite confusing, see also \cite[Remark 3.7]{antolinbarthel2019thom}. First, we start with an $(\O \otimes \E_2)$-ring $R$ out of two reasons: Taking $\LMod$ kills one copy of $\E_1$ and to run our theory of Picard spaces (\cref{corollary: GL1 monoidal}), we needed another copy of $\E_1$. Once we obtain these objects, we can freely restrict the operad structure which is why we consider a map $\P^{\otimes} \to \O^{\otimes} \otimes \E_1^{\otimes}$.
    \medskip \\Next, we demand an algebra $A \in \Alg_{\P \otimes \E_1}(\C)$ together with a $(\P \otimes \E_1)$-map $R \to A$. We needed to introduce an $\E_1$-again, so that the base change map $\Ind_R^A$ is a $\P$-map and this base change map is relevant for the theory of orientations. Note that this is different than the condition in \cref{theorem: universal property of Th} -- we only needed a weaker statement there because the base change functor $\Ind_R^A$ did not show up.}
\end{remark}

\noindent We initiate the study of orientations through a universal example, parametrizing a program from Antolín-Camarena--Barthel \cite[Section 3.2]{antolinbarthel2019thom}.

\begin{definition} \label{def: orientations}
    Let $(\O^{\otimes}, \P^{\otimes}, \C, R, A, f)$ be an orientation 
    datum. The space of \tb{$\P$-$f$-orientations} of $A$ is
    \[ \tb{\Or_A^{\P}(f)} = \Map_{\Mon_{\P}(\Sc)_{/\Pic_G^{\otimes}(R)}}\left(f, \GL_1(\Pic^{\otimes}(R)_{\downarrow A}) \to \Pic_G^{\otimes}(R) \right). \]
\end{definition}

\begin{remark}
    We warn that Antolín-Camarena--Barthel would call these $\P$-$A$-orientations of $f$, see \cite[Definition 3.14]{antolinbarthel2019thom}. If the source space of $f$ is grouplike, then we will see that an $\P$-$f$-orientation of $A$ is a $\P$-map $\mathrm{M}f \to A$, see \cref{lemma: grouplike orientation}.  For example, a map $\MU \to E$ is typically called $\MU$-orientation, so we have decided to adapt our terminology to this situation.
\end{remark}

\noindent We first show that this coincides with the approach from Antolín-Camarena--Barthel involving their $B(R, A)$ \cite[Definition 3.12]{antolinbarthel2019thom}.

\begin{lemma} \label{lemma: GL vs B}
    Let $(\O^{\otimes}, \P^{\otimes}, \C, R, A, f)$ be an orientation datum and $H \leq G$, then the underlying space of $\GL_1(\Pic^{\otimes}(R)_{\downarrow A})^H$ is the subgroupoid of $\Pic(R_H)_{\downarrow A_H}$ consisting of $R_H$-module maps $M \to A_H$ such that the adjoint $\Ind_{R_H}^{A_H}M \to A_H$ is an equivalence. 
\end{lemma}

\noindent The latter is also called $B(R_H, A_H)$ in the language of Antolín-Camarena--Barthel \cite[Definition 3.12]{antolinbarthel2019thom}.

\begin{proof}
    Since this is a completely levelwise statement, it amounts to an entirely classical check $\GL_1(\Pic^{\otimes}(R)_{\downarrow A}) \simeq B(R, A)$ for the subgroup $H = e$. To do so, we use  the equivalence $\GL_1(\Pic^{\otimes}(R)_{\downarrow A}) \simeq \Pic(R) \times_{\Pic{A}} *$ from \cref{corollary: GL1 Pic(R)A as pullback}, which is $B(R, A)$, as noted in the proof of \cite[Proposition 3.16]{antolinbarthel2019thom}.
\end{proof}

\noindent The upshot is that the $\GL_1$-business allows us to bypass some technical checks in parametrized higher category theory such as $B(R, A)$ admitting a parametrized version which furthermore admits a parametrized multiplicative enhancement. All of this was already done by our hard work involving $\GL_1$. 
\medskip \\Let us now demonstrate that this recovers the classical interpretations of (multiplicative) orientations.

\begin{corollary} \label{corollary: characterization of orientations}
    Let $(\O^{\otimes}, \P^{\otimes}, \C, R, A, f)$ be an orientation datum. The following are equivalent characterizations of $\P$-$f$-orientations of $A$:
    \begin{enumerate}[(i)]
        \item A $\P$-lift
        \begin{center}
        \begin{tikzcd}
        &\GL_1(\Pic^{\otimes}_G(R)_{\downarrow A}) \arrow[d]
                \\ X^{\otimes} \arrow[r, "f", swap] \arrow[ur, dashed] & \Pic_G^{\otimes}(R)
            \end{tikzcd}
        \end{center}
        of $f$.
        \item A nullhomotopy of the composite 
        \begin{center}
            \begin{tikzcd}
                X^{\otimes} \arrow[r, "f"] & \Pic_G^{\otimes}(R) \arrow[r, "\Ind_R^A"] & \Pic_G^{\otimes}(A).
            \end{tikzcd}
        \end{center}
        in $\Mon_{\P}(\Sc)$.
        \item A map $\Th_G^{\otimes}(f) \to A$ in $\Alg_{\P}(\LMod_R)$ such that for every $x\colon * \to X$ the adjoint $A$-module map corresponding to the $R$-module map 
        \begin{center}
            \begin{tikzcd}
                \Th_G(f \circ x) \arrow[r] & \Th_G(f) \arrow[r] & A
            \end{tikzcd}
        \end{center} is an equivalence.
    \end{enumerate}
\end{corollary}

\begin{proof}
    \hfill 
    \begin{enumerate}[(i)]
        \item This is \cref{def: orientations} (after unravelling the mapping spaces of slice $\infty$-categories).
        \item This follows from $\GL_1(\Pic_G^{\otimes}(R)_{\downarrow A}) \simeq \Pic_G^{\otimes}(R) \times_{\Pic_G^{\otimes}(A)} *$ (see \cref{corollary: GL1 Pic(R)A as pullback}).
        \item This follows from our explicit description of $\GL_1(\Pic_G^{\otimes}(R)_{\downarrow A})$ (see \cref{lemma: GL vs B}) and the universal property of multiplicative parametrized Thom spectra (\cref{theorem: universal property of Th}).
    \end{enumerate}
\end{proof}

\noindent The description (ii) already featured in the intro of this subsection and (iii) lifts the classical notion of Thom classes.

\begin{lemma}
    Let $(\O^{\otimes}, \P^{\otimes}, \C, R, A, f)$ be \label{lemma: grouplike orientation} an orientation datum. Let $X \in \Mon_{\P}(\Sc)$. Suppose one of the following conditions.
    \begin{enumerate}[(i)]
        \item Assume that $X$ is levelwise connected.
        \item Assume that there is a map $\E_1^{\otimes} \to \O^{\otimes}$ and that $X$ is grouplike.
    \end{enumerate}
    Then, $\Or_A^{\P}(f) \simeq \Map_{\Alg_{\P}(\LMod_R(\C))}(\Th_G^{\otimes}(f), A)$.
\end{lemma}

\begin{proof}
    By the universal property (\cref{theorem: universal property of Th}) of $\Th_{G}^{\otimes}$ we have
    \[ \Map_{\Alg_{\P}(\LMod_R(\C))}(\Th_{G}^{\otimes}(f), A) \simeq \Map_{\Mon_{\P}(\Sc)_{/\Pic_{G}^{\otimes}(R)}} \left(X, \Pic_G^{\otimes}(R)_{\downarrow A} \right) \]
    so we need to show that every map $X \to \Pic_{G}^{\otimes}(R)_{\downarrow A}$ factors through $\GL_1(\Pic^{\otimes}_G(R)_{\downarrow A})$ for grouplike $X$. This is a levelwise statement which with \cref{lemma: GL vs B} becomes \cite[Lemma 3.15]{antolinbarthel2019thom}.
\end{proof}

\noindent With this we can give a variation of \cref{corollary: Ind commutes with Th}.

\begin{corollary} \label{corollary: Ind of Th}
    Let $(\O^{\otimes} \otimes \E_1^{\otimes}, \C^{\otimes}, R)$ be a distributive module datum. Let $X$ be a (levelwise) connected pointed $G$-space and consider a map $f\colon X \to \ul{\Pic}_{G}(R)$ of pointed $G$-spaces. Suppose that there is a map $\E_V^{\otimes} \to \O^{\otimes} \otimes \E_1^{\otimes}$. Then, there is a natural equivalence 
    \[ \Ind^{\E_V}_{\E_0}\left(\Th_{G}^{\otimes}(f) \right) \simeq \Th_{G}^{\otimes}\left(\ol{f} \colon \Omega^V \Sigma^V X \to \ul{\Pic}_G(R) \right) \]
    where $\ol{f}$ is the $V$-fold loop map extension \cite[Theorem 3.15]{juran2025genuineequivariantrecognitionprinciple}.
\end{corollary}

\begin{proof}
    The proof is essentially the same as in \cref{corollary: Ind commutes with Th} besides the new ingredients coming from grouplikeness. We include the full proof for the convenience of the reader. 
    \medskip \\The equivariant approximation theorem \cite[Theorem 3.15]{juran2025genuineequivariantrecognitionprinciple} allows us to define $\ol{f}$ as stated. Let $A \in \Alg_{\E_V}(\LMod_R)$. We wish to employ a Yoneda argument and proceed by using the universal property (\cref{theorem: universal property of Th}). Indeed, we compute
    \begin{align*}
        \Map_{\Alg_{\E_V}(\LMod_R)} \left(\Ind_{\E_0}^{\E_V} \Th^{\otimes}_G(f), A \right) &\simeq \Map_{\Alg_{\E_0}(\LMod_R)}(\Th^{\otimes}_G(f), A)
        \\ &\simeq \Map_{\Mon_{\E_0}(\Sc)_{/\Pic_G^{\otimes}(R)}} \left(X, \GL_1(\Pic_G^{\otimes}(R)_{\downarrow A}) \right)
        \\ &\simeq \Map_{\Mon_{\E_V}(\Sc)_{/\Pic_G^{\otimes}(R)}}\left(\Omega^V \Sigma^V X, \GL_1(\Pic_G^{\otimes}(R)_{\downarrow A}) \right)
        \\ &\simeq \Map_{\Alg_{\E_V}(\LMod_R)}(\Th_G^{\otimes}(\ol{f}), A).
    \end{align*}
    The third equivalence is obtained by writing out the mapping spaces of slice categories and then applying the universal property of $\Omega^V \Sigma^V$ by noting that on the right side we have grouplike objects. The second and fourth equivalence use the grouplike/connectedness condition (\cref{lemma: grouplike orientation}).
\end{proof}

\noindent Furthermore, \cref{lemma: grouplike orientation} gives rise to a simple example, which is a surprisingly useful orientation, see \cref{corollary: lifting idempotents}.

\begin{example} \label{example: id orientation}
    Let $(\O^{\otimes}, \P^{\otimes}, \C, R, A, f)$ be an orientation datum. Suppose that $X^{\otimes}$ is grouplike. Then, $\id_{\Th_G^{\otimes}(f)}\colon  \Th_G^{\otimes}(f) \to \Th_G^{\otimes}(f)$ is a $\P$-$\Th_G^{\otimes}(f)$-orientation of $\Th_G^{\otimes}(f)$.
\end{example}

\begin{theorem} \label{theorem: monoidal thom iso}
    Let $(\O^{\otimes}, \P^{\otimes}, \C, R, A, f)$ be an orientation datum. A $\P$-$f$-orientation of $A$ gives rise to a Thom isomorphism
    \[ A \otimes_R \Th_G^{\otimes}(f) \simeq A \otimes \Sigma_+^{\infty} X^{\otimes} \]
    in $\Alg_{\P}(\LMod_A)$.
\end{theorem}

\begin{proof}
    Since left Kan extensions commute with left adjoints, we conclude 
    \[ A \otimes_R \Th_{G}^{\otimes}(f) \simeq \Th_{G}^{\otimes}(A \otimes_R f) \] 
    in $\Alg_{\P}(\LMod_A)$. On the other hand, an orientation is a nullhomotopy of $A \otimes_R f$ (see \cref{corollary: characterization of orientations} \textcolor{chilligreen}{(ii)}). We can write this null map as $\Ind_{\one}^R \circ c$ where $c\colon  X^\otimes \to \Pic_G(\one)$ is the null map. Thus, it is equivalent to 
    \[ A \otimes_R \Th_{G}^{\otimes}(f) \simeq \Th_{G}^{\otimes}(A \otimes_R f) \simeq \Ind_{\one}^A (\one \otimes \Sigma_+^{\infty}X^{\otimes}) \simeq A \otimes \Sigma_+^{\infty}X^{\otimes}. \]
    by \cref{example: thom along trivial}. 
\end{proof}

\begin{corollary} \label{corollary: chadwick-mandell}
    Let $(\O^{\otimes}, \P^{\otimes}, \C, R, A, f)$ be an orientation datum.
    \begin{enumerate}[(i)]
        \item There is a pullback square
        \begin{center}
            \begin{tikzcd}
                \Or_A^{\P}(f) \arrow[r] \arrow[d] \arrow[dr, phantom, very near start, "\lrcorner"] & * \arrow[d, "\const_A"]
                \\ * \arrow[r, "(R \otimes_A -)_* \circ f", swap] & \Map_{\Mon_{\P}(\Sc)}(X^{\otimes}, \Pic_{G}^{\otimes}(A))
            \end{tikzcd}
        \end{center}
        In particular, it is either empty or $\Omega \Map_{\P}(X^{\otimes}, \Pic_{G}^{\otimes}(A))$.
        \item Suppose that $X^{\otimes}$ is furthermore grouplike. Then, $\Map_{\Alg_{\P}(\LMod_R)}(\Th_{G}^{\otimes}(f), A)$ is empty or 
        \[ \Map_{\Alg_{\P}(\LMod_R)}(\Th_{G}^{\otimes}(f), A) \simeq \Map_{\Alg_{\P}(\LMod_R)}(R \otimes \Sigma_+^{\infty} X, A). \]
        This equivalence is also true for parametrized mapping spaces.
    \end{enumerate}
\end{corollary}

\begin{proof}
    \hfill
    \begin{enumerate}[(i)]
        \item This follows immediately from pullback pasting
        \begin{center}
            \begin{tikzcd}
                \Or_A^{\P}(f) \arrow[r] \arrow[d] \arrow[dr, phantom, very near start, "\lrcorner"] & \Map_{\Mon_{\P}(\Sc)}(X^{\otimes}, \GL_1(\Pic^{\otimes}_G(R)_{\downarrow A})) \arrow[r] \arrow[d] \arrow[dr, phantom, very near start, "\lrcorner"] & * \arrow[d, "\const_A"]
                \\ * \arrow[r, "f", swap] & \Map_{\Mon_{\P}(\Sc)}(X^{\otimes}, \Pic_{G}^{\otimes}(R)) \arrow[r, "(R \otimes_A -)_*", swap] & \Map_{\Mon_{\P}(\Sc)}(X^{\otimes}, \Pic_{G}^{\otimes}(A))
            \end{tikzcd}
        \end{center}
        The left square is a pullback by definition of $\Or_A^{\P}(-)$ and the universal property of slice $\infty$-categories. The right square is a pullback by \cref{corollary: GL1 Pic(R)A as pullback}. Thus, the composite rectangle is a pullback.
        \item By \cref{lemma: grouplike orientation} we are analyzing the space of orientations. If there is one, then the monoidal Thom isomorphism (\cref{theorem: monoidal thom iso}) yields $A \otimes_R \Th_{G}^{\otimes}(f) \simeq A \otimes \Sigma_+^{\infty}X$, so an adjunction argument gives the desired mapping space equivalence.
    \end{enumerate}
\end{proof}

\newpage
\part{Applications}
\label{part:applications}
\begin{description}
    \item[\cref{sec:StronglyEvenTowers_CohomologicalSlice}] This section sets up the main organizational and computational prerequisites for our obstruction theory.
    
    In \cref{sec:strongly_even} we review Hill--Meier's notion of strongly even spectra and sharpen some of the basic results.

    In \cref{subsec:strongly_even_filtered} we extend this to the setting of towers of \(G\)-spectra.
    The expert is free to skip this section by agreeing that, under suitable conditions, the limit of a tower of strongly even spectra is strongly even.

    In \cref{sec:cohomological_slice} we study a cohomological version of the equivariant slice tower.
    We start with a quick review of the slice tower; then we set up the cohomological slice tower and finish with providing criteria to check when the cohomological slice tower is strongly even.

    In \cref{subsec:liftingorientations} we produce an obstruction theory for lifting non-equivariant \(\E_{n}\)-orienations to equivariant \(\E_V\)-orientations. This is our main obstruction theory for studying structured orientations of equivariant Thom spectra. The main idea in this section is that: the multiplicative Thom isomorphism, combined with the recognition principle, reduces the study of multiplicative orientations to Bredon computations of deloopings.

%    \item[\cref{sec: structured orientations}] This section sets up our main obstruction theory for studying structured orientations of equivariant Thom spectra.
%    The main idea in this section is that: the multiplicative Thom isomorphism, combined with the recognition principle, reduces the study of multiplicative orientations to Bredon computations of deloopings.

%    In \cref{subsec:liftingorientations} we produce an obstruction theory for lifting non-equivariant \(\E_{n}\)-orienations to equivariant \(\E_V\)-orientations. 

%    In \cref{subsec:liftingidempotents} we specialize to studying self-orientations of equivariant Thom spectra. We give a criterion for lifting multiplicative splittings of Thom spectra to multiplicative splittings of equivariant Thom spectra.
    
    \item[\cref{sec: Real orienations}] This section proves our main results. We apply the the obstruction theory developed in the previous sections ot the study of multiplicative Real orientations.

    In \cref{subsec:universal_orientation} we produce \(\E_\rho\)-orientations for strongly even \(\E_{\infty}^{C_2}\)-rings. We do this by producing an example in a versal case through Real Wilson spaces.

    In \cref{subsec:multiplicative_real_orientations}, armed with the existence of orientations from the previous section, we deploy our obstruction theory to prove \cref{mainthm: lifting orientations}, namely that underlying complex orientations $\MU \to E^e$ of strongly even $\E_{\infty}^{C_2}$-ring spectra lift to $\E_{\rho}$-Real orientations $\MU_{\R} \to E$. We use this to enhance various orientations of interest like the Hahn--Shi Real orientations of Lubin--Tate theory. For $C_2 \leq G$ we discuss an $\Coind_{C_2}^G \E_{\rho}$-structure on the normed versions.

    In \cref{subsec:multiplication_on_BPR}, we produce an \(\E_\rho\)-multiplication on \(\BPR\), proving \cref{mainthm:BPR} from the introduction.
    We also produce \(\E_\rho\)-Adams operations on \(\BPR\).
    
    \item[\cref{sec:factorization_and_EV}] This section considers applications of equivariant Thom spectra outside the realm of structured orientations. 

    In \cref{sec:factorization} we prove formulas for the equivariant factorization homology of equivariant Thom spectra. We work in the generality of \(R\)-module Thom spectra. In particular, this specializes to give formulas for relative Real Topological Hochschild Homology \(\THR({-}/R)\). 

    In \cref{sec:EV_quotients} we study \(\E_V\)-quotients via equivariant Thom spectra. In particular we rephrase Levy's equivariant Hopkins--Mahowald Theorem in terms of \(\E_V\)-quotients. 

    In \cref{sec:nilpotence} we prove some basic facts about nilpotence for \(\E_V\)-algebras.
    We show that \(\MUR\) detects nilpotence for \(\E_\sigma\)-rings. We show that \(\uHZ\) detects nilpotence for \(\E_\sigma\otimes \E_\infty\)-rings
\end{description}
\newpage 
\noindent As a guide to \cref{part:applications} we include a sketch proof of \cref{mainthm:BPR}.
Although \cref{mainthm:BPR} follows formally from \cref{mainthm: lifting orientations}, 
exhibiting the proof strategy directly for \cref{mainthm:BPR} helps to quarantine certain technical aspects from the key ideas.
We hope this also acts as a reasonable summary of the parameterized theory from \cref{part:foundations} for the reader who skipped directly to \cref{part:applications}.

\begin{thmG*}[\cref{thm:multiplication_on_BPR}]
    The $C_2$-spectrum $\BP_{\R}$ admits an $\E_{\rho}$-algebra structure.
    Let \(G\geq C_2\). The \(G\)-spectrum \(\BP^{(\!(G)\!)}\coloneqq N^G_{C_2}\BPR\) admits a $\Coind^G_{C_2}\E_{\rho}$-algebra structure
\end{thmG*}

\begin{proof}[Proof Sketch.]
The $\Coind^G_{C_2}\E_{\rho}$-algebra structure on \(\BP^{(\!(G)\!)}\) follows formally once we produce an $\E_{\rho}$-algebra structure on $\BP_{\R}$ (\cref{construction: Coind}). So we focus on \(\BPR\).

\medskip
\noindent 
The Real Brown--Peterson spectrum \(\BPR\) can be defined as a sequential colimit along iterates of a (Real) equivariant version of Quillen's idempotent \cite[Section 7]{araki1979orientations}.
Since filtered colimits of algebras are computed on underlying, to produce an \(\E_{\rho}\)-algebra structure on \(\BPR\), it suffices to lift Quillen's idempotent to an \(\E_{\rho}\)-algebra map \(\MUR\to \MUR\).

\medskip
\noindent
For the reader unfamiliar with parameterized higher category theory, for the purpose of this sketch, the most important aspect -- aside from the parameterized theory being essential to constructing a multiplicative equivariant Thom isomorphism (\cref{theorem: monoidal thom iso}) -- is that mapping spaces are replaced with mapping \(C_2\)-spaces.
In particular, the space of \(\E_{\rho}\)-algebra maps \(\MUR\to \MUR\) refines to a \(C_{2}\)-space 
\[\myuline{\Map}_{\myuline{\Alg}_{\E_{\rho}}(\myuline{\Sp}_{C_2})}(\MUR,\MUR);\]
the \(C_{2}\)-fixed points is precisely the space of \(\E_{\rho}\)-algebra maps \(\MUR\to\MUR\), i.e.
\[\myuline{\Map}_{\myuline{\Alg}_{\E_{\rho}}(\myuline{\Sp}_{C_2})}(\MUR,\MUR)^{C_2}\simeq {\Map}_{{\Alg}_{\E_{\rho}}(\myuline{\Sp}_{C_2})}(\MUR,\MUR);\]
the underlying non-equivariant space is precisely the space of \(\E_{2}\)-algebra maps \(\MU\to\MU\)
\[\myuline{\Map}_{\myuline{\Alg}_{\E_{\rho}}(\myuline{\Sp}_{C_2})}(\MUR,\MUR)^{e}\simeq {\Map}_{{\Alg}_{\E_{2}}({\Sp})}(\MU,\MU);\]
and the map induced be the inclusion of the fixed points
\[
\myuline{\Map}_{\myuline{\Alg}_{\E_{\rho}}(\myuline{\Sp}_{C_2})}(\MUR,\MUR)^{C_2}\to 
\myuline{\Map}_{\myuline{\Alg}_{\E_{\rho}}(\myuline{\Sp}_{C_2})}(\MUR,\MUR)^{e}
\]
is precisely the restriction map
\[
{\Map}_{{\Alg}_{\E_{\rho}}(\myuline{\Sp}_{C_2})}(\MUR,\MUR)\to {\Map}_{{\Alg}_{\E_{2}}({\Sp})}(\MU,\MU).
\]
This parameterized perspective motivates a natural strategy: lift Chadwick--Mandell's \(\E_2\)-version of Quillen's idempotent \(\MU\to \MU\) \cite[Theorem 1.2]{chadwickmandell} along the above restriction to an \(\E_\rho\)-map \(\MUR\to \MUR\).
In particular, to produce a such a lift, it suffices to show that the restriction map in the Mackey functor \(\myuline{\pi}_{0}\,\myuline{\Map}_{\myuline{\Alg}_{\E_{\rho}}(\myuline{\Sp}_{C_2})}(\MUR,\MUR)\) is a surjective; we will in fact show this is an isomorphism.

\medskip
\noindent
Using the multiplicative parameterized Thom isomorphism (\cref{theorem: monoidal thom iso}), we make this computation more tractable. We have the following equivalences of \(C_2\)-spaces:
\begin{align*}
    \myuline{\Map}_{\myuline{\Alg}_{\E_{\rho}}(\myuline{\Sp}_{C_2})}(\MUR,\MUR) & \simeq \myuline{\Map}_{\myuline{\Alg}_{\E_{\rho}}(\ul{\Sp}_{C_2})}(\Sigma^\infty_+\BUR,\MUR) \\
    & \simeq \myuline{\Map}_{\myuline{\Alg}_{\E_{\rho}}(\myuline{\Spaces_{C_2}})}(\BUR,\Omega^\infty\MUR) \\
    & \simeq \myuline{\Map}_{\myuline{\Alg}^{\grp}_{\E_{\rho}}(\myuline{\Spaces_{C_2}})}(\BUR,\GL_1\MUR).
\end{align*}
The last equivalence uses that \(\BUR\) is grouplike, and the adjunction involving \(\GL_1\) (\cref{prop: inclusion grouplike}). 
The reason this reduction is useful is we can now use equivariant loop space theory and the equivariant recognition principle \cite{Guillou2012EquivariantIL,cnossen2024normedequivariantringspectra,juran2025genuineequivariantrecognitionprinciple}.
We have further equivalences of \(C_2\)-spaces
\begin{align*}
\myuline{\Map}_{\myuline{\Alg}^{\grp}_{\E_{\rho}}(\myuline{\Spaces_{C_2}})}(\BUR,\GL_1\MUR) 
    & \simeq \myuline{\Map}_{\ul{\Spaces}^{C_2}_*}(\mathrm{B}^\rho\BUR,\mathrm{B}^\rho\GL_1\MUR) \\
    & \simeq \myuline{\Map}_{\ul{\Sp}_{C_2}}(\Sigma^\infty \mathrm{B}^\rho\BUR,\Sigma^\rho\gl_1\MUR) \\
    & \simeq \myuline{\Map}_{\ul{\Sp}_{C_2}}(\Sigma^\infty \BSUR,\Sigma^\rho\gl_1\MUR). 
\end{align*}
The first equivalence is the equivariant recognition principle (\cref{thm: recognition}). 
For the second, we use the equivalence between equivariant infinite loop spaces and connective equivariant spectra (\cref{thm: recognition}).
The third equivalence uses equivariant Bott periodicity (\cref{example: equivariant Bott periodicity ku}). 

\medskip
\noindent
In summary, we have reduced the problem to computing the \(\gl_1\MUR\)-cohomology of \(\BSUR\).
More precisely, to produce an \(\E_\rho\)-lift of Quillen's idempotent, it suffices to show that the restriction map 
\[
\res_e^{C_2}\colon \pi^{C_2}_{0}(\myuline{\map}_{\ul{\Sp}_{C_2}}(\Sigma^\infty \BSUR,\Sigma^{\rho}\gl_{1}\MUR))\to
\pi^{e}_{0}(\myuline{\map}_{\ul{\Sp}_{C_2}}(\Sigma^\infty \BSUR,\Sigma^{\rho}\gl_{1}\MUR))
\]
is surjective.
However, it is no harder, and in fact it is more suggestive of a strategy, to show that
the restriction maps
\[
\res_e^{C_2} \colon \pi^{C_2}_{*\rho}(\myuline{\map}_{\ul{\Sp}_{C_2}}(\Sigma^\infty \BSUR,\Sigma^{\rho}\gl_{1}\MUR))\to
\pi^{e}_{2*}(\myuline{\map}_{\ul{\Sp}_{C_2}}(\Sigma^\infty \BSUR,\Sigma^{\rho}\gl_{1}\MUR))
\]
are isomorphisms;
it suggests to approach this by showing \(\myuline{\map}_{\ul{\Sp}_{C_2}}(\Sigma^\infty \BSUR,\Sigma^{\rho}\gl_{1}(\MUR))\) is a strongly even \(C_2\)-spectrum in the sense of Hill--Meier (\cref{def: strongly even C2}).

\medskip
\noindent
Our goal is to reduce the claim that \(\myuline{\map}_{\ul{\Sp}_{C_2}}(\Sigma^\infty \BSUR,\Sigma^{\rho}\gl_{1}(\MUR))\) is strongly even to a combination known computations, namely: the fact that \(\MUR\) is strongly even; and that the Bredon cohomology of \(\BSUR\) is strongly even.
The two ideas we use to achieve this reduction are: setting up a cohomological version of the slice tower (\cref{sec:cohomological_slice}); and extending the Hill--Meier's notion of strongly even \(C_2\)-spectra to the setting of filtered, or equivalently towers of \(C_2\)-spectra (\cref{def: strongly even towers}).

\medskip
\noindent
The cohomological slice tower is a cohomogical version of the homological slice tower studied in Carrick--Hill--Ravenel \cite{carrick2024homologicalslice}. For \(\myuline{\map}_{\ul{\Sp}_{C_2}}(\Sigma^\infty \BSUR,\Sigma^\rho\gl_1\MUR)\), the corresponding cohomological slice tower (\cref{construction: cohomological slice tower}) is 
\[
\CST(\Sigma^{\infty} \BSUR; \Sigma^{\rho}\gl_1 \MUR) = \myuline{\map}_{\ul{\Sp}_{C_2}}(\Sigma^\infty \BSUR,P^{\leq \bullet} \Sigma^\rho\gl_1\MUR).
\]
By construction, the associated graded of this tower compute Bredon cohomology with coefficients in the slices of \(\MUR\). By extending Hill--Meier's notion of strongly evenness to towers of \(C_2\)-spectra (\cref{def: strongly even towers}), we show that, under certain evenness conditions, the limit of a tower of \(C_2\)-spectra is strongly even if its associated graded pieces are strongly even (\cref{prop:strongly_even_c2_tower}). Hence, applied to the cohomological slice tower, we indeed achieve the desired reduction and learn that \(\myuline{\map}_{\ul{\Sp}_{C_2}}(\Sigma^\infty \BSUR,\Sigma^{\rho}\gl_{1}(\MUR))\) is strongly even.

\medskip
\noindent
Thus, we have sketched that the restriction
\[
\pi^{C_2}_{*\rho} \myuline{\Map}_{\myuline{\Alg}_{\E_{\rho}}(\myuline{\Sp}_{C_2})}(\MUR,\MUR) \to 
\pi^e_{2*} \myuline{\Map}_{\myuline{\Alg}_{\E_{\rho}}(\myuline{\Sp}_{C_2})}(\MUR,\MUR)
\]
is an isomorphism, whence producing an \(\E_\rho\)-lift of Quillen's idempotent.
\end{proof}

\section{Strongly Even Towers, the Cohomological Slice Tower \& Lifting Orientations}\label{sec:StronglyEvenTowers_CohomologicalSlice}
The essential idea in this section is the following: In Real oriented homotopy theory, Hill--Meier's notion of strong eveness allows many (Real) equivariant statements to be deduced from underlying non-equivariant statements; by extending Hill--Meier's notion to the setting of towers of spectra, we can reduce equivariant obstruction theory to ordinary obstruction theory.
\medskip \\Although elementary, we expect the results here to be of independent interest, so we develop more than is strictly needed to prove \cref{mainthm:BPR} and \cref{mainthm: lifting orientations}.

\medskip 
\noindent
In the case one is only interested in \cref{mainthm:BPR} and \cref{mainthm: lifting orientations}, the expert is free to skip this section by agreeing that $\ul{\map}_{\ul{\Sp}_{C_2}}(X, E)$ is strongly even when: \(E\) is strongly even, and slice bounded below; and \(\uHZ \otimes X\) is of finite type, is \(\uHZ\)-free, and has a basis given by spheres in dimensions \(n\rho\) (see \cref{prop:checkingSAHSS}). In particular, the restriction map
\[
    \res_e^{C_2}\colon  \pi^{C_2}_{*\rho}(\myuline{\map}_{\ul{\Sp}_{C_2}}(X,E))\to \pi_{2*}(\map_{\Sp}(X,E))
    \]
is an equivalence in this setting.

\subsection{Strongly Even Spectra}\label{sec:strongly_even}
Evenness plays a key role in chromatic homotopy theory.
Since \(\CP^\infty\) is built from even cells \(S^{2n}\), there is no obstruction to building complex orientations for even ring spectra. 
Hill--Meier \cite{hillmeier2017} introduced the following notion of equivariant evenness suited for Real oriented homotopy theory.

\begin{definition}[Hill--Meier]
    Let \(X \in \Sp_{C_2} \). It is said to be \tb{Real even}\footnote{Hill--Meier simply use the term even; we add the prefix Real to avoid overloading the term even.}
    if \(\pi^{C_2}_{n\rho-1}(X)\) and \(\pi^{e}_{2n-1}(X)\) are both zero for all \(n\).
\end{definition}

\noindent Let \(\CPR^\infty\) denote \(\CP^\infty\) equipped with the \(C_2\)-action by complex conjugation.
In the same way that complex orientations are determined by \(\CP^\infty\), Real orientations are determined by \(\CPR^\infty\). 
Since \(\CPR^\infty\) is built out of cells of the form \(S^{n\rho}\), the following is can be deduced from obstruction theory analogous to the classical statement.

\begin{lemma}[{\cite[Lemma 3.3]{hillmeier2017}}]
    Let \(X\) be a Real even \(C_2\)-spectrum.
    Then \(X\) admits a Real orientation. 
\end{lemma}

\noindent Hill--Meier also introduced a stronger notion of Real evenness, called strong evenness.
The main motivation for this definition is that, in examples of interest, often the \(*\rho\)-graded homotopy groups contain more relevant information than the integer graded homotopy groups.
For example: for \(\KR\), the equivariant Bott class lives in degree \(\rho\); for \(\MUR\), the Lazard ring lives in degrees \(*\rho\).
\begin{definition}[Hill--Meier]
    Let \label{def: strongly even C2} \(X \in \Sp_{C_2}\). We say that \(X\) is \tb{strongly even}\footnote{We were strongly tempted to instead use the term \emph{Really even}.}
    if the following conditions are met:
    \begin{enumerate}[(i)]
        \item We have both \(\pi^e_{2*-1}(X)=0\) and \(\pi^{C_2}_{*\rho-1}(X)=0\), i.e. \(X\) is Real even.
        \item The restriction maps \(\res_e^{C_2}: \pi^{C_2}_{*\rho}(X)\to \pi^{e}_{2*}(X)\) are isomorphisms.
    \end{enumerate}
    If $X$ naturally admits an underlying $C_2$-spectrum,\footnote{For example, if $X$ is an algebra or a module.} then we say that $X$ is \tb{strongly even} if the underlying $C_2$-spectrum is.
\end{definition}
\begin{example}
    All of the \(C_2\)-spectra $\mathrm{H} \ul{\Z}$, \(\MUR\), \(\kR\), \(\tmf_1(n)\), \(\BPR\), \(\BPR\langle n\rangle\), and the Lubin--Tate theories \(E_n\) are strongly even, see \cite{HuKrizReal,hillmeier2017,Greenlees_2017,hahnRealOrientationsLubin2020}.
\end{example}

\noindent One of the main benefits of working with strongly even spectra is that many equivariant statements can then be deduced from the non-equivariant counterparts.
For example, the following is particularly useful. 

\begin{lemma}[{\cite[Lemma 3.4]{hillmeier2017}}]
    If \(f\colon X\to Y\) is a map between strongly even spectra such that \(f^e\colon X^e\to Y^e\) is an equivalence, then \(f\) is an equivalence.
\end{lemma}

\noindent A notable feature of strongly even spectra is that they all poses a `gap' in their \(\operatorname{RO}(C_2)\)-graded homotopy groups. Greenlees showed this in fact provides an equivalent reformulation of strong eveness \cite{greenleesfour}.
\begin{lemma}[Greenlees' gap characterization]\label{lem:GreenleesGap}
    Let \(X \in \Sp_{C_2}\). The following are equivalent.
    \begin{enumerate}[(i)]
        \item \(X\) is strongly even.
        \item \(\pi^{C_2}_{*\rho-i}(X)=0\) for \(i=1\), \(2\), \(3\).
    \end{enumerate}
\end{lemma}
\begin{proof}
    See the proof of \cite[Lemma 1.2]{greenleesfour}.
    To read his statement note that Greenlees' definition of strongly even does not include the condition that the underlying non-equivariant spectrum is even.
\end{proof}

\noindent A caricature of the gap, heavily inspired by the diagrams in \cite{greenleesfour}, is illustrated below. The homotopy groups along the dotted lines vanish. The shaded region is the gap.
\vspace{-1em}
    \[\begin{tikzpicture}[scale=0.8]
      % Grid (6x6)
      \draw[step=1, very thin, gray!40] (-5,-5) grid (5,5);
      %fake nodes to make the graph "centered"
      \draw (-7,0) node {};
      \draw (7,0) node {};
      % Axes
      \draw (0,5.1) node[above] {$\mathbb{Z}\cdot \sigma$};
      \draw[->, thick] (-5,0) -- (5.1,0) node[right] {$\mathbb{Z}\cdot 1$};
      \draw[->, thick] (0,-5) -- (0,5.1) ;
    
      % Tick marks
      \foreach \x in {1,...,4} \draw (\x,2pt) -- (\x,-2pt) node[below] {\small \x};
      \foreach \y in {1,...,4} \draw (2pt,\y) -- (-2pt,\y) node[left] {\small \y};
      % Tick marks
      \foreach \x in {-4,...,-1} \draw (\x,2pt) -- (\x,-2pt) node[below] {\small \x};
      \foreach \y in {-4,...,-1} \draw (2pt,\y) -- (-2pt,\y) node[left] {\small \y};
    
      % Shaded region 
      \fill[chilligreen, opacity=0.4] 
        (-5,-5) -- (-5,-2) -- (2,5) -- (5,5) -- (-5,-5) -- cycle;
    
      % Lines y = x and y = x + 3
        \draw[thick] (-5,-5) -- (5,5) node[above ] {\tiny{$*\rho$}};
        \draw[thick, dashed] (-5,-2) -- (2,5) node[above] {\tiny{$*\rho-3$}};
        \draw[thick, dashed] (-5,-3) -- (3,5) node[above ] {\tiny{$*\rho-2$}};
        \draw[thick, dashed] (-5,-4) -- (4,5) node[above ] {\tiny{$*\rho-1$}};
    \end{tikzpicture}\]
The gap characterization is a convenient way to check if a \(C_2\)-spectrum is strongly even.
The following example will be used repeatedly in our applications. 
\begin{example}\label{example:gl1(R)_strongly_even}
    Let \(R\) be a strongly even \(\E_{\infty}^{C_2}\)-ring spectrum.
    Then \(\gl_1(R)\) is a strongly even \(C_2\)-spectrum.
\end{example}

\begin{proof}
    Using Greenlees' gap characterization (\cref{lem:GreenleesGap}), this follows almost immediately from the computation of \(\pi^{C_2}_{V}(\GL_1 R)\), see \cref{lem: homotopy of gl1}.
    The only potential subtlety is the case \(V=\rho-1=\sigma\) since \(\dim (\sigma^{C_2})=0\).
    \medskip \\To handle this, we consider the fiber sequence $C_2/e_+ \to S^0 \to S^{\sigma}$ which gives rise to a long exact sequence
    \[\begin{tikzcd}
    	\cdots & {\pi^{C_2}_{1}(X)} & {\pi^{e}_1(X)} & {\pi^{C_2}_\sigma(X)} & {\pi^{C_2}_0(X)} & {\pi^{e}_0(X)} & \cdots
    	\arrow[from=1-1, to=1-2]
    	\arrow[from=1-2, to=1-3]
    	\arrow[from=1-3, to=1-4]
    	\arrow[from=1-4, to=1-5]
    	\arrow[from=1-5, to=1-6]
    	\arrow[from=1-6, to=1-7]
    \end{tikzcd}\]
    for every $C_2$-spectrum $X$. This gives a short exact sequence
    \[\begin{tikzcd}[column sep=2.15em]
    	0 & {\operatorname{coker}\left(\pi^{C_2}_{1}(X)\to \pi^{e}_1(X)\right)} & {\pi^{C_2}_\sigma(X)} & {\operatorname{ker}\left(\pi^{C_2}_0(X)\to\pi^{e}_0(X)\right)} & 0.
    	\arrow[from=1-1, to=1-2]
    	\arrow[from=1-2, to=1-3]
    	\arrow[from=1-3, to=1-4]
    	\arrow[from=1-4, to=1-5]
    \end{tikzcd}\]
    For \(X=\gl_1(R)\) this becomes
        \[\begin{tikzcd}[column sep=1.4
        em]
    	0 & {\operatorname{coker}\left(\pi^{C_2}_{1}(R)\to \pi^{e}_1(R)\right)} & {\pi^{C_2}_\sigma(\gl_1(R))} & {\operatorname{ker}\left(\pi^{C_2}_0(R)^\times\to\pi^{e}_0(R)^\times\right)} & 0.
    	\arrow[from=1-1, to=1-2]
    	\arrow[from=1-2, to=1-3]
    	\arrow[from=1-3, to=1-4]
    	\arrow[from=1-4, to=1-5]
    \end{tikzcd}\]
    Since \(R\) is strongly even, the restriction \(\pi^{C_2}_0(R)\to \pi^{e}_0(R)\) is an isomorphism, and \(\pi^e_1(R)=0\).
    Hence, \(\pi^{C_2}_\sigma(\gl_1(R))=0\).
\end{proof}

    \noindent The following observation is elementary but will be used repeatedly.
    It's a more precise gap behaviour than typically noted -- namely,
    provided that the underlying non-equivariant spectrum is even, the vanishing of just \(\pi^{C_2}_{*\rho-2}(X)\) is equivalent to strong evenness.

\begin{lemma}\label{lem:C2_stronger_gap}
    Let \(X \in \Sp_{C_2} \).
    Suppose that following conditions hold:
    \begin{enumerate}[(i)]
        \item For all \(n\) we have \(\pi^e_{2n-1}(X)=0\).
        \item For all \(n\) we have \(\pi^{C_2}_{n\rho-2}(X)=0\).
    \end{enumerate}
    Then, the following conditions also hold:
    \begin{enumerate}
        \item[(iii)] For all \(n\) we have  \(\pi^{C_2}_{n\rho-1}(X)=0\).
        \item[(iv)] For all \(n\) we have \(\pi^{C_2}_{n\rho-3}(X)=0\).
    \end{enumerate}
    In particular, \(X\) is strongly even.
\end{lemma}

\begin{proof}
First we show that (i) and (ii) imply (iii), then we show that (i) and (iii) imply (iv). Both of these will follow from considerations of the cofiber sequence
\begin{center}
    \begin{tikzcd}
        C_2/e_+ \arrow[r] & S^0 \arrow[r] & S^{\sigma}
    \end{tikzcd}
\end{center}
and a part of the resulting long exact sequence 
\begin{center}
    \begin{tikzcd}
        \pi_{a+b+1}^e(X) \arrow[r] & \pi_{a+(b+1)\sigma}^{C_2}(X) \arrow[r] & \pi_{a+b\sigma}^{C_2}(X) \arrow[r, "\res_e^{C_2}"] & \pi_{a+b}^e(X)
    \end{tikzcd}
\end{center}
for $a,b \in \Z$. Indeed: 
\begin{enumerate}
    \item[(iii)] Take \(a=n-1\) and \(b=n\). Then, the latter three terms are given by
    \begin{center}
        \begin{tikzcd}
            \pi_{n \rho - 2}^{C_2}(X) \arrow[r] & \pi_{n\rho-1}^{C_2}(X) \arrow[r] & \pi_{2n-1}^e(X)
        \end{tikzcd}
    \end{center}
    and the outer two are $0$ by (1) and (2), so the middle term is also $0$.
    \item[(iv)] Take \(a=n-1\) and \(b=n+1\). Then, the first three terms are given by
    \begin{center}
        \begin{tikzcd}
            \pi_{2n+1}^e(X) \arrow[r] & \pi_{(n+2)\rho-3}^{C_2}(X) \arrow[r] & \pi_{n \rho - 2}^{C_2}(X)
        \end{tikzcd}
    \end{center}
    and the outer two are $0$ by (i) and (ii), so the middle term is also $0$.
\end{enumerate}
Greenlees' gap characterization of strong evenness \cref{lem:GreenleesGap} then gives that $X$ is strongly even.
\end{proof}

%\noindent There is a version for more general groups $G$, communicated to us by Meier.

%\begin{definition}[Strongly even \(G\)-spectra]
%    Let \(X\in \Sp_G \).
%    We say that \(X\) is \tb{strongly even} if the following conditions hold:
%    \begin{enumerate}[(i)]
%        \item We have \(\pi^e_{2*-1}(X)=0\)
%        \item For all \(e\neq H\leq G\) we have \(\pi^{H}_{*\rho_H-1}(X)=0\).
%        \item For all \(e\neq H\leq G\) the restriction maps \(\res_e^H:\pi^H_{*\rho_H}(X)\to \pi^e_{*\vert \rho_H\vert}(X)\) are isomorphisms.
%    \end{enumerate}
%\end{definition}
%\begin{remark}
%    In the case \(G=C_2\) this recovers the Hill--Meier notion of a strongly even \(C_2\)-spectrum (\cref{def: strongly even C2}).
%\end{remark}

\subsection{Strongly Even Towers of Spectra}\label{subsec:strongly_even_filtered}
In this section we extend the notion of strongly even spectra to the setting of filtered spectra, although for our purposes it is more convenient to think in terms of towers of spectra rather than filtrations. 
\medskip \\We use the following conventions.
We use a decreasing indexing. A tower \(X^\bullet\) of \(C_2\)-spectra consists of a sequence of \(C_2\)-spectra
\begin{center}
    \begin{tikzcd}
        \cdots \arrow[r] & X^s \arrow[r] & X^{s-1} \arrow[r] & \cdots 
    \end{tikzcd}
\end{center}
We write $X^{\infty} = \lim_{s \to \infty} X^s$ for the limit of the tower and think of it as the object the tower aims to compute. Moreover, we write $X^{-\infty} = \colim_{s \to-\infty} X^s$.
The associated graded pieces are defined via the fiber sequences
\begin{center}
    \begin{tikzcd}
        \operatorname{gr}_{s+1}X \arrow[r] & X^{s+1} \arrow[r] & X^s.
    \end{tikzcd}
\end{center}

\begin{definition} \label{def: strongly even towers}
    Let \(X^\bullet\) be a tower of \(C_2\)-spectra.
    We say that \(X^\bullet\) is a \tb{strongly even tower of \(C_2\)-spectra} if it satisfies the following conditions:
    \begin{enumerate}[(i)]
        \item The limit of the underlying non-equivariant tower is even.
        \item The colimit $X^{-\infty}$ is strongly even.
        \item Every associated graded piece $\gr_s X$ is strongly even.
    \end{enumerate}
\end{definition}

\noindent The upshot is that the limit of a strongly even tower is strongly even again.

\begin{proposition}\label{prop:strongly_even_c2_tower}
    Let \(X^\bullet\) be a strongly even tower of \(C_2\)-spectra. 
    %Suppose further that the limit of the underlying non-equivariant tower is even. 
    Then, $X^{\infty}$ is strongly even. In particular, the restriction map 
    \[
    \res_e^{C_2}\colon  \pi^{C_2}_{*\rho}(X^\infty)\to \pi^e_{2*}(X^\infty)
    \]
    is an isomorphism.
\end{proposition}

\begin{proof}
    By assumption \(\pi^{e}_{2*-1}(X^\infty)=0\). 
    So by \cref{lem:C2_stronger_gap} it suffices to show that \(\pi^{C_2}_{*\rho-2}(X^\infty)=0\).
    The Milnor sequence gives a short exact sequence
    \begin{center}
        \begin{tikzcd}
            0 \arrow[r] &\displaystyle{\lim_s}^1 \pi_{*\rho - 1}^{C_2}(X^s) \arrow[r] & \pi_{*\rho-2}^{C_2}(X^{\infty}) \arrow[r] & \displaystyle{\lim_s} \,  \pi_{*\rho-2}^{C_2}(X^s) \arrow[r] & 0.
        \end{tikzcd}
    \end{center}
    Thus, to show that \(\pi^{C_2}_{*\rho-2}(X^\infty)=0\) it suffices to show that for all \(s \in \Z\)
    we have \(\pi^{C_2}_{*\rho-1}(X^s)=0\) and \(\pi^{C_2}_{*\rho-2}(X^s)=0\). We will in fact show that each \(X^s\) is strongly even.
    \medskip \\Consider the fiber sequence
    \begin{center}
        \begin{tikzcd}
            \gr_{s+1}X \arrow[r] & X^{s+1} \arrow[r] & X^s.
        \end{tikzcd}
    \end{center}
    Since $\gr_{s+1}X$ is strongly even by assumption, we deduce by the associated long exact sequences that $X^{s+1} \to X^s$ injective on $\pi_{*\rho-1}, \pi_{*\rho-2}, \pi_{*\rho-3}$. The colimit $X^{-\infty}$ is $0$ on these homotopy groups by assumption and Greenlees' gap characterization (\cref{lem:GreenleesGap}). Thus, we must already have
    \[ \pi_{*\rho-1}X^s = \pi_{*\rho-2}X^s = \pi_{*\rho-3}X^s = 0 \]
    for all $s$. Hence, Greenlees' gap characterization (\cref{lem:GreenleesGap}) implies that each $X^s$ is strongly even.
\end{proof}

%\begin{proof}
%    By assumption \(\pi^{e}_{2*-1}(X^\infty)=0\). 
%    So by \cref{lem:C2_stronger_gap} it suffices to show that \(\pi^{C_2}_{*\rho-2}(X^\infty)=0\).
%    The Milnor sequence gives a short exact sequence
%    \begin{center}
%        \begin{tikzcd}
%            0 \arrow[r] &\displaystyle{\lim_s}^1 \pi_{*\rho - 1}^{C_2}(X^s) \arrow[r] & \pi_{*\rho-2}^{C_2}(X^{\infty}) \arrow[r] & \displaystyle{\lim_s} \,  \pi_{*\rho-2}^{C_2}(X^s) \arrow[r] & 0.
%        \end{tikzcd}
%    \end{center}
%    Thus, to show that \(\pi^{C_2}_{*\rho-2}(X^\infty)=0\) it suffices to show that for all \(s \in \Z\)
%    we have \(\pi^{C_2}_{*\rho-1}(X^s)=0\) and \(\pi^{C_2}_{*\rho-2}(X^s)=0\). We will in fact show that each \(X^s\) is strongly even. By assumption, $X^0$ is strongly even, so suppose that $X^s$ is strongly even and we will argue inductively using the fiber sequence
%    \begin{center}
%        \begin{tikzcd}
%            \operatorname{gr}_{s+1}X \arrow[r] & X^{s+1} \arrow[r] & X^s.
%        \end{tikzcd}
%    \end{center}
%    Thus, the sequence
%    \begin{center}
%        \begin{tikzcd}
%            \pi_V^{C_2}(\gr_{s+1}X) \arrow[r] & \pi_{V}^{C_2}(X^{s+1}) \arrow[r] & \pi_V^{C_2}(X^s)
%        \end{tikzcd}
%    \end{center}
%    is exact for every $V \in \RO(C_2)$. Since $\gr_{s+1}X$ and $X^s$ are strongly even, we deduce with Greenlees' gap characterization (\cref{lem:GreenleesGap}) that $X^{s+1}$ is strongly even.
%\end{proof}

\subsection{Cohomological Slice Tower}\label{sec:cohomological_slice}
The goal of this section is to set up a tower of spectra, thereby reducing computations of \(\myuline{\map}_{\ul{\Sp}_{G}}(X,E)\) to Bredon cohomology computations. We do this by applying the equivariant slice tower to \(E\).
\medskip \\The equivariant slice tower was introduced by Hill--Hopkins--Ravenel and was central to their solution of the Kervaire invariant one-problem \cite{HHR16}.
For the group \(C_2\), the \(C_2\)-slice tower was constructed and first studied by Dugger in \cite{dugger05slice}.
We only recall the minimal input we need to set up a cohomological version of the slice tower. 
For a more complete treatment the standard references are \cite{HHR16,hillprimer,ullman13regularslice}.
For an introduction we recommend and have benefited from \cite{bertminicourse}.
See also \cite[Section 2.1]{gabe2025realsyntomiccohomology} for a concise treatment in the same language we use.

\begin{remark}
    The formal properties of the slice tower were refined by Ullman \cite{ullman13regularslice} in his construction of the \emph{regular slice tower}. We will use Ullman's regular slice tower but for brevity sake only say slice tower.
\end{remark}

\begin{example}
The slice tower often organizes information in a more efficient way. The main motivating example to consider is Atiyah's \(K\)-theory with Reality \(\kR\).
The equivariant Bott map for \(\kR\) sits in the cofiber sequence 
\begin{center}
    \begin{tikzcd}
        S^{\rho} \otimes \kR \arrow[r, "\ol{\beta}"] & \kR \arrow[r] & \uHZ
    \end{tikzcd}
\end{center}
Since the equivariant Bott map shifts degree by \(\rho\), equivariant Bott periodicity appears complicated in the equivariant Postnikov tower. It is better visible in the slice tower.
By \cite{dugger05slice}, the slice tower for \(\kR\) is given as follows:
\[\begin{tikzcd}
	& {S^{2\rho}\otimes \uHZ} & 0 & {S^\rho\otimes \uHZ} & 0 \\
	\cdots & {P^{\leq 4}\kR} & {P^{\leq 3}\kR} & {P^{\leq 2}\kR} & {P^{\leq 1}\kR} & {P^{\leq 0}\kR\simeq \uHZ}
	\arrow[from=1-2, to=2-2]
	\arrow[from=1-3, to=2-3]
	\arrow[from=1-4, to=2-4]
	\arrow[from=1-5, to=2-5]
	\arrow[from=2-1, to=2-2]
	\arrow[from=2-2, to=2-3]
	\arrow[from=2-3, to=2-4]
	\arrow[from=2-4, to=2-5]
	\arrow[from=2-5, to=2-6]
\end{tikzcd}\]
The slice tower \(P^{\leq \bullet}\kR\) is depicted as the horizontal tower of $C_2$-spectra.
The vertical fibers are the associated graded pieces of the slice tower. The \(n\)'th associated graded piece is called the \(n\)'th \emph{slice} of \(\kR\).
This is exactly the analogue of the Postnikov tower for \(\ku\) with \(S^{2n}\otimes \HZ\) replaced by \(S^{n\rho}\otimes \uHZ\).
\end{example}

\begin{lemma} \label{rem:asgrisbredon}
    Let \(X\) be a strongly even \(C_2\)-spectrum.
    The odd slices \(P^{2n-1}_{2n-1}X\) vanish and
    the even slices are given by \(P^{2n}_{2n}X \simeq \Sigma^{n\rho}\mathrm{H} \myuline{\pi}_{n\rho}X\).
\end{lemma}
\begin{proof}
    As summarized in \cite[Proposition 2.13]{hillmeier2017} this follows from \cite[Proposition 4.20, Lemma 4.23]{hillhopkinsravenel2009Kervaire} and \cite[Corollary 2.16]{hillprimer}.
\end{proof}
\noindent Now we are ready to set up a cohomological version of the slice tower.
This is a cohomological version of the tower constructed by Carrick--Hill--Ravenel in their study of the homological slice spectral sequence \cite{carrick2024homologicalslice}.

\begin{construction}[Cohomological slice tower] \label{construction: cohomological slice tower}
    Let \(X, E \in \Sp_G \). 
    We define the \tb{cohomological slice tower for \(X\) and \(E\)} as the tower of \(G\)-spectra given by 
    \[
    \tb{\CST(X;E)}\coloneqq \myuline{\map}_{\ul{\Sp}_G}(X,P^{\leq \bullet} E).
    \]
    Note that the \(s\)'th associated graded is given by 
    \[
    \mathrm{gr}_s\CST(X;E) = \myuline{\map}_{\ul{\Sp}_G}(X,P^s_sE).
    \]
    Applying \(\myuline{\pi}_V\) this gives rise to a spectral sequence of Mackey functors of the form
    \[
    E^1_{V,s}=\myuline{\pi}_V(\myuline{\map}_{\ul{\Sp}_G}(X,P^s_sE)) \implies \myuline{\pi}_{V}(\myuline{\map}_{\ul{\Sp}_G}(X,E)).
    \]
\end{construction}

\begin{remark}
    \hfill 
    \begin{enumerate}[(i)]
        \item In general, convergence of the associated spectral sequence is subtle. As such, instead argue directly with the tower.
        \item In all our applications of interest \(E\) will be connective.
    So the cohomological slice tower takes the form
    \begin{center}
        \begin{tikzcd}
            \cdots \arrow[r] & \myuline{\map}_{\ul{\Sp}_G}(X,P^{\leq s} E) \arrow[r] & \cdots \arrow[r] & \myuline{\map}_{\ul{\Sp}_G}(X,P^{\leq 0} E) \arrow[r] & 0.
        \end{tikzcd}
    \end{center}
    In particular, \(\CST(X;E)\) is positively indexed for connective \(E\).
    \item Since the slice tower $P^{\leq \bullet}E$ is exhaustive, also $\CST(X;E)$ is exhaustive.
    \item By construction, the underlying non-equivariant tower of the cohomlogical slice tower \(\CST(X;E)\) is the tower whose associated spectral sequence is the usual cohomological Atiyah--Hirzebruch spectral sequence for \(X^e\) and \(E^e\).
\end{enumerate}
    
\end{remark}

\noindent By \cref{rem:asgrisbredon}, computing the associated graded of the cohomological slice tower involves computing Bredon cohomology.
In general, such computations can be very difficult.
The Bredon cohomology counterparts of many classical cohomology computations are not known.
However, Hill has identified a large class of spectra for which Bredon cohomology computations are largely formal \cite{hill2022freeness}.
We finish this section by leveraging the techniques of Hill to give criterion to check when the cohomological slice tower is strongly even.

\medskip \noindent 
The following definition isolates the kinds of \(C_2\)-spectra we are interested in computing Bredon cohomology of. 
\begin{definition} \label{def: finite type free homology}
    Let \(X\in \Sp_{C_2}\). 
    We say that \(X\) has \tb{finite type $\rho$-free homology} if there is an equivalence \(\uHZ\otimes X\simeq \uHZ \otimes \bigoplus_{i \in I} S^{n_i\rho}\) for some indexing set $I$ such that each dimension appears only finitely many times in the sum.
\end{definition}

\noindent Since we make use of the results of Hill's freeness arguments repeatedly, we review the main idea. We are grateful to Christian Carrick for explaining this to us.

\begin{lemma}[Hill's freeness argument]\label{lem:hills_freeness}
    Let \(X \in \Sp_{C_2} \).
    Suppose that \(X\) has finite type $\rho$-free homology with \(\uHZ\otimes X\simeq \uHZ \otimes \bigoplus_{i \in I} S^{n_i\rho}\).
    Let $M \in \Ab$ and consider the associated constant Mackey functor \(\myuline{M}\). Then, there is an equivalence 
    \[
    \myuline{\map}_{\Sp_{C_2}}(X,\mathrm{H}\myuline{M}) \simeq \mathrm{H}\myuline{M} \otimes \bigoplus_{i\in I}S^{-n_i\rho}.
    \]
\end{lemma}

\begin{proof}
    This follows from the much more general result of \cite[Theorem 3.38]{hill2022freeness};
    for the benefit of the reader we exhibit the details from the proof \cite[Theorem 3.38]{hill2022freeness} to prove the specific case we need directly.
    \medskip \\Since \(\myuline{M}\) is a constant Mackey functor, \(\mathrm{H}\myuline{M}\) is a \(\uHZ\)-module.
    We then have the following chain of equivalences: 
    \begin{align*}
        \myuline{\map}_{\ul{\Sp}_{C_2}}(X,\mathrm{H}\myuline{M})&\simeq \myuline{\map}_{\ul{\LMod}_{\uHZ}}(\uHZ\otimes X,\mathrm{H}\myuline{M})\\
        &\simeq \myuline{\map}_{\ul{\LMod}_{\uHZ}}\left(\uHZ\otimes \bigoplus_{i\in I}S^{n_i\rho},\mathrm{H}\myuline{M} \right)\\
        &\simeq \myuline{\map}_{\ul{\Sp}_{C_2}}\left(\S\otimes \bigoplus_{i\in I}S^{n_i\rho},\mathrm{H}\myuline{M} \right)\\
        &\simeq \myuline{\map}_{\ul{\Sp}_{C_2}}\left(\bigoplus_{i\in I} \S^{n_i\rho},\mathrm{H}\myuline{M} \right)\\
        &\simeq \prod_{i\in I}\myuline{\map}_{\ul{\Sp}_{C_2}}\left(\S^{n_i\rho},\mathrm{H}\myuline{M} \right)\\
        &\simeq \prod_{i\in I}\left(\mathrm{H}\myuline{M} \otimes S^{-n_i\rho}\right) \\
        &\simeq \mathrm{H}\myuline{M} \otimes \bigoplus_{i\in I}S^{-n_i\rho} 
    \end{align*}
    The finite type hypothesis is invoked for the final equivalence. 
\end{proof}

\noindent Combining Hill's freeness argument with the results on towers of strongly even spectra, we can deduce when the cohomological slice tower is strongly even. %For clarity, we treat the case \(G=C_2\) separately. 

\begin{theorem}\label{prop:checkingSAHSS}
    Let \(E \in \Sp_{C_2}\) be strongly even.
    Suppose that \(X \in \Sp_{C_2} \) has finite type $\rho$-free homology and is slice bounded below. 
    Then, $\ul{\map}_{\ul{\Sp}_{C_2}}(X, E)$ is strongly even. In particular, the restriction map
    \[
    \res_e^{C_2}\colon  \pi^{C_2}_{*\rho}(\myuline{\map}_{\ul{\Sp}_{C_2}}(X,E))\to \pi_{2*}(\map_{\Sp}(X,E))
    \]
    is an isomorphism.
\end{theorem}

\begin{proof}
    Consider the cohomological slice tower \(\CST(X;E)\). 
    We will use the conditions on \(X\) and \(E\) to show that the cohomological slice tower is a strongly even tower of \(C_2\)-spectra.
    The result then follows from \cref{prop:strongly_even_c2_tower}.
    \medskip \\First, we handle the underlying statement. Since \(E\) is strongly even, its underlying spectrum is even. Similarly, the assumption on \(X\) shows that its underlying homology is finite type free and even. % and is given by \(\HZ\otimes X\simeq \HZ\otimes \bigoplus_{i\in I}S^{2n_i}\).
    Then, \(\map_{\Sp}(X^e,E^e)\) is even by \cref{prop: map X E even}.
    \medskip \\Next, we need to see that the colimit is strongly even. In fact, we will argue that it is trivial, i.e. 
    \[ \colim_{s \to -\infty} \ul{\map}_{\ul{\Sp}_{C_2}} \left(X, P^{\leq s}E \right) \simeq 0. \]
    Indeed, $X \simeq P_{\geq n} X$ for some $n \in \Z$ by the bounded below assumption. Therefore, the mapping spectrum $\ul{\map}_{\ul{\Sp}_{C_2}} \left(X, P^{\leq s}E \right)$ is trivial for sufficiently small $s$ by the orthogonality conditions on the slices.
    \medskip \\It remains to see that the associated graded pieces of \(\CST(X;E)\) are strongly even, that is, we must show that \(\myuline{\map}_{\ul{\Sp}_{C_2}}(X,P^s_sE)
    \)
    is strongly even. 
    Since \(E\) is strongly even, its odd slices vanish, and its even slices are given by 
    \(
    P^{2s}_{2s}E=\Sigma^{s\rho}\mathrm{H}\myuline{\pi}_{s\rho}(E)
    \), as recalled in \cref{rem:asgrisbredon}.
    Strong evenness is preserved under \(s\rho\) suspensions, it suffices to show that \(
    \myuline{\map}_{\ul{\Sp}_{C_2}}(X,\mathrm{H}\myuline{\pi}_{s\rho}(E))
    \)
    is strongly even. Since \(E\) is strongly even, \(\myuline{\pi}_{s\rho}(E)\) is a constant Mackey functor.
    So the result follows from \cref{lem:hills_freeness}.
\end{proof}

\noindent We can phrase a slightly weaker, but also much longer condition. We do not need it in this article but have found uses in forthcoming work.

\begin{remark} \label{remark: weaker conditions strongly even mapping space}
    Instead of assuming that $X \in \Sp_{C_2}$ has finite type $\rho$-free homology, we may also assume the following four conditions.
    \begin{enumerate}[(a)]
        \item The spectrum $\map_{\Sp}(X^e, E^e)$ is even.
        \item The spectrum $\map_{\Sp}(X^e, \mathrm{H} \pi_{2s}^e(E))$ is even for every $s$.
        \item The $C_2$-spectrum $\ul{\map}_{\ul{\Sp}_{C_2}}(X, \uHZ)$ is strongly even.
        \item The group $\pi_{s\rho}(E)$ is finitely generated for each $s$.
    \end{enumerate}
    Conditions (a) and (b) are for example satisfied if $X^e$ has finite type free homology concentrated in even degrees.
\end{remark}

\begin{proof}
    We want to apply \cref{prop:strongly_even_c2_tower} on $\CST(X;E)$. Condition (a) already covers that the underlying mapping spectrum is even. So following the proof of \cref{prop:checkingSAHSS} it remains to show that $\ul{\map}_{\ul{\Sp}_{C_2}}(X, \mathrm{H}\ul{\pi}_{s\rho}(E))$ is strongly even for every $s$.
    \medskip \\By (b) we already assume that it is underlying even. So it suffices to prove
    \[ \pi_{*\rho-2}^{C_2} \ul{\map}_{\ul{\Sp}_{C_2}}(X, \mathrm{H}\ul{\pi}_{s\rho}(E)) = 0 \]
    by \cref{lem:C2_stronger_gap}. But this follows from the universal coefficient theorem (\cref{UCT}) combined with (c) and Greenlees' gap characterization of strong evenness (\cref{lem:GreenleesGap}). We are allowed to use the universal coefficient theorem by the finite type assumption (d).
\end{proof}

\noindent The conditions in \cref{remark: weaker conditions strongly even mapping space} are weaker than those in \cref{prop:checkingSAHSS} \textcolor{chilligreen}{(ii)} by Hill freeness (\cref{lem:hills_freeness}).

\subsection{Lifting Structured Orientations}\label{subsec:liftingorientations}
In this subsection we isolate arguments that would otherwise be repeated several times. We hope this is also a benefit for the reader with different applications in mind. For readability, we do not state these results in full generality, but our proofs suggest straightforward generalizations.
\medskip \\For future work, we require these results in the level of generality of \(R\)-module Thom spectra;
the reader only interested in \cref{mainthm:BPR} and \cref{mainthm: lifting orientations} loses nothing by assuming \(R=\mathbb{S}\).
For simplicity, we will assume that \(R\) is a \( \E_{\infty}^G \)-ring, but this is not strictly needed.
\medskip \\The arguments in this section are generalizations of the main argument in the sketch proof of \cref{mainthm:BPR} from the beginning of \cref{part:applications}.
The main idea of this section can be summarized as follows:
It is often easier to lift a non-equivariant \(\E_{\dim(V)}\)-map \(\mathrm{M}f^e\to E^e\) to an equivariant \(\E_V\)-map \(\mathrm{M}f\to E\), than it is to lift an equivariant but non-multiplicative map \(\mathrm{M}f\to E\) to an \(\E_V\)-map.
\medskip \\For historical reasons, and because we mainly have \(\MU\) in mind, in this section we will write \(\mathrm{M}f\) instead of \(\Th(f)\) for Thom spectra.

\begin{proposition}\label{thm: abstract akward lifting general}
    \fix{Let \(R\in \Alg_{\E^G_\infty}(\myuline{\Sp}_G)\) and $E \in \Alg_{\E_{\infty}^G}(\myuline{\LMod}_R^G)$.} Let $V$ be a $G$-representation and
    let $X$ be an $V$-loop space with a map
     \(f\colon X\to \ul{\Pic}_G(R)\) in $\Alg_{\E_{n \rho_G}}(\myuline{\Sc}_G)$. 
    Suppose there is an \(\E_{V}\)-\(R\)-algebra map \(\mathrm{M}f\to E\).
    Then, there is an equivalence
    \[
    \myuline{\Map}_{\ul{\Alg}_{\E_{V}}(\myuline{\LMod}_R^G)}(\mathrm{M}f,E) \simeq \Omega^\infty\myuline{\map}_{\ul{\Sp}_G}(\Sigma^\infty \mathrm{B}^{V}X, \Sigma^{V}\gl_1(E)).
    \]
    in $\Sc_G$.
\end{proposition}
\begin{proof}
    We perform the following string of equivalences:
    \begin{align*}
        \myuline{\Map}_{\ul{\Alg}_{\E_{n \rho_G}}(\myuline{\LMod}_R^G)}(\mathrm{M}f, E) &\simeq  \myuline{\Map}_{\ul{\Alg}_{\E_{n \rho_G}}(\myuline{\LMod}_R^G)}(R\otimes\Sigma^\infty_+ X,E) \\
        &\simeq \myuline{\Map}_{\ul{\Alg}_{\E_{n \rho_G}}(\myuline{\Sp}_G)}(\Sigma^\infty_+ X,E)\\ % (2)
        &\simeq \myuline{\Map}_{\ul{\Alg}_{\E_{n \rho_G}}(\myuline{\Sc}_G)}(X,\Omega^{\infty} E)\\ % (3)
        &\simeq \myuline{\Map}_{\ul{\Alg}_{\E_{n \rho_G}}^{\gp}(\myuline{\Sc}_G)}(X, \GL_1(E))\\ % (4)
        &\simeq \myuline{\Map}_{\ul{\Spaces}^G_*}(\mathrm{B}^{V}X, \mathrm{B}^{V}\GL_1(E))\\ % (5)
        &\simeq \myuline{\Map}_{\ul{\Sp}_G}(\Sigma^\infty \mathrm{B}^{V}X, \Sigma^{V}\gl_1(E))\\ % (6)
        &\simeq \Omega^\infty\myuline{\map}_{\ul{\Sp}_G}(\Sigma^\infty \mathrm{B}^{V}X, \Sigma^{V}\gl_1(E)) % (7)
    \end{align*}
    In order the equivalences are justified as follows:
    the first equivalence uses the multiplicative Thom isomorphism \cref{corollary: chadwick-mandell} \textcolor{chilligreen}{(ii)};
    the second uses the free \(R\)-module adjunction;
    the third uses the \(\Sigma^\infty_+\dashv \Omega^\infty\) adjunction;
    the fourth uses \cref{corollary: GL1 monoidal} along with the assumption that \(X\) is grouplike;
    the fifth and sixth both use \cref{thm: recognition}; 
    the final is by definition of \(\myuline{\map}_{\ul{\Sp}_G}\).
\end{proof}

\begin{remark}
    The strong multiplicative structures on $R$ and $E$ were assumed just so that $\GL_1(E)$ can be delooped to \(\gl_1(E)\).
    \fix{Suppose we only have the following weaker assumptions:
    \(R\) is an \(\E_{V}\otimes \E_2\)-ring; 
    \(E\) is an \(\E_V\otimes \E_1\)-ring, and there is an \(\E_V\otimes \E_1\)-ring map \(R\to E\);
    and 
    \(E\) admits an \(\E_V\)-\(\mathrm{M}f\)-orientation.
    Then, the first five equivalences in the proof of \cref{thm: abstract akward lifting general} hold and give an equivalence
    \[
    \myuline{\Map}_{\ul{\Alg}_{\E_{V}}(\myuline{\LMod}_R^G)}(\mathrm{M}f,E) \simeq \myuline{\Map}_{\ul{\Spaces}^G_*}(\mathrm{B}^{V}X, \mathrm{B}^{V}\GL_1(E)).
    \]
    in $\Sc_G$.}
\end{remark}

\noindent Specializing to $G = C_2$ and assembling all our work, we may now deduce the abstract lifting result.

\begin{theorem}\label{thm:abstract_lifting_clean_c2}
    Let \(R\in \Alg_{\E^{C_2}_\infty}(\myuline{\Sp}_{C_2})\) and let $E \in \Alg_{\E_{\infty}^{C_2}}(\myuline{\LMod}_R^{C_2})$ be strongly even. Let $n \in \N$ and let $X$ be an $n\rho$-loop space with a map
     \(f\colon X\to \ul{\Pic}_{C_2}(R)\) in $\Alg_{\E_{n \rho}}(\myuline{\Sc}_{C_2})$. Suppose the following:
    \begin{enumerate}[(i)]
        \item There exists an \(\E_{n\rho}\)-\(R\)-algebra map \(\mathrm{M}f\to E\).
        \item The space \(\mathrm{B}^{n\rho}X\) has finite type $\rho$-free homology.
    \end{enumerate}
    Then, the restriction map 
    \[
    \res_e^{C_2}\colon  \pi^{C_2}_{*\rho}\left(\myuline{\Map}_{\ul{\Alg}_{\E_{n \rho}}(\myuline{\LMod}_R^{C_2})}(\mathrm{M}f,E) \right)\to \pi^{e}_{*\vert\rho\vert}\left(\myuline{\Map}_{\ul{\Alg}_{\E_{n \rho}}(\myuline{\LMod}_R^{C_2})}(\mathrm{M}f,E) \right)
    \]
    is an isomorphism.
\end{theorem}

\begin{proof}
    By \cref{thm: abstract akward lifting general} we have an equivalence of \(C_2\)-spaces
    \[
    \myuline{\Map}_{\ul{\Alg}_{\E_{n \rho}}(\myuline{\LMod}_R^{C_2})}(\mathrm{M}f,E) \simeq \Omega^\infty\myuline{\map}_{\ul{\Sp}_{C_2}}(\Sigma^\infty \mathrm{B}^{n\rho}X,\Sigma^{n\rho}\gl_1(E)).
    \]
    Note that the right side has a preferred basepoint, so the equivalence imports a basepoint onto the left side. This is what the basepoint we mean when applying $\pi_{*\rho}$ for $* \neq 0$. By the assumptions on \(\mathrm{B}^{n\rho}X\) and \(E\), the statement follows from \cref{prop:checkingSAHSS}.
    Note here that strong evenness of \(E\) implies strong evenness of \(\Sigma^{n\rho}\gl_1(E)\) by \cref{lem: homotopy of gl1}. Moreover, \(\Sigma^{n\rho}\gl_1(E)\) is connective by construction.
\end{proof}

\begin{proposition} \label{prop: weaker conditions abstract lifting}
    Instead of condition (ii) we can pose the following four conditions.
    \begin{enumerate}[(a)]
        \item The spectrum $\map_{\Sp}(\Sigma^{\infty} \mathrm{B}^{2n}X^e, \Sigma^{2n} \gl_1 E^e)$ is even.\footnote{This is for example satisfied if $\mathrm{B}^{2n}X^e$ has finite type free homology concentrated in even degrees.}
        \item The spectrum $\map_{\Sp}(B^{2n}X^e, \mathrm{H} \pi_{2s}^e(\gl_1 E))$ is even for every $s$.
        \item The $C_2$-spectrum $\ul{\map}_{\ul{\Sp}_{C_2}}(\Sigma^{\infty} \mathrm{B}^{n\rho}X, \uHZ)$ is strongly even.
        \item The group $\pi_{s\rho}(\gl_1 E)$ is finitely generated for each $s$.
    \end{enumerate}
    Conditions (a) and (b) are for example satisfied if $\mathrm{B}^{2n}X^e$ has finite type free homology concentrated in even degrees.
\end{proposition}

\begin{proof}
    This follows by applying \cref{remark: weaker conditions strongly even mapping space} after seeing
    \[
    \myuline{\Map}_{\ul{\Alg}_{\E_{n \rho}}(\myuline{\LMod}_R^{C_2})}(\mathrm{M}f,E) \simeq \Omega^\infty\myuline{\map}_{\ul{\Sp}_{C_2}}(\Sigma^\infty \mathrm{B}^{n\rho}X,\Sigma^{n\rho}\gl_1(E)).
    \]
    with \cref{thm: abstract akward lifting general}.
\end{proof}

\noindent We will specialize to endomorphisms and in particular study idempotents. When we speak about structured idempotents, we mean a structured map which is an idempotent underlying.

\begin{corollary} \label{corollary: lifting idempotents}
    Let $R \in \Alg_{\E_{\infty}^{C_2}}(\ul{\Sp}_{C_2})$ and let $f \colon X \to \ul{\Pic}_{C_2}(R)$ be a map of $n\rho$-loop spaces. Suppose that $\mathrm{B}^{n\rho}X$ has finite type $\rho$-free homology and that $\mathrm{M}f$ is strongly even. Then, any $\E_{2n}$-idempotent $e \colon \mathrm{M}f^e \to \mathrm{M}f^e$ lifts uniquely to an $\E_{n\rho}$-idempotent $\ol{e} \colon \mathrm{M}f \to \mathrm{M}f$.
\end{corollary}

\begin{proof}
    The unique lift comes from \cref{thm:abstract_lifting_clean_c2} where condition (i) is always satisfied since $\id_{\mathrm{M}f}$ provides an orientation. Now suppose that $\ol{e}$ lifts an idempotent $e$. Then, $\ol{e} \circ \ol{e}$ lifts $e \circ e \simeq e$. Thus, $\ol{e} \circ \ol{e} \simeq \ol{e}$ by uniqueness of lifts. So $\ol{e}$ is an idempotent. We remark that idempotents are in general more complicated in a higher categorical setting, but no issues arise since we are working in a stable setting \cite[Lemma 1.2.4.6]{lurie2017ha}.
\end{proof}

\noindent Note that we could instead also have demanded the (weaker) conditions from \cref{prop: weaker conditions abstract lifting}.

\section{Real Orientations \& Multiplication on \texorpdfstring{\(\BPR\)}{BPR}}\label{sec: Real orienations}

\subsection{Constructing Multiplicative Real Orientations}\label{subsec:universal_orientation}
Let \(E\) be a strongly even \(\E_\infty^{C_2}\)-ring spectrum.
The goal of this section is to show \(E\) admits an \(\E_\rho\)-\(\MUR\)-orientation.
Informally speaking, our strategy is to produce an orientation via a versal example.
More precisely, we will construct an \(\E_\rho\)-map \(\MUR\to \MWR\), 
where \(\MWR\) is a \(\E_\infty^{C_2}\)-ring spectrum constructed by Angelini-Knoll--Kong--Quigley which admits \(\E_\infty^{C_2}\)-ring maps \(\MWR\to E\) \cite{gabe2025realsyntomiccohomology}.
This is a Real version of the spectrum $\MW$ considered by Hahn--Raksit--Wilson \cite[Definition 3.2.9]{hahn2024motivicfiltrationtopologicalcyclic}.

\begin{construction}[{\cite[5.8]{gabe2025realsyntomiccohomology}}] \label{construction: MWR}
    Let \(\overline{\gamma}\colon \MUR\to \kR\) be any map of \(C_2\)-spectra which is surjective on \(\pi^{C_2}_{*\rho}\).
    For example we may take Hahn--Shi's lift of the Conner--Floyd orientation \cite{hahnRealOrientationsLubin2020}.
    Applying \(\Omega^{\infty}\Sigma^\rho\), using equivariant Bott perodicity (\cref{example: equivariant Bott periodicity ku}), and postcomposing with the Real \(J\)-homomorphism (\cref{subsection: Real J}) gives an \(\E_{\infty}^{C_2}\)-map 
    \[
    \Omega^{\infty}\Sigma^\rho\MUR \to \Omega^{\infty}\Sigma^\rho\kR\simeq \BUR \to \ul{\Pic}_{C_2  }(\ul{\Sp}_{C_2}).
    \]
    The spectrum \(\tb{\MWR}\) is defined to the Thom spectrum of this map.
\end{construction}

\begin{remark}[Real Wilson Spaces]\label{rem: real wilson spaces}
    \hfill 
    \begin{enumerate}[(i)]
        \item The construction of \(\MWR\) by Angelini-Knoll--Kong--Quigley is a (Real) equivariant refinement of a construction of Hahn--Raksit--Wilson. In their work on the even filtration \cite[Definition 3.2.9]{hahn2024motivicfiltrationtopologicalcyclic}, Hahn--Raksit--Wilson constructed \(\MW\) as the Thom spectrum of
    \[
    \tb{\gamma} \colon \Omega^{\infty}\Sigma^2\MU \to \Omega^{\infty}\Sigma^2\ku\simeq \BU \to \Pic(\Sp).
    \]
    By construction, Hahn--Raksit--Wilson's \(\MW\) is the underlying non-equivariant spectrum of Angelini-Knoll--Kong--Quigley's \(\MWR\).
    \item The base space $\Omega^{\infty} \Sigma^{\rho} \MU_{\R}$ of \(\MWR\) is an example of a \emph{Real Wilson Space}.
    The non-equivariant version was first extensively studied by Wilson in his PhD thesis \cite{Wilson1973, wilson1975wilsonspaces2}, hence the nomenclature.
    The study of the Real version was undertaken by Hill--Hopkins \cite{hill2018realwilsonspaces}.
    In particular, they show that \(\Omega^\infty \Sigma^{n\rho}\MUR\) has finite type $\rho$-free homology, see \cite[Theorem 1.2]{hill2018realwilsonspaces}. %with
    %\[
    %\uHZ\otimes \Omega^\infty \Sigma^{n\rho}\MUR \simeq \uHZ \otimes \bigoplus_{i\in I} S^{n_i\rho},
    %\]
    %for a certain indexing set $I$ and $n_i \in \N$
    More precisely, the computation of Hill--Hopkins shows that \(\Omega^\infty \Sigma^{n\rho}\MUR\) has free \(\uHZ\)-homology, and that \(\uHZ\otimes \Omega^\infty \Sigma^{n\rho}\MUR\) is strongly even.
    The finite type assumption follows from the non-equivariant computation Wilson \cite[Theorem 3.3, Corollary 3.4]{Wilson1973}.
    
    Note that the work of Hill--Hopkins also computes the \(\uZ\)-homology of every \(m\rho\)-fold delooping of \(\Omega^\infty \Sigma^\rho \MUR\).
    By \cref{lem:delooping_suspension} we have
    \[
    \mathrm{B}^{m\rho}\Omega^\infty \Sigma^{n\rho}\MUR\simeq \Omega^\infty \Sigma^{(n+m)\rho}\MUR.
    \]
    So for every \(m\), we deduce that \(\mathrm{B}^{m\rho}\Omega^\infty \Sigma^{n\rho}\MUR\) has finite type $\rho$-free homology. %with
    %\[
    %\uHZ\otimes \mathrm{B}^{m\rho}\Omega^\infty \Sigma^{\rho}\MUR \simeq \uHZ \otimes \bigoplus_{j\in J} S^{n_j\rho}.
    %\]
    %for a certain indexing set \(J\).
    \end{enumerate}
\end{remark}

\noindent One of the main motivations for the construction of \(\MWR\) is that it provides multiplicative orientations for many rings of interest.

\begin{theorem}[{\cite{gabe2025realsyntomiccohomology}}] \label{theorem: C2 Einfty map from MW}
    Let \(E\) be a strongly even \(\E_{\infty}^{C_2}\)-ring spectrum. 
    There exists a \(\E_{\infty}^{C_2}\)-ring map \(\MWR\to E\).
\end{theorem}
\begin{proof}
    This follows from \cite[Theorem 5.10 (1)]{gabe2025realsyntomiccohomology}. Specifically, see the last paragraph of the proof.
\end{proof}

\begin{proposition}[\(\MWR\)-splitting]\label{prop: MWR splitting}
    There is an \(\E_\rho\)-idempotent of \(\MWR\) that splits off \(\MUR\) as an \(\E_\rho\)-retract.
    In particular, there is an \(\E_\rho\)-map \(\MUR\to \MWR\).
\end{proposition}

\begin{proof}
    There are three steps: First we discuss a non-equivariant $\E_2$-splitting after Hahn--Raksit--Wilson. Then, we lift the non-equivariant splitting to an equivariant splitting. Finally, we check that the summand split off equivariantly is indeed \(\MUR\).
    \medskip \\Hahn--Raksit--Wilson show \(\MW \simeq \MU\otimes\, \Sigma^\infty_+ \mathrm{fib}(\Omega^\infty \Sigma^2\gamma) \) as \(\E_2\)-rings in the proof of \cite[Theorem 3.2.10]{hahn2024motivicfiltrationtopologicalcyclic}.
    Now consider the map \(\mathrm{fib}(\Omega^\infty \Sigma^2\gamma)\to *\) which makes \(\Sigma^\infty_+ \mathrm{fib}(\Omega^\infty \Sigma^2\gamma)\) into an augmented \(\E_2\)-algebra.
    Using the augmentation we can construct an $\E_2$-idempotent
    \[
    e\colon\MW\simeq \MU\otimes\, \Sigma^\infty_+ \mathrm{fib}(\Omega^\infty \Sigma^2\gamma) \to \MU\otimes \,\mathbb{S} \to   \MU\otimes \,\Sigma^\infty_+ \mathrm{fib}(\Omega^\infty \Sigma^2\gamma) \simeq \MW
    \]
    which splits off $\MU$ from $\MW$. Thus,
    \[
    \MU\simeq \colim \left(\MW\xrightarrow{e} \MW\xrightarrow{e} \MW\xrightarrow{e} \cdots \right).
    \]
    This completes the first step.
    \medskip \\Now, we show that \(e\) can be lifted to an \(\E_\rho\)-idempotent \(\overline{e}\colon \MWR\to \MWR\).
    So we must check the conditions of \cref{corollary: lifting idempotents}.
    The fact that \(\MWR\) is strongly even is \cite[Theorem C, Proposition 5.9.]{gabe2025realsyntomiccohomology}.
    It remains to check the condition on the base space \(\Omega^\infty \Sigma^\rho\MUR\) of \(\MWR\).
    To apply \cref{corollary: lifting idempotents} in the case of \(\E_\rho\)-structures, we must show that \(\mathrm{B}^{\rho}\Omega^\infty \Sigma^{\rho}\MUR\) has finite type $\rho$-free homology. %with 
    %\[
    %\uHZ\otimes \mathrm{B}^{\rho}\Omega^\infty \Sigma^{\rho}\MUR \simeq \uHZ \otimes \bigoplus_{n_i\in I} S^{n_i\rho}
    %\]
    This follows from the work of Hill--Hopkins as recalled in \cref{rem: real wilson spaces}.
    So \cref{corollary: lifting idempotents} produces an \(\E_\rho\)-idempotent \(\overline{e}\colon \MWR\to \MWR\) lifting \(e\). 
    \medskip \\As an idempotent, $\ol{e}$ splits off a summand. It remains to check that this summand is indeed \(\MUR\). We temporarily make the distinction and write
    \[
    \overline{\MU}_{\R} \coloneqq \colim \left(
    \MWR \xrightarrow{\overline{e}} \MWR \xrightarrow{\overline{e}} \MWR \xrightarrow{\overline{e}}\cdots
    \right).
    \]
    As \(\overline{\MU}_{\R}\) is a filtered colimit of strongly even \(C_2\)-spectra,
    by Greenlees' gap characterization of strong evenness (\cref{lem:GreenleesGap}), it is immediate that \(\overline{\MU}_{\R}\) is strongly even.
    Thus, by \cite[Lemma 3.4]{hillmeier2017}, it suffices to produce a map \(\overline{\MU}_{\R}\to \MUR\) that is a non-equivariant equivalence.
    By construction of \(\overline{\MU}_{\R}\) we have a map \(\overline{\MU}_{\R}\to \MWR\) whose underlying non-equivariant map is the inclusion \(\MU\hookrightarrow \MW\).
    The non-equivariant splitting provides an \(\E_2\)-map \(\MW\to \MU\) such that the composite 
    \(\MU\hookrightarrow \MW \to \MU\) is equivalent to the identity.
    Using \cref{thm:abstract_lifting_clean_c2} we can produce a lift of \(\MW\to \MU\) to an \(\E_\rho\)-map \(\MWR\to \MUR\). Thus we have an \(\E_\rho\)-composite \(\overline{\MU}_{\R}\to \MWR\to \MUR\) whose underlying non-equivariant map is an equivalence. This completes the proof.
\end{proof}

\begin{corollary}\label{cor: existence of Erho-MUR-orientations}
    Any strongly even \(\E_{\infty}^{C_2}\)-ring admits an \(\E_\rho\)-\(\MUR\)-orientation.
\end{corollary}
\begin{proof}
    Since $\BU_{\R}$ is grouplike, any \(\E_\rho\)-map \(\MUR\to E\) is an \(\E_\rho\)-\(\MUR\)-orientation (\cref{lemma: grouplike orientation}).
    By \cref{theorem: C2 Einfty map from MW} there is a \(\E_{\infty}^{C_2}\)-ring map \(\MWR\to E\). Precomposing this with the map from \cref{prop: MWR splitting} gives the result.
\end{proof}

\noindent The orientation results can be extended to larger groups.
\begin{corollary} \label{corollary: MUG orientations}
Let \(C_2\leq G\) be a finite group and $E \in \Alg_{\E_{\infty}^G}(\ul{\Sp}_G)$. Suppose that \(\mathrm{Res}^G_{C_2}E\) is a strongly even.
Then, there exists an \(\Coind^G_{C_2}\E_\rho\)-algebra map \(\MU^{(\! ( G ) \! )} = N_{C_2}^G \MUR \to E\).
\end{corollary}
\begin{proof}
    By \cref{cor: existence of Erho-MUR-orientations}, there is an \(\E_\rho\)-map \(\MUR\to \mathrm{Res}^G_{C_2}E\). We then take the composite
    \begin{center}
        \begin{tikzcd}
            N_{C_2}^G \MUR \arrow[r] & N_{C_2}^G \Res_{C_2}^G E \arrow[r] & E
        \end{tikzcd}
    \end{center}
    where the first map is a $\Coind_{C_2}^G \E_{\rho}$-map (\cref{construction: Coind}) and the second map is the counit of an adjunction $N_{C_2}^G \colon \Alg_{\E_{\infty}^{C_2}}(\ul{\Sp}_{C_2}) \rightleftarrows \Alg_{\E_{\infty}^{G}}(\ul{\Sp}_{G}) \cocolon \Res_{C_2}^G$, hence an \(\E^G_\infty\)-map.
\end{proof}

\subsection{Refining Multiplicative Real Orientations}\label{subsec:multiplicative_real_orientations}

Let $E$ be a strongly even $\E_{\infty}^{C_2}$-ring. There always exists some \(\E_\rho\)-map \(\MUR\to E\) by \cref{cor: existence of Erho-MUR-orientations}. 
The main result of this section provides a way to construct preferred maps.

\begin{theorem}\label{thm: main lifting restated}
    Let \(E\) be a strongly even \(\E_{\infty}^{C_2}\)-ring spectrum.
    Any \(\E_2\)-ring map \(\MU\to E^e\) lifts uniquely to an \(\E_\rho\)-ring map \(\MUR\to E\). 
\end{theorem}
\begin{proof}
    First,
    we show that any \(\E_2\)-ring map \(\MU\to E^e\) admits a lift to an \(\E_\rho\)-ring map \(\MUR\to E\).
    Recall from \cref{prop: MWR splitting} that we have an \(\E_\rho\)-retraction \(\MUR\to \MWR\to \MUR\) lifting an \(\E_2\)-retraction \(\MU\to \MW\to \MU\). 
    Given an \(\E_2\)-ring map \(\MU\to E^e\), the composite 
    \begin{center}
        \begin{tikzcd}
            \MW \arrow[r] & \MU \arrow[r] & E^e
        \end{tikzcd}
    \end{center} 
    is also an \(\E_2\)-ring map.
    We can apply \cref{thm:abstract_lifting_clean_c2} to lift this to an \(\E_\rho\)-map \(\MWR\to E\). Here, \cref{thm:abstract_lifting_clean_c2} applies since \(E\) is a strongly even \(\E^{C_2}_\infty\)-ring spectrum admitting an \(\E_\rho\)-\(\MWR\)-orientation (\cref{theorem: C2 Einfty map from MW}), and \(\mathrm{B}^{\rho}\Omega^\infty \Sigma^{\rho}\MUR\) has finite type free homology (\cref{rem: real wilson spaces}).
    Precomposing this lift with \(\MUR\to \MWR\) gives an \(\E_\rho\)-map \(\MUR\to E\).
    By construction, this is a lift of the \(\E_2\)-map
    \begin{center}
        \begin{tikzcd}
            \MU \arrow[r] & \MW \arrow[r] & \MU \arrow[r] & E.
        \end{tikzcd}
    \end{center}
    Since \(\MU\to \MW\to \MU\) is $\id_{\MU}$, the \(\E_\rho\)-map \(\MUR\to E\) we produced is indeed a lift of the given \(\E_2\)-map \(\MU\to E^e\).

\medskip
\noindent
    It remains to show that such lifts are unique. 
    Let \(f\colon \MU\to E^e\) be an \(\E_2\)-ring map.
    Let \(\overline{f},\tilde{f}\colon \MUR\to E\) be two \(\E_\rho\)-lifts of \(f\).
    Since \(\overline{f}\) and \(\tilde{f}\) both restrict to \(f\), their precompositions with \(\MWR\to \MUR\) both restrict to the precomposition of \(f\) with \(\MW\to \MU\).
    By \cref{thm:abstract_lifting_clean_c2}, any \(\E_2\)-map \(\MW\to E^e\) admits a unique lift to an \(\E_\rho\)-map \(\MWR\to E\).
    By precomposing with \(\MUR\to \MWR\) we learn that \(\overline{f}\) and \(\tilde{f}\) are equivalent.
\end{proof}

\noindent In particular, this provides a Real version of Chadwick--Mandell's result \cite[Theorem 1.2]{chadwickmandell}.

\begin{corollary}
    Let $E$ \label{remark: Real version of Chadwick Mandell} be a strongly even $\E_{\infty}^{C_2}$-ring spectrum. 
    \begin{enumerate}[(i)]
        \item Any homotopy ring map $\MU \to E^e$ lifts uniquely to an $\E_{\rho}$-ring map $\MUR \to E$.
        \item Any homotopy ring map $\MUR \to E$ lifts uniquely to an $\E_{\rho}$-ring map $\MUR \to E$. \qedhere
    \end{enumerate}
\end{corollary}

\begin{proof}
    \hfill 
    \begin{enumerate}[(i)]
        \item By Chadwick--Mandell $\MU \to E^e$ lifts uniquely to an $\E_2$-ring map $\MU \to E^e$ \cite[Theorem 1.2]{chadwickmandell}. Apply \cref{thm: main lifting restated}.
        \item Forgetting down to a homotopy ring map $\MU \to E^e$ happens in a unique fashion. Indeed, Real orientations are cohomology classes in $E^{\rho}(\mathbb{CP}_{\mathbb{R}}^{\infty})$ which restrict to the unit along the inclusion map $S^{\rho} \simeq \mathbb{CP}_{\mathbb{R}}^1 \to \mathbb{CP}_{\R}^{\infty}$, but $\ul{\map}_{\ul{\Sp}_{C_2}}(\Sigma^{\infty} \mathbb{CP}_{\R}^{\infty}, E)$ and $\ul{\map}_{\ul{\Sp}_{C_2}}(\Sigma^{\infty} \mathbb{CP}_{\R}^{1}, E)$ are strongly even, so in the diagram
        \begin{center}
            \begin{tikzcd}
                \widetilde{E}^{\rho}(\mathbb{CP}_{\R}^{\infty}) \arrow[r] \arrow[d, "\res_e^{C_2}", swap] & \widetilde{E}^{\rho}(\mathbb{CP}_{\R}^{1}) \arrow[d, "\res_e^{C_2}"]
                \\ \widetilde{E}^2(\mathbb{CP}^{\infty}) \arrow[r] & \widetilde{E}^2(\mathbb{CP}^1)
            \end{tikzcd}
        \end{center}
        the restriction maps $\res_e^{C_2}$ are isomorphisms. Hence, Real orientations correspond to complex orientations for strongly even rings $E$. Now apply (i). \qedhere
    \end{enumerate}
\end{proof}

%\noindent Our lifting results begs the following question and we intend to return to it in future work:

%\begin{question}
%    Let $E$ be a strongly even $\E_{\infty}^{C_2}$-ring spectrum.
%    \begin{enumerate}[(i)]
%        \item Let $n > 1$. Do $\E_{2n}$-ring maps $\MU \to E^e$ lift to $\E_{n\rho}$-ring maps $\MU_{\R} \to E$?
%        \item Do $\E_{\infty}$-ring maps $\MU \to E^e$ lift to $\E_{\infty}^{C_2}$-ring maps $\MU_{\R} \to E$?
%    \end{enumerate}
%\end{question}
%\noindent We will return to this question in future work.

\noindent Now, we can now leverage the Real Chadwick--Mandell result (\cref{remark: Real version of Chadwick Mandell}) and immediately obtain structured versions of orientations of interest, for example:

\begin{corollary} \label{corollary: hahn-shi hirzebruch n}
    \hfill 
    \begin{enumerate}[(i)]
        \item The Hahn--Shi Real orientations \cite{hahnRealOrientationsLubin2020} of Lubin--Tate theories \(\MUR\to E_n\) admit lifts to $\E_{\rho}$-maps.
        \item The Hirzebruch level-\(n\) genera \cite{MeierHirzebruch} \(\MUR\to \tmf_1(n)\) admit lifts to $\E_{\rho}$-maps for \(n>1\). 
    \end{enumerate}
\end{corollary}
\begin{proof}
    The \(\E^{C_2}_\infty\)-structures on \(E_n\) and \(\tmf_1(n)\) come from the cofree structures \cite[Theorem 2.4, 2.7]{hillmeier2017}. It's a result from Hahn--Shi that $E_n$ is strongly even \cite[Theorem 1.9]{hahnRealOrientationsLubin2020}, and a result of Meier that $\tmf_1(n)$ is strongly even \cite[Theorem 2.22]{MeierHirzebruch}.
\end{proof}

\begin{corollary} \label{corollary: Hahn--Shi orientation for G}
    Let $C_2 \leq G$.
    \begin{enumerate}[(i)]
        \item The Hahn--Shi Real orientations $\MU^{(\! ( G ) \! )} =  N_{C_2}^G \MU_{\R} \to E_n$ refine to $\Coind_H^G \E_{\rho}$-maps.
        \item Let $n > 1$ and $G = (\Z/n)^{\times}$. The Hirzebruch level-$n$ genera induce $\Coind_H^G \E_{\rho}$-maps $\MU^{(\! ( G ) \! )} \to \tmf_1(n)$ for $n > 1$.
    \end{enumerate}
\end{corollary}

\begin{proof}
    These orientations are constructed as in \cref{corollary: MUG orientations}.
\end{proof}

\noindent Previously, the Hahn--Shi Real orientations were only known to be $\E_{\sigma}$ and the normed versions were only known to be homotopy associative \cite{hahnRealOrientationsLubin2020}. The Hirzebruch level $n$ genera were not known to admit any \(\E_V\)-structure and were only constructed on the level of homotopy commutative \(C_2\)-ring spectra \cite{MeierHirzebruch}.
This is in stark contrast to what is known non-equivariantly: Senger shows that the underlying map $\MU \to \tmf_1(n)$ refines to an $\E_{\infty}$-map \cite[Theorem 1.7]{sengerleveln}; 
Burklund--Schlank--Yuan show that the Lubin--Tate theories admit \(\E_\infty\)-complex orientations \cite[Corollary 8.13]{burklund2022chromaticnullstellensatz}.
We expect that the Real orientations in \cref{corollary: hahn-shi hirzebruch n} admit further refinements to \(\E^{C_2}_\infty\)-maps, and intend to return to this problem in future work.
\begin{remark}
\change{
    Carrick showed in recent work that taking equivariant slice filtrations is lax $G$-symmetric monoidal \cite[Definition 6.4]{carrick2025slice}. Thus, taking the slice filtration of the orientations constructed in \cref{corollary: MUG orientations} yields multiplicative maps in filtered \(G\)-spectra.
    In particular, \cref{corollary: Hahn--Shi orientation for G}
    produces $\Coind_H^G \E_{\rho}$-maps of filtered \(G\)-spectra \(\mathrm{Slice}(\MU^{(\! ( G ) \! )})\to \mathrm{Slice}(E_n)\).}
\end{remark}

\noindent Let $p$ be a prime. We recall from \cite[Construction 5.3]{BHLS} that the stable Adams conjecture yields a commuting diagram of $\E_{\infty}$-maps
\begin{center}
    \begin{tikzcd}
        \BU_{(p)} \arrow[rr, "\Psi^{\ell}"] \arrow[dr, "J", swap] & & \BU_{(p)} \arrow[dl, "J"]
        \\ & \BGL_1(\S_{(p)})
    \end{tikzcd}
\end{center}
where $\Psi^{\ell} \colon \BU_{(p)} \to \BU_{(p)}$ denotes an Adams operation. Thomifying yields \(\E_\infty\)-maps
\[ \Psi^{\ell} \colon \MU_{(p)}\longrightarrow \MU_{(p)}, \] 
also called \emph{Adams operations on $\MU_{(p)}$} in \cite[Construction 5.3]{BHLS}.
We are grateful to Christian Carrick for pointing out to us that \cref{thm: main lifting restated} can be used to construct \(\E_\rho\)-lifts of these.

\begin{corollary}\label{cor:adamsopsMUR}
    The Burklund--Hahn--Levy--Schlank \cite{BHLS} Adams operations on \(\MU_{(2)}\) admit lifts to \(\E_\rho\)-operations on \({\MUR}_{(2)}\).
\end{corollary}
\begin{proof}
    The proof of \cref{thm: main lifting restated} works in the \(2\)-local case.
    So we can produce \(\E_\rho\)-maps $\overline{\Psi}^\ell\colon \MUR_{(2)}\to\MUR_{(2)}$
    lifting the underlying \(\E_2\)-map of the Burklund--Hahn--Levy--Schlank Adams operations $\Psi^\ell\colon \MU_{(2)}\to\MU_{(2)}$.
\end{proof}

\noindent For $C_2 \leq G$ applying $N_{C_2}^G$ also yields an $\Coind_{C_2}^G \E_{\rho}$-structured version.

\begin{remark}
    The Adams operations constructed in \cite{BHLS} are \(\E_\infty\)-operations.
    It is expected that \(\MUR_{(2)}\) should admit \(\E_{\infty}^{C_2}\)-lifts of these Adams operations. 
    % However, for many applications \(\E_\rho\) is sufficient.
\end{remark}

\subsection{Multiplication on \texorpdfstring{\(\BPR\)}{BPR}}\label{subsec:multiplication_on_BPR}
%Let $p$ be an even prime.
Everything is implicitly $(2)$-local in this subsection.
\medskip \\We use our $\MUR$ results to learn about the Real Brown--Peterson spectrum $\BPR$ which can be constructed as a colimit along an iteration of the Real Quillen idempotent \cite[Section 7]{araki1979orientations}. 
The following provides the first structured version of \(\BPR\) and its norms.
\begin{theorem}\label{thm:multiplication_on_BPR}
    \hfill 
    \begin{enumerate}[(i)]
        \item The Real Brown--Peterson spectrum \(\BPR\) admits an \(\E_\rho\)-algebra structure.
        \item Let $C_2 \leq G$. Then, $\BP^{(\!(G )\!)} = N_{C_2}^G \BP_{\R}$ admits a $\Coind_{C_2}^G \E_{\rho}$-algebra structure.
    \end{enumerate}
\end{theorem}
\begin{proof}
    \hfill 
    \begin{enumerate}[(i)]
        \item By \cref{thm: main lifting restated}, the Real Quillen idempotent lifts to an \(\E_\rho\)-algebra map \(\MUR\to \MUR\). 
    The result follows since filtered colimits of algebras are computed on underlying \cite[Theorem 5.1.4]{nardinshah2022equivarianttopos}.
    \item This follows immediately from (i) and \cref{construction: Coind}. \qedhere
    \end{enumerate}
\end{proof}

\begin{remark}
    Since the multiplicative structure on \(\BPR\) is constructed as an \(\E_\rho\)-retract of \(\MUR\), the $\E_{\rho}$-Adams operations on \(\MUR\) from \cref{cor:adamsopsMUR} induce $\E_{\rho}$-Adams operations on \(\BPR\).
\end{remark}

\noindent In a similar fashion, we deduce a $\BPR$-version of \cref{thm: main lifting restated}. 

\begin{corollary}
    Let \(E\) be a \(2\)-local strongly even \(\E^{C_2}_\infty\)-ring. Any homotopy ring map \(\BP\to E^e\) can be lifted to an \(\E_\rho\)-algebra  map \(\BPR\to E\).
\end{corollary}

\begin{proof}
    We can lift the map $\MU \to \BP \to E^e$ to an $\E_{\rho}$-algebra map $\MU_{\R} \to E$ by \cref{thm: main lifting restated}. Precomposing with $\BP_{\R} \to \MU_{\R} \to E$ yields the desired lift. This is really a lift since the underlying map is 
    \begin{center}
        \begin{tikzcd}
            \BP \arrow[r] & \MU \arrow[r] & \BP \arrow[r] & E^e
        \end{tikzcd}
    \end{center}
    where the first two maps compose to $\id_{\BP}$, i.e.~this composite is the original map.
\end{proof}

\noindent This, for example, gives \(\E_\rho\)-structured \(\BPR\)-versions of the Hahn--Shi Real orientations of Lubin--Tate theories, and the Hirzebruch level-\(n\) genera  from \cref{corollary: hahn-shi hirzebruch n}.

\section{Factorization Homology of Thom Spectra \& \texorpdfstring{\(\E_V\)}{EV}-Quotients}\label{sec:factorization_and_EV}
\subsection{Real Topological Hochschild Homology of $R$-Module Thom Spectra}\label{sec:factorization}
This section will reap some benefits of developing the theory of left modules (\cref{sec:monoidal_LMod}) for non-trivial bases. 
\medskip \\Conditional on monoidal parameterized straightening Horev--Klang--Zou gave formulas for equivariant factorization homology of Thom spectra over \(\mathbb{S}\) \cite[Theorem 7.1.1]{hahn2024equivariantnonabelianpoincareduality}.
In this section -- conditional on the monoidal parametrized straightening  -- we generalize this to equivariant \(R\)-module Thom spectra.
In particular, this specializes to give formulas for relative versions of Real Topological Hochschild Homology \(\THR({-}/R)\) of \(R\)-module Thom spectra.

\begin{assumption}\label{folklore}
    Let $\O^{\otimes} \in \Op_{G, \infty}$ and $X$ be an $\O$-monoidal $G$-space. There is an \(\O\)-monoidal equivalence \(\myuline{\PSh}_G^{\O}(X)^{\otimes} \simeq (\ul{\Spaces}^G_{/X})^{\otimes}\).\footnote{The left side is equipped with the Day convolution monoidal structure (\cref{theorem: omnibus day convolution}) and the right side is equipped with the slice monoidal structure (\cref{def: slice pattern}).}
\end{assumption}
\begin{remark}
    \cref{folklore}, which is an equivariant version of \cite[Corollary 4.9]{ramzi2022monoidalgrothendieckconstructioninftycategories}, was stated in \cite[Theorem A.6.1]{hahn2024equivariantnonabelianpoincareduality} as folklore.
    In the Horev--Klang--Zou approach to equivariant Thom spectra, this result is essential to constructing the equivariant Thom spectrum functor as a \(G\)-symmetric monoidal functor. In the approach we take in this paper, a microcosmic version (\cref{theorem: microcosmic straightening-unstraightening}) is sufficient for most applications.
    However, generalizing \cite[Theorem 7.1.1]{hahn2024equivariantnonabelianpoincareduality} requires this stronger form of straightening.
    We include this section, conditional on \cref{folklore} for two reasons:
    \begin{itemize}
        \item We hope this encourages a written proof of \cref{folklore} appear in the literature.
        \item The results of \cite{hahn2024equivariantnonabelianpoincareduality} are used extensively in the literature. For those willing to accept \cref{folklore} without proof, this section provides additional computationally useful results.
    \end{itemize}
\end{remark}

\noindent Let $V$ be a $G$-representation and $X \in \Sp_G$. We say that $X$ is \tb{$V$-connective} if $\pi_k(X^H) = 0$ for all $H \leq G$ and $k < \dim \left(V^H \right)$. We will use this terminology to state the following results.

\begin{theorem}[Conditional on \cref{folklore}]
    Let \(R\) be an \(\E_{\infty}^G\)-ring and $V$,  $W$ be $G$-representations.
    Let 
    \(
    f\colon \Omega^{V+W}X\to \ul{\Pic}_G(R)
    \)
    be an \(\E_{V+W}\)-map.
    Suppose that \(X\) is \((V+W)\)-connective. 
    Let \(M\) be a \(G\)-manifold the same dimension as \(V\). 
    Suppose that \(M\times W\) embeds equivariantly into \(V\times W\), and that there is an equivariant embedding \(D(V)\hookrightarrow M\).
    Then, 
    \[
    \int_{M\times W}\Th_G(f)\simeq \Th_G(f) \otimes \Sigma^\infty_+\myuline{\Map}_*\left(M^+-D(V),\Omega^W X \right)
    \]
    is an equivalence of \(R\)-modules.
\end{theorem}
\begin{proof}
    By \cref{theorem: LMod distributive}, \(\myuline{\LMod}^G_R\) is a distributive \(G\)-symmetric monoidal $G$-$\infty$-category.
    Conditional on \cref{folklore} the proof of \cite[Theorem 7.1.1]{hahn2024equivariantnonabelianpoincareduality} goes through mutatis mundatis replacing \(\myuline{\Sp}_G\) with \(\myuline{\LMod}^G_R\).
\end{proof}

\noindent The main corollary of interest occurs when \(M\) is a representation sphere \(S^V\).

\begin{corollary}[Conditional on \cref{folklore}]
    Let \(R\) be an \(\E_{\infty}^G\)-ring and $V$ be a representation of \(G\).
    Let 
    \(
    f\colon \Omega^{V+1}X\to \ul{\Pic}_G(R)
    \)
    be an \(\E_{V+1}\)-map.
    Suppose that \(X\) is \((V+1)\)-connective. Then, there is an equivalence 
    \[
    \int_{S^V\times \mathbb{R}}\Th_G(f)\simeq \Th_G(f)\otimes\Sigma^\infty_+\Omega X
    \]
    of \(R\)-modules.
\end{corollary}

\noindent Specializing further to the \(C_2\)-sign sphere \(S^\sigma\) we obtain formulas for relative Real Topological Hochschild Homology since $\THR \simeq \int_{S^{\sigma}}$, see \cite[Remark 7.1.2]{horev2019genuineequivariantfactorizationhomology}.
\begin{corollary}[Conditional on \cref{folklore}]
    Let \(R\) be an \(\E_{\infty}^{C_2}\)-ring and consider an $\E_{\sigma+1}$-map
    \(
    f\colon \Omega^{\sigma+1}X\to \ul{\Pic}_{C_2}(R)
    \).
    Suppose that \(X\) is \((\sigma+1)\)-connective. Then, there is an equivalence 
    \[
    \THR(\Th_{C_2}(f)/R)\simeq \Th_{C_2}(f)\otimes\Sigma^\infty_+\Omega X
    \]
    of \(R\)-modules.
\end{corollary}

\subsection{Equivariant Thom Spectra as \texorpdfstring{\(\E_V\)}{EV}-Quotients}\label{sec:EV_quotients}
Classically, the Hopkins--Mahowald theorem states that $\mathrm{H}\F_p$ can be constructed as a Thom spectrum over $\Omega^2 S^3$ \cite[Theorem 4.8]{MNNMayNilpotence}, \cite[Theorem 5.1]{antolinbarthel2019thom}.
Equivalently,
using the theory of $\E_n$-quotients or versal algebras as Thom spectra \cite[Section 4]{antolinbarthel2019thom},
this describes $\mathrm{H}\F_p$ as a free \(\E_2\)-ring with a null-homotopy of \(p\).

\medskip
\noindent
Equivariant versions of the Hopkins--Mahowald theorem were studied by Behrens--Wilson \cite{behrenswilson2018C2mahowald}, Hahn--Wilson \cite{hahnwilson2020eilenberg}, Devalapurkar \cite{devalapurkar2024higherchromaticthom}, and Levy \cite{levy2022Eilenberg}, the latter of which ultimately unifies and extends all known results. For any faithful \(2\)-dimensional representation \(V\) of a \(p\)-group \(G\), Levy shows that \(\mathrm{H}\myuline{\mathbb{F}}_p\) can be constructed as the equivariant Thom spectrum of a \(V\)-fold loop map 
\(\Omega^V\Sigma^V S^1\to \BGL_1(\mathbb{S}_{(p)})\).

\medskip
\noindent
In this section we record, a fact known to experts, that Antolín-Camarena--Barthel's treatment of \(\E_n\)-quotients in terms of Thom spectra generalizes to the equivariant setting. Our hard work on the equivariant version of Antolín-Camarena--Barthel's universal property (\cref{part:foundations}) allows us to make this rigorous.

\begin{definition}
    Let \(V, W\) be $G$-representations and $R \in \Alg_{\E_{V} \otimes \E_1}(\ul{\Sp}_G)$.
    Consider an element \(\chi\in \pi^G_W(R)\).
    Write \(\tilde{\chi}\colon \Sigma^W R\to R\) for the adjoint map of \(R\)-modules.
    We say that \(A\in \Alg_{\E_V}(\ul{\LMod}_R)\) is of \tb{characteristic} \(\chi\) if the composite
    \(
    \eta \circ \tilde{\chi} \colon \Sigma^W R\to R \to A
    \)
    is null-homotopic where \(\eta\colon R\to A\) is the unit.
\end{definition}

\begin{definition}[\(\E_V\)-quotient] \label{def: EV quotient}
    Let \(V, W\) be $G$-representations and $R \in \Alg_{\E_{V}}(\ul{\Sp}_G)$.
    Consider an element \(\alpha\in \pi^G_W(R)\).
    The \tb{\(\E_V\)-quotient of \(R\) by \(\alpha\)}, denoted \tb{\(\smash{R \sslash_{\E_V}\alpha}\)}, is defined as the following pushout in \(\E_V\)-algebras
    \[
    \begin{tikzcd}
        \mathrm{free}^{\E_V}(\Sigma^{W}R) \ar[r,"\alpha"]\ar[d,"0", swap] & R \ar[d]\\
        R \ar[r] & R \sslash_{\E_V}\alpha. \arrow[ul, phantom, very near start, "\ulcorner"]
    \end{tikzcd}
    \]
\end{definition}

\noindent The first hint that \(\E_V\)-quotients are related to equivariant Thom spectra is that they have Thom isomorphisms.
\begin{lemma}
    Let $V$ be a $G$-representation and \(\eta\colon R\to A\) be a map of \(\E_{V}\otimes \E_1\)-ring spectra.
    Let \(\chi\in \pi^G_V(R)\).
    Suppose that \(A\) has characteristic \(\chi\), then there is an equivalence
    \[
    A\otimes_R \mathrm{free}^{\E_V}(\Sigma^{W +1}R)\simeq A\otimes_R \smash{(R \sslash_{\E_V} \chi)}
    \]
    of \(\E_V\)-\(A\)-algebras.
\end{lemma}
\begin{proof}
    The proof is analogous to \cite[Lemma 4.5]{antolinbarthel2019thom}
\end{proof}

\begin{definition} \label{def: characteristic of f}
    Let \(V, W\) be $G$-representations and $R \in \Alg_{\E_{V} \otimes \E_2}(\ul{\Sp}_G)$. Consider a map \(f\colon \Sigma S^W\to \operatorname{BGL}_1(R)\) of pointed \(G\)-spaces and let \(\tilde{f}\colon \Sigma^W R\to R\) the induced map of \(R\)-modules.
    The \tb{characteristic of \(f\)} is defined to be
    \[
    \tb{\chi(f)} = 
    \begin{cases}
        \tilde{f}-1 & \text{if }W=0, \\
        \tilde{f} & \text{if }W \neq 0. \\
    \end{cases}
    \]
\end{definition}

\begin{remark}
    The discrepancy in the definition of the characteristic of $f$ comes from understanding the equivariant analogue of \cite[Proposition 4.9]{antolinbarthel2019thom}. Working out the argument, we find that the characteristic should be defined as $\widetilde{f} - a_{-W}$ where 
    \[ [a_{-W}] = \left[S^W \xrightarrow{a_{-W}} S^0 \to \Omega^{\infty}R \right] \in \pi_W^G(R) \] 
    is induced by the pre-Euler class $a_{-W} \colon S^W \to S^0$. On the other hand, 
    \[ \pi_W^{G}(S^0) \cong [S^W, S^0]_*^G \cong \begin{cases}
        \{\id_{S^0}, * \} \quad & W = 0,
        \\ * & W \neq 0
    \end{cases} \]
    by a connectivity argument.
\end{remark}

\begin{lemma}
Let \(V, W\) be $G$-representations and $R \in \Alg_{\E_{V} \otimes \E_2}(\ul{\Sp}_G)$. Consider a map \(f\colon S^{W+1}\to \operatorname{BGL}_1(R)\) and let \(\overline{f}\colon \Omega^V\Sigma^V S^{W+1}\to \operatorname{BGL}_1(R)\) denote the associated $V$-fold loop map \cite[Theorem 3.15]{juran2025genuineequivariantrecognitionprinciple}. Then,
\[ \Th_G \left(\overline{f}\colon \Omega^V\Sigma^V S^{W+1}\to \operatorname{BGL}_1(R) \right) \simeq R \sslash_{\E_V} \chi(f) \]
in $\Alg_{\E_V}(\LMod_R)$.
\end{lemma}
\begin{proof}
    The proof is analogous to \cite[Theorem 4.10]{antolinbarthel2019thom}.
\end{proof}

\begin{theorem}[{\cite[Theorem A]{levy2022Eilenberg}}]
    Fix a \(p\)-group \(G\). 
    Let \(V\) be a faithful \(2\)-dimensional representation of \(G\).
    Then, there is an equivalence of \(\E_V\)-algebras between \(\mathrm{H}\myuline{\mathbb{F}}_p\) and \(\mathbb{S}\sslash_{\E_V}p\)
\end{theorem}

\begin{proof}
    \change{In \cite[Theorem A]{levy2022Eilenberg} Levy constructs \(\mathrm{H}\myuline{\mathbb{F}}_p\) as the equivariant Thom spectrum of a \(V\)-fold loop map \(\Omega^V\Sigma^V S^1\to \BGL_1(\mathbb{S}_{(p)})\).}
    Our discussion of $\E_V$-algebras verifies the $\S \sslash_{\E_V} p$ description.
\end{proof}

\noindent In particular, any $\E_V$-algebra of characteristic $p$ is an $\mathrm{H}\ul{\F}_p$-algebra.

\subsection{Nilpotence for \texorpdfstring{\(\mathbb{E}_\sigma\)}{Eσ}-Algebras}\label{sec:nilpotence}
Non-equivariantly, the Hopkins--Mahowald Theorem (\cref{sec:EV_quotients}) is intimately connected to nilpotence phenomena; for this perspective we recommend \cite{MNNMayNilpotence,hahn2017nilpotenceenalgebras,devalapurkar2024higherchromaticthom}.
Thus, it is natural to ask if the same is true in the equivariant setting.
Obvious analogues of the various famous nilpotence theorems are well-known to fail in the Real equivariant setting;
see \cite[Remark 3.20]{barthelGreenleesHausmann2020balmercompactlie}, and \cite[Remark 4.3]{carrickSmashingLocalizationsEquivariant2022} for counterexamples.
\medskip \\To demonstrate the utility of \(\E_V\)-ring structures, we show that an \(\MUR\) detects nilpotence for \(\E_\sigma\)-rings; and that \(\uHZ\) detects nilpotence for \(\E_\sigma \otimes \E_\infty\)-rings.

\begin{remark}
    These results are surely known to the experts.
    The authors deduced these from a remark in an expository talk by Hahn \cite[\href{https://youtu.be/yqb0-7jZFmY?t=2760}{t=2760}]{Hahn2020Nishida} in which he remarks that an \(\E_{\infty}^{C_2}\)-ring is trivial if and only if its underlying non-equivariant ring is trivial. 
    Our only contribution is the easy observation that \(\E_\sigma\) is sufficient.
\end{remark}

\begin{theorem}
    Let \(R \in \Alg_{\E_{\sigma}}(\ul{\Sp}_{C_2})\).
    \begin{enumerate}[(i)]
        \item Then \(R\otimes\MUR\simeq 0\) if and only if \(R\simeq 0\).
        \item Then \(R\otimes K(n)_{\mathbb{R}}\simeq 0\) for all \(n \in \N \) if and only if \(R\simeq 0\).
        \item Suppose further that \(R \in \Alg_{\E_{\sigma} \otimes \E_{\infty}}(\ul{\Sp}_{C_2}) \).
    Then \(R\otimes \uHZ \simeq 0\) if and only if \(R\simeq 0\).
    \end{enumerate}
    
\end{theorem}

\begin{proof}
    In each case, the if direction is clear.
    We will only show the converse directions.
    Each converse direction uses the same key ingredient: since \(R\) is (at least) an \(\E_\sigma\)-ring, \(R\) is a module over \(N^{C_2}_e\mathrm{Res}^{C_2}_eR\) \cite[Section 7.1]{horev2019genuineequivariantfactorizationhomology};\footnote{Let us suggestively write that the action map is induced by the embedding $\left( \bigsqcup_{C_2} \R^1 \right) \sqcup \R^{\sigma} \hookrightarrow \R^{\sigma}$ in $\E_{\sigma}$, see \cite[Section 7.1]{horev2019genuineequivariantfactorizationhomology}.} so to show that \(R\simeq 0\) it suffices to show that \(\mathrm{Res}^{C_2}_eR\simeq 0\). To do so, we will use that $\Res_e^{C_2}R$ is an $\E_1$-ring in (i) and (ii) to apply the ring spectrum version of the nilpotence theorem.
    
    \begin{enumerate}[(i)]
        \item Suppose \(R\otimes \MUR\simeq 0\).
    Hence, \(\mathrm{Res}^{C_2}_eR\otimes \MU\simeq 0\). 
    By Devinatz--Hopkins--Smith Nilpotence \cite{DHS88} it follows that \(\mathrm{Res}^{C_2}_eR\simeq 0\). 
    Since \(R\) is an \(\E_\sigma\)-ring, it follows that \(R\simeq 0\).
    \item Suppose $R \otimes K(n)_{\R} \simeq 0$ for all $n \in \N$. Then, $\Res_e^{C_2} R \otimes K(n) \simeq 0$ for all $n \in \N$. By Devinatz--Hopkins--Smith Nilpotence \cite{DHS88} it follows that $\Res_e^{C_2}R \simeq 0$. Since \(R\) is an \(\E_\sigma\)-ring, it follows that \(R\simeq 0\).
    \item Suppose \(R\otimes \uHZ\simeq 0\).
    Therefore, \(\mathrm{Res}^{C_2}_eR\otimes \HZ\simeq 0\).
    Since \(\mathrm{Res}^{C_2}_eR\) is an \(\E_\infty\)-ring, by May Nilpotence \cite{brunerMaySteinberger1986Hinfty,MNNMayNilpotence} it follows that \(\mathrm{Res}^{C_2}_eR\simeq 0\). 
    Since \(R\) is in particular an \(\E_\sigma\)-ring, it follows that \(R\simeq 0\).\qedhere
    \end{enumerate}
\end{proof}

\begin{remark}
    This result works in greater generality than just for $\E_{\sigma}$-rings by generalizing the module result that we used in the proof, see e.g. \cite[Corollary 2.2, Remark 2.4.1]{levy2022Eilenberg}.
    The exact same proof e.g.~also works for \(\MU^{(\!(G)\!)}= \smash{N^G_{C_2}\MUR}\) instead of \(\MUR\) under the assumption that \(R\) is a module over \(N^G_{e}\mathrm{Res}^G_eR\).
\end{remark}

\newpage
\part*{Appendix}
\begin{description}
    \item[\cref{sec: more equivariant higher algebra}] includes two further checks in equivariant higher algebra.

    In \cref{A:monoidality_of_res_and_norm} we spell out that restrictions and norms of $G$-symmetric monoidal $G$-$\infty$-categories are symmetric monoidal (ordinary) functors.

    % In \cref{B:UCT} we record a proof we learnt from Meier of a version of the Universal Coefficient Theorem for constant Mackey functors. 

    In \cref{subsection: Real J} we explain an $\E_{\infty}^{C_2}$-structured version of the Real $J$-homomorphism $J_{\R}$ based on forthcoming work of Brink--Lenz about equivariant $J$-homomorphisms.
    \item[\cref{sec: more homological algebra}] includes two results on homological algebra.

    In \cref{B:UCT} we record a proof of an equivariant cohomological universal coefficient theorem and relax certain finiteness assumptions from Lewis--Mandell.

    In \cref{subsection: even cohomology} we carefully discuss that for bounded below $X \in \Sp$ such that $H_*(X;\Z)$ is free, of finite type and even, one obtains that $E^*(X)$ is even for every $E \in \Sp$.
\end{description}

\appendix

\section{More Equivariant Higher Algebra}
\label{sec: more equivariant higher algebra}

\subsection{Monoidality of Restrictions and Norms}\label{A:monoidality_of_res_and_norm}
We show that the structure maps of $G$-symmetric monoidal $G$-$\infty$-categories are compatible with the levelwise symmetric monoidal structures.
\begin{lemma} \label{lemma: Res N sym mon}
    Let $\C^{\otimes} \in \Mon_{\Span(\F_G)}(\Cat_{\infty})$. Then, the restriction and norm functors of $\C^{\otimes}$ are symmetric monoidal.
\end{lemma}

\begin{proof}
    Let $H \leq K \leq G$. Recall first that we obtain symmetric monoidal structures on each level by pulling back:
    \begin{center}
        \begin{tikzcd}
            \C_H^{\otimes} \arrow[r] \arrow[d] \arrow[dr, phantom, very near start, "\lrcorner"] & \C^{\otimes} \arrow[d]
            \\ \Span(\F) \arrow[r, "G/H", swap] & \Span(\F_G)
        \end{tikzcd}
    \end{center}
    In other words, it is the image of $\C^{\otimes}$ under the functor 
    \[ \Alg_{\Span(\F_G)}(\Cat_{\infty}) \to \Alg_{\Span(\F)}(\Cat_{\infty}) \simeq \mathbf{CMon}(\Cat_{\infty}) \]
    by \cref{prop: Alg(...) = Mon(...)}. So maps $N_H^K\colon  \C_H^\otimes \to \C_K^{\otimes}$ and $\Res_H^K\colon  \C_K^{\otimes} \to \C_H^{\otimes}$ are induced by the following natural transformations. 
    \begin{center}
    \begin{tikzcd}
        \Span(\F) \arrow[r, "G/H"{name=a}] \arrow[d, equal]
        & \Span(\F_G) \arrow[d, equal]
        & \Span(\F) \arrow[r,"G/K"{name=A}] \arrow[d,equal]
        & \Span(\F_G) \arrow[d, equal]
        \\
        \Span(\F) \arrow[r, "G/K"'{name=b}] 
        & \Span(\F_G) 
        & \Span(\F) \arrow[r,"G/H"'{name=B}] 
        & \Span(\F_G) 
        \arrow[Rightarrow, shorten <=4pt, shorten >=4pt, from=a, to=b]
        \arrow[Rightarrow, shorten <=4pt, shorten >=4pt, from=A, to=B]
    \end{tikzcd}
    \end{center}
    These are induced by $G/H \to G/K$ resp. $G/K \leftarrow G/H$ in $\Span(\F_G)$.
\end{proof}

\subsection{Real \texorpdfstring{\(J\)}{J}-Homomorphism}
\label{subsection: Real J}There is a Real $J$-homomorphism $J_{\R}$ such that $\MU_{\R} = \Th_{C_2}(J_{\R})$. This map also appears in the definition of $\MWR$ (\cref{construction: MWR}). To the best of the authors' knowledge, there is no construction of an $\E_{\infty}^{C_2}$-Real $J$-homomorphism in the literature. Here, we will explain how to obtain such a map using forthcoming work of Emma Brink and Tobias Lenz. We are grateful to Emma for suggesting this approach.

\begin{fact}[Brink--Lenz, forthcoming] \label{fact: brink lenz}
    There is an $\E_{\infty}^{C_2}$-map $J_{C_2}\colon  \BOP_{C_2} \to \ul{\Pic}_{C_2}(\ul{\Sp}_{C_2})$ whose $C_2$-Thom spectrum is $\MOP_{C_2}$.
\end{fact}

\noindent Their result works in much greater generality but this is all we will need.
\medskip \\So, it suffices to construct an $\E_{\infty}^{C_2}$-map $\BUR \to \BOP_{C_2}$. Postcomposing by $J_{C_2}$ yields $J_{\R}$.

\begin{construction} \label{construction: BUPR to BOP}
    There exist $C_2$-orthogonal space models $\widetilde{\BUP}_{\R}$ and of $\widetilde{\BOP}_{C_2}$ by Schwede \cite[Example VI.7.1]{schwede2014global}\cite[Example 2.4.1]{schwede_2018}, which for an inner product space $V$ are given by
    \[ \widetilde{\BUP}_{\R}(V) = \Gr^{\mathbb C}(V_{\mathbb C}^2) \quad \text{and} \quad \widetilde{\BOP}_{C_2}(V) = \Gr(V^2), \]
    where $(-)_{\mathbb C}$ denotes the complexification functor. All the structure maps are defined essentially in the same way. Moreover, the $C_2$-action on $\widetilde{\BUP}_{\R}(V)$ is by complex conjugation. We define a map of orthogonal spaces $\widetilde{\BUP}_{\R} \to \widetilde{\BOP}_{C_2}$ given by forgetting the complex structure. This is compatible with complex conjugation and since all structure maps are defined in the essentially same way, this defines a map of ultracommutative orthogonal $C_2$-spaces. In particular, it induces an $\E_{\infty}^{C_2}$-map $\BUP_{\R} \to \BOP_{C_2}$ by \cite{lenz2025normsequivarianthomotopytheory}. Precomposing by $\BUR \to \BUP_{\R}$ yields a map $\BUR \to \BOP_{C_2}$.
\end{construction}

\begin{definition}
    The \tb{Real $J$-homomorphism} is the composition 
    \[ \tb{J_{\R}} \colon \BU_{\R} \to \BOP_{C_2} \xrightarrow{J_{C_2}} \ul{\Pic}_{C_2}(\ul{\Sp}_{C_2}) \] 
    combining \cref{fact: brink lenz} and \cref{construction: BUPR to BOP}.
\end{definition}

\section{More Homological Algebra}
\label{sec: more homological algebra}

\subsection{Universal Coefficient Theorems for Constant Mackey Functors}\label{B:UCT}
\tikzcdset{
  cells={font=\everymath\expandafter{\the\everymath\displaystyle}},
}
Lewis--Mandell prove general equivariant versions of the Universal Coefficient Theorem in the form of a Universal Coefficient spectral sequence \cite{equivUCT} for finite $G$-spectra.
The main reason the equivariant version of the Universal Coefficient theorem is more difficult than its non-equivariant counterpart is that projective resolutions of Mackey functors are more complicated than that of abelian groups. In particular, a Mackey functor in general does not admit a two-step free resolution.
\medskip \\Here we record a proof we learnt from Meier of a version of \cite[Theorem 1.5]{equivUCT} that relaxes the finiteness assumption.
The main idea is that projective resolutions of constant Mackey functors can be obtained by resolving the underlying abelian group. We denote by $\myuline{\map}_G$ the mapping $G$-spectrum.
\begin{lemma}\label{UCT}
    Let \(\myuline{M}\) be a constant Mackey functor. 
    Let \(X\) be a \(G\)-spectrum. Suppose one of the following conditions:
    \begin{itemize}
        \item[(i)] The $\myuline{\Z}$-homology of $X$ is free and of finite type.
        \item[(ii)] The abelian group $M$ is finitely generated.
    \end{itemize}
    It follows that for every $G$-representation \(V\) there is a short exact sequence %\vspace{-1em}
    \begin{center}
        \begin{tikzcd}
            \pi_{-V}^G \myuline{\map}_G(X, \uHZ) \otimes_{\Z} M \arrow[r] & \pi_{-V}^G \myuline{\map}_G(X, \mathrm{H} \myuline{M}) \arrow[r] & \operatorname{Tor}_1^{\Z} \left(M, \pi_{-V-1}^G \myuline{\map}_G(X, \uHZ) \right)
        \end{tikzcd}
    \end{center}
    in $\Ab$.
\end{lemma}
\begin{remark}
    In cohomological notation the above short exact sequence is
    \begin{center}
        \begin{tikzcd}
            0 \arrow[r] & H^V(X; \myuline{\Z}) \otimes_{\Z} M \arrow[r] & H^V(X;\myuline{M}) \arrow[r] & \operatorname{Tor}_1^{\Z} \left(M, H^{V+1}(X; \myuline{\Z}) \right) \arrow[r] & 0.
        \end{tikzcd}
    \end{center}
\end{remark}

\begin{proof}[Proof of \cref{UCT}]
    The abelian group \(M\) admits a free resolution of length two
    \begin{center}
        \begin{tikzcd}
            \bigoplus_J \Z \arrow[r, "f"] & \bigoplus_I \Z \arrow[r] & M.
        \end{tikzcd}
    \end{center}
    Taking constant Mackey functors, this gives a two-step resolution of \(\myuline{M}\).
    Taking equivariant Eilenberg--MacLane spectra, applying \(\myuline{\map}_G(X,{-})\), and using the finiteness assumption gives a cofiber sequence of \(G\)-spectra 
    % Applying \(F(X,H{-})\) gives a cofiber sequence of \(G\)-spectra
    % \[
    % F(X,\bigoplus_{J}H\uZ)\to F(X\bigoplus_{I},H\uZ)\to F(X,H\myuline{M}).
    % \]
    % By the finite type assumption we get a cofiber sequence
    \begin{center}
        \begin{tikzcd}
            \bigoplus_J \myuline{\map}_G(X, \uHZ) \arrow[r] & \bigoplus_I \myuline{\map}_G(X, \uHZ) \arrow[r] & \myuline{\map}_G(X, \mathrm{H}\ul{M}).
        \end{tikzcd}
    \end{center}
    Now we apply \(\pi^G_{-V}\) to obtain a long exact sequence, and we consider the subsequent short exact sequence:
    \begin{center}
        \begin{tikzcd}
            \coker\! \left( \pi_{-V}^G \myuline{\map}_G(X, \uHZ) \otimes_{\Z} f \right) \arrow[r] & \pi_{-V}^G \myuline{\map}_G(X, \mathrm{H}\myuline{M}) \arrow[r] & \ker \!\left(\pi_{-V-1}^G \myuline{\map}_G(X, \uHZ) \otimes_{\Z} f \right)\!.
        \end{tikzcd}
    \end{center} 
\noindent Note that
\[
\coker \left( \pi_{-V}^G \myuline{\map}_G(X, \uHZ) \otimes_{\Z} f \right) 
\cong \pi^G_{-V}\myuline{\map}_G(X,\uHZ)\otimes_{\Z}M.
\]
Moreover, by definition of $\operatorname{Tor}_1$ through projective resolutions, we have
\[
\ker \left(\pi_{-V-1}^G \myuline{\map}_G(X, \uHZ) \otimes_{\Z} f \right)
\cong \mathrm{Tor}^{\Z}_1 \left(M, \pi^G_{-V-1}\myuline{\map}_G(X,\uHZ) \right).
\]
This finishes the proof.
\end{proof}

\noindent For the homology version we can drop the finiteness conditions.
\begin{lemma}\label{homoUCT}
    Let \(\myuline{M}\) be a constant Mackey functor.
    Let \(X\) be \(G\)-spectrum.
    For every $G$-representation \(V\) there is a short exact sequence
    \begin{center}
    \begin{tikzcd}
        0 \arrow[r] & \pi_V^G(\uHZ \otimes X) \otimes_{\Z} M \arrow[r] &\pi_V^G(\mathrm{H}\myuline{M} \otimes X) \arrow[r]&      \mathrm{Tor}^{\Z}_1(M, \pi^G_{V-1}(\mathrm{H}\uZ\otimes X))\arrow[r] & 0
    \end{tikzcd}
    \end{center}
    in $\Ab$.
\end{lemma}

\begin{remark}
    In homological notation the above short exact sequence is
    \begin{center}
        \begin{tikzcd}
            0 \arrow[r] & H_V(X; \myuline{\Z}) \otimes_{\Z} M \arrow[r] & H_V(X;\myuline{M}) \arrow[r] & \operatorname{Tor}_1^{\Z} \left(M, H_{V-1}(X;\myuline{\Z}) \right) \arrow[r] & 0
        \end{tikzcd}
    \end{center}
\end{remark}
\begin{proof}[Proof of \cref{homoUCT}]
    The proof is the same as in \cref{UCT} but with \(\myuline{\map}_G(X,{-})\) replaced by \({-}\otimes X\).
    This is also the difference that allows us to drop the finitely generated condition.
\end{proof}

\subsection{Even Cohomology}
\label{subsection: even cohomology}
Let $X,E \in \Sp$. We will prove the evenness of $E^*(X)$ given sufficient finiteness and evenness conditions on $X$ and $E$. If we were to assume strong enough finiteness conditions, then we could argue directly via the Atiyah--Hirzebruch spectral sequence. In our more general setting, we have to be more careful and will work on the level of towers.

\begin{definition}
    A spectrum $X$ has \tb{finite type free even $\Z$-homology} if there is an equivalence $\mathrm{H} \Z \otimes X \simeq \mathrm{H} \Z \otimes \bigoplus_{i \in I} S^{2n_i}$ for some $n_i \in \N$ and an indexing set $I$.
\end{definition}

\begin{lemma} \label{lemma: map X HA even}
    Let $X$ have finite type free even $\Z$-homology and let $A \in \Ab$. Then, $\map_{\Sp}(X, \mathrm{H}A)$ is even.
\end{lemma}

\begin{proof}
    Suppose that $\mathrm{H} \Z \otimes X \simeq \mathrm{H} \Z \otimes \bigoplus_{i \in I} S^{2n_i}$ for some indexing set $I$. Then, 
    \[ \map_{\Sp}(X, \mathrm{H}A) \simeq \mathrm{H}A \otimes \bigoplus_{i \in I} S^{-2n_i}, \] 
    e.g.~by \cref{lem:hills_freeness}. This is even. 
    %An adjunction argument shows that $\map_{\Sp}(X, \mathrm{H}\Z)$ is even using that $X$ has free even $\Z$-homology. The conditions on $X$ further allow us to use the purely cohomological universal coefficient theorem: For $k \in \N$ we have a short exact sequence
    %\begin{center}
    %    \begin{tikzcd}
    %        0 \arrow[r] & H^{2k+1}(X;\Z) \otimes_{\Z} A \arrow[r] & H^{2k+1}(X;A) \arrow[r] & \operatorname{Tor}_1^{\Z}\left(A, H^{2k+2}(X;\Z) \right) \arrow[r] & 0.
    %    \end{tikzcd}
    %\end{center}
    %The conditions on $X$ imply that the first and last term vanish, so $H^{2k+1}(X;A) = 0$.
\end{proof}

\begin{proposition} \label{prop: map X E even}
    Let $E \in \Sp$ be even and let $X \in \Sp$ have finite type free even $\Z$-homology and be bounded below. Then, $\map_{\Sp}(X, E)$ is even.
\end{proposition}

\begin{proof}
    We apply the Postnikov tower to $E$ to obtain a tower together with the fibers:
    \begin{center}
        \begin{tikzcd}
            \cdots \arrow[r] & \map_{\Sp}(X, \tau_{\leq 2}E) \arrow[r] & \map_{\Sp}(X, \tau_{\leq 1}E) \arrow[r]  & \map_{\Sp}(X, \tau_{\leq 0}E) \arrow[r] & \cdots 
            \\ & \map_{\Sp}(X, \Sigma^2 \pi_2 E) \arrow[u] & \map_{\Sp}(X, \Sigma \pi_1 E) \arrow[u] & \map_{\Sp}(X, \pi_0 E) \arrow[u] 
        \end{tikzcd}
    \end{center}
    Let us first note that $\colim_{s \to -\infty} \map_{\Sp}(X, \tau_{\leq s} E) \simeq 0$. Since $X$ is bounded below, there exists some $n \in \Z$ such that $X \simeq \tau_{\geq n}X$. Therefore, the orthogonality conditions of $t$-structures imply $\map_{\Sp}(X, \tau_{\leq s}E) \simeq 0$ for sufficiently small $s$.
    \medskip \\Consider the fiber sequences
    \begin{center}
        \begin{tikzcd}
            \map_{\Sp}(X, \Sigma^s \pi_s E) \arrow[r] & \map_{\Sp}(X, \tau_{\leq s}E) \arrow[r] & \map_{\Sp}(X, \tau_{\leq s-1}E).
        \end{tikzcd}
    \end{center}
    The fiber is even for all $s$: Indeed, they vanish for odd $s$ be evenness of $E$ and are even for even $s$ by the assumption on $X$, see \cref{lemma: map X HA even}. This implies that the map
    \[ \map_{\Sp}(X, \tau_{\leq s}E) \to \map_{\Sp}(X, \tau_{\leq s-1}E) \]
    is injective on odd homotopy groups. On the other hand, we have seen above that the sequential colimit is $0$. Thus, these terms must all already have trivial odd homotopy groups. So every object of the tower is even.
    \medskip \\It remains to pass to the limit. There is a Milnor $\lim^1$-sequence
    \begin{center}
        \begin{tikzcd}
            0 \arrow[r] & \displaystyle{\lim_s}^1 (\tau_{\leq s}E)^{2k}(X) \arrow[r] & E^{2k+1}(X) \arrow[r] & \displaystyle\lim_s (\tau_{\leq s}E)^{2k+1}(X) \arrow[r] & 0
        \end{tikzcd}
    \end{center}
    for all $k \in \Z$. We have shown above that the right term vanishes, so it remains to discuss the $\lim^1$-term. It vanishes if the structure maps $(\tau_{\leq s+1}E)^{2k}(X) \to (\tau_{\leq s}E)^{2k}(X)$ are surjective for all $s \in \Z$. On the other hand, this follows from the long exact sequence associated to the fiber sequence
    \begin{center}
        \begin{tikzcd}
            \map_{\Sp}(X, \Sigma^{s+1}\pi_{s+1}E) \arrow[r] & \map_{\Sp}(X, \tau_{\leq s+1}E) \arrow[r] & \map_{\Sp}(X, \tau_{\leq s}X).
        \end{tikzcd}
    \end{center}
    We win.
\end{proof}

\newpage 
\section*{Index of Notation}\label{ION}
\markboth{Index of Notation}{Index of Notation}
\addcontentsline{toc}{section}{Index of Notation} For many of the notions in the following there are versions enhanced with an underline $\ul{-}$ or a tensor $(-)^{\otimes}$ which denote a parametrized version or a monoidal version. Let $H \leq K \leq G$ be finite groups.
\begin{description}[leftmargin=!, labelwidth=\widthof{\bfseries The Largest Named Item}]
    \item[$\Alg_{\P/\Q}(\C)$] the functor $\infty$-category of maps $\P^{\otimes} \to \C^{\otimes}$ over $\Q^{\otimes}$, see \cref{sec:param_operads}
    \item[\({\Alg}_{\O}(\C)\)] the \(\infty\)-category of \(\O\)-algebras in \(\C\), also given by $\Alg_{\O/\Span(\F_G)}(\C)$, see \cref{sec:param_operads}
%    \item[\(\myuline{\Alg}_{\O}(\C)\)] the \(G\)-\(\infty\)-category of \(\O\)-algebras in \(\C\)
    \item[$\ul{\Alg}_{\O}^{\gp}(\ul{\Sc}_G)$] the $G$-$\infty$-category of grouplike $\O$-algebras, see \cref{def: grouplike}
    %\item[\({\Alg}^R_{\O}(\C)\)] the \(\infty\)-category of \(\O\)-algebras in \(\myuline{\LMod}_R(\C)\), i.e.,~\(\O\)-\(R\)-algebras in \(\C\)
    %\item[\(\myuline{\Alg}^R_{\O}(\C)\)] the \(G\)-\(\infty\)-category of \(\O\)-algebras in \(\myuline{\LMod}_R(\C)\), i.e.,~\(\O\)-\(R\)-algebras in \(\C\)
    \item[\(\AlgPatt\)] the \(\infty\)-category of algebraic patterns, see \cref{sec:param_operads}
    \item[$\Ar$] the arrow $\infty$-category given by $\Fun([1], -)$
    \item[$\mathrm{B}^V$] the $V$-th delooping functor, see \cref{thm: recognition}
%    \item[$\BP$] Brown--Peterson theory
%    \item[$\BPR$] Real Brown--Peterson theory
    \item[$C_2$] the cyclic group on two elements
    \item[\(\myuline{\C}_{\myuline{d}}\)] parameterized fiber of a \(G\)-\(\infty\)-category \(\myuline{\C}\) over \(d\), see \cref{def: parameterized fiber}
    \item[\(\Cat_\infty\)] the \(\infty\)-category of \(\infty\)-categories
    \item[\(\Cat_{G,\infty}\)] the \(\infty\)-category of \(G\)-\(\infty\)-categories
    \item[$\coCart{\C}$] the subcategory of $\Cat_{\infty/\C}$ spanned by the coCartesian fibrations and maps preserving coCartesian edges
    \item[$\Coind_H^K$] the right adjoint of the restriction functor of $G$-$\infty$-categories, e.g.~see \cref{construction: Coind}
%    \item[\(\myuline{\Cat}_{G,\infty}\)] the \(G\)-\(\infty\)-category of \(G\)-\(\infty\)-categories
    \item[$(-)^{\core}$] the maximal subgroupoid of an $\infty$-category
    \item[$\CST$] the cohomological slice tower, see \cref{construction: cohomological slice tower}
    \item[$\D_{/A}$] the slice $\infty$-catgory in $\Cat_{\infty/\O}$ for $\O \in \Cat_{\infty}$ and $A \colon \O \to \D$, see \cref{def: slice pattern}
    \item[$\Ec_{\B}^{\C}$] the cotensor of $\C \in \Cat_{\infty}$ with $(\Ec \to \B) \in \Cat_{\infty/\B}$, see \cref{prop: omnibus tensor-cotensor}
    \item[$\E_V$] the equivariant little disk operad for a $G$-representation $V$
    \item[$\E_{\infty}^G$] the terminal $G$-$\infty$-operad
    \item[\({\F}_{G}\)] the \(\infty\)-category of finite \(G\)-sets
    \item[\(\myuline{\F}_{G,*}\)] the \(G\)-\(\infty\)-category of finite pointed \(G\)-sets
    \item[\(\Fbrs(\O)\)] the \(\infty\)-category of fibrous \(\O\)-patterns for  an algebraic pattern \(\O\)
    \item[$\free^{\E_V}$] the free $\E_V$-algebra on a $G$-spectrum
    \item[$\ul{\Fun}^{\O}(\C, \D)^{\otimes}$] the functor $G$-$\infty$-category with Day convolution, see \cref{theorem: omnibus day convolution}
    \item[$\GL_1$] the right adjoint to the inclusion of grouplike algebras, see \cref{def: GL1}
    \item[$\gl_1$] a delooping of $\GL_1$, see \cref{remark: gl1}
    \item[$\gr_s$] the $s$-th associated graded piece of a tower
%    \item[$G$] a finite group
    \item[$H^V$] the $V$-graded $G$-equivariant cohomology groups, given by $\pi_{-V} \ul{\map}_G$
    \item[$\Ind_R^A$] the left adjoint to the restriction functor $\ul{\LMod}_A \to \ul{\LMod}_R$, given a map $R \to A$
    \item[$\Ind_H^K$] the left adjoint of the restriction functor of various $G$-$\infty$-categories
    \item[$\Ind_{\O}^{\Q}$] the left adjoint of restriction of algebra structures, see \cref{construction: Ind operads}
    \item[$\Infl_G$] the inflation functor from $\infty$-operads to $G$-$\infty$-operads, see \cref{construction: inflated operads}
    \item[$J_{\R}$] the Real $J$-homomorphism, see \cref{subsection: Real J}
    \item[\({\LMod}^G_R(\C)\)] the \(\infty\)-category of left modules over \(R\), see \cref{construction: LMod}
%    \item[\(\myuline{\LMod}^G_R(\C)\)] the \(G\)-\(\infty\)-category of left modules over \(R\)
    \item[$\mathrm{M}$] Thom spectrum functor, see \cref{def: parametrized thom spectrum}
    \item[\(\myuline{\Map}_{\ul{\C}}({-},{-})\)] the mapping \(G\)-space in a \(G\)-\(\infty\)-category \(\ul{\C}\)
    \item[\(\myuline{\map}_{\ul{\C}}({-},{-})\)] the mapping \(G\)-spectrum in a \(G\)-stable \(G\)-\(\infty\)-category \(\ul{\C}\)
%    \item[\(\myuline{\Map}_{G}({-},{-})\)] the mapping \(G\)-space in \(\myuline{\Sp}_G\) or in $\ul{\Sc}_G$ or $\ul{\Sc}_*^G$
%    \item[\(\myuline{\map}_{G}({-},{-})\)] the mapping \(G\)-spectrum in \(\myuline{\Sp}_G\)
    \item[$\Mon_{\O}(\Cat_{\infty})$] the subcategory of $\Fbrs(\O^{\otimes})$ spanned by the coCartesian fibrations with morphisms the ones preserving coCartesian edges
    \item[$\Mon_{\ul{\O}}^{\NS}(\ul{\C})$] the $\infty$-category of $\ul{\O}$-monoids in  $\ul{\C}$ in the Nardin--Shah formalism
%    \item[$\MU$] the complex bordism spectrum
%    \item[$\MUR$] the Real complex bordism spectrum
    \item[$N_H^G$] the norm functor for $H \leq G$ from equivariant monoidal structures
    \item[\(\Op_{G,\infty}\)] the \(\infty\)-category of \(G\)-\(\infty\)-operads, i.e.,~\(\Fbrs(\Span(\F_G))\)
    \item[$\Or_A^{\P}(f)$] the space of $\P$-$A$-orientations, see \cref{def: orientations}
    \item[$\Op_{G, \infty}^{\NS}$] the $\infty$-category of $G$-$\infty$-operads in the Nardin--Shah formalism
    \item[\(\Orb_{G}\)] the orbit category for a finite group \(G\)
    % \item[\(\Pic(R)\)] Picard space of invertible objects in \(\LMod_R(\C)\) for an $\O$-monoidal \(\infty\)-category $\C$
    % \item[\(\Pic_G(R)\)] Picard space of invertible objects in \(\LMod^G_R(\C)\) for an $\O$-monoidal \(G\)-\(\infty\)-category $\C$
    \item[$P^{\leq \bullet}$] the (regular) slice tower
    \item[$P^s_s$] the slices of the (regular) slice tower
    
    \item[\(\myuline{\Pic}_G(R)\)] the \(G\)-space of invertible objects in \(\myuline{\LMod}^G_R(\C)\), see \cref{construction: Pic} 
    \item[$\ul{\Pic}_G(R)_{\downarrow A}$] the comma-$G$-$\infty$-category of $\ul{\Pic}_G(R)$ over $A \in \ul{\LMod}_R$, see \cref{construction: Pic}
    \item[$\ul{\PSh}_G^{\O}(X)^{\otimes}$] the presheaf $G$-$\infty$-category with Day convolution, see \cref{theorem: omnibus day convolution}
    \item[\(\Spaces_G, \Sc^G\)] the \(\infty\)-category of \(G\)-spaces
    \item[$\Seg_{\O}(\C)$] the $\infty$-category of $\O$-Segal objects in $\C$
%    \item[\(\myuline{\Spaces}_G, \ul{\Sc}^G\)] the \(G\)-\(\infty\)-category of \(G\)-spaces
    \item[\(\Sp_G, \Sp^G\)] the \(\infty\)-category of \(G\)-spectra
%    \item[\(\myuline{\Sp}_G, \ul{\Sp}^G\)] the \(G\)-\(\infty\)-category of \(G\)-spectra
    \item[\(\Span(\C)\)] the $\infty$-category of spans of an $\infty$-category $\C$ with pullbacks
    \item[\(\myuline{\St}\)] parameterized straightening functor, see \cref{prop: parametrized straightening}
    \item[\(\Th_G({-})\)] the Thom spectrum \(G\)-functor, see \cref{def: parametrized thom spectrum}, \cref{construction: multiplicative thom spectra}
    % \item[\(\Th_G^\otimes({-})\)] the monoidal \(G\)-equivariant Thom spectrum \(G\)-functor taking values in \(({\LMod}^G_R)^\otimes\)
%    \item[$\tmf_1(n)$] the topological modular forms spectrum with level structure; often equipped with a Borel equivariant structure
    \item[$\ul{\Un}$] parametrized unstraightening functor, see \cref{prop: parametrized straightening}
    \item[$\gamma$] the composite $\Omega^{\infty}\Sigma^2 \MU \to \Omega^{\infty} \Sigma^2 \ku \simeq \BU \to \Pic(\Sp)$ induced by the Conner--Floyd orientation $\MU \to \ku$, see \cref{rem: real wilson spaces}
    \item[$\pi_V^G$] the $V$-graded $G$-homotopy groups, given by $\pi^G_V = [S^V, -]_*^G$
    \item[\(\rho\)] the regular representation of \(C_2\)
    \item[\(\rho_G\)] the regular representation of \(G\)
    \item[\(\sigma\)] the sign representation of \(C_2\)
    \item[$\Psi^{\ell}$] the Adams operations on $\BU$ or $\MU$, see \cref{cor:adamsopsMUR}
    \item[$\chi(f)$] the characteristic of $f$, see \cref{def: characteristic of f}
    \item[$\yo$] the Yoneda embedding
    \item[$R \sslash_{\E_V} \alpha$] the $\E_V$-quotient of $R$ by $\alpha$, see \cref{def: EV quotient}
    \item[$\langle n \rangle$] the pointed finite set $\{1,2,\cdots,n \} \sqcup \{* \} \in \F_*$ 
    \item[$\bigotimes_{o \to o'}$] the induced functor $\C_o^{\otimes} \to \C_{o'}^{\otimes}$ over $o \to o'$ of an $\O$-monoidal $\infty$-category $\C^{\otimes} \to \O^{\otimes}$, see \cref{sec:param_operads}
    \item[$\underline{\bigotimes}_{o \to o'}$] the indexed tensor product, see \cref{def: indexed tensor product}
\end{description}

\newpage
        
\bibliographystyle{alpha}
\bibliography{main}

\Addresses
\end{document}